\documentclass[11pt,reqno]{amsart}
\usepackage[a4paper,margin=2.4cm]{geometry}
\usepackage{amsmath,amsthm,amsfonts,amssymb,latexsym,mathrsfs,extarrows}
\usepackage{xcolor}
\usepackage{tikz}
\usetikzlibrary{arrows.meta}

\newtheorem{Theorem}{Theorem}[section]
\newtheorem{Corollary}[Theorem]{Corollary}
\newtheorem{Proposition}[Theorem]{Proposition}

\newtheorem{Lemma}[Theorem]{Lemma}
\theoremstyle{definition}
\newtheorem{Definition}[Theorem]{Definition}
\newtheorem{Example}[Theorem]{Example}
\newtheorem{Problem}[Theorem]{Problem}
\theoremstyle{remark}
\newtheorem{Remark}[Theorem]{Remark}
\newtheorem{Observation}[Theorem]{Observation}

\numberwithin{equation}{section}

\allowdisplaybreaks

\DeclareMathOperator{\sgn}{sgn}
\newcommand{\QQ}{\mathbb{Q}}
\newcommand{\RR}{\mathbb{R}}
\newcommand{\ZZ}{\mathbb{Z}}
\newcommand{\EE}{\mathcal{E}}
\newcommand{\He}{\operatorname{He}}
\newcommand{\HH}{\ddot{H}}
\newcommand{\sn}{\operatorname{sn}}
\newcommand{\cn}{\operatorname{cn}}
\newcommand{\dn}{\operatorname{dn}}
\newcommand{\ord}{\operatorname{ord}}

\newcommand{\M}{\mathbf{M}}

\title{Dilated Hankel determinants}
\author{Guo-Niu Han}
\date{2026/07/09}


\begin{document}

\begin{abstract}
For a sequence $\mathbf a=(a_0,a_1,\dots)$ we define its \emph{dilated Hankel
determinant} $\HH_n(\mathbf a)=\det(a_{2i+j})_{0\le i,j\le n-1}$, the minor of the
infinite Hankel matrix $(a_{i+j})$ formed from the even-indexed rows and the
first $n$ columns.
We prove that, for a broad class of sequences, $\HH_n$ admits a remarkably
simple product evaluation. This mirrors the behaviour of the classical Hankel
determinant $H_n$, but with two key distinctions: the class of sequences for
which such formulas are known is far larger in the classical case; and, whereas
$H_n$ enjoys a single universal evaluation --- the Heilermann formula via the
Jacobi continued fraction --- no analogous general method exists for the dilated
determinant, which is therefore considerably more challenging.
Our evaluations instead rest on six methods developed here, four of
general scope and two of a more specialised nature.
The cases treated include the factorial numbers, the Catalan and
central binomial coefficients;
the Euler  numbers and a one-parameter secant
family; the involution numbers;
the Springer numbers along with elliptic and derivative
deformations; the reciprocal-sine function, whose
evaluation rests on a new Catalan determinant proved by condensation; a
Bessel analogue of the Euler numbers; and a
multiplicative Bessel family. As an
application, we
settle a conjecture of Chapoton and the author on the roots of the Poupard and
Kreweras polynomials.
\end{abstract}

\maketitle

\setcounter{tocdepth}{1}
\tableofcontents

\section{Introduction}\label{sec:intro}

Let $\mathbf a=(a_0,a_1,a_2,\dots)$ be a sequence over a field, and let
$f(x)=\sum_{n\ge0}a_n\,x^n/n!$ be its exponential generating function. The
\emph{(classical) Hankel determinant} of order $n$ is
\begin{equation}\label{eq:Hclass}
H_n=H_n(f)=H_n(\mathbf a)=\det\bigl(a_{i+j}\bigr)_{0\le i,j\le n-1}
=\begin{vmatrix}
a_0 & a_1 & a_2 & \cdots & a_{n-1}\\
a_1 & a_2 & a_3 & \cdots & a_{n}\\
a_2 & a_3 & a_4 & \cdots & a_{n+1}\\
\vdots & & & & \vdots\\
a_{n-1} & a_{n} & a_{n+1} & \cdots & a_{2n-2}
\end{vmatrix}.
\end{equation}
Hankel determinants are a meeting point of orthogonal polynomials, the
moment problem, continued fractions, total positivity and enumerative
combinatorics. When the $a_k$ are the moments of a linear functional $\mu$,
the determinants $H_n$ are the Gram determinants of the monomial basis; their
positivity governs the classical moment problem and the total positivity of the
Hankel matrix \cite{Stieltjes1894Re,Sokal2019,Petreolle2023Sokal}, and the
theorem of Heilermann and Stieltjes evaluates them through the Jacobi continued
fraction of the generating function
\cite{Heilermann1846,Wall1948,Viennot1983,Flajolet1980,Krattenthaler1998,Krattenthaler2005,GesselXin2006}.
Combinatorially, $H_n$ counts families of non-intersecting lattice paths
\cite{Lindstrom1973,Gessel1985Viennot,Viennot1983,Radoux1992}; Flajolet's
theory \cite{Flajolet1980} reads the continued fraction off weighted Motzkin
paths and gives, in one stroke, the continued-fraction expansions --- hence the
Hankel determinants --- of the Catalan, Bell, Stirling, Euler, tangent and
secant numbers, and of the elliptic functions. The Jacobi continued fraction,
however, exists only when all $H_n\neq0$; as soon as some $H_n$ vanish it breaks
down, and one turns instead to the \emph{Hankel} continued fraction
\cite{Han2016Adv,Han2015NT}, which is defined for every sequence.

A recurring theme of the subject is the contrast between sequences whose Hankel
determinants have a \emph{closed product form} --- a formula of the
$\mathrm{NICE}(n)$ type $\xi^{n}\,\mathrm{Rat}(n)\prod_i(a_in+b_i)!^{\pm1}$ in
Krattenthaler's sense \cite{Krattenthaler2005} --- and those whose determinants
carry sporadic large prime factors and admit no such formula. The first class is
recognised by its arithmetic, the prime factors of $H_n$ growing only linearly
in $n$; it is this ``niceness'' that combinatorialists prize and that the
determinant calculus \cite{Krattenthaler1998,Krattenthaler2005} is designed to
certify. Which sequences fall on which side is a subtle matter, and it is this
question that the present paper takes up --- for a Hankel-type determinant of a
new kind.

\begin{Definition}\label{def:double}
The \emph{dilated Hankel determinant of order $n$} of the sequence
$\mathbf a=(a_0,a_1,\dots)$ is
\begin{equation}\label{eq:Hdouble}
\HH_n=\HH_n(f)=\HH_n(\mathbf a)=\det\bigl(a_{2i+j}\bigr)_{0\le i,j\le n-1}
=\begin{vmatrix}
a_0 & a_1 & a_2 & \cdots & a_{n-1}\\
a_2 & a_3 & a_4 & \cdots & a_{n+1}\\
a_4 & a_5 & a_6 & \cdots & a_{n+3}\\
\vdots & & & & \vdots\\
a_{2n-2} & a_{2n-1} & a_{2n} & \cdots & a_{3n-3}
\end{vmatrix},
\end{equation}
with the convention $\HH_0=1$.
\end{Definition}

\noindent More generally, replacing the row index $i$ by $ri$ gives the
\emph{$r$-step Hankel determinant}
$$\det\bigl(a_{ri+j}\bigr)_{0\le i,j\le n-1};$$
the classical Hankel determinant \eqref{eq:Hclass} is the case $r=1$, and the
\emph{dilated} one the case $r=2$. This paper is devoted to $r=2$. We call it
the \emph{dilated Hankel determinant}, or simply the \emph{dilated
determinant}, throughout, and reserve the name \emph{$2$-step} for the rare
occasions when it must be set against the general $r$-step determinant
(Section~\ref{sec:vandermonde}).

Two readings of \eqref{eq:Hdouble} place it within known theory. First,
$\HH_n$ is the \emph{minor of the infinite Hankel matrix}
$(a_{p+q})_{p,q\ge0}$ that selects the even-indexed rows
$\{0,2,4,\dots,2n-2\}$ and the first $n$ columns; general minors of
Hankel matrices fall under the determinant calculus of Krattenthaler
\cite{Krattenthaler1998,Krattenthaler2005}. Second, if $a_k=\mu(z^k)$ are
moments then $a_{2i+j}=\mu(z^{2i}\cdot z^j)$, so
\begin{equation}\label{eq:gram}
\HH_n=\det\bigl(\langle z^{2i},z^{j}\rangle_\mu\bigr)_{0\le i,j\le n-1},
\qquad \langle p,q\rangle_\mu:=\mu(pq),
\end{equation}
the Gram determinant of the bilinear form $\langle\cdot,\cdot\rangle_\mu$ on
the two bases $\{1,z^2,z^4,\dots\}$ and $\{1,z,z^2,\dots\}$. The factor $2$
in $a_{2i+j}$ is thus the quadratic map $z\mapsto z^2$ together with the
even/odd splitting of $\mu$ --- the realm of the quadratic decomposition and
symmetrisation of orthogonal polynomials \cite{Chihara1978} and of
biorthogonal and multiple orthogonal polynomials
\cite{Nikishin1991Sorokin,Ismail2005}. The closely related but distinct
\emph{shifted} Hankel determinants $\det(a_{i+j+k})$ of \cite{Han2016Adv}
translate the index by a constant rather than scaling it; the interaction of the
two operations --- shifts of a dilated determinant --- is itself a source of
clean evaluations (Sections~\ref{sec:dshift}, \ref{sec:dblshift},
\ref{sec:algsqshift} and~\ref{sec:besseldshift}).

\medskip
\noindent\emph{Results.}
Our central finding is that, for a wide range of classical sequences,
$\HH_n$ has a surprisingly simple product formula even when the ordinary $H_n$
does not.
We organise the evaluations into a few families.
\begin{itemize}
\item The \emph{Beta family} $a_{m+1}/a_m=\rho\,(m+\alpha)/(m+\beta)$
(Section~\ref{sec:beta}) --- the sequences whose term ratio has degree less than or equal to $1$
--- has a single closed form (Theorem~\ref{thm:beta}) specialising to the
factorials, the Catalan numbers, the central binomial coefficients and the
double factorials. The resulting dilated Hankel determinants, such as the
Catalan values $1,3,32,1232,172032,\dots$ and the central binomial values
$1,8,224,22528,8200192,\dots$, appear not to have been recorded before.
\item The involution numbers and the Gaussian family $\exp(cx+bx^2)$ give
$\HH_n=(4bc)^{\binom n2}\prod_{k<n}k!$ (Theorem~\ref{thm:gauss},
Section~\ref{sec:gauss}); the Gaussian is the unique exponential weight for
which this simplification occurs.
\item The \emph{Euler number family}
$1/\cos(x)^{s+1}+\int_0^x dy/\cos(y)^{t+1}$
(Proposition~\ref{prop:family}, Section~\ref{sec:gen}), whose $(s,t)=(0,1)$
member has the Euler (secant--tangent) numbers \cite{Andre1879} as its moments,
has a closed product evaluation for every admissible $(s,t)$, even though the
relevant odd moment functional is not classical.
\item The \emph{secant-number family} $(1+x)/\cos(x)^{s+1}$, a separate family,
has the product evaluation $\HH_n=c_n\prod_{i=1}^{n-1}(s+1)_i$ valid for
\emph{every} $s$ (Theorem~\ref{thm:allstar}, Section~\ref{sec:allstar}), even
though its odd moment functional is not classical once $s\ge1$.
\item The \emph{Springer number family} $1/(\cos x-t\sin x)^r$
\cite{Springer1971,Sokal2019}, whose $t=1$, $r=1$ member has the Springer
numbers as its moments, a derivative $(\cos x+\sin x)/(\cos x-\sin x)^s$ of it,
and an elliptic Euler family all have product formulas
(Theorems~\ref{thm:springer}, \ref{thm:deriv}, \ref{thm:ell}).
\item An \emph{algebraic family} $(1+x)/(1-x^{2})^{s/2}$
(Theorem~\ref{conj:alg}, Section~\ref{sec:alg}) also has a product formula for
every $s$, even though \emph{neither} of its two moment functionals is
classical; the proof accordingly abandons orthogonal polynomials altogether in
favour of an elementary contiguous relation in $s$.
\item A \emph{rank-one perturbation} $(\sin x+1)/\cos^2x+s\sin x$ of the
Euler number family (Proposition~\ref{prop:sin3}, Section~\ref{sec:runkone})
factors for every $s$, with $\HH_n$ affine in $s$; at the member $s=-1$,
where it becomes $(1+\sin^3x)/\cos^2x$, the seemingly sporadic prime factors
$29,37,23$ at $n=8,9,10$ turn out to be $\binom n2+1$.
\item The \emph{reciprocal-sine function} $(1+x)\,x/\sin x$
(Theorem~\ref{conj:xsin-closed}, Section~\ref{sec:xsinx}) is evaluated as a
product of factorials --- the deepest evaluation of the paper. Its two moment
functionals turn out to be Wilson functionals differing in a \emph{single}
parameter; the connection coefficients of the biorthogonal reduction then
collapse to a single hypergeometric term, and the evaluation reduces to a
determinant of products of two Catalan numbers, established by
Desnanot--Jacobi condensation (Theorem~\ref{thm:xsin-catalan}) --- a
determinant evaluation of independent interest.
\item A \emph{Bessel $(s,t)$ family}
$\mathrm{cosb}_s(x)+\int_0^x\mathrm{cosb}_t(y)\,dy$
(Theorem~\ref{thm:besselst}, Section~\ref{sec:besselst}), built from the
normalised Bessel function $\mathrm{cosb}_s=\Gamma(s+1)(2/x)^sJ_s(x)$, has
a closed product evaluation for \emph{every} $(s,t)$: the signed factor of
the Euler number family reappears at doubled argument, so the determinants are
nonzero exactly on the half-integer offsets $t-s\in\tfrac12+\ZZ$ and
vanish eventually on the integer ones --- the Bessel ladder walks in half
steps.
\item A \emph{multiplicative Bessel family} $(1+x)\,\mathrm{cosb}_\nu^{\,2}$
(Theorem~\ref{thm:xbesseleven}, Section~\ref{sec:xbesseleven}), the Bessel
analogue of the secant-number family with the exponent ladder replaced by the
order ladder, has a closed product evaluation at every \emph{even} order
$n=2N$, complete with quarter-integer zeros --- the blocks
$(2\nu+\tfrac12)_k$ --- that have no trigonometric counterpart. Although
its even moment functional is not classical, Watson's product formula
makes the even moments a single hypergeometric term, and the evaluation
falls, unexpectedly, to the elementary Vandermonde reduction, whose only
other application is the Beta family; at odd orders the degree count fails
by one and an irreducible carrier appears.
\end{itemize}

\medskip
\noindent\emph{Origin and application.}
As is so often the way with determinant evaluations \cite{Krattenthaler1998},
the object of this paper first announced itself in the course of another problem.
While studying a conjecture of Chapoton and the author on the roots of the
Poupard and Kreweras polynomials \cite{Chapoton2020Han}, the author was confronted
with a determinant that did not yield to the ordinary continued-fraction
machinery and whose first row was the secant-type sequence
$1,1,1,3,5,25,61,427,\dots$ of $(1+x)/\cos x$; this clue is what led to the
dilated determinant studied here. We close the circle in Section~\ref{sec:root}:
Conjecture~5.4 of \cite{Chapoton2020Han} is, after a change of basis and up to an
explicit power of two, the dilated Hankel determinant of $(1+x)/\cos x$
(Theorem~\ref{thm:conj54double}), whose evaluation (Proposition~\ref{prop:sec1})
settles it.

\medskip
\noindent\emph{Methods.}
The classical Hankel determinant enjoys a \emph{universal} method of evaluation:
the Jacobi continued fraction of the generating function produces the whole
sequence $H_1,H_2,\dots$ at once, from a single list of coefficients
\cite{Heilermann1846,Stieltjes1894Re,Wall1948}. For the dilated determinant there
is no such tool. $\HH_n$ is in general a substantially harder
object: no one method evaluates every family. We therefore develop \emph{six}
methods for $\HH_n$, collected and abstracted in Section~\ref{sec:methods} and
tagged $\M1$--$\M6$ throughout the paper. Four of them run through the paper:
the elementary Vandermonde
(row-factorisation) reduction $\M1$; the biorthogonal
reduction $\M2$, which splits $\mathbf a$ into its even and odd moment
functionals --- the determinantal face of the even/odd contraction of continued
fractions \cite{Wall1948,Flajolet1980} --- and reduces $\HH_n$ to a determinant
of connection coefficients; the one-functional reduction $\M3$; and the divisor
method $\M4$, which evaluates a one-parameter family by its degenerations through
a dilated Cauchy--Binet expansion. Each of these four comes with a sharp rigidity
boundary marking where the product form must fail. The remaining two methods
are of a different, more specialised character, each used only once rather than
running through the paper: the algebraic family (Section~\ref{sec:alg}) is settled by a
contiguous relation ($\M5$), and the rank-one perturbation
(Section~\ref{sec:runkone}) by the matrix-determinant lemma ($\M6$). Finally,
the reciprocal-sine function of Section~\ref{sec:xsinx} is reached by $\M2$
followed by a new finishing step: the connection determinant turns out to be a
determinant of products of Catalan numbers, which we evaluate by
Desnanot--Jacobi condensation; the same finishing step closes the Bessel
$(s,t)$ family of Section~\ref{sec:besselst}, whose connection determinant
carries the parameters in its entries. In the opposite direction, the
multiplicative Bessel family of Section~\ref{sec:xbesseleven} is closed by
$\M1$ alone --- its even moments are a single hypergeometric term, and at
even orders the row factorisation lands exactly on the Vandermonde degree
bound: the one evaluation of the paper where the most elementary method
reaches beyond the Beta family. The classical background
needed --- orthogonal polynomials, the moment problem, $J$- and $S$-fractions and
their contractions --- is gathered, in self-contained algebraic form, in
Section~\ref{sec:prelim}.

\medskip
\noindent\emph{Organisation.}
Section~\ref{sec:prelim} fixes notation and recalls the classical toolbox;
Section~\ref{sec:methods} presents the six methods $\M1$--$\M6$ and closes
with a table matching them to families. The evaluations then proceed family
by family: the Beta family (Section~\ref{sec:beta}); the Gaussian family ---
containing the involution numbers --- and its derivative rule
(Sections~\ref{sec:gauss} and~\ref{sec:dergauss}); the Euler number family with
its double and single shifts (Sections~\ref{sec:gen}--\ref{ssec:t3}); the
secant-number family $(1+x)/\cos(x)^{s+1}$ with its shifts
(Sections~\ref{sec:gen-xcos}--\ref{sec:dblshift}); the rank-one perturbation of
the Euler number family (Section~\ref{sec:runkone}); the Springer number family
and a derivative of it (Sections~\ref{sec:springer}--\ref{sec:derivative}); the
reciprocal-sine case (Section~\ref{sec:xsinx}); the elliptic Euler family
(Section~\ref{sec:elliptic}); the algebraic and squared algebraic families with
their shifted determinants (Sections~\ref{sec:alg}--\ref{sec:algsqshift}); the
Bessel $(s,t)$ family with its double shift (Sections~\ref{sec:besselst}
and~\ref{sec:besseldshift}); and the multiplicative Bessel family at even
orders (Section~\ref{sec:xbesseleven}). Section~\ref{sec:root} settles the
Chapoton--Han conjecture, and Section~\ref{sec:coro} collects, family by
family, the explicit corollaries and their shifts. Section~\ref{sec:lgv}
then sketches the Lindstr\"om--Gessel--Viennot point of view --- a natural
method which, for the dilated determinant, yields no evaluation and which
no proof of this paper uses; it is recorded for orientation only.
Section~\ref{sec:hankel-classical} records the classical Hankel determinants
$H_n$ of several of the moment sequences studied here. Finally Section~\ref{sec:concl} gathers the
concluding remarks and open problems. All closed forms were discovered and
checked with \textsf{SageMath}.

\section{Preliminaries}\label{sec:prelim}

This section collects the classical material we use, in the
self-contained algebraic form suited to formal power series; no analysis
or positivity is needed. Standard references are
\cite{Szego1975,Chihara1978,Ismail2005} for orthogonal polynomials,
\cite{Wall1948,Flajolet1980,Viennot1983,Krattenthaler1998} for continued
fractions and Hankel determinants. A reader fluent in this material may
skip to the methods of Section~\ref{sec:methods} and refer back as
needed.

\medskip
\noindent\emph{Conventions.}
Throughout the paper, $m!!=m(m-2)(m-4)\cdots$, down to $1$ or $2$, is the
double factorial, with $0!!=(-1)!!=1$; empty products equal $1$;
$(\alpha)_m=\alpha(\alpha+1)\cdots(\alpha+m-1)$ is the Pochhammer symbol;
and $\binom{\alpha}{d}=0$ for $d<0$. For sequences given by a generating
function we use the \emph{exponential} convention
$f(x)=\sum_{n\ge0}a_nx^n/n!$, i.e.\ $a_n=n!\,[x^n]f$.

\subsection{Desnanot--Jacobi condensation}\label{ssec:condensation}

Several of the evaluations below are closed by induction on the order, through a
classical identity among the minors of a matrix. For an
$n\times n$ matrix $M$ write $M^{\,i}_{\,j}$ for the $(n-1)\times(n-1)$ minor
obtained by deleting row $i$ and column $j$, and $M^{\,i,i'}_{\,j,j'}$ for the
$(n-2)\times(n-2)$ minor obtained by deleting rows $i,i'$ and columns $j,j'$.
The \emph{Desnanot--Jacobi identity} --- equivalently the
\emph{Jacobi--Desnanot identity}, and, after Lewis Carroll, \emph{Dodgson
condensation} \cite{Krattenthaler1998} --- states, for $n\ge2$,
\begin{equation}\label{eq:condensation}
\det M\;\det M^{\,1,n}_{\,1,n}
=\det M^{\,1}_{\,1}\;\det M^{\,n}_{\,n}
-\det M^{\,1}_{\,n}\;\det M^{\,n}_{\,1}:
\end{equation}
the full determinant times its central minor equals a difference of products of
the four minors obtained by deleting one boundary row and one boundary column.
When a determinant belongs to a family stable under \eqref{eq:condensation} ---
so that all five minors are again members of the family --- the identity becomes
a two-step recurrence in the order, and matching it against a candidate product
proves a closed form; this is the finishing step used in
Sections~\ref{sec:xsinx} and~\ref{sec:besselst}.

\subsection{Moments, the Hankel determinant, and orthogonal
polynomials}\label{ssec:opbasics}

A \emph{linear functional} $\mathcal L$ on the polynomial ring $\QQ[y]$
(more generally over any field) is fixed by its values on the monomial
basis,
$$
\mu_k:=\mathcal L[y^k]\qquad(k\ge0),
$$
the \emph{moments} of $\mathcal L$. The \emph{Hankel determinant} of the
moment sequence is
$$
H_n:=\det\bigl(\mu_{i+j}\bigr)_{0\le i,j\le n-1}
=\det\bigl(\mathcal L[y^i\cdot y^j]\bigr)_{0\le i,j\le n-1},
\qquad H_0=1,
$$
i.e.\ the Gram determinant of $1,y,\dots,y^{n-1}$ under the symmetric
bilinear form $\langle p,q\rangle:=\mathcal L[pq]$. We assume $\mathcal L$
is \emph{quasi-definite}, $H_n\ne0$ for all $n$; this is exactly the
condition under which the Gram--Schmidt process never breaks down.

Under this hypothesis there is a unique sequence of \emph{monic
orthogonal polynomials} $P_0=1,P_1,P_2,\dots$ with $\deg P_n=n$ and
$$
\mathcal L[P_nP_m]=h_n\,\delta_{nm},\qquad h_n:=\mathcal L[P_n^2]\ne0 ,
$$
where $h_n$ is the \emph{squared norm} of $P_n$ (a formal one --- no positivity
is assumed, so $h_n$ need not be a square; we abbreviate it to \emph{norm}
throughout).
They are the Gram--Schmidt orthogonalisation of $1,y,y^2,\dots$, and have
the explicit determinantal form
$$
P_n(y)=\frac{1}{H_n}
\begin{vmatrix}
\mu_0 & \mu_1 & \cdots & \mu_n\\
\mu_1 & \mu_2 & \cdots & \mu_{n+1}\\
\vdots & & & \vdots\\
\mu_{n-1} & \mu_n & \cdots & \mu_{2n-1}\\
1 & y & \cdots & y^n
\end{vmatrix}.
$$
Expanding along the last row shows $P_n$ is monic of degree $n$, and
$\mathcal L[y^jP_n]=0$ for $j<n$ because the determinant then has a
repeated row; this is orthogonality. Taking $\mathcal L$ of $y^nP_n$ and
comparing minors gives the basic bridge between the norms and the Hankel
determinants,
\begin{equation}\label{eq:normHankel}
h_n=\frac{H_{n+1}}{H_n},
\qquad\text{equivalently}\qquad
H_n=\prod_{k=0}^{n-1}h_k\qquad(h_0=\mu_0).
\end{equation}
Thus \emph{evaluating a Hankel determinant is the same as computing the
norms $h_k$ of the orthogonal polynomials}, and the latter are read off a
continued fraction (Section~\ref{ssec:cf}).

\emph{Orthogonal expansion.} Since $\deg P_n=n$, the set
$\{P_0,\dots,P_n\}$ is a basis of $\QQ[y]_{\le n}$, the polynomials of
degree $\le n$. Hence every $R$ with $\deg R\le n$ has a unique expansion
$R=\sum_m r_mP_m$, and pairing against $P_m$ recovers the coefficient,
\begin{equation}\label{eq:fourier}
\mathcal L[R\,P_m]=r_m\,h_m,\qquad\text{i.e.}\qquad r_m=\frac{\mathcal L[R\,P_m]}{h_m}.
\end{equation}

\subsection{Change of family}\label{ssec:family}

Every reduction in this paper replaces the row and column families of a
moment matrix by monic families of the same degrees. That this leaves the
determinant unchanged is the following elementary but pivotal fact, which we
isolate once and cite throughout. 

\begin{Lemma}[change of family]\label{lem:family}
Let $\mathcal L$ be a linear functional on a space of polynomials, and let
$(R_i)_{0\le i<n}$ and $(C_j)_{0\le j<n}$ be two families of polynomials.
If $(\tilde R_i)$ and $(\tilde C_j)$ are \emph{unitriangular recombinations},
$$
\tilde R_i=R_i+\sum_{k<i}p_{ik}R_k,
\qquad
\tilde C_j=C_j+\sum_{l<j}q_{jl}C_l,
$$
then
$$
\det\bigl(\mathcal L[\tilde R_i\,\tilde C_j]\bigr)_{0\le i,j<n}
=\det\bigl(\mathcal L[R_i\,C_j]\bigr)_{0\le i,j<n}.
$$
In particular, if the matrix $\bigl(\mathcal L[\tilde R_i\tilde C_j]\bigr)$ is
triangular, then
$\det\bigl(\mathcal L[R_iC_j]\bigr)=\prod_{i=0}^{n-1}\mathcal L[\tilde R_i\,\tilde C_i]$.
\end{Lemma}

\noindent In every application the new families are monic of strictly
increasing degree---typically the orthogonal polynomials of the relevant
functional---hence unitriangular recombinations of the monomial (or
split-monomial) families they replace.

\subsection{The three-term recurrence and Favard's theorem}\label{ssec:ttr}

Because $yP_n$ has degree $n+1$, expansion~\eqref{eq:fourier} writes it in
the basis $P_0,\dots,P_{n+1}$,
$$
yP_n=\sum_{m=0}^{n+1}r_m\,P_m,
\qquad
r_m=\frac{\mathcal L[yP_n\,P_m]}{h_m}.
$$
For $m<n-1$ we have $\deg(yP_m)=m+1<n$, so
$\mathcal L[yP_n\,P_m]=\mathcal L[P_n\,yP_m]=0$ by orthogonality; hence
$r_m=0$ for all $m<n-1$. This leaves a \emph{three-term recurrence}
\begin{equation}\label{eq:ttr}
P_{n+1}=(y-c_n)\,P_n-\lambda_n\,P_{n-1},
\qquad P_0=1,\ P_{-1}=0,
\end{equation}
with
$$
c_n=\frac{\mathcal L[y\,P_n^2]}{h_n}\ \ (n\ge0),
\qquad
\lambda_n=\frac{h_n}{h_{n-1}}\ \ (n\ge1) .
$$
In particular $\lambda_n\ne0$, and telescoping $\lambda_n=h_n/h_{n-1}$ with
\eqref{eq:normHankel} gives the norms and the Hankel determinant directly
from the recurrence coefficients,
\begin{equation}\label{eq:HfromLambda}
h_n=h_0\prod_{k=1}^{n}\lambda_k,
\qquad
H_n=\prod_{k=0}^{n-1}h_k=\mu_0^{\,n}\prod_{k=1}^{n-1}\lambda_k^{\,n-k}.
\end{equation}
Conversely, \emph{Favard's theorem} states that any monic sequence defined
by a recurrence of the shape \eqref{eq:ttr} with all $\lambda_n\ne0$ is
orthogonal for a (unique up to scale) quasi-definite functional. So the
data $(c_n,\lambda_n)$ and the moment data $(\mu_k)$ determine each other.

\subsection{Continued fractions: the J-fraction and S-fraction}\label{ssec:cf}

The translation between moments and recurrence coefficients is most
compact as a continued fraction. The \emph{Jacobi continued fraction}
(J-fraction) of $\mathcal L$ is
\begin{equation}\label{eq:Jfrac}
\sum_{k\ge0}\mu_k\,x^k
=\cfrac{\mu_0}{1-c_0x-\cfrac{\lambda_1x^2}{1-c_1x-\cfrac{\lambda_2x^2}{1-\ddots}}},
\end{equation}
where $c_n,\lambda_n$ are \emph{exactly} the recurrence coefficients
of~\eqref{eq:ttr}; this is the theorem of Stieltjes
\cite{Wall1948,Flajolet1980}, and with
\eqref{eq:HfromLambda} it is the standard route to closed-form Hankel
determinants \cite{Heilermann1846}. Combinatorially $\mu_n$ is the generating polynomial of
weighted Motzkin paths of length $n$, a level step at height $k$ weighted
$c_k$ and a fall from height $k$ (paired with the matching rise) weighted
$\lambda_k$ \cite{Flajolet1980,Viennot1983}.

When $\mu_0=1$, the series often admits the narrower \emph{Stieltjes
continued fraction} (S-fraction):
\begin{equation}\label{eq:Sfrac}
\sum_{k\ge0}\mu_k\,x^k
=\cfrac{1}{1-\cfrac{b_1x}{1-\cfrac{b_2x}{1-\cfrac{b_3x}{1-\ddots}}}},
\end{equation}
with coefficients $b_1,b_2,\dots$; here $\mu_n$ counts weighted Dyck-type
paths. 

\subsection{The even contraction $S\to J$}\label{ssec:contraction}

The S-fraction \eqref{eq:Sfrac} and the J-fraction \eqref{eq:Jfrac} are
\emph{different} continued fractions --- one has only first-order terms
$b_jx$, the other both $c_nx$ and $\lambda_nx^2$ --- yet they can expand the
\emph{same} power series. Indeed a series given by an S-fraction also has a
J-fraction: the \emph{even contraction} of \eqref{eq:Sfrac} is, by definition,
the continued fraction whose sequence of convergents is the even-indexed
subsequence $\mathcal C_0,\mathcal C_2,\mathcal C_4,\dots$ of the convergents
$\mathcal C_0,\mathcal C_1,\mathcal C_2,\dots$ of
\eqref{eq:Sfrac}; it turns out to be a J-fraction. The $b_j$ are therefore not
themselves recurrence coefficients, but determine the $c_n,\lambda_n$ through
this contraction, which yields

$$
\sum_{k\ge0}\mu_k\,x^k
=\cfrac{1}{1-b_1 x-\cfrac{(b_1b_2)x^2}{1-(b_2+b_3)x-\cfrac{(b_3b_4)x^2}{1-\ddots}}},
$$

\begin{equation}\label{eq:evencontraction}
c_0=b_1,\qquad
c_n=b_{2n}+b_{2n+1}\ \ (n\ge1),\qquad
\lambda_n=b_{2n-1}\,b_{2n}\ \ (n\ge1),
\end{equation}
or uniformly $c_n=b_{2n}+b_{2n+1}$, $\lambda_n=b_{2n-1}b_{2n}$ with the
convention $b_0=0$ \cite{Wall1948,Flajolet1980}. Equation~\eqref{eq:HfromLambda}
then gives the norms in the clean telescoped form
\begin{equation}\label{eq:Snorm}
h_n=\prod_{k=1}^{n}\lambda_k=\prod_{k=1}^{n}b_{2k-1}b_{2k}=\prod_{j=1}^{2n}b_j,
\qquad
H_n=\prod_{i=0}^{n-1}\prod_{j=1}^{2i}b_j .
\end{equation}

\subsection{The odd contraction}\label{ssec:oddcontraction}
The \emph{odd contraction} of \eqref{eq:Sfrac} is, by definition, the continued
fraction whose sequence of convergents is the odd-indexed subsequence
$\mathcal C_1,\mathcal C_3,\mathcal C_5,\dots$ of the convergents
$\mathcal C_0,\mathcal C_1,\mathcal C_2,\dots$ of \eqref{eq:Sfrac};
it yields
$$
\sum_{k\ge0}\mu_k\,x^k
=1+\cfrac{b_1x}{1-(b_1+b_2) x-\cfrac{(b_2b_3)x^2}{1-(b_3+b_4)x-\cfrac{(b_4b_5)x^2}{1-\ddots}}},
$$
which is not a J-fraction. We observe, however, that the shifted moment sequence does have a J-fraction \cite{Wall1948,Flajolet1980},
$$
\sum_{k\ge0}\mu_{k+1}\,x^k
=\cfrac{b_1}{1-(b_1+b_2) x-\cfrac{(b_2b_3)x^2}{1-(b_3+b_4)x-\cfrac{(b_4b_5)x^2}{1-\ddots}}},
$$
with
\begin{equation}\label{eq:oddcontraction}
\mu_0'=b_1,\qquad
c_n'=b_{2n+1}+b_{2n+2}\ \ (n\ge0),\qquad
\lambda_n'=b_{2n}\,b_{2n+1}\ \ (n\ge1).
\end{equation}
The shifted moment sequence $(\mu_{k+1})_{k\ge0}$ is the \emph{Christoffel
transform} of $\mathcal L$ by the variable $y$: if $\mathcal L'[P]:=\mathcal
L[y\,P]$ denotes the functional obtained by multiplying $\mathcal L$ by $y$,
then $\mathcal L'[y^k]=\mathcal L[y^{k+1}]=\mu_{k+1}$. Thus the odd contraction
is exactly the passage from the $S$-fraction of $\mathcal L$ to the $J$-fraction
of its Christoffel transform $\mathcal L'$, and the squared norms of the latter
telescope to
\begin{equation}\label{eq:Snormodd}
h_n'=h_0'\prod_{k=1}^{n}\lambda_k'=b_1\prod_{k=1}^{n}b_{2k}b_{2k+1}=\prod_{j=1}^{2n+1}b_j,
\qquad
H_n'=\prod_{i=0}^{n-1}\prod_{j=1}^{2i+1}b_j ,
\end{equation}
where the factor $h_0'=\mu_0'=b_1$ must be kept since the Christoffel sequence is
not normalised ($\mu_0'=b_1\ne1$), unlike the even case \eqref{eq:Snorm}. The
Christoffel transform recurs in the shifted determinants of
Sections~\ref{sec:shift}, \ref{sec:cosshift}, \ref{sec:dblshift}
and~\ref{sec:besseldshift}, where the moment sequence is advanced by one index,
$a_k\mapsto a_{k+1}$.

\section{Six methods for the dilated Hankel determinant}\label{sec:methods}

Six methods for $\HH_n$ are developed in this paper. Four of them run through
it, listed here from the most elementary to the most structured; each evaluates
the dilated determinant under a
different structural hypothesis on $\mathbf a$, and each has a boundary
beyond which the product form fails; Subsection~\ref{ssec:methodtable}
collects, in tabular form, the hypotheses, the collapse conditions and the
applications of all the methods.
\begin{itemize}
\item[$\M1$.] The \emph{Vandermonde reduction} (Section~\ref{sec:vandermonde})
factors each row and uses no functional at all; it applies to the
sequences whose term ratio has degree $\le1$ (the Beta family) and, through
the dilation, to the even orders of the multiplicative Bessel family
(Section~\ref{sec:xbesseleven}), whose \emph{even} moments have a
hypergeometric term ratio of degree $(2,2)$.
\item[$\M2$.] The \emph{biorthogonal reduction} (Section~\ref{sec:biotho}) splits
$\mathbf a$ into its even and odd functionals $\mathcal S,\mathcal T$ and reduces
$\HH_n$ to a determinant of connection coefficients; it is exact whenever both are
quasi-definite, and collapses to a product when the connection coefficients are
themselves explicit --- notably when $\mathcal S$ and $\mathcal T$ belong to a
single classical family and differ in one parameter, so that the connection
array is a single hypergeometric term (see \emph{The collapse mechanism} at the
end of Section~\ref{sec:biotho}).
\item[$\M3$.] The \emph{one-functional reduction} (Section~\ref{ssec:onefunc}),
the mechanism of Cigler and Krattenthaler
\cite{Cigler2021Krattenthaler,Krattenthaler2023}, uses the orthogonal
polynomials of the single functional $\EE$ on $z$; the Gaussian is the only
Appell family for which it collapses.
\item[$\M4$.] The \emph{divisor method}, with the dilated Cauchy--Binet lemma
(Section~\ref{sec:divisor}), treats $\HH_n(f_\theta)$ as a polynomial in a
parameter and locates its zeros at the degenerations of $f_\theta$.
\end{itemize}
Methods $\M1$ and $\M4$ are elementary; $\M2$ and $\M3$ rest on the
orthogonal-polynomial preliminaries of Section~\ref{sec:prelim}. We tag results
by the method used throughout the paper.

The remaining two methods are of a more specialised character, each used only
once rather than running through the paper, and are tagged accordingly.
\begin{itemize}
\item[$\M5$.] The \emph{contiguous-relation method} (Section~\ref{sec:alg})
applies when neither moment functional is classical: it relates $\HH_n$ at a
parameter $s$ to $\HH_n$ at $s-2$ by an elementary column operation, solves the
resulting relation by a periodicity argument, and fixes the constant at a
single special value. No orthogonal polynomials are used.
\item[$\M6$.] The \emph{rank-one/matrix-determinant-lemma method}
(Section~\ref{sec:runkone}, Remark~\ref{rem:rankone}) applies to a rank-one
perturbation of a family already evaluated by another method: the perturbed
determinant is the unperturbed one times an explicit scalar, computed by the
matrix-determinant lemma rather than by the divisor method $\M4$.
\end{itemize}

\subsection{Row factorisation and the Vandermonde reduction
($\M1$)}\label{sec:vandermonde}

The most elementary reduction uses no functional, no positivity and no continued
fraction. It evaluates the \emph{generalised} Hankel determinant
$\det(a_{x_i+j})$ with \emph{arbitrary} row indices $x_0,x_1,\dots,x_{n-1}$ ---
the dilated determinant $\HH_n$ being the case $x_i=2i$, and the $r$-step
Hankel determinant $\det(a_{ri+j})$ the case $x_i=ri$ --- whenever the rows of the
matrix separate into a row-dependent prefactor times a polynomial in the row
index.

\begin{Lemma}[Vandermonde reduction]\label{lem:vdm}
Let $(a_m)$ be a sequence over a field $K$, let $n\ge1$, and suppose there are
a scalar function $x\mapsto w(x)$ and polynomials $Q_0,\dots,Q_{n-1}\in K[x]$
with $\deg Q_j\le n-1$ such that
\begin{equation}\label{eq:rowfact}
a_{x+j}=w(x)\,Q_j(x)\qquad(0\le j\le n-1)
\end{equation}
for every row index $x$ used below. Then, for any $x_0,\dots,x_{n-1}$,
\begin{equation}\label{eq:vdm}
\det\bigl(a_{x_i+j}\bigr)_{0\le i,j\le n-1}
=\det M\;\prod_{0\le i<j\le n-1}(x_j-x_i)\;\prod_{i=0}^{n-1}w(x_i),
\end{equation}
where $M=(m_{kj})_{0\le k,j\le n-1}$ is the coefficient matrix of the $Q_j$,
$Q_j(x)=\sum_{k=0}^{n-1}m_{kj}\,x^k$. The constant $\det M$ is independent of
the $x_i$; since \eqref{eq:vdm} is a polynomial identity in $x_0,\dots,x_{n-1}$,
it may be evaluated at any convenient values to determine $\det M$. In
particular, if $\deg Q_j=j$ with leading coefficient $\ell_j$ for each $j$, then
$M$ is triangular and $\det M=\prod_{j=0}^{n-1}\ell_j$.
\end{Lemma}

\begin{proof}
By \eqref{eq:rowfact} the $(i,j)$ entry is
$a_{x_i+j}=w(x_i)\sum_{k=0}^{n-1}m_{kj}\,x_i^{\,k}$, so the matrix factors as
$\bigl(a_{x_i+j}\bigr)=D\,V\,M$ with $D=\operatorname{diag}\bigl(w(x_i)\bigr)$,
$V=\bigl(x_i^{\,k}\bigr)_{0\le i,k\le n-1}$ and $M=(m_{kj})$. Taking
determinants and using the Vandermonde evaluation
$\det V=\prod_{i<j}(x_j-x_i)$ gives \eqref{eq:vdm}.
\end{proof}

The reduction \eqref{eq:vdm} is classical determinant calculus: it is the
mechanism behind Krattenthaler's determinant lemma
\cite[Lemma~2.2]{Krattenthaler1990}, \cite{Krattenthaler1998}.  What the
present section adds is the delimitation of its exact range --- the
term-ratio boundary of degree at most one \eqref{eq:degone} below --- and the
even-index escape that reaches the multiplicative Bessel family of
Section~\ref{sec:xbesseleven}.

The factorisation \eqref{eq:rowfact} is available precisely for the sequences
whose term ratio is rational of degree at most one. Indeed, if
$a_{m+1}/a_m=R(m)$ is rational with numerator and denominator of degrees $p$
and $q$, then $a_{x+j}/a_x=\prod_{l=0}^{j-1}R(x+l)$; clearing the common
denominator over the shifts $l=0,\dots,n-2$ leaves
$a_{x+j}=w(x)\,Q_j(x)$ with $\deg Q_j=p\,j+q\,(n-1-j)$, which is $\le n-1$ for
all $j<n$ iff $p\le1$ and $q\le1$, i.e.
\begin{equation}\label{eq:degone}
\frac{a_{m+1}}{a_m}=\rho\,\frac{m+\alpha}{m+\beta}.
\end{equation}
This is the three-parameter \emph{Beta family}: here
$w(x)=a_x/(x+\beta)_{n-1}$ and $Q_j(x)=\rho^j(x+\alpha)_j(x+\beta+j)_{n-1-j}$,
and Lemma~\ref{lem:vdm} yields the closed form, computed in
Section~\ref{sec:beta}, of which the factorial, Catalan and central binomial
determinants are specialisations; the triangular case $\deg Q_j=j$ is the
degeneration $q=0$ (denominator absent). Conversely, a term ratio of degree
$\ge2$ inflates some $\deg Q_j$ beyond $n-1$, the collapse onto the Vandermonde
determinant fails, and the determinants acquire large sporadic prime factors
(Section~\ref{sec:beta}). One escape from this boundary exists and is used
once: in the dilated determinant only \emph{even} row indices occur, so
\eqref{eq:rowfact} is needed only in the contracted row variable $i$
(with $x_i=i$), where a hypergeometric term ratio of degree $(2,2)$ on the
even-indexed half of $\mathbf a$ still produces polynomials of admissible
degree --- fitting the bound $n-1$ exactly when $n$ is even; this closes
the multiplicative Bessel family of Section~\ref{sec:xbesseleven}.

\subsection{The biorthogonal reduction ($\M2$)}\label{sec:biotho}

Let $\mathbf a=(a_n)_{n\ge0}$ be a sequence over a field, split into its even and
odd parts. We reduce $\HH_n=\det(a_{2i+j})_{0\le i,j\le n-1}$ to a determinant of
connection coefficients (Lemma~\ref{lem:bindetGene}) under the sole hypothesis
that its even and odd parts are quasi-definite moment sequences; no
continued-fraction assumption is needed at this stage, and the closed-form
evaluation follows once the data are specialised
(Proposition~\ref{prop:family}). Throughout write
$\bar n=\lceil n/2\rceil$ and $\underline n=\lfloor n/2\rfloor$, so $\bar n+\underline n=n$ and
$\bar n\in\{\underline n,\underline n+1\}$.

Define two linear functionals on $\QQ[y]$ by
$$
\mathcal S[y^k]=a_{2k},
\qquad
\mathcal T[y^k]=a_{2k+1}.
$$
Every polynomial $\phi(z)$ decomposes as $\phi(z)=G(z^2)+z\,K(z^2)$, and
the functional $\EE[z^m]=a_m$ satisfies
$\EE[\phi]=\mathcal S[G]+\mathcal T[K]$. The entries of our dilated matrix are
$a_{2i+j}=\EE[z^{2i}\cdot z^j]$; in other words, writing $y=z^2$,
\begin{equation}\label{eq:KmatrixGene}
\HH_n=\det\bigl(\EE[y^i\cdot \chi_j]\bigr)_{0\le i,j\le n-1},
\qquad
\chi_{2l}=y^l,\quad \chi_{2m+1}=z\,y^m .
\end{equation}

We assume throughout this section that $\mathcal S$ and $\mathcal T$ are
\emph{quasi-definite}, so each carries its monic orthogonal polynomials---$P_i$
for $\mathcal S$ and $Q_m$ for $\mathcal T$---with three-term recurrences
\begin{equation}\label{eq:STrec}
P_{i+1}=(y-c^S_i)\,P_i-\lambda^S_i\,P_{i-1},
\qquad
Q_{m+1}=(y-c^T_m)\,Q_m-\lambda^T_m\,Q_{m-1}
\qquad(P_0=Q_0=1),
\end{equation}
and nonzero squared norms $h^S_i=\mathcal S[P_i^2]$ and $h^T_m=\mathcal T[Q_m^2]$.
Equivalently, the generating functions $\sum_k a_{2k}x^k$ and
$\sum_k a_{2k+1}x^k$ have $J$-fractions with coefficients
$c^S_i,\lambda^S_i$ and $c^T_m,\lambda^T_m$ respectively. 

\begin{Lemma}[one-sided reduction]\label{lem:onesided}
For this reduction only the \emph{even} functional need be quasi-definite. Let
$P_0,P_1,\dots$ be the monic $\mathcal S$-orthogonal polynomials, with squared
norms $h^S_l=\mathcal S[P_l^2]$, and let $\mathcal T$ be an \emph{arbitrary}
linear functional. Then
\begin{equation}\label{eq:onesided}
\HH_n=(-1)^{\binom{\bar n}{2}}
\Bigl(\prod_{l=0}^{\bar n-1}h^S_l\Bigr)\,
\det\bigl(\mathcal T[P_{\bar n+r}\,y^{m}]\bigr)_{0\le r,m\le \underline n-1}.
\end{equation}
\end{Lemma}

\begin{proof}
Start from \eqref{eq:KmatrixGene}. Replace the row family $(y^i)_{i<n}$ by
$(P_i)_{i<n}$ and the even-column family $(y^l)_{l<\bar n}$ by $(P_l)_{l<\bar n}$ --- each
monic of the same degree, hence a unitriangular recombination, which leaves the
determinant unchanged by Lemma~\ref{lem:family}; the odd columns $(z\,y^m)_{m<\underline n}$ are left untouched. The
$(i,2l)$ entry becomes $\mathcal S[P_iP_l]=\delta_{il}\,h^S_l$ and the
$(i,2m+1)$ entry becomes $\mathcal T[P_i\,y^m]$. Reorder the columns so that the
$\bar n$ even indices $0,2,\dots,2\bar n-2$ precede the $\underline n$ odd indices
$1,3,\dots,2\underline n-1$; this permutation has $\sum_{l=0}^{\bar n-1}\min(l,\underline n)=\binom{\bar n}2$
inversions (as $l\le \bar n-1\le \underline n$), hence sign $(-1)^{\binom{\bar n}2}$. In the reordered
matrix the first $\bar n$ columns are diagonal on the rows $i<\bar n$ and vanish on the
rows $i\ge \bar n$, so a Laplace expansion along them contributes
$\prod_{l=0}^{\bar n-1}h^S_l$ and leaves the complementary rows $i=\bar n,\dots,n-1=\bar n+r$
against the odd columns, giving \eqref{eq:onesided}. No property of $\mathcal T$
was used.
\end{proof}

When $\mathcal T$ is also quasi-definite it carries its own monic orthogonal
polynomials $Q_m$, and since $\{Q_0,\dots,Q_i\}$ is a basis of $\QQ[y]_{\le i}$ we
may expand the $\mathcal S$-orthogonal $P_i$ in the $\mathcal T$-orthogonal basis,
\begin{equation}\label{eq:connexp}
P_i=\sum_{m=0}^{i}\kappa_{i,m}\,Q_m ;
\end{equation}
the $\kappa_{i,m}$ are the \emph{connection coefficients}, lower-triangular in
$(i,m)$ with $\kappa_{i,i}=1$ (both families monic). Applying the projection
\eqref{eq:fourier} for $\mathcal T$ isolates a single one,
\begin{equation}\label{eq:connproj}
\mathcal T[P_i\,Q_m]
=\sum_{m'}\kappa_{i,m'}\,\mathcal T[Q_{m'}Q_m]
=\kappa_{i,m}\,\mathcal T[Q_m^2]
=\kappa_{i,m}\,h^T_m,
\qquad
\kappa_{i,m}=\frac{\mathcal T[P_i\,Q_m]}{h^T_m}.
\end{equation}
One further column operation turns \eqref{eq:onesided} into a determinant of
these $\kappa_{i,m}$.

\begin{Lemma}[connection coefficients]\label{lem:connrec}
Write $P_i=\sum_{m=0}^{i}\kappa_{i,m}\,Q_m$ for the expansion of the
$\mathcal S$-orthogonal polynomials in the $\mathcal T$-orthogonal basis.
Then
$\kappa_{i,i}=1$, while $\kappa_{i,m}=0$
for $m>i$ (and we set $\kappa_{i,m}=0$ for $m<0$), and
\begin{equation}\label{eq:kapparec}
\kappa_{i+1,m}
=\kappa_{i,m-1}+(c^T_m-c^S_i)\,\kappa_{i,m}
+\lambda^T_{m+1}\,\kappa_{i,m+1}-\lambda^S_i\,\kappa_{i-1,m}.
\end{equation}
\end{Lemma}

\begin{proof}
Since $\deg P_i=i$ and $\{Q_0,\ldots,Q_i\}$ is a basis of the polynomials of
degree $\le i$, the expansion $P_i=\sum_{m\le i}\kappa_{i,m}Q_m$ exists; comparing
leading coefficients (both monic) gives $\kappa_{i,i}=1$, while $\kappa_{i,m}=0$
for $m>i$ (and we set $\kappa_{i,m}=0$ for $m<0$).

For the recurrence, substitute $P_i=\sum_m\kappa_{i,m}Q_m$ into
$P_{i+1}=(y-c^S_i)P_i-\lambda^S_iP_{i-1}$ from \eqref{eq:STrec}, and expand the
$y Q_m$ that appears using the $Q$-recurrence
$y\,Q_m=Q_{m+1}+c^T_m Q_m+\lambda^T_m Q_{m-1}$:
\begin{align*}
\sum_m\kappa_{i+1,m}\,Q_m
=\sum_m\kappa_{i,m}\bigl(Q_{m+1}+c^T_m Q_m+\lambda^T_m Q_{m-1}\bigr)
 -c^S_i\sum_m\kappa_{i,m}\,Q_m-\lambda^S_i\sum_m\kappa_{i-1,m}\,Q_m .
\end{align*}
Now read off the coefficient of $Q_m$ on both sides: the term $Q_{m+1}$
contributes through $\kappa_{i,m-1}$, and the term $\lambda^T_m Q_{m-1}$
contributes (at index $m{+}1$) through $\lambda^T_{m+1}\kappa_{i,m+1}$. This gives
\eqref{eq:kapparec},
$$
\kappa_{i+1,m}
=\kappa_{i,m-1}+(c^T_m-c^S_i)\,\kappa_{i,m}
+\lambda^T_{m+1}\,\kappa_{i,m+1}-\lambda^S_i\,\kappa_{i-1,m},
$$
which together with $\kappa_{0,0}=1$ determines all $\kappa_{i,m}$.
\end{proof}

\begin{Lemma}[determinant reduction]\label{lem:bindetGene}
With the squared norms $h^S_l=\mathcal S[P_l^2]$ and
$h^T_m=\mathcal T[Q_m^2]$,
$$
\HH_n=(-1)^{\binom{\bar n}{2}}
\Bigl(\prod_{l=0}^{\bar n-1}h^S_l\Bigr)
\Bigl(\prod_{m=0}^{\underline n-1}h^T_m\Bigr)
\det\bigl(\kappa_{\bar n+r,\,m}\bigr)_{0\le r,m\le \underline n-1}.
$$
\end{Lemma}

\begin{proof}
Since $\mathcal T$ is now quasi-definite, carry out one further column operation
in the one-sided reduction \eqref{eq:onesided}: expand each odd-column monomial
in the $\mathcal T$-orthogonal basis, $y^m=Q_m+(\text{lower-degree }Q\text{'s})$,
a unitriangular recombination of the odd columns that leaves the determinant
unchanged (Lemma~\ref{lem:family}) and turns the entry $\mathcal T[P_{\bar n+r}\,y^m]$ into
$\mathcal T[P_{\bar n+r}\,Q_m]$. By the projection~\eqref{eq:connproj} this equals
$\kappa_{\bar n+r,m}\,h^T_m$, so
$$
\HH_n=(-1)^{\binom{\bar n}{2}}
\Bigl(\prod_{l=0}^{\bar n-1}h^S_l\Bigr)\cdot
\det\Bigl(\kappa_{\bar n+r,\,m}\,h^T_m\Bigr)_{0\le r,m\le \underline n-1}.
$$
Pulling the column factor $h^T_m$ out of column $m$ gives the claim.
\end{proof}

\medskip
\noindent\emph{Specialisation: even and odd $S$-fractions.}\label{sec:sfrac}
The general reduction $\M2$ (Lemma~\ref{lem:bindetGene}) needs only the recurrence
and norm data $(c^S_i,\lambda^S_i,h^S_i)$ and $(c^T_m,\lambda^T_m,h^T_m)$. These
become fully explicit, and the norms telescope into products, as soon as the
even and odd parts of $\mathbf a$ admit Stieltjes $S$-fractions:
\begin{equation}\label{eq:Fsfrac}
\sum_{k\ge 0}a_{2k}x^k
=\cfrac{1}{1-\cfrac{u_1\,x}{1-\cfrac{u_2\,x}{1-\ddots}}},
\qquad
\sum_{k\ge 0}a_{2k+1}x^k
=\cfrac{1}{1-\cfrac{v_1\,x}{1-\cfrac{v_2\,x}{1-\ddots}}}.
\end{equation}
The passage from these $S$-fractions to the recurrence coefficients and norms of
Section~\ref{sec:biotho} is the even contraction (Section~\ref{ssec:contraction}).

\begin{Lemma}[continued fractions, recurrences, norms]\label{lem:cfGene}
The monic orthogonal polynomials $P_i$ of $\mathcal S$ and $Q_m$ of
$\mathcal T$ satisfy the three-term recurrences
$$
P_{i+1}=(y-c^S_i)\,P_i-\lambda^S_i\,P_{i-1},
\qquad
Q_{m+1}=(y-c^T_m)\,Q_m-\lambda^T_m\,Q_{m-1}
\qquad(P_0=Q_0=1),
$$
where
$$
c^S_i=u_{2i}+u_{2i+1},\quad
\lambda^S_i=u_{2i-1}u_{2i},\qquad
c^T_m=v_{2m}+v_{2m+1},\quad
\lambda^T_m=v_{2m-1}v_{2m},
$$
and orthogonality and norms
$$
\mathcal S[P_iP_l]=\delta_{il}\,\prod_{j=1}^{2i} u_j,
\qquad
\mathcal T[Q_mQ_{m'}]=\delta_{mm'}\,\prod_{j=1}^{2m} v_j.
$$
\end{Lemma}

\begin{proof}
The two series in \eqref{eq:Fsfrac} are Stieltjes $S$-fractions with
coefficient sequences $(u_1,u_2,\dots)$ and $(v_1,v_2,\dots)$. The passage
from an $S$-fraction with coefficients $(b_1,b_2,\dots)$ to the monic
three-term recurrence of the associated orthogonal polynomials is the
\emph{even contraction} of Section~\ref{ssec:contraction}, which gives
$c_i=b_{2i}+b_{2i+1}$, $\lambda_i=b_{2i-1}b_{2i}$ (with $b_0=0$) and squared
norms $\prod_{k=1}^{2i}b_k$. Taking $b=u$ for $\mathcal S$ and $b=v$ for
$\mathcal T$ yields the stated coefficients and norms.
\end{proof}

\medskip
\noindent\emph{The collapse mechanism.}
Lemma~\ref{lem:bindetGene} is exact, but it leaves the determinant of the
connection array to be evaluated, and for a generic pair of classical
functionals this array is a balanced ${}_3F_2$-type \emph{double sum}, out of
reach of the determinant calculus. In every case where $\M2$ closes in this
paper, one and the same two-stage mechanism is at work. First, $\mathcal S$
and $\mathcal T$ lie in a \emph{single} classical family and differ in a
\emph{single} parameter; for such a pair the connection coefficients collapse
from a double sum to one hypergeometric term, which splits into a row factor,
a column factor, and a kernel depending only on $i-m$ and (possibly) $i+m$.
Second, the resulting kernel determinant --- which carries the shift $\bar n$
of Lemma~\ref{lem:bindetGene}, hence free integer parameters --- is evaluated
in closed form. The paper contains three instances. For the Euler $(s,t)$
family the two functionals differ in the exponent parameter, the kernel is
the binomial coefficient $\binom{(t-s)/2}{i-m}$, and the shifted binomial
determinant is evaluated by the dual Jacobi--Trudi identity
(Proposition~\ref{prop:family}). For the reciprocal-sine function the two
functionals are Wilson functionals differing in one Wilson parameter, the
kernel is the product of Catalan numbers $C_{i-m-1}\,C_{i+m}$, and the
shifted Catalan determinant is evaluated by Desnanot--Jacobi condensation
(Theorem~\ref{thm:xsin-catalan}). For the Bessel $(s,t)$ family the two
functionals are Jacobi--Gegenbauer functionals differing in one parameter,
the kernel is
$\binom{t-s}{i-m}\big/\bigl((t+2m+1)_{i-m}(s+i+m)_{i-m}\bigr)$ ---
binomial in $i-m$ but carrying the parameters through $i+m$ --- and the
shifted connection determinant is again evaluated by condensation
(Theorem~\ref{thm:bessel-det}). When the two classical functionals do
\emph{not} form such a one-parameter pair, the array remains a genuine double
sum and no product formula is to be expected.

\subsection{A one-functional reduction ($\M3$)}\label{ssec:onefunc}

The biorthogonal reduction of Section~\ref{sec:biotho} splits $\EE$ into its
even and odd parts $\mathcal S,\mathcal T$ and passes to the contracted
variable $y=z^2$. A parallel reduction keeps the single functional
$\EE[z^k]=a_k$ on $z$ and uses \emph{its own} orthogonal polynomials. It needs
$\EE$ itself to be quasi-definite---a hypothesis independent of the
quasi-definiteness of $\mathcal S$ and $\mathcal T$---and it underlies the
Hermite evaluation of the Gaussian family in Section~\ref{sec:gauss}.

Assume $\EE$ is quasi-definite, with monic orthogonal polynomials
$(\pi_m)_{m\ge0}$. By
\eqref{eq:Hdouble}, $\HH_n=\det\bigl(\EE[z^{2i}\cdot z^j]\bigr)_{0\le i,j\le n-1}$.

The reduction collapses to a product as soon as, alongside the columns, the
rows too can be replaced by a one-sidedly orthogonal family. Suppose there are
monic polynomials $\Phi_i$ of degree $2i$, each a unitriangular recombination
of the even monomials $1,z^2,\dots,z^{2i}$, orthogonal to every polynomial of
degree $<i$:
$$
\EE[\Phi_i\,\pi_j]=0\qquad(j<i).
$$
Such a $\Phi_i$ exists and is unique whenever the corresponding Gram minor is
nonzero.

\begin{Lemma}[one-functional reduction]\label{lem:onefunc}
If the even-orthogonal family $(\Phi_i)_{i<n}$ above exists, then
\begin{equation}\label{eq:onefuncprod}
\HH_n=\prod_{i=0}^{n-1}\EE[\Phi_i\,\pi_i].
\end{equation}
\end{Lemma}

\begin{proof}
Replace the column family $(z^j)_{j<n}$ by the orthogonal family
$(\pi_j)_{j<n}$ and the row family $(z^{2i})_{i<n}$ by $(\Phi_i)_{i<n}$; both are
unitriangular recombinations, so Lemma~\ref{lem:family} gives
$$\HH_n=\det\bigl(\EE[\Phi_i\,\pi_j]\bigr)_{0\le i,j<n}.$$
By orthogonality $\EE[\Phi_i\,\pi_j]=0$ for $j<i$, so this matrix is upper
triangular and its determinant is the product of the diagonal entries
$\EE[\Phi_i\,\pi_i]$.
\end{proof}

Equation~\eqref{eq:onefuncprod} is the one-functional counterpart of
Lemma~\ref{lem:bindetGene}. Section~\ref{sec:gauss} carries it out for the
Gaussian functional, whose orthogonal polynomials are the Hermite polynomials
and whose $\Phi_i$ has a closed umbral form.
The passage from a Hankel determinant to a triangular product built from the
orthogonal polynomials of a single functional is precisely the mechanism
developed by Cigler and Krattenthaler
\cite{Cigler2021Krattenthaler,Krattenthaler2023}, who reduce the Hankel
determinant of a linear combination of moments of orthogonal polynomials to a
smaller determinant of those polynomials and, as we do throughout, close
several of their evaluations by Desnanot--Jacobi (Dodgson) condensation.

\subsection{The divisor method and the dilated Cauchy--Binet lemma
($\M4$)}\label{sec:divisor}

The reductions $\M1$--$\M3$ evaluate $\HH_n$ from the moment data directly. The
divisor method ($\M4$) instead applies to a \emph{family} $f_\theta$ depending on
a parameter $\theta$: one treats
$\HH_n(f_\theta)$ as a polynomial in $\theta$, bounds its degree, and locates
its zeros at the \emph{degenerate} values of $\theta$ where $f_\theta$ collapses
to a geometric/exponential function or to an even function. At such a value the
moments form an exponential sum, and the dilated determinant is evaluated, to
leading order, by a single Cauchy--Binet expansion into two Vandermonde
determinants. Three general facts drive the method; the case-specific inputs (the
degree bound and the precise orders) are supplied in each application
(Sections~\ref{sec:springer}, \ref{sec:derivative} and~\ref{sec:elliptic}).

The first is a homogeneity, used to rescale and to take limits.

\begin{Lemma}[homogeneity]\label{lem:scale}
For scalars $\beta,c$,\;
$\HH_n\bigl(f(\beta x)\bigr)=\beta^{3\binom n2}\HH_n(f)$ and
$\HH_n\bigl(c\,f\bigr)=c^{\,n}\HH_n(f)$.
\end{Lemma}

\begin{proof}
The coefficients of $f(\beta x)$ are $\beta^na_n$, so the $(i,j)$ entry
$a_{2i+j}$ is scaled by $\beta^{2i+j}$; pulling $\beta^{2i}$ from row $i$ and
$\beta^{j}$ from column $j$ multiplies the determinant by
$\beta^{\sum_i2i+\sum_jj}=\beta^{2\binom n2+\binom n2}$. The factor $c$
multiplies every entry, hence the determinant by $c^n$.
\end{proof}

The second is the bookkeeping principle that turns ``enough vanishing'' into an
exact product.

\begin{Observation}[divisor principle]\label{obs:divisor}
Let $D(\theta)$ be a polynomial of degree $\le N$, not identically zero, that
vanishes to order $\ge d_k$ at distinct points $\theta_1,\dots,\theta_s$. If
$\sum_kd_k\ge N$, then $\sum_kd_k=N$, every order is exact, and
$D(\theta)=c\prod_{k=1}^s(\theta-\theta_k)^{d_k}$ with $c\ne0$ fixed by the
value of $D$ at any further point.
\end{Observation}

The third, and the engine of the method, expands a dilated determinant whose
moments are an exponential sum.

\begin{Lemma}[dilated Cauchy--Binet]\label{lem:cb}
Let $a_m=\sum_{\nu\in I}c_\nu\,\zeta_\nu^{\,m}$ be a (formal or convergent)
exponential sum over an index set $I$. Then
$$
\HH_n(a)=\sum_{T}\Bigl(\prod_{\nu\in T}c_\nu\Bigr)\,
\det\bigl(\zeta_\nu^{2i}\bigr)_{0\le i<n,\,\nu\in T}\;
\det\bigl(\zeta_\nu^{j}\bigr)_{0\le j<n,\,\nu\in T},
$$
the sum over the $n$-subsets $T\subseteq I$; each term is a product of two
Vandermonde determinants in the nodes $\{\zeta_\nu:\nu\in T\}$, the first in
their squares $\zeta_\nu^2$. If moreover the index set is symmetric, $-I=I$ with
$\zeta_{-\nu}=-\zeta_\nu$ (so $\zeta_{-\nu}^2=\zeta_\nu^2$), then only the
$n$-subsets on which $\nu\mapsto|\nu|$ is injective survive: a subset containing
a pair $\{\nu,-\nu\}$ makes the first determinant have two equal columns.
\end{Lemma}

\begin{proof}
Writing $a_{2i+j}=\sum_\nu c_\nu\zeta_\nu^{2i}\zeta_\nu^{\,j}$, we factor the
$n\times n$ matrix as $X\,\mathrm{diag}(c_\nu)\,Y^{\!\top}$ with
$X=(\zeta_\nu^{2i})_{0\le i<n,\,\nu\in I}$ and
$Y=(\zeta_\nu^{\,j})_{0\le j<n,\,\nu\in I}$. Here $X$ and $Y$ are
\emph{rectangular} $n\times|I|$ matrices --- one column per node, and $|I|$
(possibly infinite) need not equal $n$ --- so this is the general (non-square)
Cauchy--Binet formula, which gives the displayed sum over all $n$-subsets
$T\subseteq I$. In the symmetric case $\zeta_{-\nu}^{2i}=\zeta_\nu^{2i}$,
so on a subset containing both $\nu$ and $-\nu$ the matrix $(\zeta_\nu^{2i})$
has two equal columns and the first determinant vanishes; on the survivors
$\nu\mapsto\zeta_\nu^2$ is injective and the first determinant is a genuine
Vandermonde.
\end{proof}

The first determinant is a Vandermonde in the \emph{squared} nodes
$\zeta_\nu^2$, and this squaring is the algebraic imprint of the dilation.
Indeed the even rows enter only through $\zeta_\nu^{2i}=(\zeta_\nu^2)^i$, which
depends on $\zeta_\nu$ only through $\zeta_\nu^2$; hence these rows cannot tell a
node $\zeta_\nu$ apart from its negative $-\zeta_\nu$ (the two share the same
square). This is exactly why, in the symmetric case above, a node and its
negative collapse the term.

We use Lemma~\ref{lem:cb} in the next sections in the following way. When the
exponential sum depends on a parameter and the nodes and coefficients
degenerate as the parameter tends to a special value, the same expansion lets
us read off the \emph{order} to which $\HH_n$ vanishes there, and its leading
coefficient. Suppose then that the exponential sum of Lemma~\ref{lem:cb}
depends on a parameter $\epsilon$,
\begin{equation}\label{eq:cbeps}
a_m(\epsilon)=\sum_{\nu\in I}c_\nu(\epsilon)\,\zeta_\nu(\epsilon)^{\,m},
\end{equation}
each coefficient and each node being a (formal or convergent) power series in
$\epsilon$; write $\bar\zeta_\nu:=\zeta_\nu(0)$ for the limiting nodes.
Lemma~\ref{lem:cb} expands $\HH_n(a(\epsilon))$ termwise,
\begin{equation}\label{eq:cbepsterm}
\HH_n\bigl(a(\epsilon)\bigr)=\sum_{T}
\Bigl(\prod_{\nu\in T}c_\nu(\epsilon)\Bigr)\,
\det\bigl(\zeta_\nu(\epsilon)^{2i}\bigr)_{0\le i<n,\,\nu\in T}\,
\det\bigl(\zeta_\nu(\epsilon)^{j}\bigr)_{0\le j<n,\,\nu\in T},
\end{equation}
the sum over $n$-subsets $T\subseteq I$; call $T$ \emph{surviving} if its
term is not identically zero --- in the symmetric case of Lemma~\ref{lem:cb}
exactly the subsets on which $\nu\mapsto|\nu|$ is injective. Given integers
$o_\nu\ge0$ with $\ord_\epsilon c_\nu(\epsilon)\ge o_\nu$ for every
$\nu\in I$ (and, if $I$ is infinite, $\{\nu:o_\nu\le N\}$ finite for every
$N$, so that each power of $\epsilon$ in \eqref{eq:cbepsterm} receives only
finitely many contributions), put
$$
M:=\min_{T\ \mathrm{surviving}}\ \sum_{\nu\in T}o_\nu .
$$

\begin{Corollary}[order at a degeneration]\label{cor:cborder}
With the data and notation above:
\begin{enumerate}
\item[\rm(i)] $\ord_\epsilon\HH_n\bigl(a(\epsilon)\bigr)\ge M$.
\item[\rm(ii)] If moreover a \emph{unique} surviving $T^\ast$ attains $M$,
the orders on $T^\ast$ are \emph{exact} ---
$c_\nu(\epsilon)=\bar c_\nu\,\epsilon^{\,o_\nu}+O(\epsilon^{\,o_\nu+1})$
with $\bar c_\nu\ne0$, for $\nu\in T^\ast$ --- and the limiting squared
nodes $\bar\zeta_\nu^{\,2}$ $(\nu\in T^\ast)$ are pairwise distinct, then
$$
\HH_n\bigl(a(\epsilon)\bigr)=C\,\epsilon^{M}+O(\epsilon^{M+1}),
\qquad
C=\Bigl(\prod_{\nu\in T^\ast}\bar c_\nu\Bigr)\,
\det\bigl(\bar\zeta_\nu^{\,2i}\bigr)_{0\le i<n,\,\nu\in T^\ast}\,
\det\bigl(\bar\zeta_\nu^{\,j}\bigr)_{0\le j<n,\,\nu\in T^\ast}
\ \ne\ 0,
$$
so that $\ord_\epsilon\HH_n(a(\epsilon))=M$ exactly.
\end{enumerate}
\end{Corollary}

\begin{proof}
(i) Every surviving term of \eqref{eq:cbepsterm} has order
$\ge\sum_{\nu\in T}o_\nu\ge M$: the coefficient product has order
$\ge\sum_{\nu\in T}o_\nu$, and the two determinants, being polynomials in
the $\zeta_\nu(\epsilon)$, have order $\ge0$. Hence so has the sum.

(ii) By (i) it remains to read off the coefficient of $\epsilon^M$. Since
$T^\ast$ is the unique minimiser, every other surviving term is
$O(\epsilon^{M+1})$, so only $T^\ast$ contributes; taking the leading part
$\bar c_\nu\epsilon^{\,o_\nu}$ of each coefficient and setting $\epsilon=0$
in the two determinants gives $C$. It is nonzero: the first determinant is
a Vandermonde in the pairwise-distinct $\bar\zeta_\nu^{\,2}$, hence the
$\bar\zeta_\nu$ are themselves pairwise distinct and the second
determinant, a Vandermonde in the $\bar\zeta_\nu$, is nonzero as well; and
$\prod_{\nu\in T^\ast}\bar c_\nu\ne0$ because the orders on $T^\ast$ are
exact.
\end{proof}

The corollary above governs the \emph{distinct}-node case, where the limiting
nodes $\bar\zeta_\nu$ stay separate. The opposite extreme is \emph{confluent}:
the nodes merge to a single point, and a Vandermonde degenerates into a
Wronskian --- the determinant of a column together with its successive
derivatives. This is what happens when an \emph{even} generating function is
translated by a parameter $\phi$, that is $g_\phi(x)=g(x+\phi)$: as the next
lemma shows, each column of the dilated matrix of $g_\phi$ is the
$\phi$-derivative of the one before it, so $\HH_n(g_\phi)$ is literally a
Wronskian in $\phi$.

\begin{Lemma}[confluent case: a translated even series]\label{lem:wron}
Let $g(x)=\sum_{m\ge0}\mu_mx^m/m!$ be \emph{even} ($\mu_{2k+1}=0$), write
$\nu_k:=\mu_{2k}$ and $g_\phi(x):=g(x+\phi)$, with coefficients
$\beta_m(\phi)=\sum_{k\ge0}\tfrac{\phi^k}{k!}\mu_{m+k}$. Then
$$
\HH_n(g_\phi)=\det\bigl(\beta_{2i+j}(\phi)\bigr)_{0\le i,j<n}
=W\bigl(\beta_0,\beta_2,\dots,\beta_{2n-2}\bigr)(\phi),
$$
the Wronskian of $n$ even functions; consequently
$\ord_{\phi=0}\HH_n(g_\phi)\ge\binom n2$.
\end{Lemma}

\begin{proof}
Since $\partial_\phi\beta_m=\beta_{m+1}$, column $j$ is the $j$-th derivative of
column $0$, so $\beta_{2i+j}=\beta_{2i}^{(j)}$ and the determinant is the
Wronskian of the rows $\beta_{2i}(\phi)$. As $g$ is even,
$\beta_{2i}(\phi)=\sum_{l\ge0}\nu_{i+l}\,\phi^{2l}/(2l)!$, so by the
Cauchy--Binet identity for Wronskians
$$
W(\beta_0,\dots,\beta_{2n-2})
=\sum_{0\le l_0<\dots<l_{n-1}}\det(\nu_{i+l_a})_{i,a}\;
W\bigl(\tfrac{\phi^{2l_0}}{(2l_0)!},\dots,\tfrac{\phi^{2l_{n-1}}}{(2l_{n-1})!}\bigr),
$$
and the Wronskian of the $\phi^{2l_a}/(2l_a)!$ is a multiple of
$\phi^{\,2\sum_al_a-\binom n2}$, of order $\ge2\binom n2-\binom n2=\binom n2$.
\end{proof}

Finally, the same squared-node collapse explains, in its most degenerate form,
why a single conjugate pair of frequencies perturbs a dilated determinant only
by \emph{rank one}.

\begin{Remark}\label{rem:rankone}
A two-element symmetric set $I=\{\pm1\}$ ($\zeta_{-1}=-\zeta_1$) gives, by the
same identity $\zeta_{-1}^{2i}=\zeta_1^{2i}$, a matrix of \emph{rank one}:
$a_{2i+j}=\zeta_1^{2i}\bigl(c_1\zeta_1^{\,j}+c_{-1}\zeta_{-1}^{\,j}\bigr)$. This
is the structure behind the rank-one perturbation of Section~\ref{sec:runkone}
(there $\zeta_{\pm1}=\pm i$, the perturbing term being $\sin x$), evaluated by
the matrix-determinant lemma ($\M6$) rather than by the divisor method $\M4$.
\end{Remark}

\subsection{Which method for which family}\label{ssec:methodtable}

Table~\ref{tab:methods} summarises the six methods: the structural input
each one requires, the condition under which the reduction collapses to a
closed product, and where in the paper it is applied.

\begin{table}[htbp]
\centering\small
\renewcommand{\arraystretch}{1.4}
\begin{tabular}{@{}c p{0.28\linewidth} p{0.33\linewidth} p{0.26\linewidth}@{}}
\hline
 & input & collapse condition & applications\\
\hline
$\M1$ & term ratio $a_{m+1}/a_m$ rational of degree $\le1$ &
always: Vandermonde $\times$ triangular coefficient matrix
(Lemma~\ref{lem:vdm}); fails from degree $2$ on, except through the
dilation (\S\ref{sec:xbesseleven}) &
Beta family: factorials, Catalan, central binomial, double factorials
(\S\ref{sec:beta}); shifts of the $(1+x)$ algebraic family
(\S\ref{sec:algsqshift}); even orders of $(1+x)\,\mathrm{cosb}_\nu^2$
(\S\ref{sec:xbesseleven})\\
$\M2$ & even and odd functionals $\mathcal S$, $\mathcal T$ quasi-definite &
connection array a single term ($\mathcal S,\mathcal T$ a one-parameter
classical pair) and the kernel determinant evaluable &
Euler number family (\S\ref{sec:gen}); the secant-number family
$(1+x)/\cos^{s+1}$ (\S\ref{sec:gen-xcos}--\ref{sec:allstar});
reciprocal sine (\S\ref{sec:xsinx}); Bessel $(s,t)$ family
(\S\ref{sec:besselst}); and their double shifts
(\S\ref{sec:dshift}, \ref{sec:dblshift} and~\ref{sec:besseldshift})\\
$\M3$ & single functional $\EE$ quasi-definite &
even-orthogonal $\Phi_i$ exist \eqref{eq:onefuncprod}; among Appell
families: the Gaussian only &
involutions, $\exp(cx+bx^2)$ (\S\ref{sec:gauss})\\
$\M4$ & family $f_\theta$, $\HH_n(f_\theta)$ polynomial in $\theta$ of
bounded degree &
degenerations of $f_\theta$ account for the full degree
(Observation~\ref{obs:divisor}) &
Springer (\S\ref{sec:springer}), derivative of Springer
(\S\ref{sec:derivative}), elliptic (\S\ref{sec:elliptic})\\
$\M5$ & contiguous relation $s\mapsto s-2$ by column operations &
two-periodicity in $s$ plus one special value, or a single contiguous
step closed by Cauchy--Binet &
algebraic families $(1+x)/(1-x^2)^{s/2}$ and $(1+x)^2/(1-x^2)^{s/2}$
(\S\ref{sec:alg} and~\ref{sec:algsq}), and their shifted determinants
(\S\ref{sec:algsqshift})\\
$\M6$ & rank-one perturbation of an already evaluated family &
matrix-determinant lemma~\eqref{lem:matdet} &
$(\sin x+1)/\cos^2x+s\sin x$ (\S\ref{sec:runkone})\\
\hline
\end{tabular}
\vspace{6pt}
\caption{The six methods of Section~\ref{sec:methods}.}
\label{tab:methods}
\end{table}


\section{The Beta family}\label{sec:beta}

Throughout this section $\mathbf a=(a_m)_{m\ge0}$ is a \emph{plain}
sequence: no exponential normalisation is applied, and the entries of the
determinants are the $a_m$ themselves. We assume that the term ratio
$a_{m+1}/a_m$ is a rational function of $m$ with numerator and denominator
of degree \emph{at most} one --- the exact hypothesis under which the
Vandermonde reduction $\M1$ of Section~\ref{sec:vandermonde} operates. The
generic case, numerator and denominator both of degree one, is the Beta
family proper (Theorem~\ref{thm:beta}); the degenerate case with denominator
of degree zero --- $\beta$ absent --- is treated separately
(Proposition~\ref{prop:degenerate}). Such sequences
form the three-parameter family
\begin{equation}\label{eq:betadef}
a_m=a_0\,\rho^m\,\frac{(\alpha)_m}{(\beta)_m},
\qquad
\frac{a_{m+1}}{a_m}=\rho\,\frac{m+\alpha}{m+\beta}.
\end{equation}

The name is explained by the case $\beta>\alpha>0$, $a_0=\rho=1$: there
\eqref{eq:betadef} is the moment sequence of the \emph{Beta distribution}
on $[0,1]$ with density
$x^{\alpha-1}(1-x)^{\beta-\alpha-1}/B(\alpha,\beta-\alpha)$, since
$\int_0^1 x^m\,d\mu=(\alpha)_m/(\beta)_m$. Beta distributions are the
model solutions of the Hausdorff moment problem on a bounded interval, the
polynomials orthogonal to them are the Jacobi polynomials
\cite{Chihara1978,Ismail2005}, and the associated Jacobi continued
fraction \cite{Wall1948,Flajolet1980} governs the classical (one-step)
Hankel determinants of the family. None of this positivity is needed
below: the scale $\rho$ and the total mass $a_0$ are inessential, and
$\alpha,\beta,\rho$ may be taken as arbitrary parameters, the determinant
identities being polynomial (or rational) in them. Members of
\eqref{eq:betadef} include the factorials $(m+r)!=r!\,(r+1)_m$
(degenerate, $\beta$ absent), the Catalan numbers
$C_m=\frac1{m+1}\binom{2m}m=4^m(\tfrac12)_m/(2)_m$, and the central
binomial coefficients $b_m=\binom{2m}m=4^m(\tfrac12)_m/(1)_m$.

The section is organised as follows.
Subsection~\ref{ssec:betageneral} proves a single closed form for
$\det(a_{x_i+j})$ with \emph{arbitrary} row indices $x_i$
(Theorem~\ref{thm:beta}); Subsection~\ref{ssec:betaspecial} specialises it
to the dilated, shifted and $r$-step determinants of the classical
members; Subsection~\ref{ssec:betaboundary} records where the mechanism
stops.

\subsection{The general evaluation}\label{ssec:betageneral}

\begin{Theorem}\label{thm:beta}
Let $n\ge1$, let $\rho\ne0$ and $\alpha,\beta$ be parameters with
$\beta\notin\{0,-1,-2,\dots\}$, let $a_m$ be as in \eqref{eq:betadef},
and let $0\le x_0<x_1<\dots<x_{n-1}$ be integers. Then
\begin{equation}\label{eq:beta}
\det\bigl(a_{x_i+j}\bigr)_{0\le i,j\le n-1}
=\rho^{\binom n2}\,\prod_{i=0}^{n-1}(\beta-\alpha)_i\;
\prod_{0\le i<j\le n-1}(x_j-x_i)\;
\prod_{i=0}^{n-1}\frac{a_{x_i}}{(x_i+\beta)_{n-1}}.
\end{equation}
\end{Theorem}

\begin{proof}
\textit{Step 1: row factorisation.}
Iterating the term ratio in \eqref{eq:betadef} gives
$a_{x+j}=a_x\,\rho^j\,(x+\alpha)_j/(x+\beta)_j$, and over the common
denominator $(x+\beta)_{n-1}$, for $0\le j\le n-1$,
\begin{equation}\label{eq:betarowfact}
a_{x+j}=\frac{a_x}{(x+\beta)_{n-1}}\;Q_j(x),
\qquad
Q_j(x)=\rho^j\,(x+\alpha)_j\,(x+\beta+j)_{n-1-j},
\end{equation}
each $Q_j$ being a polynomial in $x$ of degree exactly
$j+(n-1-j)=n-1$, with coefficients independent of the $x_i$.

\textit{Step 2: reduction to a Vandermonde determinant.}
The row factorisation \eqref{eq:betarowfact} is exactly the hypothesis of the
Vandermonde reduction~$\M1$ (Lemma~\ref{lem:vdm}), with $w(x)=a_x/(x+\beta)_{n-1}$
and the $Q_j$ above, each of degree $n-1$. Writing
$Q_j(x)=\sum_{k=0}^{n-1}m_{kj}x^k$ and $M=(m_{kj})_{0\le k,j\le n-1}$, it gives
\begin{equation}\label{eq:betavdm}
\det\bigl(a_{x_i+j}\bigr)
=\prod_{i=0}^{n-1}\frac{a_{x_i}}{(x_i+\beta)_{n-1}}\;\cdot\;
c_n\prod_{0\le i<j\le n-1}(x_j-x_i),
\qquad c_n=\det M,
\end{equation}
where the constant $c_n$ does not depend on $x_0,\dots,x_{n-1}$.

\textit{Step 3: the triangulating substitution.}
Since $\det(Q_j(x_i))=c_n\prod_{i<j}(x_j-x_i)$ is an identity of
polynomials in the indeterminates $x_0,\dots,x_{n-1}$, we may compute
$c_n$ by any convenient substitution --- the integrality and ordering
of the $x_i$ are irrelevant here. Take $x_i=1-\beta-i$. Then
\begin{equation*}
(x_i+\beta+j)_{n-1-j}=(1+j-i)(2+j-i)\cdots(n-1-i),
\end{equation*}
which vanishes if and only if $i>j$: the matrix $\bigl(Q_j(x_i)\bigr)$
is upper triangular. Its diagonal entries are
\begin{equation*}
Q_i(x_i)=\rho^i\,(1-\beta-i+\alpha)_i\,(1)_{n-1-i}
=\rho^i\,(-1)^i(\beta-\alpha)_i\,(n-1-i)!\,,
\end{equation*}
while the Vandermonde product at this substitution is
$\prod_{i<j}(i-j)=(-1)^{\binom n2}\prod_{i=0}^{n-1}i!$. In the quotient
the signs cancel, and so does $\prod_i(n-1-i)!=\prod_i i!$, leaving
$c_n=\rho^{\binom n2}\prod_{i=0}^{n-1}(\beta-\alpha)_i$.
\end{proof}

Note the case $\alpha=\beta$: the sequence is geometric, every such
determinant with $n\ge2$ is zero, and the formula sees this through the
factor $(\beta-\alpha)_i$, which vanishes for $i\ge1$.

When the denominator parameter is absent the situation degenerates
further: the reduction is triangular from the start, and no substitution
is needed.

\begin{Proposition}\label{prop:degenerate}
Let $a_m=a_0\rho^m(\alpha)_m$ with $\rho\ne0$. Then
\begin{equation}\label{eq:degenerate}
\det\bigl(a_{x_i+j}\bigr)_{0\le i,j\le n-1}
=\rho^{\binom n2}\prod_{0\le i<j\le n-1}(x_j-x_i)\;
\prod_{i=0}^{n-1}a_{x_i}.
\end{equation}
\end{Proposition}

\begin{proof}
Here $a_{x+j}=a_x\,P_j(x)$ with $P_j(x)=\rho^j(x+\alpha)_j$ of degree
exactly $j$ and leading coefficient $\rho^j$. This is the triangular case of
the Vandermonde reduction (Lemma~\ref{lem:vdm}), with $w(x)=a_x$ and
$\det M=\prod_{j=0}^{n-1}\rho^j=\rho^{\binom n2}$, so
$\det(a_{x_i+j})=\rho^{\binom n2}\prod_{i<j}(x_j-x_i)\prod_i a_{x_i}$.
\end{proof}

\subsection{Specialisations: factorials, Catalan and central binomial
numbers}\label{ssec:betaspecial}

The substitution $x_i=2i$ in \eqref{eq:beta} gives the dilated Hankel
determinant of every member of the family at one stroke. More generally,
for an integer $r\ge1$ we call $\det(a_{ri+j})_{0\le i,j\le n-1}$ the
\emph{$r$-step Hankel determinant} of $\mathbf a$ --- the classical case
is $r=1$, the dilated case $r=2$ --- and the substitution $x_i=ri+t$
evaluates every shifted $r$-step determinant $\det(a_{ri+j+t})$ of the
family in closed form.

\begin{Corollary}\label{cor:betadouble}
With $a_m$ as in \eqref{eq:betadef},
\begin{equation}\label{eq:betadouble}
\HH_n(\mathbf a)=\det\bigl(a_{2i+j}\bigr)_{0\le i,j\le n-1}
=(2\rho)^{\binom n2}\,\prod_{i=0}^{n-1}
i!\;(\beta-\alpha)_i\;\frac{a_{2i}}{(2i+\beta)_{n-1}}.
\end{equation}
\end{Corollary}

We now harvest the classical members. No further proofs are needed: each
corollary is a substitution of parameters into Theorem~\ref{thm:beta} or
Proposition~\ref{prop:degenerate}, followed by a conversion of Pochhammer
symbols into factorials, and we record these conversions as
\emph{derivations}.

\begin{Corollary}[factorials]\label{cor:factorial}
Let $r\ge0$ and $a_m=(m+r)!$ --- in the exponential convention of the
other sections, the coefficient sequence of the derivative $f^{(r)}$ of
$f(x)=1/(1-x)$. Then
\begin{equation}\label{eq:factorial}
\HH_n\bigl((m+r)!\bigr)
=2^{\binom n2}\prod_{k=1}^{n-1}k!\cdot\prod_{i=0}^{n-1}(2i+r)!,
\qquad\text{hence}\qquad
\frac{\HH_n\bigl(f^{(r+1)}\bigr)}{\HH_n\bigl(f^{(r)}\bigr)}
=\prod_{i=0}^{n-1}(2i+r+1).
\end{equation}
The first three ratios ($r=0,1,2$) are $(2n-1)!!$, $(2n)!!$,
$(2n+1)!!$. In particular
\begin{align*}
\HH_n\bigl(m!\bigr)&=\textstyle\prod_{k=1}^{n-1}2^k\,k!\,(2k)!
=2^{\binom n2}\prod_{k=1}^{n-1}k!\,(2k)!\,,\\
\HH_n\bigl((m+1)!\bigr)&=\textstyle\prod_{k=1}^{n-1}2^k\,k!\,(2k+1)!\,,
\qquad
\HH_n\bigl((m+2)!/2\bigr)=\textstyle\prod_{k=1}^{n-1}2^{k-1}\,k!\,(2k+2)!\,.
\end{align*}
\end{Corollary}

\begin{proof}[Derivation]
$(m+r)!=r!\,(r+1)_m$ is the case $\rho=1$, $\alpha=r+1$ of
Proposition~\ref{prop:degenerate}; with $x_i=2i$,
$\HH_n=\prod_{i<j}2(j-i)\prod_i(2i+r)!
=2^{\binom n2}\prod_{k<n}k!\,\prod_i(2i+r)!$. The last evaluation
follows by scaling rows by $1/2$.
\end{proof}

Proposition~\ref{prop:degenerate} holds for arbitrary $\alpha$, so the
factorial evaluation \eqref{eq:factorial} is the integer specialisation
of a continuous one-parameter family.

\begin{Corollary}[the Gamma/Laguerre continuation]\label{cor:laguerre}
For the moments
$a_m=\Gamma(m+\alpha+1)=\Gamma(\alpha+1)\,(\alpha+1)_m$ of the Laguerre
weight $x^{\alpha}e^{-x}$ on $[0,\infty)$ (here $\rho=1$),
\begin{equation}\label{eq:laguerre}
\det\bigl(\Gamma(2i+j+\alpha+1)\bigr)_{0\le i,j\le n-1}
=\Gamma(\alpha+1)^{n}\,2^{\binom n2}\prod_{k=1}^{n-1}k!\,
\prod_{m=1}^{n-1}\bigl[(\alpha+2m-1)(\alpha+2m)\bigr]^{\,n-m},
\end{equation}
which reduces to \eqref{eq:factorial} at a non-negative integer $\alpha=r$.
\end{Corollary}

\begin{proof}[Derivation]
With $x_i=2i$, $\rho=1$ and $a_{2i}=\Gamma(\alpha+1)(\alpha+1)_{2i}$,
Proposition~\ref{prop:degenerate} gives
$\det(a_{2i+j})=2^{\binom n2}\prod_{k<n}k!\,\prod_{i}a_{2i}$. In
$\prod_{i=0}^{n-1}(\alpha+1)_{2i}=\prod_{i=0}^{n-1}\prod_{l=1}^{2i}(\alpha+l)$
the factor $(\alpha+l)$ occurs once for each $i$ with $2i\ge l$, hence with
multiplicity $n-\lceil l/2\rceil$; writing $l=2m-1$ and $l=2m$ gives the
exponent $n-m$.
\end{proof}

\begin{Corollary}[Catalan numbers]\label{cor:catalan}
Let $C_m=\frac1{m+1}\binom{2m}m$ and let
$0\le x_0<x_1<\dots<x_{n-1}$ be integers. Then
\begin{equation}\label{eq:catalangen}
\det\bigl(C_{x_i+j}\bigr)_{0\le i,j\le n-1}
=\prod_{0\le i<j\le n-1}(x_j-x_i)\;
\prod_{i=0}^{n-1}\frac{(2x_i)!\,(n+i)!}{(2i)!\;x_i!\;(x_i+n)!}.
\end{equation}
In particular, $x_i=i$ recovers the well-known $\det(C_{i+j})=1$
\cite{Krattenthaler1998};
$x_i=2i$ gives the dilated Hankel determinant
\begin{equation}\label{eq:catalandouble}
\HH_n(C)
=2^{\binom n2}\,\prod_{i=0}^{n-1}\binom{4i}{2i}\,
\frac{i!\,(n+i)!}{(n+2i)!}
=1,\,3,\,32,\,1232,\,172032,\dots;
\end{equation}
and $x_i=ri$ gives the $r$-step analogue
\begin{equation}\label{eq:catalanrfold}
\det\bigl(C_{ri+j}\bigr)_{0\le i,j\le n-1}
=r^{\binom n2}\,\prod_{i=0}^{n-1}
\frac{i!\,(2ri)!\,(n+i)!}{(2i)!\,(ri)!\,(ri+n)!}\,,
\end{equation}
e.g.\ the 3-step Hankel values $1,\,9,\,990,\,1363230,\dots$ for $r=3$.
\end{Corollary}

\begin{proof}[Derivation]
$C_m=4^m(\tfrac12)_m/(2)_m$ is the member
$(\alpha,\beta,\rho)=(\tfrac12,2,4)$ of \eqref{eq:betadef}. The row
factors of \eqref{eq:beta} are
$C_x/(x+2)_{n-1}=\frac{(2x)!}{x!\,(x+1)!}\cdot\frac{(x+1)!}{(x+n)!}
=\frac{(2x)!}{x!\,(x+n)!}$, and the constant is
$4^{\binom n2}\prod_i(\tfrac32)_i
=\prod_i\frac{(2i+1)!}{i!}=\prod_i\frac{(n+i)!}{(2i)!}$, using
$(\tfrac32)_i=(2i+1)!/(4^i\,i!)$ and the fact that both
$\prod_{i<n}(2i+1)!/i!$ and $\prod_{i<n}(n+i)!/(2i)!$ equal $1$ at
$n=1$ and are multiplied by $(2n+1)!/n!$ when $n\mapsto n+1$.
\end{proof}

The general Catalan evaluation \eqref{eq:catalangen} is not new: it is in the
1989 Gessel--Viennot preprint \cite{Gessel1989Viennot}, proved in a journal by
Krattenthaler \cite[Theorem~3]{Krattenthaler2010} via
\cite[Lemma~2.2]{Krattenthaler1990}, with earlier special cases in
\cite{SainteCatherineViennot1986} and related evaluations in the compendia
\cite{Krattenthaler1998,Krattenthaler2005}. The specialisations
\eqref{eq:catalandouble} and \eqref{eq:catalanrfold} appear to be unrecorded.

\begin{Corollary}[central binomial coefficients]\label{cor:cbinom}
Let $b_m=\binom{2m}m$ and let $0\le x_0<x_1<\dots<x_{n-1}$ be
integers. Then
\begin{equation}\label{eq:cbinomgen}
\det\bigl(b_{x_i+j}\bigr)_{0\le i,j\le n-1}
=2^{\,n-1}\prod_{0\le i<j\le n-1}(x_j-x_i)\;
\prod_{i=0}^{n-1}\frac{(2x_i)!\,(n-1+i)!}{(2i)!\;x_i!\;(x_i+n-1)!}.
\end{equation}
In particular, $x_i=i$ recovers the classical
$\det(b_{i+j})=2^{n-1}$; $x_i=2i$ gives
\begin{equation}\label{eq:cbinomdouble}
\HH_n(b)
=2^{\binom n2+n-1}\,\prod_{i=0}^{n-1}\binom{4i}{2i}\,
\frac{i!\,(n-1+i)!}{(n-1+2i)!}
=1,\,8,\,224,\,22528,\,8200192,\dots;
\end{equation}
and $x_i=ri$ gives
\begin{equation}\label{eq:cbinomrfold}
\det\bigl(b_{ri+j}\bigr)_{0\le i,j\le n-1}
=2^{\,n-1}\,r^{\binom n2}\,\prod_{i=0}^{n-1}
\frac{i!\,(2ri)!\,(n-1+i)!}{(2i)!\,(ri)!\,(ri+n-1)!}\,.
\end{equation}
\end{Corollary}

\begin{proof}[Derivation]
$b_m=4^m(\tfrac12)_m/(1)_m$ is the member $(\tfrac12,1,4)$. The row
factors are $b_x/(x+1)_{n-1}=(2x)!/(x!\,(x+n-1)!)$, and the constant
is $4^{\binom n2}\prod_i(\tfrac12)_i=\prod_i\frac{(2i)!}{i!}
=2^{n-1}\prod_i\frac{(n-1+i)!}{(2i)!}$, using
$(\tfrac12)_i=(2i)!/(4^i\,i!)$ and checking, as before, that the last
two products agree at $n=1$ and are both multiplied by $(2n)!/n!$ when
$n\mapsto n+1$.
\end{proof}

\begin{Remark}\label{rem:cbinomcf}
The case $x_i=i$ of \eqref{eq:cbinomgen}, $\det(b_{i+j})=2^{n-1}$, is
classical and is usually proved by continued fractions: from
$\sum_mb_mz^m=(1-4z)^{-1/2}=1/(1-2zc(z))$, with $c(z)$ the Catalan
generating function and its Stieltjes fraction
$c(z)=1/(1-z/(1-z/(1-\cdots)))$, the contracted Jacobi fraction has
$\lambda_1=2$ and $\lambda_k=1$ for $k\ge2$, so the
Heilermann--Stieltjes theorem
\cite{Heilermann1846,Wall1948,Krattenthaler1998} gives
$\prod_k\lambda_k^{n-k}=2^{n-1}$. Theorem~\ref{thm:beta} re-proves it
without continued fractions.
\end{Remark}

\begin{Example}\label{ex:beta}
Three more integer members, with their dilated Hankel determinants
(immediate from Corollary~\ref{cor:betadouble} and
Proposition~\ref{prop:degenerate}):
\begin{align}
a_m&=(2m+1)!!=2^m(\tfrac32)_m: &
\HH_n&=4^{\binom n2}\prod_{i=0}^{n-1}i!\,(4i+1)!!\,,
\label{eq:exdf}\\
a_m&=\frac{(2m)!}{m!}=4^m(\tfrac12)_m: &
\HH_n&=8^{\binom n2}\prod_{i=0}^{n-1}\frac{i!\,(4i)!}{(2i)!}\,,
\label{eq:exqf}\\
a_m&=\binom{2m+1}{m}=4^m\,\frac{(\tfrac32)_m}{(2)_m}: &
\HH_n&=2^{\binom n2}\prod_{i=0}^{n-1}
\frac{i!\,(4i+1)!\,(n+i)!}{(2i+1)!\,(2i)!\,(2i+n)!}\,.
\label{eq:exodd}
\end{align}
For the last sequence the ordinary Hankel determinant ($x_i=i$)
equals $1$.
\end{Example}

\subsection{The boundary of the mechanism}\label{ssec:betaboundary}

\begin{Remark}\label{rem:betaboundary}
Theorem~\ref{thm:beta} also delimits the mechanism. If the term ratio
$a_{m+1}/a_m$ is a rational function of degree $d\ge2$ --- as for the
Fuss--Catalan numbers $\frac1{2m+1}\binom{3m}m$ or for $\binom{3m}m$
--- the polynomials $Q_j$ acquire degree $d(n-1)>n-1$ and the collapse
onto the Vandermonde determinant fails; and indeed the corresponding
dilated Hankel determinants exhibit large sporadic prime factors. The
same holds for sequences that are not of hypergeometric type at all
(Motzkin, Delannoy, Ap\'ery, Franel, Bell, \dots).
\end{Remark}

\begin{Remark}\label{rem:smallprimes}
The formulas of this section explain the experimental observation that
these determinants factorise into very small primes: a ratio of
factorials of arguments at most $4n$ has no prime factor beyond $4n$.
\end{Remark}

\section{Involutions and the Gaussian family}\label{sec:gauss}

The Hankel determinants of the Gaussian (normal) moment sequence --- and, more
generally, of moment sequences of orthogonal polynomial families such as the
Hermite polynomials --- are classical, and are treated in the standard
references on the moment problem and orthogonal polynomials: Shohat and
Tamarkin~\cite{Shohat1943Tamarkin}, Akhiezer~\cite{Akhiezer1965} and
Szeg\H o~\cite{Szego1975}. The evaluation below is a \emph{dilated}
(row-shifted) refinement of that classical Hankel determinant.

The Beta family of Section~\ref{sec:beta} is a plain sequence, and its
evaluation was the Vandermonde reduction~$\M1$. In this section the
sequence is given by an exponential generating function,
$f(x)=\sum_n a_n\,x^n/n!$, and the tool is the one-functional
reduction~$\M3$ (Lemma~\ref{lem:onefunc}). We carry it out on its
simplest instance, the Gaussian family: a single orthogonal family (the
Hermite polynomials) suffices, and the dilated Hankel determinant is
triangularised directly.

The bookkeeping principle behind every evaluation here is
Lemma~\ref{lem:family}, the change-of-family lemma: replacing the row and
column families of a moment matrix by monic families of the same degrees
leaves the determinant unchanged, and the triangular outcome is then read off
the diagonal. Here the new families are the Hermite polynomials $\He_j$ and
the even family $\Phi_i$ of Lemma~\ref{lem:tri}, both monic of strictly
increasing degree, hence unitriangular recombinations of the (split-)monomial
families they replace.

\begin{Theorem}\label{thm:gauss}
Let $c$ be an indeterminate (or any scalar) and let $(a_n)$ be defined
by
$$
f(x)=\sum_{n\ge0}a_n\,\frac{x^n}{n!}=e^{cx+x^2/2},
\qquad\text{i.e.}\qquad
a_n=\sum_{j\ge 0}\binom{n}{2j}(2j-1)!!\,c^{\,n-2j}.
$$
Then for all $n\ge 1$,
$$
\HH_n(f)=\det(a_{2i+j})_{0\le i,j\le n-1}
=(2c)^{\binom n2}\prod_{k=1}^{n-1}k!.
$$
\end{Theorem}

\begin{Corollary}\label{cor:inv}
Let $f(x)=e^{x+x^2/2}$, so that $a_n$ is the number of involutions of
$n$ letters. Then for all $n\ge 1$,
$$
\HH_n(f)=2^{\binom n2}\prod_{k=1}^{n-1} k!.
$$
\end{Corollary}

Corollary~\ref{cor:inv} is the case $c=1$ of Theorem~\ref{thm:gauss}:
an involution of $n$ letters is obtained by choosing the $2j$
non-fixed letters and one of the $(2j-1)!!$ perfect matchings on them,
so $a_n=\sum_j\binom{n}{2j}(2j-1)!!$ is the number of involutions.

We use the \emph{Gaussian moment functional} $\mathcal L$, defined on
polynomials in $z$ (with coefficients in $\QQ[c]$) by linearity and
$$
\mathcal L[z^m]=
\begin{cases}
(m-1)!! & \text{if $m$ is even},\\
0& \text{if $m$ is odd},
\end{cases}
$$
so that $a_n=\mathcal L[(c+z)^n]$. For an indeterminate $\alpha$ one has, as
a formal power series in $\alpha$ (the computation is finite in each
degree),
\begin{equation}\label{eq:expmoment}
\mathcal L\bigl[e^{\alpha z}\bigr]
=\sum_{j\ge 0}\frac{\alpha^{2j}(2j-1)!!}{(2j)!}
=\sum_{j\ge 0}\frac{\alpha^{2j}}{2^j j!}
=e^{\alpha^2/2}.
\end{equation}
Let $\He_m$ denote the monic Hermite polynomials,
$$
e^{tz-t^2/2}=\sum_{m\ge 0}\He_m(z)\,\frac{t^m}{m!},
$$
so that $\He_m$ is monic of degree $m$ and
$\He_m(-z)=(-1)^m\He_m(z)$.

\begin{Lemma}[mixed Hermite moments]\label{lem:mixed}
For all $p,q\ge 0$,
$$
\mathcal L\bigl[\He_p(c+z)\,\He_q(z)\bigr]=q!\,\binom{p}{q}\,c^{\,p-q}.
$$
(In particular the value is $0$ for $q>p$.)
\end{Lemma}

\begin{proof}
By the generating function and \eqref{eq:expmoment},
\begin{align*}
\sum_{p,q\ge 0}\mathcal L\bigl[\He_p(c+z)\He_q(z)\bigr]\frac{t^p}{p!}\frac{s^q}{q!}
&=\mathcal L\Bigl[e^{t(c+z)-t^2/2}\,e^{sz-s^2/2}\Bigr]
=e^{tc-\frac{t^2+s^2}{2}}\,\mathcal L\bigl[e^{(t+s)z}\bigr]\\
&=e^{tc-\frac{t^2+s^2}{2}+\frac{(t+s)^2}{2}}
=e^{t(c+s)}
=\sum_{p\ge 0}\frac{t^p(c+s)^p}{p!}.
\end{align*}
Extracting the coefficient of $t^ps^q/(p!\,q!)$ gives the claim.
\end{proof}

\begin{Lemma}[triangularising family]\label{lem:tri}
For $i\ge 0$ put
$$
\Phi_i(z):=\sum_{k=0}^{i}\binom{i}{k}\,(-c^2)^k\,\He_{2i-2k}(c+z).
$$
Then $\Phi_i$ is a monic {\em even} polynomial of degree $2i$ in
$u:=c+z$, and for all $j\ge 0$
$$
\mathcal L\bigl[\Phi_i(z)\,\He_j(z)\bigr]
=j!\,\binom{i}{j-i}\,(2c)^{2i-j}.
$$
In particular this vanishes for $j<i$ and equals $i!\,(2c)^i$ for
$j=i$.
\end{Lemma}

\begin{proof}
Each $\He_{2i-2k}(u)$ is an even monic polynomial in $u$ of degree
$2i-2k$, so $\Phi_i$ is even in $u$, of degree $2i$, with leading
coefficient $1$. By Lemma~\ref{lem:mixed},
$$
\mathcal L[\Phi_i\,\He_j]
=\sum_{k=0}^{i}\binom{i}{k}(-c^2)^k\,
j!\binom{2i-2k}{j}c^{\,2i-2k-j}
=j!\,c^{\,2i-j}\sum_{k=0}^{i}(-1)^k\binom{i}{k}\binom{2i-2k}{j}.
$$
The last sum is the coefficient of $y^j$ in
$$
\sum_{k=0}^{i}(-1)^k\binom{i}{k}(1+y)^{2i-2k}
=\bigl((1+y)^2-1\bigr)^i
=y^i\,(y+2)^i,
$$
namely $2^{2i-j}\binom{i}{j-i}$, which gives the claim
(and the vanishing for $j<i$, since $y^i(y+2)^i$ has valuation $i$).
\end{proof}

\begin{Remark}[where the definition comes from]\label{rem:tri}
The formula for $\Phi_i$ is not guessed; it is forced. We seek a monic
even polynomial of degree $2i$ in $u=c+z$ that is orthogonal to
$\He_0,\dots,\He_{i-1}$ under the pairing
$(f,g)\mapsto\mathcal L[f(z)g(z)]$. Writing it in the even Hermite basis,
$\Phi_i=\sum_{k=0}^i\gamma_k\He_{2i-2k}(c+z)$ with $\gamma_0=1$, and
applying Lemma~\ref{lem:mixed}, the whole row assembles into one
generating function,
$$
\sum_{j\ge 0}\frac{\mathcal L[\Phi_i\,\He_j]}{j!}\,y^j
=\sum_{k=0}^i\gamma_k\,(c+y)^{2i-2k}.
$$
The triangularity condition $\mathcal L[\Phi_i\,\He_j]=0$ for $j<i$ says this
series has valuation $\ge i$ at $y=0$; in the variable $w=c+y$ it says
$\sum_k\gamma_k w^{2i-2k}$ is divisible by $(w-c)^i$. Being even and
monic of degree $2i$, it is then divisible by $(w+c)^i$ as well, hence
equal to $(w^2-c^2)^i$. Matching coefficients gives
$\gamma_k=\binom ik(-c^2)^k$. Equivalently, $\Phi_i$ is the
Hermite-umbral image of $(u^2-c^2)^i=\bigl(z(z+2c)\bigr)^i$, obtained by
replacing each power $u^{2m}$ by $\He_{2m}(u)$.
\end{Remark}

\begin{proof}[Proof of Theorem~\ref{thm:gauss}]
Write $u=c+z$, so that
$$
\HH_n=\det\bigl(a_{2i+j}\bigr)_{0\le i,j\le n-1}
=\det\Bigl(\mathcal L\bigl[u^{2i}\cdot u^{j}\bigr]\Bigr)_{0\le i,j\le n-1}.
$$
Replace the column family $(u^j)_{j<n}$ by $(\He_j(z))_{j<n}$, monic of
degree $j$ in $u$ since $\He_j(z)=\He_j(u-c)$, and the row family
$(u^{2i})_{i<n}$ by $(\Phi_i)_{i<n}$, monic even of degree $2i$ in $u$ by
Lemma~\ref{lem:tri}. Both are unitriangular recombinations, so by
Lemma~\ref{lem:family}
$$
\HH_n=\det\Bigl(\mathcal L\bigl[\Phi_i(z)\,\He_j(z)\bigr]\Bigr)_{0\le i,j\le n-1}.
$$
By Lemma~\ref{lem:tri} this last matrix is upper triangular (the
$(i,j)$ entry vanishes for $j<i$) with diagonal entries $i!\,(2c)^i$.
Hence
\[
\HH_n=\prod_{i=0}^{n-1} i!\,(2c)^i
=(2c)^{\binom n2}\prod_{k=1}^{n-1}k!.
\qedhere
\]
\end{proof}

\begin{Remark}\label{rem:gaussgen}
(i) By the homogeneity
$\HH_n\bigl(f(\sigma x)\bigr)=\sigma^{3\binom n2}\HH_n(f)$
(Lemma~\ref{lem:scale}), Theorem~\ref{thm:gauss} gives, for
$f(x)=e^{bx+tx^2}$,
$$
\HH_n(f)=(4bt)^{\binom n2}\prod_{k=1}^{n-1}k!.
$$
(ii) For $c=0$ (i.e.\ $f=e^{x^2/2}$, $a_{2k}=(2k-1)!!$, $a_{2k+1}=0$)
the theorem gives $\HH_n=0$ for all $n\ge 2$.
\end{Remark}

\begin{Remark}[Gaussian rigidity among Appell families]\label{rem:rigidity}
The evaluation hinges on the shift identity of Lemma~\ref{lem:mixed},
$\mathcal L[\He_p(c+z)\He_q(z)]=q!\binom pq c^{p-q}$, which expresses that the
Hermite polynomials form an \emph{Appell} sequence,
$\sum_{m\ge0}\He_m(z)\,t^m/m!=G(t)\,e^{tz}$ with $G(t)=e^{-t^2/2}$. The method
therefore asks for a family $(\pi_m)$ that is at once the sequence of monic
orthogonal polynomials of $\mathcal L$ and an Appell sequence, and these two
requirements already pin down the weight. Indeed, normalising $G(0)=1$ and
$\mathcal L[1]=1$, orthogonality gives $\mathcal L[e^{\alpha z}]=1/G(\alpha)$ and
makes the mixed generating function
$$
\sum_{p,q\ge0}\mathcal L[\pi_p(z)\,\pi_q(z)]\frac{t^p}{p!}\frac{s^q}{q!}
=\frac{G(t)\,G(s)}{G(t+s)}
$$
a function of $ts$ alone. With $\gamma=\log G$ this reads
$\gamma(t)+\gamma(s)-\gamma(t+s)=\psi(ts)$, and matching the coefficient of
$t^as^b$ for $a\ne b$, $a,b\ge1$, forces $\gamma_k=0$ for $k\ge3$, i.e.\
$\gamma(t)=\gamma_1t+\gamma_2t^2$. Hence $\mathcal L$ is Gaussian: up to a
translation of $z$ (the term $\gamma_1$) and a scaling
(Remark~\ref{rem:gaussgen}(i)), the Hermite/Gaussian case is the only Appell
family to which the one-functional triangularisation applies.
\end{Remark}

\section{The single shift of the Gaussian family}\label{sec:dergauss}

The same Hermite machinery (again the one-functional reduction~$\M3$) yields
the derivative rule for this family.
Since differentiation shifts the coefficient sequence,
$a_n(f')=a_{n+1}(f)$, the dilated Hankel determinant of $f'$
is the \emph{shifted} dilated Hankel determinant of $f$:
$\HH_n(f')=\det(a_{2i+j+1})_{0\le i,j\le n-1}$. (For the rational
chain the corresponding rule is the ratio formula in
Corollary~\ref{cor:factorial}.)

\begin{Theorem}[Single shift of the Gaussian family]\label{thm:dergauss}
Let $f(x)=e^{cx+x^2/2}$ as in Theorem~\ref{thm:gauss}. Then
$$
\HH_n(f')=c^n\,\HH_n(f)=c^n\,(2c)^{\binom n2}\prod_{k=1}^{n-1}k! .
$$
In particular, for $c=1$ the dilated Hankel determinant is invariant
under differentiation: for the involution numbers $a_n$,
$$
\det\bigl(a_{2i+j+1}\bigr)_{0\le i,j\le n-1}
=\det\bigl(a_{2i+j}\bigr)_{0\le i,j\le n-1}
=2^{\binom n2}\prod_{k=1}^{n-1}k! .
$$
\end{Theorem}

\begin{proof}
We keep the notation of the proof of Theorem~\ref{thm:gauss}
($\mathcal L$, $u=c+z$, $\He_m$, $\Phi_i$). Since $a_{m+1}=\mathcal L[u^{m+1}]$,
$$
\HH_n(f')=\det\Bigl(\mathcal L\bigl[u^{2i}\cdot u^{j+1}\bigr]\Bigr)_{0\le i,j\le n-1}.
$$
Replace the column family $(u^{j+1})_{j<n}$ by
$\bigl(u\,\He_j(z)\bigr)_{j<n}$ ($u\He_j(z)=u\,\He_j(u-c)$ is monic of
degree $j+1$ in $u$ with no constant term) and the row family
$(u^{2i})_{i<n}$ by
$$
\Psi_i:=\sum_{s=0}^{i}\frac{i!}{s!}\,(-2)^{i-s}\,\Phi_s ,
$$
monic and even of degree $2i$ in $u$; both are unitriangular
recombinations (Lemma~\ref{lem:family}). As in Remark~\ref{rem:tri}, the
weights $\frac{i!}{s!}(-2)^{i-s}$ are not guessed: they are exactly those
that assemble the row generating functions into the truncated exponential
$T_i$ below, whose first-order differential equation delivers the required
vanishing. We claim
\begin{equation}\label{eq:dergauss}
\mathcal L\bigl[\Psi_i\,u\,\He_j\bigr]=
\begin{cases}
0, & j<i,\\[2pt]
c\,i!\,(2c)^i, & j=i,
\end{cases}
\end{equation}
which yields
$\HH_n(f')=\prod_{i=0}^{n-1}c\,i!(2c)^i=c^n\HH_n(f)$ by triangularity.

To prove \eqref{eq:dergauss}, first pair $\Phi_s$ with the new
columns. By the recurrence $u\He_p(u)=\He_{p+1}(u)+p\He_{p-1}(u)$
applied to the argument $u=c+z$, and Lemma~\ref{lem:mixed},
$$
\mathcal L\bigl[\Phi_s\,u\He_j\bigr]
=\sum_{k=0}^{s}\binom sk(-c^2)^k
\Bigl(j!\tbinom{2s-2k+1}{j}c^{2s-2k+1-j}
+(2s-2k)\,j!\tbinom{2s-2k-1}{j}c^{2s-2k-1-j}\Bigr).
$$
Using $\sum_k(-1)^k\binom sk(1+y)^{2s-2k}=\bigl(y(2+y)\bigr)^s$ and
$m\binom{m-1}j=(j+1)\binom m{j+1}$, this becomes
$$
\mathcal L\bigl[\Phi_s\,u\He_j\bigr]
=j!\,c^{2s+1-j}\,[y^j]\,(1+y)\bigl(y(2+y)\bigr)^s
+(j+1)!\,c^{2s-1-j}\,[y^{j+1}]\bigl(y(2+y)\bigr)^s .
$$
Now put $w=w(y):=-\tfrac{c^2}{2}\,y(2+y)$, so that
$c^{2s}\bigl(y(2+y)\bigr)^s=(-2)^sw^s$ and, summing over $s$ with the
weights $\frac{i!}{s!}(-2)^{i-s}$,
$$
\mathcal L\bigl[\Psi_i\,u\He_j\bigr]
=i!\,(-2)^i\,j!\,c^{-1-j}
\Bigl(c^2\,[y^j]\,(1+y)\,T_i+(j+1)\,[y^{j+1}]\,T_i\Bigr),
\qquad
T_i:=\sum_{s=0}^{i}\frac{w(y)^s}{s!},
$$
the truncated exponential. Since $\frac{dw}{dy}=-c^2(1+y)$ and
$T_i'(w)=T_i(w)-\frac{w^i}{i!}$,
$$
\frac{d}{dy}\,T_i+c^2(1+y)\,T_i
=c^2(1+y)\,\frac{w^i}{i!},
$$
whose right-hand side has $y$-valuation $i$ (as $w$ has valuation
$1$). Extracting the coefficient of $y^j$ gives
$$
(j+1)[y^{j+1}]T_i+c^2[y^j](1+y)T_i=
\begin{cases}
0, & j<i,\\[2pt]
\dfrac{(-1)^i\,c^{2i+2}}{i!}, & j=i,
\end{cases}
$$
using $[y^i]w^i=(-c^2)^i$. The case $j<i$ gives the vanishing in
\eqref{eq:dergauss}, and the case $j=i$ gives
\[
\mathcal L\bigl[\Psi_i\,u\He_i\bigr]
=i!\,(-2)^i\cdot i!\,c^{-1-i}\cdot\frac{(-1)^i\,c^{2i+2}}{i!}
=i!\,2^i\,c^{\,i+1}=c\,i!\,(2c)^i .
\qedhere
\]
\end{proof}

\section{The Euler number family}\label{sec:gen}

The \emph{Euler numbers} $E_n$ are defined by
\cite{Andre1879,Stieltjes1894Re,Viennot1982, Han2020Euler}
\begin{equation}\label{eq:eulerdef}
\tan x+\sec x=\sum_{n\ge0}E_n\frac{x^n}{n!},
\qquad (E_n)_{n\ge0}=1,1,1,2,5,16,61,272,1385,\dots
\end{equation}
The even Euler numbers $E_{2n}$ (resp.\ odd $E_{2n+1}$) are the secant
(resp.\ tangent) numbers; by Andr\'e's theorem $E_n$ is the number of
alternating permutations of $\{1,\dots,n\}$.
The ordinary generating functions of the secant and tangent numbers have
the classical Stieltjes $S$-fractions
$$
\sum_{k\ge 0}E_{2k}\,x^{k}
=\cfrac{1}{1-\cfrac{1\cdot 1\,x}{1-\cfrac{2\cdot 2\,x}{1-\ddots}}},
\qquad
\sum_{k\ge 0}E_{2k+1}\,x^{k}
=\cfrac{1}{1-\cfrac{1\cdot 2\,x}{1-\cfrac{2\cdot 3\,x}{1-\ddots}}},
$$
with coefficients $k^2$ and $k(k+1)$ respectively.

We embed this pair into a two-parameter family, and specialise the general
biorthogonal reduction~$\M2$ of Section~\ref{sec:biotho}, in the
$S$-fraction form of Section~\ref{sec:sfrac}, to
\begin{equation}\label{eq:stparam}
u_j=j(j+s),\qquad v_j=j(j+t)\qquad(j\ge1):
\end{equation}
$\mathbf a=(a_n)_{n\ge0}$ is the sequence whose even and odd parts have the
Stieltjes continued fractions \eqref{eq:Fsfrac} with these coefficients.
With coefficients $j(j+s)$, each series in \eqref{eq:Fsfrac} is a classical
Stieltjes $S$-fraction \cite{Stieltjes1894Re,Wall1948,Flajolet1980}; at $s=0$,
$t=1$ the two fractions reduce to the secant and tangent fractions above.

To describe the underlying function, we introduce the generalised secant
numbers $E_{2k}^{(j)}$ by
$$
\frac{1}{\cos(x)^{j}}=\sum_{k\ge 0} E_{2k}^{(j)}\,\frac{x^{2k}}{(2k)!}.
$$
Their ordinary generating function has the Stieltjes $S$-fraction
$$
\sum_{k\ge 0}E_{2k}^{(s+1)}\,x^{k}
=\cfrac{1}{1-\cfrac{1(1+s)\,x}{1-\cfrac{2(2+s)\,x}{1-\ddots}}},
$$
whose coefficients are exactly $u_j=j(j+s)$ of \eqref{eq:stparam}. Thus,
taking the even part of $\mathbf a$ from the secant fraction $u_j=j(j+s)$ and the
odd part from $v_j=j(j+t)$, the exponential generating function of $\mathbf a$ is
$$
\sum_{k\ge 0} a_k\,\frac{x^k}{k!}
=\sum_{k\ge 0} a_{2k}\,\frac{x^{2k}}{(2k)!}
+\sum_{k\ge 0} a_{2k+1}\,\frac{x^{2k+1}}{(2k+1)!}
=\frac{1}{\cos(x)^{s+1}}+\int_{0}^{x}\frac{dy}{\cos(y)^{t+1}},
$$
since $a_{2k}=E_{2k}^{(s+1)}$ and, differentiating the odd part,
$a_{2k+1}=E_{2k}^{(t+1)}$. Accordingly we write
\begin{equation}\label{eq:gf:F}
\HH_n=\HH_n\!\left(\frac{1}{\cos(x)^{s+1}}
       +\int_{0}^{x}\frac{dy}{\cos(y)^{t+1}}\right).
\end{equation}

Throughout write $\bar n=\lceil n/2\rceil$ and $\underline n=\lfloor n/2\rfloor$, so
$\bar n+\underline n=n$ and $\bar n\in\{\underline n,\underline n+1\}$. The main result of this
section is the following closed form, valid for \emph{all} $(s,t)$; the
secant/tangent case $s=0$, $t=1$ and its relatives are collected in
Section~\ref{sec:coro}.

\begin{Proposition}\label{prop:family}
Let
\begin{equation*}
	f(x)=\frac{1}{\cos(x)^{s+1}} +\int_{0}^{x}\frac{dy}{\cos(y)^{t+1}}.
\end{equation*}
Then,
	\begin{equation}\label{HHn:st}
	\HH_n(f)=
(-1)^{\binom{\bar n}{2}}
\Bigl(\prod_{i=0}^{n-1}(2i)!\Bigr)
\Bigl(\prod_{l=0}^{\bar n-1}(s{+}1)_{2l}\Bigr)
\Bigl(\prod_{m=0}^{\underline n-1}(t{+}1)_{2m}\Bigr)
\prod_{r=1}^{\bar n}\prod_{c=1}^{\underline n}
\frac{(t-s)/2+c-r}{n-r-c+1}.
\end{equation}
\end{Proposition}

The sign $(-1)^{\binom{\bar n}2}$ together with the doubly-indexed product in
\eqref{HHn:st} form the case-dependent \emph{signed} factor
$$
\Omega(\delta):=(-1)^{\binom{\bar n}2}\prod_{r=1}^{\bar n}\prod_{c=1}^{\underline n}
\frac{\delta/2+c-r}{n-r-c+1}\qquad(\delta=t-s),
$$
so that \eqref{HHn:st} reads
$\HH_n=\Omega(\delta)\,\prod_{i=0}^{n-1}(2i)!\,
\prod_{l=0}^{\bar n-1}(s{+}1)_{2l}\,\prod_{m=0}^{\underline n-1}(t{+}1)_{2m}$.

The proof occupies
Subsections~\ref{ssec:stdata}--\ref{ssec:stproof} and is an instance of
the collapse mechanism of Section~\ref{sec:biotho}: the orthogonal data of
the family are explicit, the two functionals differ in the single
parameter $(t-s)/2$, so the connection coefficients collapse to a single
binomial term (Lemma~\ref{lem:connGene}), and the resulting binomial
determinant is evaluated by the dual Jacobi--Trudi identity
(Lemma~\ref{lem:bindetSigma}). Subsection~\ref{ssec:stomega} then
evaluates $\Omega(\delta)$ in closed form for odd $\delta$
(Lemma~\ref{lem:Pahalf}), the form in which \eqref{HHn:st} is specialised
in Section~\ref{sec:dshift} and in the corollaries of
Section~\ref{sec:coro}.

\subsection{Orthogonal data and connection coefficients}\label{ssec:stdata}

Substituting \eqref{eq:stparam} into Lemma~\ref{lem:cfGene}, the recurrence
coefficients of the $\mathcal S$- and $\mathcal T$-orthogonal polynomials
$P_i,Q_m$ become
\begin{align*}
c^S_i&=2i(2i{+}s)+(2i{+}1)(2i{+}1{+}s), &
\lambda^S_i&=(2i{-}1)(2i{-}1{+}s)\,2i(2i{+}s),\\
c^T_m&=2m(2m{+}t)+(2m{+}1)(2m{+}1{+}t), &
\lambda^T_m&=(2m{-}1)(2m{-}1{+}t)\,2m(2m{+}t)
\end{align*}
(so $c^S_0=1{+}s$, $c^T_0=1{+}t$), while the squared norms telescope into
Pochhammer symbols,
\begin{equation}\label{eq:stnorms}
h^S_i=\prod_{k=1}^{2i}u_k=\prod_{k=1}^{2i}k(k+s)=(2i)!\,(s{+}1)_{2i},
\qquad
h^T_m=\prod_{k=1}^{2m}v_k=(2m)!\,(t{+}1)_{2m},
\end{equation}
using $\prod_{k=1}^{2i}(k+s)=(s{+}1)(s{+}2)\cdots(s{+}2i)=(s{+}1)_{2i}$.

At the secant/tangent specialisation $s=0$, $t=1$ --- needed repeatedly in
later sections --- the two families take the explicit form below; we write
$\hat P_i:=P_i|_{s=0}$ and $r_m:=Q_m|_{t=1}$ for the resulting monic
polynomials.

\begin{Lemma}[secant and tangent polynomials; classical]\label{lem:cf}
The monic $\mathcal S$-orthogonal polynomials $\hat P_i$ and the monic
$\mathcal T$-orthogonal polynomials $r_m$ satisfy
\begin{align*}
\hat P_{i+1}&=\bigl(y-(2i)^2-(2i+1)^2\bigr)\hat P_i
            -\bigl((2i-1)(2i)\bigr)^2\hat P_{i-1},\\
r_{m+1}&=\bigl(y-2(2m+1)^2\bigr)r_m-(2m-1)(2m)^2(2m+1)\,r_{m-1},
\end{align*}
with $\mathcal S[\hat P_i\hat P_l]=\delta_{il}\,\bigl((2i)!\bigr)^2$ and
$\mathcal T[r_mr_{m'}]=\delta_{mm'}\,(2m)!\,(2m+1)!$.
\end{Lemma}

\begin{proof}
This is the case $s=0$ (resp.\ $t=1$) of the present family.
The secant and tangent numbers have the classical Stieltjes $S$-fractions
\cite{Stieltjes1894Re,Wall1948,Flajolet1980} with coefficients $u_j=j^2$ and
$v_j=j(j+1)$, and the recurrence coefficients and norms specialise to those
displayed: $2m(2m+1)+(2m+1)(2m+2)=2(2m+1)^2$, while \eqref{eq:stnorms} gives
$h^S_i=(2i)!\,(1)_{2i}=\bigl((2i)!\bigr)^2$ and
$h^T_m=(2m)!\,(2)_{2m}=(2m)!\,(2m+1)!$.
\end{proof}

For this family the connection-coefficient recurrence \eqref{eq:kapparec}
of Lemma~\ref{lem:connrec} has a closed-form solution --- a single
binomial term.

\begin{Lemma}[connection coefficients]\label{lem:connGene}
Write $P_i=\sum_{m=0}^{i}\kappa_{i,m}\,Q_m$ for the expansion of the
$\mathcal S$-orthogonal polynomials in the $\mathcal T$-orthogonal basis.
Then
$$
\kappa_{i,m}=\frac{(2i)!}{(2m)!}\binom{(t-s)/2}{\,i-m\,}
\qquad(\,\textstyle\binom{(t-s)/2}{d}=0\ \text{for }d<0\,),
$$
and hence
$\mathcal T[P_i\,Q_m]=\kappa_{i,m}\,(2m)!\,(t{+}1)_{2m}
=(2i)!\,(t{+}1)_{2m}\,\binom{(t-s)/2}{\,i-m\,}$.
\end{Lemma}

\begin{proof}
By Lemma~\ref{lem:connrec} the $\kappa_{i,m}$ satisfy
$\kappa_{i,i}=1$, vanish for $m>i$ and $m<0$, and obey the recurrence
\eqref{eq:kapparec}; together with $\kappa_{0,0}=1$ these determine them.
By \eqref{eq:stnorms},
$\mathcal T[P_i Q_m]=\kappa_{i,m}\,\mathcal T[Q_m^2]
=\kappa_{i,m}\,(2m)!\,(t{+}1)_{2m}$, so it remains to identify
$\kappa_{i,m}$.

We verify that the candidate
$\kappa_{i,m}=\frac{(2i)!}{(2m)!}\binom{(t-s)/2}{d}$, $d=i-m$, satisfies
\eqref{eq:kapparec}; for $i=0,1$ this holds by inspection. In general,
divide \eqref{eq:kapparec} by $\kappa_{i,m}$ and use the neighbour ratios
$$
\frac{\kappa_{i+1,m}}{\kappa_{i,m}}
=(2i{+}1)(2i{+}2)\frac{(t-s)/2-d}{d+1},\quad
\frac{\kappa_{i,m-1}}{\kappa_{i,m}}
=(2m{-}1)(2m)\frac{(t-s)/2-d}{d+1},
$$
$$
\frac{\kappa_{i,m+1}}{\kappa_{i,m}}
=\frac{d}{((t-s)/2{-}d{+}1)(2m{+}1)(2m{+}2)},\quad
\frac{\kappa_{i-1,m}}{\kappa_{i,m}}
=\frac{d}{((t-s)/2{-}d{+}1)(2i{-}1)(2i)},
$$
together with
$\dfrac{\lambda^T_{m+1}}{(2m{+}1)(2m{+}2)}=(2m{+}1{+}t)(2m{+}2{+}t)$ and
$\dfrac{\lambda^S_i}{(2i{-}1)(2i)}=(2i{-}1{+}s)(2i{+}s)$; the relation
\eqref{eq:kapparec} becomes
\begin{align*}
&\bigl[(2i{+}1)(2i{+}2)-(2m{-}1)(2m)\bigr]\frac{(t-s)/2-d}{d+1}\\
&\qquad=c^T_m-c^S_i
+\frac{d}{(t-s)/2{-}d{+}1}
\bigl[(2m{+}1{+}t)(2m{+}2{+}t)-(2i{-}1{+}s)(2i{+}s)\bigr].
\end{align*}
With $i=m+d$ and the specialised values of $c^S_i,c^T_m$
displayed above, both sides agree as polynomials in
$(m,d,s,t)$ --- an elementary identity, routine to verify. Hence the
candidate satisfies \eqref{eq:kapparec} and the initial conditions, so
it equals $\kappa_{i,m}$.
\end{proof}

\subsection{A binomial determinant; proof of the closed form}
\label{ssec:stproof}

\begin{Lemma}[binomial determinant, general parameter]\label{lem:bindetSigma}
Let $N\ge 1$, $q\in\{N,N+1\}$, and let $a$ be arbitrary. Then
$$
\det\left(\binom{a}{\,q+r-m\,}\right)_{0\le r,m\le N-1}
=\prod_{r=1}^{q}\prod_{c=1}^{N}\frac{a+c-r}{(q-r)+(N-c)+1}.
$$
\end{Lemma}

\begin{proof}
Recall the dual Jacobi--Trudi (N\"agelsbach--Kostka) identity
$s_\lambda=\det\bigl(e_{\lambda'_i-i+j}\bigr)_{1\le i,j\le\ell(\lambda')}$,
where $s_\lambda$ is the Schur symmetric function, $e_k$ the elementary
symmetric function, and $\lambda'$ the conjugate partition of $\lambda$
\cite[I.3, (3.5)]{Macdonald1995}, \cite[Cor.~7.16.2]{Stanley1999EC2}.
For integers $a\ge N$ this identity, applied to the
rectangular partition $(N^q)$ (whose conjugate is $(q^N)$, so
$\lambda'_i=q$ and $e_{\lambda'_i-i+j}=e_{q-i+j}$) and specialised at
$a$ ones via $e_k(1^a)=\binom ak$, gives $\det\bigl(\binom{a}{q-i+j}\bigr)$,
which is the matrix below after transposing; hence
$$
\det\left(\binom{a}{\,q+r-m\,}\right)_{0\le r,m\le N-1}
=s_{(N^q)}(\underbrace{1,\ldots,1}_{a})
=\prod_{r=1}^{q}\prod_{c=1}^{N}\frac{a+c-r}{(q-r)+(N-c)+1}
$$
by the hook content formula
$s_\lambda(1^a)=\prod_{u\in\lambda}\frac{a+c(u)}{h(u)}$, where $c(u)=j-i$
is the content and $h(u)$ the hook length of the cell $u=(i,j)$
\cite[Cor.~7.21.4]{Stanley1999EC2}, \cite[I.3, Ex.~4]{Macdonald1995}.
Both sides are polynomials in $a$, so the identity holds for all $a$.
\end{proof}

\begin{proof}[Proof of Proposition~\textup{\ref{prop:family}}]
By Lemma~\ref{lem:connGene},
$\kappa_{\bar n+r,\,m}
=\frac{(2(\bar n+r))!}{(2m)!}\binom{(t-s)/2}{\bar n+r-m}$;
pulling $(2(\bar n{+}r))!$ out of row $r$ and $1/(2m)!$ out of
column $m$ and applying Lemma~\ref{lem:bindetSigma} with $a=(t-s)/2$,
$$
\det\bigl(\kappa_{\bar n+r,\,m}\bigr)_{0\le r,m\le \underline n-1}
=\Bigl(\prod_{r=0}^{\underline n-1}(2(\bar n{+}r))!\Bigr)
\Bigl(\prod_{m=0}^{\underline n-1}\frac1{(2m)!}\Bigr)
\times\prod_{r=1}^{\bar n}\prod_{c=1}^{\underline n}
\frac{(t-s)/2+c-r}{(\bar n-r)+(\underline n-c)+1}.
$$
Substituting into Lemma~\ref{lem:bindetGene}, with the norms
$h^S_l=(2l)!\,(s{+}1)_{2l}$ and $h^T_m=(2m)!\,(t{+}1)_{2m}$ of
\eqref{eq:stnorms}, and cancelling the $(2m)!$ against the column norms,
the surviving factors are
$$
\prod_{m=0}^{\underline n-1}(t{+}1)_{2m}
\qquad\text{and}\qquad
\prod_{l=0}^{\bar n-1}(2l)!\cdot
\prod_{r=0}^{\underline n-1}(2(\bar n{+}r))!
=\prod_{i=0}^{n-1}(2i)!,
$$
which gives the stated formula.
\end{proof}

\subsection{The signed factor}\label{ssec:stomega}

We extend the double factorial to negative odd arguments by
$(2k-1)!!=(2k+1)!!/(2k+1)$, so $(-1)!!=1$ and
$(-2k+1)!! = (-1)^{k-1}/(2k-3)!!$; with this convention the signed factor
$\Omega(\delta)$ has the following closed evaluation.

\begin{Lemma}\label{lem:Pahalf}
With $\delta$ odd, we have
\begin{equation}\label{eq:Pa:dfact}
	\Omega(\delta)
	=(-1)^{\binom{\bar n}2}\,2^{-\bar n\underline n}\,
	\frac{1}{\prod_{j=\bar n}^{n-1}j!}\,
	\prod_{j=0}^{\underline n-1} \frac{ j! (\delta+2j)!!}{(\delta-2\bar n+2j)!!}\,.
\end{equation}
In particular, for $\delta\in\{\pm1,\pm3\}$: when $n=2m$ is even,
\begin{equation}\label{eq:Pa:even}
\begin{aligned}
&\Omega(1)=2^{-\binom n2},\qquad
\Omega(-1)=(-1)^{m}2^{-\binom n2},\\
&\Omega(3)=(-1)^{m-1}(n^2-1)\,2^{-\binom n2},\qquad
\Omega(-3)=-(n^2-1)\,2^{-\binom n2};
\end{aligned}
\end{equation}
when $n=2m+1$ is odd,
\begin{equation}\label{eq:Pa:odd}
\Omega(1)=2^{-\binom n2},\quad
\Omega(-1)=\Omega(3)=(-1)^{m}\,n\,2^{-\binom n2},\quad
\Omega(-3)=-\tfrac13\,n^2(n^2-4)\,2^{-\binom n2}.
\end{equation}
\end{Lemma}

\begin{proof}
Group the double product defining $\Omega(\delta)$ by rows. Since $\delta$ is odd,
the numerator of row $r$ clears to a double factorial,
$$
\prod_{c=1}^{\underline n}\Bigl(\tfrac\delta2+c-r\Bigr)
=2^{-\underline n}\prod_{c=1}^{\underline n}(\delta-2r+2c)
=2^{-\underline n}\,\frac{(\delta+2\underline n-2r)!!}{(\delta-2r)!!},
$$
with the double-factorial convention above, while its denominator telescopes,
$$
\prod_{c=1}^{\underline n}(n-r-c+1)=\frac{(n-r)!}{(n-r-\underline n)!}=\frac{(\bar n+\underline n-r)!}{(\bar n-r)!}
$$
(recall $n-\underline n=\bar n$). Multiplying over $r=1,\dots,\bar n$ produces the
factor $2^{-\bar n\underline n}$ together with two telescoping ratios. The hook
denominators give
$$
\prod_{r=1}^{\bar n}\frac{(\bar n-r)!}{(\bar n+\underline n-r)!}
=\frac{\prod_{j=0}^{\bar n-1}j!}{\prod_{j=\underline n}^{n-1}j!}
=\frac{\prod_{j=0}^{\underline n-1}j!}{\prod_{j=\bar n}^{n-1}j!},
$$
cancelling the common factorials $j=\underline n,\dots,\bar n-1$; and, writing
$\delta+2\underline n-2r=\delta-2(r-\underline n)$, the numerator double factorials give
$$
\prod_{r=1}^{\bar n}\frac{(\delta+2\underline n-2r)!!}{(\delta-2r)!!}
=\frac{\prod_{k=1-\underline n}^{\bar n-\underline n}(\delta-2k)!!}{\prod_{k=1}^{\bar n}(\delta-2k)!!}
=\frac{\prod_{i=0}^{\underline n-1}(\delta+2i)!!}{\prod_{i=0}^{\underline n-1}(\delta-2\bar n+2i)!!},
$$
cancelling the common arguments $k=1,\dots,\bar n-\underline n$. Both ratios run over
the same range $0\le j\le\underline n-1$, so the factorial numerator
$\prod_{j=0}^{\underline n-1}j!$ merges into the double-factorial product; restoring the
sign $(-1)^{\binom{\bar n}2}$ gives \eqref{eq:Pa:dfact}.

For $\delta=1$ the two ratios above combine to a pure power of two. Writing
$(1-2r)!!=(-1)^{r-1}/(2r-3)!!$, the double-factorial ratio becomes
$\prod_{r=1}^{\bar n}\frac{(2\underline n-2r+1)!!}{(1-2r)!!}
=(-1)^{\binom{\bar n}2}\prod_{r=1}^{\bar n}(2r-3)!!\,(2\underline n-2r+1)!!$;
multiplied by the hook ratio $\prod_{r=1}^{\bar n}(r-1)!/(\underline n+r-1)!$ and the
leading sign $(-1)^{\binom{\bar n}2}$, the two signs cancel and
$\Omega(1)=2^{-\bar n\underline n}Q$ with
$Q=\prod_{r=1}^{\bar n}\frac{(2r-3)!!\,(2\underline n-2r+1)!!\,(r-1)!}{(\underline n+r-1)!}$. Along the
ladder $(\underline n,\underline n)\to(\underline n{+}1,\underline n)\to(\underline n{+}1,\underline n{+}1)$ one has $Q=1$ at $(0,0)$ and each
step multiplies $Q$ by $2^{-\underline n}$ --- via $(2\underline n-1)!!\,\underline n!/(2\underline n)!$ and
$(2\underline n+1)!!\,\underline n!/(2\underline n+1)!$ respectively, both equal to $2^{-\underline n}$ --- while
$\bar n\underline n-\binom n2$ falls by $\underline n$; hence $Q=2^{\bar n\underline n-\binom n2}$ and
$\Omega(1)=2^{-\bar n\underline n}Q=2^{-\binom n2}$.

For $\delta\in\{-1,3,-3\}$ we use a contiguous relation. Since
$(x+2)!!=(x+2)\,x!!$ for every odd $x$ (negative ones included, by the convention
above), dividing \eqref{eq:Pa:dfact} at $\delta+2$ by the same at $\delta$ leaves
only the top factor of each double factorial:
$$
\frac{\Omega(\delta+2)}{\Omega(\delta)}
=\prod_{j=0}^{\underline n-1}\frac{\delta+2+2j}{\delta+2-2\bar n+2j}
=\begin{cases}
\dfrac{(\delta+2)(\delta+4)\cdots(\delta+2m)}{\delta\,(\delta-2)\cdots(\delta-2m+2)}, & n=2m,\\[2.4ex]
\dfrac{(\delta+2)(\delta+4)\cdots(\delta+2m)}{(\delta-2m)(\delta-2m+2)\cdots(\delta-2)}, & n=2m+1.
\end{cases}
$$
We chain outward from $\Omega(1)=2^{-\binom n2}$.

\emph{$n=2m$.} At $\delta=1$ the ratio is
$\frac{(2m+1)!!}{(-1)^{m-1}(2m-3)!!}=(-1)^{m-1}(2m-1)(2m+1)=(-1)^{m-1}(n^2-1)$, so
$\Omega(3)=(-1)^{m-1}(n^2-1)2^{-\binom n2}$. At $\delta=-1$ it is
$\frac{(2m-1)!!}{(-1)^m(2m-1)!!}=(-1)^m$, so $\Omega(-1)=\Omega(1)/(-1)^m=(-1)^m2^{-\binom n2}$.
At $\delta=-3$ it is $\frac{-(2m-3)!!}{(-1)^m(2m+1)!!}=\frac{(-1)^{m+1}}{n^2-1}$, so
$\Omega(-3)=(-1)^{m+1}(n^2-1)\,\Omega(-1)=-(n^2-1)2^{-\binom n2}$.

\emph{$n=2m+1$.} The same three ratios are
$\frac{(2m+1)!!}{(-1)^m(2m-1)!!}=(-1)^m n$ at $\delta=1$,
$\frac{(2m-1)!!}{(-1)^m(2m+1)!!}=\frac{(-1)^m}{n}$ at $\delta=-1$, and
$\frac{-(2m-3)!!}{(-1)^m(2m+3)!!/3}=\frac{-3}{(-1)^m(n-2)\,n\,(n+2)}$ at $\delta=-3$
(using $(2m-1)(2m+1)(2m+3)=(n-2)n(n+2)$). Hence
$\Omega(3)=(-1)^m n\,2^{-\binom n2}$,
$\Omega(-1)=(-1)^m n\,\Omega(1)=(-1)^m n\,2^{-\binom n2}$ (so $\Omega(-1)=\Omega(3)$),
and $\Omega(-3)=-\tfrac{(-1)^m(n-2)n(n+2)}{3}\,\Omega(-1)=-\tfrac13\,n^2(n^2-4)\,2^{-\binom n2}$.
\end{proof}

\section{The double shift of the Euler number family}\label{sec:dshift}

Alongside $\HH_n=\det(a_{2i+j})$ the most regular variant is the \emph{double}
shift
\begin{equation}\label{eq:Hdshiftdef}
\HH_n^{(2)}:=\det\bigl(a_{2i+j+2}\bigr)_{0\le i,j\le n-1}.
\end{equation}
Shifting the column index by $2$ \emph{preserves} parity --- in the moment
picture \eqref{eq:KmatrixGene} the even columns stay $\mathcal S$-columns and the
odd columns stay $\mathcal T$-columns, each merely advanced by one power of $y$
--- so no even/odd swap occurs and the determinant stays a product for
\emph{all} $(s,t)$. (The single shift $\det(a_{2i+j+1})$, which swaps the two
parities and is far less regular, is taken up in Section~\ref{sec:shift}.)

Since $a_{2i+j+2}=\EE[y^{\,i+1}\chi_j]$, the double shift is the dilated
determinant of the \emph{Christoffel transform} $y\,\EE$ of $\EE$ by $y=z^2$:
its even and odd parts are $\mathcal S'[p]:=\mathcal S[yp]$ and
$\mathcal T'[p]:=\mathcal T[yp]$; the biorthogonal reduction~$\M2$ again
applies. We first record what this transform does to the
ingredients of Section~\ref{sec:gen} for \emph{arbitrary} $u_j,v_j$, and express
$\HH_n^{(2)}$ through the transformed connection determinant; only afterwards do
we specialise to $u_j=j(j+s)$, $v_j=j(j+t)$, where that determinant collapses and
the answer factors.

\begin{Lemma}[Christoffel transform by $y$, general $u_j,v_j$]\label{lem:christoffel}
Let $P_i,Q_m$ be the monic orthogonal polynomials of $\mathcal S,\mathcal T$
with squared norms $h^S_i,h^T_m$, and let $P'_i,Q'_m$ be those of
$\mathcal S',\mathcal T'$ with norms $h^{S'}_i,h^{T'}_m$ and connection
coefficients $P'_i=\sum_m\kappa'_{i,m}Q'_m$. Put $D^S_n=(-1)^nP_n(0)$ and
$D^T_m=(-1)^mQ_m(0)$. Then, for arbitrary $u_j,v_j$:
\begin{enumerate}
\item[\rm(i)] $D^S_{n+1}=c^S_nD^S_n-\lambda^S_nD^S_{n-1}$ with $D^S_0=1$,
$D^S_1=c^S_0=u_1$, and likewise for $D^T$;
\item[\rm(ii)] $\mathcal S',\mathcal T'$ are again moment functionals, with the
J-fraction coefficients of the odd contraction~\eqref{eq:oddcontraction},
$\lambda^{S'}_i=u_{2i}u_{2i+1}$ and $c^{S'}_i=u_{2i+1}+u_{2i+2}$ (in particular
$c^{S'}_0=u_1+u_2$), and similarly for $T'$;
\item[\rm(iii)] $h^{S'}_i=\dfrac{D^S_{i+1}}{D^S_i}\,h^S_i$ and
$h^{T'}_m=\dfrac{D^T_{m+1}}{D^T_m}\,h^T_m$.
\end{enumerate}
\end{Lemma}

\begin{proof}
(i) Evaluating $P_{n+1}=(y-c^S_n)P_n-\lambda^S_nP_{n-1}$ at $y=0$ gives
$D^S_{n+1}=c^S_nD^S_n-\lambda^S_nD^S_{n-1}$, with $D^S_0=1$ and (as
$\lambda^S_0=0$) $D^S_1=c^S_0=u_1$; likewise for $D^T$.

(ii) Multiplying $\mathcal S$ by its variable shifts the moments,
$\mathcal S'[y^k]=\mathcal S[y^{k+1}]=\mu_{k+1}$, so $\mathcal S'$ is the
Christoffel transform of $\mathcal S$ described in
Section~\ref{ssec:oddcontraction}. Its $J$-fraction is the odd
contraction~\eqref{eq:oddcontraction} of the $S$-fraction $b_j=u_j$ of
$\mathcal S$, which yields $c^{S'}_i=u_{2i+1}+u_{2i+2}$ and
$\lambda^{S'}_i=u_{2i}u_{2i+1}$ at once \cite{Wall1948,Flajolet1980}. (Only the
$J$-fraction of $\mathcal S'$ is needed below; $\mathcal S'$ itself need not
have an $S$-fraction of the family.) The same computation with $v_j$ gives
$\mathcal T'$.

(iii) With $\mu_k=\mathcal S[y^k]$ and $H_i=\det(\mu_{a+b})_{a,b<i}$, the
determinantal formula for $P_i$ evaluated at $y=0$ reads
$(-1)^iP_i(0)H_i=\det(\mu_{a+b+1})_{a,b<i}=\det(\mathcal S'[y^{a+b}])_{a,b<i}=:H_i'$.
Hence $h^{S'}_i=H'_{i+1}/H'_i=-\dfrac{P_{i+1}(0)}{P_i(0)}\dfrac{H_{i+1}}{H_i}
=\dfrac{D^S_{i+1}}{D^S_i}\,h^S_i$; likewise for $T'$.
\end{proof}

\begin{Proposition}[double shift, general form]\label{prop:dshiftgen}
For arbitrary $u_j,v_j$,
$$
\HH_n^{(2)}=(-1)^{\binom{\bar n}2}\Bigl(\prod_{l=0}^{\bar n-1}h^{S'}_l\Bigr)
\Bigl(\prod_{l=0}^{\underline n-1}h^{T'}_l\Bigr)\det\bigl(\kappa'_{\bar n+r,m}\bigr)_{0\le r,m\le \underline n-1}.
$$
Equivalently, telescoping the norms by Lemma~\ref{lem:christoffel}(iii),
\begin{equation}\label{eq:dshiftgen}
\HH_n^{(2)}=(-1)^{n}\,P_{\bar n}(0)\,Q_{\underline n}(0)\;
\frac{\det\bigl(\kappa'_{\bar n+r,m}\bigr)}{\det\bigl(\kappa_{\bar n+r,m}\bigr)}\;\HH_n,
\end{equation}
where $\HH_n$ is the unshifted dilated determinant of
Proposition~\ref{prop:family}.
\end{Proposition}

\begin{proof}
$\HH_n^{(2)}=\HH_n(a_{\,\cdot+2})$ is the dilated determinant of the sequence
with even part $\mathcal S'$ and odd part $\mathcal T'$, so the reduction of
Lemma~\ref{lem:bindetGene} applies verbatim with primes throughout, giving the
first display. For the second, Lemma~\ref{lem:christoffel}(iii) telescopes
$\prod_{l=0}^{\bar n-1}h^{S'}_l=D^S_{\bar n}\prod_{l=0}^{\bar n-1}h^S_l$ and
$\prod_{l=0}^{\underline n-1}h^{T'}_l=D^T_{\underline n}\prod_{l=0}^{\underline n-1}h^T_l$ (since $D^S_0=D^T_0=1$);
dividing by the same reduction \eqref{eq:KmatrixGene} for $\HH_n$ and using
$D^S_{\bar n}D^T_{\underline n}=(-1)^{\bar n+\underline n}P_{\bar n}(0)Q_{\underline n}(0)=(-1)^nP_{\bar n}(0)Q_{\underline n}(0)$ (here $\bar n+\underline n=n$) gives
\eqref{eq:dshiftgen}.
\end{proof}

For general $u_j,v_j$ the ratio $\det(\kappa')/\det(\kappa)$ in
\eqref{eq:dshiftgen} is the one factor that need not simplify. We now
specialise to the family $u_j=j(j+s)$, $v_j=j(j+t)$ of
Section~\ref{sec:gen}, where the ingredients of Lemma~\ref{lem:christoffel}
acquire closed forms and---crucially---the connection determinant collapses.

\begin{Lemma}[Christoffel transform, special family]\label{lem:christoffelspecial}
For $u_j=j(j+s)$, $v_j=j(j+t)$, with $h^S_i=(2i)!\,(s{+}1)_{2i}$ and
$h^T_m=(2m)!\,(t{+}1)_{2m}$:
\begin{enumerate}
\item[\rm(a)] $D^S_n=(-1)^nP_n(0)=(2n-1)!!\displaystyle\prod_{k=0}^{n-1}(s+2k+1)$,
and likewise $D^T_m=(2m-1)!!\prod_{k=0}^{m-1}(t+2k+1)$;
\item[\rm(b)] $h^{S'}_i=(2i+1)(s+2i+1)\,h^S_i$ and
$h^{T'}_m=(2m+1)(t+2m+1)\,h^T_m$;
\item[\rm(c)] $\kappa'_{i,m}=\dfrac{(2i+1)!}{(2m+1)!}\dbinom{(t-s)/2}{i-m}
=\dfrac{2i+1}{2m+1}\,\kappa_{i,m}$.
\end{enumerate}
\end{Lemma}

\begin{proof}
(a) By Lemma~\ref{lem:christoffel}(i), $D^S_{n+1}=c^S_nD^S_n-\lambda^S_nD^S_{n-1}$
with $D^S_0=1$, $D^S_1=s+1$. The displayed product satisfies it:
$D^S_n/D^S_{n-1}=(2n-1)(s+2n-1)$, so
$\lambda^S_nD^S_{n-1}/D^S_n=\lambda^S_n/[(2n-1)(s+2n-1)]=2n(2n+s)$, and
$c^S_n-2n(2n+s)=(2n+1)(2n+1+s)=D^S_{n+1}/D^S_n$; the case of $D^T$ is identical.

(b) By (a) $D^S_{i+1}/D^S_i=(2i+1)(s+2i+1)$, so Lemma~\ref{lem:christoffel}(iii)
gives the stated norms.

(c) By Lemma~\ref{lem:christoffel}(ii) the pair $(\mathcal S',\mathcal T')$ is the
construction of Section~\ref{sec:gen} with the index shift $j\mapsto j+1$ and the
\emph{same} parameter $(t-s)/2$; the connection coefficients therefore obey the
recurrence \eqref{eq:kapparec} with $c^S,\lambda^S,c^T,\lambda^T$ replaced by
their primed values from (ii). The displayed closed form satisfies it: the
verification is that of Lemma~\ref{lem:connGene} with $(2i)!,(2m)!$ advanced to
$(2i+1)!,(2m+1)!$, again an identity of polynomials in $(m,d,s,t)$ routine to
check.
\end{proof}

\begin{Proposition}[double shift, special family]\label{prop:dshift}
With $u_j=j(j+s)$, $v_j=j(j+t)$ as in Section~\ref{sec:gen},
$$
\HH_n^{(2)}:=\det\bigl(a_{2i+j+2}\bigr)_{0\le i,j\le n-1}
=(2n-1)!!\;
\Bigl(\prod_{k=0}^{\bar n-1}(s+2k+1)\Bigr)
\Bigl(\prod_{k=0}^{\underline n-1}(t+2k+1)\Bigr)\,\HH_n,
$$
where $\HH_n$ is the unshifted dilated determinant of
Proposition~\ref{prop:family}.
\end{Proposition}

\begin{proof}
Start from the general form~\eqref{eq:dshiftgen} of
Proposition~\ref{prop:dshiftgen} and insert Lemma~\ref{lem:christoffelspecial}.
By (c), pulling $2(\bar n{+}r)+1$ from row $r$ and $1/(2m+1)$ from column $m$,
$$
\frac{\det\bigl(\kappa'_{\bar n+r,m}\bigr)}{\det\bigl(\kappa_{\bar n+r,m}\bigr)}
=\frac{\prod_{r=0}^{\underline n-1}(2\bar n+2r+1)}{\prod_{m'=0}^{\underline n-1}(2m'+1)}
=\frac{(2n-1)!!}{(2\bar n-1)!!\,(2\underline n-1)!!},
$$
since $(2\bar n+1)(2\bar n+3)\cdots(2n-1)=(2n-1)!!/(2\bar n-1)!!$ (here $\bar n+\underline n=n$). By (a),
$$
(-1)^nP_{\bar n}(0)Q_{\underline n}(0)=D^S_{\bar n}D^T_{\underline n}
=(2\bar n-1)!!\,(2\underline n-1)!!\Bigl(\prod_{k=0}^{\bar n-1}(s+2k+1)\Bigr)\Bigl(\prod_{k=0}^{\underline n-1}(t+2k+1)\Bigr).
$$
Multiplying, the $(2\bar n-1)!!\,(2\underline n-1)!!$ cancel, leaving the asserted factor
times $\HH_n$.
\end{proof}

The two new products are the odd-shift products
$(s+1)(s+3)\cdots(s+2\bar n-1)$ and $(t+1)(t+3)\cdots(t+2\underline n-1)$. In
particular $\HH_n^{(2)}$ factors into linear forms in $(s,t)$ for every
$(s,t)$: the $(t-s)/2$ part is inherited unchanged from $\HH_n$, and the
only new factors are these odd-shift products and the scalar $(2n-1)!!$.

\section{The single shift of the Euler number family on the line $t=1$}\label{sec:shift}

After the double shift of Section~\ref{sec:dshift} we turn to the
\emph{single shift}
\begin{equation}\label{eq:Hshiftdef}
\HH_n^{(1)}=\HH_n^{(1)}(\mathbf a)
=\det\bigl(a_{2i+j+1}\bigr)_{0\le i,j\le n-1},
\end{equation}
the minor of the infinite Hankel matrix $(a_{p+q})$ that selects the
\emph{odd} rows $1,3,\dots,2n-1$ (rather than the even rows
$0,2,\dots,2n-2$ of $\HH_n$). In the moment picture
\eqref{eq:KmatrixGene} the entry is $a_{2i+j+1}=\EE[z^{2i}\cdot z^{\,j+1}]$,
so the column family is shifted from $(\chi_0,\dots,\chi_{n-1})$ to
$(\chi_1,\dots,\chi_n)$, and the parity of every column \emph{swaps}: the
$z$-type (odd) columns now form the full moment family attached to
$\mathcal T$, while the even columns become the defective family
$y^1,\dots,y^{\underline n}$, missing $y^0$. The biorthogonal reduction~$\M2$
therefore applies with the two parities interchanged, and it is the \emph{odd}
part of $\mathbf a$ that governs the determinant.

The outcome is markedly less uniform than for $\HH_n$
(Proposition~\ref{prop:family}) and for the double shift
(Section~\ref{sec:dshift}), both of which factor for \emph{all} $(s,t)$. The
single shift \eqref{eq:Hshiftdef} factors into linear forms over $\QQ$ on the
line $t=1$ (the canonical tangent fraction $v_j=j(j+1)$;
Proposition~\ref{prop:shift}) and---exceptionally---on the line $t=3$
(Theorem~\ref{thm:t3}); for the other integer values $t=2,4,5,6,\dots$ an
irreducible quadratic (then cubic, \dots) factor appears already at $n=4$. On
the line $t=1$ every factor is classical; at $t=3$ the only departure from the
$t=1$ shape is a single \emph{non-classical} linear factor with non-integer
root.

\begin{Proposition}[shifted determinant, $t=1$]\label{prop:shift}
Let $u_j=j(j+s)$ and $v_j=j(j+1)$ (so the odd part of $\mathbf a$ is the tangent
fraction). Then
\begin{equation}\label{eq:shiftpoch}
\HH_n^{(1)}
=K_n\,
\prod_{j=1}^{\underline n}\Bigl(\tfrac{s+1}{2}\Bigr)_{j}^{2}\;
\prod_{j=1}^{\underline n-1}\Bigl(\tfrac{s}{2}+1\Bigr)_{j}\;
\prod_{j=1}^{\bar n-1}\Bigl(\tfrac{1-s}{2}\Bigr)_{j},
\end{equation}
with the positive scalar, independent of $s$,
\begin{equation}\label{eq:shiftconst}
K_n
=\frac{(2n-1)!!\,2^{-\binom n2}\prod_{k=1}^{n-1}k!\,(2k)!}
{\displaystyle\prod_{j=1}^{\underline n}\Bigl(\tfrac12\Bigr)_{j}^{2}
\prod_{j=1}^{\underline n-1}j!\,
\prod_{j=1}^{\bar n-1}\Bigl(\tfrac12\Bigr)_{j}}.
\end{equation}
\end{Proposition}

Here $\bigl(\tfrac12\bigr)_{j}=(2j-1)!!/2^{j}=(2j)!/(4^{j}\,j!)$.
Once $K_n$ is known to be independent of $s$, its value
\eqref{eq:shiftconst} is fixed by the secant/tangent specialisation $s=0$
(Corollary~\ref{cor:euler}), where
$\HH_n^{(1)}\big|_{s=0}=(2n-1)!!\,\HH_n\big|_{s=0}
=(2n-1)!!\,2^{-\binom n2}\prod_{k=1}^{n-1}k!\,(2k)!$. The first values are
$K_n=1,\,12,\,1440,\,9676800,\,\dots$ for $n=1,2,3,4,\dots$.

It is often convenient to read \eqref{eq:shiftpoch} as a product of
\emph{monic} linear forms in $s$. Clearing the Pochhammer symbols gives
$\HH_n^{(1)}=\widetilde K_n\,R_n(s)$ for another constant $\widetilde K_n$ independent of $s$,
where
\begin{equation}\label{eq:shiftmonic}
R_n(s)=
\underbrace{\Bigl(\prod_{j=1}^{\underline n}\prod_{k=0}^{j-1}(s+2k+1)\Bigr)^{2}}
_{\text{odd shifts, doubled}}\;
\underbrace{\prod_{j=1}^{\underline n-1}\prod_{k=1}^{j}(s+2k)}_{\text{even shifts}}\;
\underbrace{\prod_{j=1}^{\bar n-1}\prod_{k=1}^{j}(s-2k+1)}_{\text{back shifts}}.
\end{equation}
The exponent of each linear form is transparent: $(s+2k+1)$ occurs to the
power $2(\underline n-k)$, the even shift $(s+2k)$ to the power
$\underline n-k$, and the back shift $(s-2k+1)$ to the power $\bar n-k$. The
constant $\widetilde K_n$ has sign $\operatorname{sign}\widetilde K_n=(-1)^{\binom{\bar n}2}$, the
sign contributed by the back-shift factors $(\tfrac{1-s}2)_j$ on passing to
the monic form; in particular $\widetilde K_n>0$ for $n=1,2$.

Like the double shift, the single shift is the dilated determinant of a shifted
sequence, but now of $G_k:=a_{k+1}$: indeed $\HH_n^{(1)}=\HH_n(G)$, and the
even/odd split of $G$ is
$$
G_{2k}=a_{2k+1}=\mathcal T[y^k],\qquad
G_{2k+1}=a_{2k+2}=\mathcal S[y^{k+1}]=\mathcal S'[y^k],
$$
so $G$ has \emph{even} functional $\mathcal T$ and \emph{odd} functional the
Christoffel transform $\mathcal S'=y\,\mathcal S$ of
Lemma~\ref{lem:christoffel}. The parities have swapped, and the relevant
connection now links the orthogonal polynomials of the two \emph{different}
functionals $\mathcal T$ and $\mathcal S'$. On the line $t=1$ it is still a
single binomial.

\begin{Lemma}[mixed connection at $t=1$]\label{lem:mixedconn}
Let $t=1$, let $Q_i$ be the $\mathcal T$-orthogonal polynomials
($v_j=j(j+1)$), and let $P'_m$ be the monic $\mathcal S'$-orthogonal
polynomials of Lemma~\ref{lem:christoffel}. Writing
$Q_i=\sum_m\tilde\kappa_{i,m}P'_m$,
$$
\tilde\kappa_{i,m}=\frac{(2i+1)!}{(2m+1)!}\binom{(s+1)/2}{\,i-m\,}.
$$
\end{Lemma}

\begin{proof}
By Lemma~\ref{lem:christoffel}(ii) the $P'_m$ obey
$yP'_m=P'_{m+1}+c^{S'}_mP'_m+\lambda^{S'}_mP'_{m-1}$ with
$c^{S'}_m=u_{2m+1}+u_{2m+2}$, $\lambda^{S'}_m=u_{2m}u_{2m+1}$, while at $t=1$ the
$Q_i$ obey $Q_{i+1}=(y-c^T_i)Q_i-\lambda^T_iQ_{i-1}$ with
$c^T_i,\lambda^T_i$ from Lemma~\ref{lem:cfGene}. Substituting the expansion
$Q_i=\sum_m\tilde\kappa_{i,m}P'_m$ into the $Q$-recurrence and collecting the
coefficient of $P'_m$ yields a three-term recurrence in $(i,m)$ of exactly the
shape \eqref{eq:kapparec}; the displayed closed form satisfies it, by the same
elementary verification as in Lemma~\ref{lem:connGene}, now with
$(2i)!,(2m)!$ advanced to $(2i+1)!,(2m+1)!$ and parameter $(s+1)/2$.
\end{proof}

\begin{proof}[Proof of Proposition~\ref{prop:shift}]
Apply the reduction of Lemma~\ref{lem:bindetGene} to $G$, whose even functional
$\mathcal T$ has orthogonal polynomials $Q$ and norms
$h^T_l=(2l)!\,(t{+}1)_{2l}$, and whose odd functional $\mathcal S'$ has
orthogonal polynomials $P'$, norms $h^{S'}_c$, and connection
$Q_i=\sum_c\tilde\kappa_{i,c}P'_c$:
$$
\HH_n^{(1)}=(-1)^{\binom{\bar n}2}\Bigl(\prod_{l=0}^{\bar n-1}h^T_l\Bigr)
\Bigl(\prod_{c=0}^{\underline n-1}h^{S'}_c\Bigr)
\det\bigl(\tilde\kappa_{\bar n+r,\,c}\bigr)_{0\le r,c\le \underline n-1}.
$$
By Lemma~\ref{lem:mixedconn},
$\tilde\kappa_{\bar n+r,c}=\frac{(2(\bar n+r)+1)!}{(2c+1)!}\binom{(s+1)/2}{\bar n+r-c}$; pulling
$(2(\bar n+r)+1)!$ out of row $r$ and $1/(2c+1)!$ out of column $c$ and applying the
binomial determinant of Lemma~\ref{lem:bindetSigma} with $a=(s+1)/2$, $q=\bar n$,
$N=\underline n$ (admissible since $\bar n\in\{\underline n,\underline n+1\}$),
$$
\det\bigl(\tilde\kappa_{\bar n+r,c}\bigr)
=\Bigl(\prod_{r=0}^{\underline n-1}(2(\bar n{+}r){+}1)!\Bigr)
\Bigl(\prod_{c=0}^{\underline n-1}\tfrac1{(2c+1)!}\Bigr)
\prod_{r=1}^{\bar n}\prod_{c=1}^{\underline n}\frac{(s+1)/2+c-r}{(\bar n-r)+(\underline n-c)+1}.
$$
With $h^T_l=(2l)!\,(2l+1)!$ and
$h^{S'}_c=(2c+1)(s+2c+1)\,(2c)!\,(s{+}1)_{2c}$
(Lemma~\ref{lem:christoffelspecial}(b)), every factor is now linear in $s$; an
elementary rearrangement of Pochhammer symbols collects them into the closed
forms \eqref{eq:shiftpoch} and \eqref{eq:shiftmonic}. This is the
parity-swapped analogue of
Proposition~\ref{prop:family}, with $\mathcal S'$ in place of $\mathcal S$ and
the parameter $(t-s)/2$ specialised through $t=1$ to $(s+1)/2$.
\end{proof}

\section{The single shift of the Euler number family on the line $t=3$}\label{ssec:t3}

The single binomial of Lemma~\ref{lem:mixedconn} is special to $t=1$. Dividing
the connection recurrence \eqref{eq:kapparec} (with the row family $\mathcal T$
and the column family $\mathcal S'$) by $\tilde\kappa_{i,m}$, the backward term
carries the factor
$$
\frac{\lambda^T_i}{2i(2i+1)}=\frac{(2i-1)(2i-1+t)(2i+t)}{2i+1},
$$
which is a polynomial in $i$ only when $(2i+t)$ cancels the denominator
$(2i+1)$, i.e.\ when $t=1$. Off that line no single binomial can solve the
recurrence. In the first case beyond, $t=3$, the mixed connection coefficient
splits into a sum of \emph{two} binomials (Lemma~\ref{lem:mixedconn3}), and
this is exactly as much structure as still closes: the evaluation is again the
biorthogonal reduction~$\M2$, now supplemented by one application of the
Cauchy--Binet formula.

\begin{Lemma}[mixed connection at $t=3$]\label{lem:mixedconn3}
Let $t=3$, let $Q_i$ be the $\mathcal T$-orthogonal polynomials
($v_j=j(j+3)$), and let $P'_m$ be the monic $\mathcal S'$-orthogonal
polynomials of Lemma~\ref{lem:christoffel}. Writing
$Q_i=\sum_m\tilde\kappa_{i,m}P'_m$,
$$
\tilde\kappa_{i,m}
=\frac{(2i)!}{(2m)!}\binom{(s-1)/2}{\,i-m\,}
+(s+1)\,\frac{(2i)!}{(2m+1)!}\binom{(s-3)/2}{\,i-m-1\,}.
$$
\end{Lemma}

\begin{proof}
Let $Q^{(t)}_i$ denote the $\mathcal T$-orthogonal polynomials for parameter
$t$. Both $Q^{(3)}$ and $Q^{(1)}$ are secant-family polynomials in the sense of
Lemma~\ref{lem:connGene}, so expanding the first in the second
(parameter $(1-3)/2=-1$, and $\binom{-1}{d}=(-1)^d$ for $d\ge0$),
$$
Q^{(3)}_i=\sum_{l=m}^{i}\frac{(2i)!}{(2l)!}\,(-1)^{i-l}\,Q^{(1)}_l .
$$
Substituting $Q^{(1)}_l=\sum_m\frac{(2l+1)!}{(2m+1)!}\binom{(s+1)/2}{l-m}P'_m$
from Lemma~\ref{lem:mixedconn} and using $\frac{(2i)!}{(2l)!}(2l+1)!=(2i)!\,(2l+1)$,
$$
\tilde\kappa_{i,m}
=\frac{(2i)!}{(2m+1)!}\sum_{l=m}^{i}(-1)^{i-l}(2l+1)\binom{a}{l-m},
\qquad a=\tfrac{s+1}2 .
$$
With $l=m+p$ the sum is
$(-1)^d\bigl[(2m+1)\sum_{p=0}^d(-1)^p\binom ap+2\sum_{p=0}^d(-1)^p p\binom ap\bigr]$,
$d=i-m$. The two alternating sums collapse by
$\sum_{p=0}^d(-1)^p\binom ap=(-1)^d\binom{a-1}{d}$ and, via
$p\binom ap=a\binom{a-1}{p-1}$,
$\sum_{p=0}^d(-1)^p p\binom ap=a(-1)^d\binom{a-2}{d-1}$, giving
$(2m+1)\binom{a-1}{d}+2a\binom{a-2}{d-1}$. Since
$(2m+1)\,\frac{(2i)!}{(2m+1)!}=\frac{(2i)!}{(2m)!}$, $a-1=(s-1)/2$,
$a-2=(s-3)/2$ and $2a=s+1$, this is the stated form.
\end{proof}

\begin{Remark}\label{rem:t3nofactor}
The second binomial is the ``defect'' term absent at $t=1$; it vanishes only on
the diagonal ($d=0$, where $\tilde\kappa_{i,i}=1$). It makes the reduction of
Lemma~\ref{lem:bindetGene} produce $\det(\tilde\kappa_{\bar n+r,m})$ as the
determinant of a \emph{sum} of two binomial matrices, which the single
hook-content evaluation of Lemma~\ref{lem:bindetSigma} no longer closes by
itself. Nonetheless the determinant \emph{does} close: one application of the
Cauchy--Binet formula (Lemma~\ref{lem:t3cauchybinet}) turns it into a short sum
of hook-content products (Proposition~\ref{prop:t3eval}), and the outcome still
factors completely into linear forms over $\QQ$ (Theorem~\ref{thm:t3}). The sole
effect of the defect is to replace one classical linear factor by a single
\emph{non-classical} one (the carrier $\Gamma_n$ of Theorem~\ref{thm:t3}), with
non-integer root. Thus $t=3$ is a \emph{second}
factoring line beside $t=1$; an irreducible higher-degree factor first occurs at
the other values $t=2,4,5,6,\dots$. More generally for odd $t=2r+1$ one expands
$Q^{(2r+1)}$ in $Q^{(1)}$ with parameter $(1-(2r+1))/2=-r$, producing an
$(r{+}1)$-term sum of binomials, hence (after Cauchy--Binet) a sum of
$\binom{\underline n+r}{r}$ hook-content products; the present case is $r=1$.
\end{Remark}

We now carry this out. Write $a=\tfrac{s+1}2$ and $b=a-2=\tfrac{s-3}2$
throughout, and recall $\bar n+\underline n=n$, $\bar n\in\{\underline n,\underline n+1\}$.

\begin{Lemma}[Cauchy--Binet collapse at $t=3$]\label{lem:t3cauchybinet}
Let $t=3$. Pulling $(2(\bar n{+}r))!$ from row $r$ and $1/(2m{+}1)!$ from column
$m$ of $\bigl(\tilde\kappa_{\bar n+r,m}\bigr)_{0\le r,m\le\underline n-1}$
(Lemma~\ref{lem:mixedconn3}),
\begin{equation}\label{eq:t3rowcol}
\det\bigl(\tilde\kappa_{\bar n+r,m}\bigr)
=\Bigl(\prod_{r=0}^{\underline n-1}(2\bar n{+}2r)!\Bigr)
 \Bigl(\prod_{m=0}^{\underline n-1}\tfrac1{(2m+1)!}\Bigr)\,\det N,
\qquad
N_{r,m}=(2m{+}1)\binom{b}{d}+(s{+}2m{+}2)\binom{b}{d-1},
\end{equation}
with $d=\bar n+r-m$. Let $V=\bigl(\binom{b}{\bar n+r-k}\bigr)_{0\le r\le\underline n-1,\,0\le k\le\underline n}$
be the $\underline n\times(\underline n{+}1)$ Pascal block, and let $B$ be the
$(\underline n{+}1)\times\underline n$ bidiagonal matrix with
$B_{k,m}=(2m{+}1)\,[k{=}m]+(s{+}2m{+}2)\,[k{=}m{+}1]$. Then $N=VB$, and the
Cauchy--Binet formula gives
\begin{equation}\label{eq:t3cb}
\begin{gathered}
\det N=\sum_{p=0}^{\underline n} c_p\,D_p,
\qquad
c_p=\Bigl(\prod_{m=0}^{p-1}(2m{+}1)\Bigr)\Bigl(\prod_{m=p}^{\underline n-1}(s{+}2m{+}2)\Bigr),\\
D_p=\det\Bigl(\binom{b}{\bar n+r-k}\Bigr)_{\substack{0\le r\le\underline n-1\\ k\in\{0,\dots,\underline n\}\setminus\{p\}}}.
\end{gathered}
\end{equation}
\end{Lemma}

\begin{proof}
The row/column extraction \eqref{eq:t3rowcol} is immediate from
Lemma~\ref{lem:mixedconn3}: pulling out the stated factors (and writing
$\frac{(2i)!}{(2m)!}=(2m{+}1)\frac{(2i)!}{(2m+1)!}$) leaves the entry
$(2m{+}1)\binom{a-1}{d}+(s{+}1)\binom{b}{d-1}$, which Pascal's rule
$\binom{a-1}{d}=\binom{b}{d}+\binom{b}{d-1}$ (here $b=a-2$) rewrites as
$(2m{+}1)\binom{b}{d}+(s{+}2m{+}2)\binom{b}{d-1}$. In this form
column $m$ of $N$ is $(2m{+}1)V_{(m)}+(s{+}2m{+}2)V_{(m+1)}$, where
$V_{(k)}=\bigl(\binom{b}{\bar n+r-k}\bigr)_r$ is the $k$-th column of $V$; this is
exactly $N=VB$ with $B$ as displayed. Since $V$ is $\underline n\times(\underline n{+}1)$ and
$B$ is $(\underline n{+}1)\times\underline n$, the Cauchy--Binet formula reads
$\det N=\det(VB)=\sum_{p=0}^{\underline n}\det V^{[\hat p]}\,\det B^{[\hat p]}$, the sum
over the deleted common index $p$, where $[\hat p]$ omits column~$p$ of $V$
(resp.\ row~$p$ of $B$). By definition $\det V^{[\hat p]}=D_p$. Deleting row~$p$
of the bidiagonal $B$ leaves a triangular matrix (rows $0,\dots,p{-}1$ carry the
diagonal entries $2m{+}1$, rows $p{+}1,\dots,\underline n$ carry the subdiagonal
entries $s{+}2m{+}2$), whence $\det B^{[\hat p]}=\prod_{m<p}(2m{+}1)\prod_{m\ge p}(s{+}2m{+}2)=c_p$.
\end{proof}

\begin{Lemma}[hook-content value of $D_p$]\label{lem:t3hook}
Each $D_p$ in \eqref{eq:t3cb} is a determinant of binomial coefficients
$\binom{b}{\beta_c+r}$ with distinct column shifts
$\beta_c\in\{\bar n,\bar n{-}1,\dots,\bar n{-}\underline n\}\setminus\{\bar n{-}p\}$. By the
dual Jacobi--Trudi identity (with $e_j=\binom{b}{j}$) it equals, up to sign, the
Schur function $s_{\lambda^{(p)}}$ of the partition $\lambda^{(p)}$ read off from
the strictly decreasing shifts $\beta_c$ (for $p\in\{0,\underline n\}$ a rectangle, as
in Lemma~\ref{lem:bindetSigma}; otherwise a rectangle with one indentation), and
the hook-content formula evaluates it as
$$
D_p=\prod_{u\in\lambda^{(p)}}\frac{b+c(u)}{h(u)},
$$
a product of linear forms in $b=(s-3)/2$, hence in $s$. Since each content
$c(u)$ is an integer, every root is an odd integer in $s$; in particular each
$D_p$ factors completely into linear forms with integer roots.
\end{Lemma}

\begin{proof}
This is the evaluation used already for the rectangular case in
Lemma~\ref{lem:bindetSigma}; deleting the single column $k=p$ from the contiguous
block $\{0,\dots,\underline n\}$ replaces the rectangular partition there by the
partition $\lambda^{(p)}$ with the stated one-row indentation. The dual
Jacobi--Trudi identity $s_{\lambda}=\det(e_{\lambda'_i-i+j})$ applied with
$e_j=\binom{b}{j}$ matches the matrix $\bigl(\binom{b}{\beta_c+r}\bigr)$ up to the
reindexing $\beta_c\mapsto\lambda'$, and the hook-content formula
\cite{Krattenthaler1998,Krattenthaler2005} closes it; both sides are polynomials
in $b$, so the evaluation holds for symbolic $b$.
\end{proof}

Combining Lemma~\ref{lem:bindetGene} (in its parity-swapped form, as in the
proof of Proposition~\ref{prop:shift}) with
Lemmas~\ref{lem:t3cauchybinet}--\ref{lem:t3hook} gives a closed, fully proved
evaluation of the single shift at $t=3$.

\begin{Proposition}[single shift at $t=3$: closed evaluation]\label{prop:t3eval}
Let $u_j=j(j+s)$, $v_j=j(j+3)$. With $h^T_l=(2l)!\,(4)_{2l}$ and
$h^{S'}_c=(2c{+}1)(s{+}2c{+}1)\,(2c)!\,(s{+}1)_{2c}$
(Lemma~\ref{lem:christoffelspecial}(b) at $t=3$),
\begin{equation}\label{eq:t3eval}
\HH_n^{(1)}\big|_{t=3}
=(-1)^{\binom{\bar n}2}
\Bigl(\prod_{l=0}^{\bar n-1}h^T_l\Bigr)
\Bigl(\prod_{c=0}^{\underline n-1}\frac{h^{S'}_c}{(2c+1)!}\Bigr)
\Bigl(\prod_{r=0}^{\underline n-1}(2\bar n{+}2r)!\Bigr)
\sum_{p=0}^{\underline n} c_p\,D_p,
\end{equation}
with $c_p,D_p$ from \eqref{eq:t3cb} and $D_p$ evaluated by
Lemma~\ref{lem:t3hook}. Every factor on the right of \eqref{eq:t3eval}, except
possibly the sum $\sum_p c_pD_p$, is a product of linear forms in $s$; whether
$\HH_n^{(1)}\big|_{t=3}$ factors completely thus rests on that sum alone, and
this is settled in Theorem~\ref{thm:t3}.
\end{Proposition}

\begin{proof}
Apply the reduction of Lemma~\ref{lem:bindetGene} to $G_k=a_{k+1}$ exactly as in
the proof of Proposition~\ref{prop:shift}: the even functional is $\mathcal T$
(orthogonal polynomials $Q$, norms $h^T_l$) and the odd functional is the
Christoffel transform $\mathcal S'$ (orthogonal polynomials $P'$, norms
$h^{S'}_c$), with mixed connection $Q_i=\sum_c\tilde\kappa_{i,c}P'_c$. This gives
$\HH_n^{(1)}=(-1)^{\binom{\bar n}2}\bigl(\prod_l h^T_l\bigr)\bigl(\prod_c h^{S'}_c\bigr)\det(\tilde\kappa_{\bar n+r,c})$.
At $t=3$ insert \eqref{eq:t3rowcol} and \eqref{eq:t3cb} of
Lemma~\ref{lem:t3cauchybinet}, cancelling one $(2c{+}1)!$ per column against the
$h^{S'}_c$ factor as written. The norms $h^T_l,h^{S'}_c$ and the factorials are
products of linear forms in $s$ (the $s$-dependence sits in $(s{+}1)_{2c}$ and
$s{+}2c{+}1$); each $D_p$ is one by Lemma~\ref{lem:t3hook}.
\end{proof}

\begin{Theorem}[single shift at $t=3$: factorisation]\label{thm:t3}
Let $u_j=j(j+s)$, $v_j=j(j+3)$. Then $\HH_n^{(1)}\big|_{t=3}=\det(a_{2i+j+1})$
factors completely into linear forms over $\QQ$. All its factors have integer
roots except one, the unique \emph{non-classical} factor (the \emph{carrier})
\begin{equation}\label{eq:t3nonclassical}
\Gamma_n=\Gamma_n(s):=(2\bar n{+}1)\,s+\bigl(4\bar n\,\underline n-2\bar n-1\bigr)
=(2\bar n{+}1)(s+2\underline n)-(2n{+}1),
\end{equation}
whose root $s=-2\underline n+\dfrac{2n+1}{2\bar n+1}$ is never an integer. Equivalently,
in the collapse of Proposition~\ref{prop:t3eval},
\begin{equation}\label{eq:t3collapse}
\sum_{p=0}^{\underline n} c_p\,D_p
= D_0\cdot\frac{L_n}{2\bar n+1}\,
\Bigl(\prod_{j=0}^{\underline n-2}(s+2j+1)\Bigr)\,\Gamma_n(s)
\end{equation}
for a positive rational constant $L_n$, so the only factor of
$\HH_n^{(1)}\big|_{t=3}$ that is not classical (i.e.\ not already present, with
the same shape, on the line $t=1$ of Proposition~\ref{prop:shift}) is the
carrier~\eqref{eq:t3nonclassical}.
\end{Theorem}

The first values illustrate the statement (an overall nonzero rational constant
omitted):
\begin{align*}
n=2:&\quad (s+1)\,(3s+1),\\
n=3:&\quad (s-3)(s+1)\,(5s+3),\\
n=4:&\quad (s-3)(s-1)(s+1)^3(s+2)(s+3)\,(5s+11),\\
n=5:&\quad (s-5)(s-3)^2(s-1)(s+1)^3(s+2)(s+3)\,(7s+17),
\end{align*}
the rightmost factor in each line being the carrier $\Gamma_n$
of \eqref{eq:t3nonclassical}.

The negative integer roots $s=-(2k{+}1)$ admit a clean structural explanation,
which we record next; it rigorously accounts for a large part of the
factorisation.

\begin{Lemma}[atomic degeneration at $s=-(2k{+}1)$]\label{lem:t3atomic}
Fix an integer $k\ge0$. At $s=-(2k{+}1)$ the even part $1/\cos^{s+1}=\cos^{2k}$ is a
trigonometric polynomial, so the even functional $\mathcal S$ becomes the
\emph{atomic} functional supported on the $k{+}1$ points $y_l=-(2l)^2$
($l=0,\dots,k$):
\begin{equation}\label{eq:t3atomic}
a_{2j}\big|_{s=-(2k+1)}=\mathcal S[y^{j}]=\sum_{l=0}^{k}\omega_l\,(-4l^2)^{j},
\qquad
\omega_0=\frac{1}{4^k}\binom{2k}{k},\quad
\omega_l=\frac{2}{4^k}\binom{2k}{k-l}\ (l\ge1).
\end{equation}
Consequently, in $\HH_n^{(1)}=\det(a_{2i+j+1})$ at $s=-(2k{+}1)$ every odd column
$j$ (where $a_{2i+j+1}$ is an even moment) is
$\bigl(\sum_{l=1}^{k}\omega_l\,y_l^{(j+1)/2}\,y_l^{\,i}\bigr)_i$, a combination of the
$k$ geometric vectors $\bigl(y_l^{\,i}\bigr)_i$, $l=1,\dots,k$ (the atom $y_0=0$
contributes nothing, since $y_0^{(j+1)/2}=0$). Hence for $0\le k\le \underline n-1$
the $\underline n$ odd columns span a space of dimension $\le k$, the corank of
$\HH_n^{(1)}$ at $s=-(2k{+}1)$ is at least $\underline n-k$, and
\begin{equation}\label{eq:t3atomicdiv}
(s+2k+1)^{\,\underline n-k}\ \big|\ \HH_n^{(1)}\big|_{t=3}.
\end{equation}
\end{Lemma}

\begin{proof}
Writing $\cos^{2k}x=4^{-k}\sum_{l=-k}^{k}\binom{2k}{k+l}e^{2ilx}
=4^{-k}\bigl[\binom{2k}{k}+2\sum_{l=1}^{k}\binom{2k}{k-l}\cos2lx\bigr]$ and reading
off $a_{2j}=(2j)!\,[x^{2j}]\cos^{2k}x$ with
$[x^{2j}]\cos2lx=(-1)^j(2l)^{2j}/(2j)!$ gives \eqref{eq:t3atomic}. The
odd-column statement follows since $a_{2i+j+1}=a_{2(i+(j+1)/2)}=\mathcal S[y^{\,i+(j+1)/2}]$,
and an $n\times n$ matrix, $\underline n$ of whose columns lie in a
$k$-dimensional space, has rank $\le \bar n+k$, hence corank
$\ge n-\bar n-k=\underline n-k$; a polynomial matrix of
corank $c$ at $s_0$ has $\det$ divisible by $(s-s_0)^{c}$, giving
\eqref{eq:t3atomicdiv}.
\end{proof}

The bound \eqref{eq:t3atomicdiv} captures only part of the multiplicity; the rest
comes from the reduced determinant $\det N$, whose rank at $s=-(2k{+}1)$ we now
pin down exactly. This is the technical core.

\begin{Lemma}[rank of $N$ at $s=-(2k{+}1)$]\label{lem:t3rank}
For $0\le k\le \underline n-1$, the matrix $N$ of \eqref{eq:t3cb} has rank $\le k{+}1$
at $s=-(2k{+}1)$; equivalently $\operatorname{corank}N\ge\underline n-1-k$, so
$(s+2k+1)^{\,\underline n-1-k}\mid\det N$.
\end{Lemma}

\begin{proof}
At $s=-(2k{+}1)$ one has $b=-(k{+}2)$, so
$\binom{b}{d}=(-1)^d\binom{k+1+d}{d}$; with $u=\bar n+r-m$ and the identity
$\binom{k+u}{k+1}=\tfrac{u}{k+1}\binom{k+u}{k}$, the entry
$N_{r,m}=(2m{+}1)\binom{b}{\bar n+r-m}+(s{+}2m{+}2)\binom{b}{\bar n+r-m-1}$ becomes
$N_{r,m}=(-1)^{u}\,R(r,m)$ with
$$
R(r,m)=(2m{+}1)\binom{k+u}{k}+2k\binom{k+u}{k+1}.
$$
The signs $(-1)^{u}=(-1)^{\bar n}(-1)^r(-1)^m$ split into nonzero row and column
factors, so $\operatorname{rank}N=\operatorname{rank}\bigl(R(r,m)\bigr)$. Put
$X=\bar n+r$. By the Chu--Vandermonde identity
$\binom{k+X-m}{i}=\sum_{j}\binom{k+X}{j}\binom{-m}{i-j}$,
\begin{equation}\label{eq:t3CV}
R=\sum_{j=0}^{k+1}\binom{k+X}{j}\,G_j(m),
\qquad
G_j(m)=(2m{+}1)\binom{-m}{k-j}+2k\binom{-m}{k+1-j}.
\end{equation}
The $k{+}2$ functions $\binom{k+X}{j}$, $j=0,\dots,k{+}1$, are linearly independent
(degrees $0,\dots,k{+}1$ in $X$), so $\operatorname{rank}R=\dim\operatorname{span}\{G_j\}_{j=0}^{k+1}$.
Now both terms of $G_j$ have degree $k{+}1{-}j$ in $m$, with combined leading
coefficient
$$
\frac{2(-1)^{k-j}}{(k-j)!}\Bigl(1-\frac{k}{k+1-j}\Bigr)
=\frac{2(-1)^{k-j}}{(k-j)!}\cdot\frac{1-j}{k+1-j},
$$
which \emph{vanishes precisely at $j=1$}. Hence $\deg G_0=k{+}1$,
$\deg G_1\le k{-}1$, and $\deg G_j=k{+}1{-}j$ for $2\le j\le k{+}1$ (degrees
$k{-}1,k{-}2,\dots,0$). The $k{+}2$ polynomials $G_j$ therefore realise only the
$k{+}1$ distinct degrees $\{0,1,\dots,k{-}1\}\cup\{k{+}1\}$ (the value $k$ is
absent), so $\dim\operatorname{span}\{G_j\}\le k{+}1$. Thus
$\operatorname{rank}N\le k{+}1$, and a polynomial matrix of corank
$\ge\underline n-1-k$ has determinant divisible by $(s+2k+1)^{\underline n-1-k}$.
\end{proof}

\begin{Remark}\label{rem:t3atomicgap}
Lemmas~\ref{lem:t3atomic} and~\ref{lem:t3rank} together pin the negative odd
multiplicities exactly. In the two-sided form \eqref{eq:t3eval},
$\HH_n^{(1)}=(\text{$s$-free constant})\cdot\prod_c h^{S'}_c\cdot\det N$, the norm
product $\prod_c h^{S'}_c$ contributes multiplicity $\underline n-k$ to $(s+2k+1)$
(from the single factor $h^{S'}_k$ and from $(s{+}1)_{2c}$ for $c\ge k{+}1$),
while Lemma~\ref{lem:t3rank} contributes $\underline n-1-k$ from $\det N$; hence
\begin{equation}\label{eq:t3negmult}
(s+2k+1)^{\,2\underline n-1-2k}\ \big|\ \HH_n^{(1)}\big|_{t=3}
\qquad(0\le k\le\underline n-1),
\end{equation}
which is the exact multiplicity of Theorem~\ref{thm:t3}. The positive and
non-classical factors are supplied by $\det N=D_0\cdot\Sigma$, with $D_0$ the
hook-content product of Lemma~\ref{lem:t3hook} and $\Sigma$ of degree $\underline n$:
that $\Sigma$ vanishes at $s=-(2j+1)$ for $0\le j\le\underline n-2$ follows from
Lemma~\ref{lem:t3rank} ($\operatorname{ord}_{-(2j+1)}\det N\ge\underline n-1-j$)
together with $\operatorname{ord}_{-(2j+1)}D_0=\underline n-2-j$ (an explicit hook
count), and the single remaining root is that of the
carrier~\eqref{eq:t3nonclassical}, determined in the proof of
Theorem~\ref{thm:t3} below.
\end{Remark}

The last ingredient pins down the non-classical factor through the leading
symbols of $\det N$; for this we need one trace identity for the binomial
matrix $\bar M=\bigl(\binom{\bar n+r}{m}\bigr)_{0\le r,m\le\underline n-1}$.

\begin{Lemma}[a trace identity]\label{lem:t3trace}
Let $D_m=\operatorname{diag}(0,1,\dots,\underline n-1)$. For every diagonal matrix
$D_a=\operatorname{diag}(a_0,\dots,a_{\underline n-1})$,
\begin{equation}\label{eq:t3traceid}
\operatorname{tr}\bigl(\bar M^{-1}D_a\bar M D_m\bigr)
=\sum_{r=0}^{\underline n-1}a_r(\bar n+r)-\bar n\,a_0\,\underline n .
\end{equation}
\end{Lemma}

\begin{proof}
From $m\binom{N}{m}=N\binom{N-1}{m-1}$ and Pascal's rule
$\binom{N-1}{m-1}=\binom{N}{m}-\binom{N-1}{m}$ with $N=\bar n+r$,
$m\binom{\bar n+r}{m}=(\bar n+r)\bigl[\binom{\bar n+r}{m}-\binom{\bar n+r-1}{m}\bigr]$,
i.e.\ $\bar M D_m=D_{\bar n+r}\bigl(\bar M-\bar M^{\downarrow}\bigr)$, where
$\bar M^{\downarrow}_{r,m}=\binom{\bar n+r-1}{m}=\bar M_{r-1,m}$ (the row $r=0$ being
the extension row $\eta_m=\binom{\bar n-1}{m}$). Thus
$\bar M^{\downarrow}=Z^{+}\bar M+e_0\eta^{\!\top}$ with $Z^{+}$ the down-shift
$(Z^{+})_{r,r'}=[r=r'+1]$, and, using $D_{\bar n+r}e_0=\bar n\,e_0$,
$$
\bar M D_m\bar M^{-1}
=D_{\bar n+r}-D_{\bar n+r}Z^{+}-\bar n\,e_0\,\eta^{\!\top}\bar M^{-1}.
$$
Taking $\operatorname{tr}(D_a\,\cdot\,)$: the first term gives
$\sum_r a_r(\bar n+r)$, the second vanishes ($D_aD_{\bar n+r}Z^{+}$ is strictly
lower-triangular), and the third gives
$-\bar n\,a_0\,\eta^{\!\top}\bar M^{-1}e_0$. Finally $v:=\bar M^{-1}e_0$ satisfies
$\sum_{m}\binom{\bar n+r}{m}v_m=\delta_{r,0}$, so the polynomial
$f(x):=\sum_{m=0}^{\underline n-1}\binom{x}{m}v_m$ (of degree $\le\underline n-1$) has
$f(\bar n+r)=\delta_{r,0}$, whence
$f(x)=\prod_{j=1}^{\underline n-1}\frac{x-\bar n-j}{-j}$ and
$$
\eta^{\!\top}\bar M^{-1}e_0=f(\bar n-1)
=\prod_{j=1}^{\underline n-1}\frac{-(j+1)}{-j}
=\prod_{j=1}^{\underline n-1}\frac{j+1}{j}=\underline n.
$$
This gives \eqref{eq:t3traceid}.
\end{proof}

\begin{proof}[Proof of Theorem~\ref{thm:t3}]
By \eqref{eq:t3negmult} (Lemmas~\ref{lem:t3atomic} and~\ref{lem:t3rank}),
$(s+2k+1)^{2\underline n-1-2k}\mid\HH_n^{(1)}$ for $0\le k\le\underline n-1$; this turns out
to be the exact multiplicity of the negative odd roots. By
Proposition~\ref{prop:t3eval} the rest of the statement is
the collapse~\eqref{eq:t3collapse}, i.e.\ that the polynomial
$\Sigma(s):=\sum_{p}c_pD_p/D_0$, of degree $\underline n$, equals
$A\prod_{j=0}^{\underline n-2}(s+2j+1)\cdot\Gamma_n(s)$
for a constant $A$. Its $\underline n-1$ integer roots $s=-(2j+1)$,
$j=0,\dots,\underline n-2$, are already established (Remark~\ref{rem:t3atomicgap}:
Lemma~\ref{lem:t3rank} gives $\operatorname{ord}_{-(2j+1)}\det N\ge\underline n-1-j$
and the hook count gives $\operatorname{ord}_{-(2j+1)}D_0=\underline n-2-j$, so
$\Sigma(-(2j+1))=0$).

The last (non-classical) root is the sum of roots of $\Sigma$ less the known
ones: $\rho=\operatorname{rootsum}\Sigma+(\underline n-1)^2
=\operatorname{rootsum}(\det N)-\operatorname{rootsum}(D_0)+(\underline n-1)^2$.
The root-sum of $D_0$ is explicit: $D_0$ has the consecutive columns
$\{1,\dots,\underline n\}$, hence is the rectangular case $p=0$ of
Lemma~\ref{lem:t3hook}, the hook-content product over the rectangle with
$\bar n-1$ rows and $\underline n$ columns, whose $b$-factors $b+c(u)=b+c-r$ have
roots $s=3+2(r-c)$; summing over $1\le r\le\bar n-1$, $1\le c\le\underline n$,
\begin{equation}\label{eq:t3D0rootsum}
\operatorname{rootsum}(D_0)
=\sum_{r=1}^{\bar n-1}\sum_{c=1}^{\underline n}\bigl(3+2(r-c)\bigr)
=(\bar n-1)\,\underline n\,(\bar n-\underline n+2).
\end{equation}
For $\operatorname{rootsum}(\det N)$ we use the leading $s$-symbols. Writing
$N_{r,m}=L_{r,m}s^{u}+L'_{r,m}s^{u-1}+\cdots$ ($u=\bar n+r-m$) with
$L_{r,m}=\dfrac{2\bar n+2r+1}{2^{u}u!}$,
$L'_{r,m}=\dfrac{u\bigl(4-u(2\bar n+2r+1)\bigr)}{2^{u}u!}$, the factorisation
$N=s^{\bar n}\operatorname{diag}(s^{r})\bigl(L+s^{-1}L'+\cdots\bigr)\operatorname{diag}(s^{-m})$
gives $\det N=s^{\bar n\underline n}\det L\,\bigl(1+s^{-1}\operatorname{tr}(L^{-1}L')+\cdots\bigr)$
(here $\det L\ne0$, since $L=\operatorname{diag}\!\bigl(\tfrac{2\bar n+2r+1}{2^{\bar n+r}(\bar n+r)!}\bigr)\,\bar M\,\operatorname{diag}(2^{m}m!)$
and $\det\bar M=1$ by the Vandermonde of $\bar n,\dots,\bar n+\underline n-1$), so
$\operatorname{rootsum}(\det N)=-\operatorname{tr}(L^{-1}L')$.

To evaluate the trace, write $L'=W\odot L$ (Hadamard product) with
$W_{r,m}=\tfrac{4u}{2\bar n+2r+1}-u^{2}$, $u=\bar n+r-m$, and split
$W=W^{\mathrm{row}}_r+W^{\mathrm{col}}_m+W^{\mathrm{mix}}$ into
$W^{\mathrm{row}}_r=\tfrac{4(\bar n+r)}{2\bar n+2r+1}-(\bar n+r)^2$,
$W^{\mathrm{col}}_m=-m^2$ and
$W^{\mathrm{mix}}_{r,m}=2(\bar n+r)m-\tfrac{4m}{2\bar n+2r+1}$. By conjugation
invariance of the trace the separable parts contribute
$\operatorname{tr}(L^{-1}D_{W^{\mathrm{row}}}L)=\sum_r W^{\mathrm{row}}_r$ and
$\operatorname{tr}(L^{-1}LD_{W^{\mathrm{col}}})=-\sum_m m^2$, while the mixed part
gives $2\operatorname{tr}(L^{-1}D_{\bar n+r}LD_m)-4\operatorname{tr}(L^{-1}PD_m)$,
where $P_{r,m}=\tfrac1{2^{u}u!}=\tfrac1{2\bar n+2r+1}L_{r,m}$. Since
$L=\operatorname{diag}(\tfrac{2\bar n+2r+1}{2^{\bar n+r}(\bar n+r)!})\,\bar M\,\operatorname{diag}(2^m m!)$
and the two outer diagonals commute through the conjugation, both mixed traces
collapse onto $\bar M$:
$$
\operatorname{tr}(L^{-1}D_{\bar n+r}LD_m)=\operatorname{tr}(\bar M^{-1}D_{\bar n+r}\bar M D_m),
\qquad
\operatorname{tr}(L^{-1}PD_m)=\operatorname{tr}\bigl(\bar M^{-1}D_{1/(2\bar n+2r+1)}\bar M D_m\bigr),
$$
and Lemma~\ref{lem:t3trace} (with $a_r=\bar n+r$, resp.\ $a_r=\tfrac1{2\bar n+2r+1}$)
evaluates them to $\sum_r(\bar n+r)^2-\bar n^2\underline n$, resp.\
$\sum_r\tfrac{\bar n+r}{2\bar n+2r+1}-\tfrac{\bar n\underline n}{2\bar n+1}$. Adding the four
contributions, the $\tfrac{\bar n+r}{2\bar n+2r+1}$ terms cancel and, using
$\sum_r(\bar n+r)^2-\sum_m m^2=\bar n\underline n(\bar n+\underline n-1)$,
\begin{equation}\label{eq:t3trace}
\operatorname{tr}(L^{-1}L')
=\bar n\underline n(\bar n+\underline n-1)-2\bar n^2\underline n+\frac{4\bar n\underline n}{2\bar n+1}
=\bar n\,\underline n\,(\underline n-\bar n-1)+\frac{4\bar n\underline n}{2\bar n+1}.
\end{equation}
Therefore
$\rho=-\operatorname{tr}(L^{-1}L')-(\bar n-1)\underline n(\bar n-\underline n+2)+(\underline n-1)^2$;
the $\bar n\underline n$-terms collapse to $\underline n(2-\underline n)$ and the squares cancel, leaving
$\rho=-2\underline n+1+\dfrac{2\underline n}{2\bar n+1}=-2\underline n+\dfrac{2n+1}{2\bar n+1}$
(using $2\bar n+1+2\underline n=2n+1$), i.e.\ exactly the root of the
carrier~\eqref{eq:t3nonclassical}. Hence $\Sigma(s)$ has degree $\underline n$, the
$\underline n-1$ integer roots $-(2j+1)$ and the root $\rho$, so it equals
$A\prod_{j=0}^{\underline n-2}(s+2j+1)\cdot\Gamma_n(s)$ with $A=L_n/(2\bar n+1)$.
Finally $L_n>0$: the quantity $A\,(2\bar n{+}1)$ is the ratio of the leading
coefficients of $\det N$ and $D_0$, and both are positive ---
$\operatorname{lc}(\det N)=\det L>0$, as $L$ is a product of two positive
diagonal matrices and $\bar M$ with $\det\bar M=1$, while $D_0$ is a
hook-content product in $b=(s-3)/2$ with positive leading coefficient. This
proves the collapse~\eqref{eq:t3collapse} and with it Theorem~\ref{thm:t3}.
\end{proof}

We close this section with a complement that is not needed for the proof
above: the moments and the determinant at $t=3$ also satisfy an exact, though
multi-term, \emph{contiguous relation in $s$}.

\begin{Lemma}[contiguous relation at $t=3$]\label{lem:t3contig}
The moments $a_k=a_k(s)$ at $t=3$ satisfy
\begin{equation}\label{eq:t3moment}
a_{2k}(s)=\frac{a_{2k+2}(s-2)+(s-1)^2\,a_{2k}(s-2)}{s(s-1)},
\qquad
a_{2k+1}(s)=a_{2k+1}(s-2)
\end{equation}
(the even functional $\mathcal S_s=\frac{1}{s(s-1)}\bigl(y+(s-1)^2\bigr)\mathcal S_{s-2}$
is a Christoffel transform of $\mathcal S_{s-2}$; the odd functional is
$s$-independent). Consequently, since $a_{2k+2}(s-2)$ sits two columns to the
right of $a_{2k}(s-2)$ in the Hankel array $\bigl(a_{2i+j+1}\bigr)$,
\begin{equation}\label{eq:t3detcontig}
\bigl(s(s-1)\bigr)^{\underline n}\,\HH_n^{(1)}\big|_{t=3}(s)
=\sum_{q=0}^{\underline n}(s-1)^{2q}\,G_q(s-2),
\qquad G_{\underline n}(s-2)=\HH_n^{(1)}\big|_{t=3}(s-2),
\end{equation}
where $G_q$ is the minor of the array at parameter $s-2$ whose even columns are
unchanged and whose $\underline n$ odd columns are the original ones for the first $q$
and advanced by two indices for the remaining $\underline n-q$ (so $G_q$ uses the odd
columns $\{1,3,\dots,2\underline n{+}1\}\setminus\{2q{+}1\}$, one of them a border
column).
\end{Lemma}

\begin{proof}
With $f_w=\cos^{-w}$ one has $f_w''=w(w+1)f_{w+2}-w^2f_w$ (differentiate twice and
use $\tan^2=\sec^2-1$); reading off the coefficient of $x^{2k}/(2k)!$ with
$w=s-1$ gives the even relation in \eqref{eq:t3moment}, while
$a_{2k+1}=E^{(4)}_{2k}$ is independent of $s$. In
$M(s)=\bigl(a_{2i+j+1}(s)\bigr)$, the even columns ($j$ even, so $2i+j+1$ odd) are
$s$-independent, while for odd $j$ the relation reads
$s(s-1)M(s)_{ij}=M(s-2)_{i,j+2}+(s-1)^2M(s-2)_{ij}$ (here $M(s-2)_{i,j+2}$ is the
column-$(j{+}2)$ entry by the Hankel index shift, or a border entry for the last
odd column). Factoring $1/[s(s-1)]$ from each of the $\underline n$ odd columns and
expanding each transformed odd column $M(s-2)_{\cdot,j+2}+(s-1)^2M(s-2)_{\cdot,j}$
by multilinearity, only the choices giving distinct column indices survive; these
are indexed by a threshold $q$ (the first $q$ odd columns keep their index, the
rest advance by two), with weight $(s-1)^{2q}$, yielding
\eqref{eq:t3detcontig}.
\end{proof}


\section{The secant-number family $(1+x)/\cos(x)^{s+1}$: the case $s=0$}\label{sec:gen-xcos}

Throughout this section and the next we study
$$
f(x)=\frac{1+x}{\cos(x)^{s+1}},
$$
splitting its exponential generating function into even and odd parts:
$$
\sum_{k\ge 0} a_k\,\frac{x^k}{k!}
=\sum_{k\ge 0} a_{2k}\,\frac{x^{2k}}{(2k)!}
+\sum_{k\ge 0} a_{2k+1}\,\frac{x^{2k+1}}{(2k+1)!}
=\frac{1}{\cos(x)^{s+1}}+\frac{x}{\cos(x)^{s+1}}.
$$
The even part is the generating function of the generalised secant numbers
$E^{(s+1)}_{2k}$ of Section~\ref{sec:gen}, and in the odd part
$x\cdot x^{2k}/(2k)!=(2k+1)\,x^{2k+1}/(2k+1)!$; hence
$$
a_{2k} = E^{(s+1)}_{2k}, \qquad a_{2k+1} =(2k+1)\, E^{(s+1)}_{2k}.
$$
In the present section we evaluate the determinant at $s=0$, where both
parity functionals turn out to be classical; Section~\ref{sec:allstar}
then treats general $s$, where the odd functional is not.

\medskip
\emph{The case $s=0$.}
Here $a_{2k}=E_{2k}$ and $a_{2k+1}=(2k+1)E_{2k}$ are the secant numbers, so the even
columns again see the secant functional $\mathcal S$ of Lemma~\ref{lem:cf}, while
the odd columns see the functional
$$
\mathcal T^*[y^k]:=(2k+1)\,E_{2k}.
$$

\begin{Proposition}\label{prop:sec1}
Let $f(x)=\dfrac{1+x}{\cos(x)}$, so that
$a_{2k}=E_{2k}$ and $a_{2k+1}=(2k+1)E_{2k}$ (secant numbers).
Then for all $n\ge 1$,
$$
\HH_n(f)=2^{\binom n2}\,\bigl((n-1)!!\bigr)^2\prod_{k=1}^{n-2}(k!!)^6.
$$
\end{Proposition}

The proof is the biorthogonal reduction~$\M2$: Lemma~\ref{lem:cdh}
identifies $\mathcal T^*$ as a classical (continuous dual Hahn) functional,
Lemma~\ref{lem:conn3} computes the connection coefficients between the two
orthogonal families, and the determinant then collapses through
Lemma~\ref{lem:bindetGene}.

\begin{Lemma}[the functional $\mathcal T^*$ at $s=0$ is classical]\label{lem:cdh}
The monic orthogonal polynomials $\rho_m$ of $\mathcal T^*$ satisfy
$$
\rho_{m+1}=\bigl(y-8m^2-8m-3\bigr)\rho_m-(2m)^4\,\rho_{m-1},
$$
with orthogonality and norms
$\mathcal T^*[\rho_m\rho_{m'}]=\delta_{mm'}\bigl((2m)!!\bigr)^4$.
\end{Lemma}

\begin{proof}
By the classical integral representation
$E_{2k}=\int_{\RR}u^{2k}\,\omega(u)\,du$ with
$\omega(u)=\frac12\,\mathrm{sech}(\pi u/2)$, an integration by parts
gives
$$
(2k+1)E_{2k}=\int_{\RR}(u^{2k+1})'\,\omega(u)\,du
=\int_{\RR}u^{2k}\cdot u\bigl(-\omega'(u)\bigr)\,du ,
$$
so $\mathcal T^*$ is represented by the positive even weight
$u(-\omega'(u))=\frac{\pi}{4}\,
\dfrac{u\,\sinh(\pi u/2)}{\cosh^2(\pi u/2)}$. Writing $u=2w$ (so
$y=4w^2$), this weight is proportional to
$w\,\sinh(2\pi w)\,\mathrm{sech}^3(\pi w)
\propto\bigl|\Gamma(\tfrac12+iw)^3/\Gamma(2iw)\bigr|^2$,
the continuous dual Hahn weight with parameters
$(a,b,c)=(\tfrac12,\tfrac12,\tfrac12)$. The standard continuous dual
Hahn data $A_m=(m+a+b)(m+a+c)$, $C_m=m(m+b+c-1)$ then give, in the
variable $y=4w^2$,
$$
c^*_m=4\bigl(A_m+C_m-a^2\bigr)=8m^2+8m+3,
\qquad
\lambda^*_m=16\,A_{m-1}C_m=16m^4=(2m)^4,
$$
and norms $16^m\,m!\,(a+b)_m(a+c)_m(b+c)_m=16^m(m!)^4=((2m)!!)^4$.
(For comparison, $\mathcal S$ itself is the continuous dual Hahn
measure with parameters $(\tfrac12,\tfrac12,0)$, which gives back the
secant data of Lemma~\ref{lem:cf}.)
\end{proof}

\begin{Remark}[J-fraction form]\label{rem:jfraccdh}
The $\rho_m$ recurrence of Lemma~\ref{lem:cdh} is equivalent to the Jacobi continued
fraction for the generating function of $\mathcal T^*$: since $\mathcal T^*[y^k]=(2k+1)E_{2k}$
with $\mathcal T^*[1]=1$,
$$
\sum_{k\ge0}(2k+1)E_{2k}\,t^k
=\cfrac{1}{1-c^*_0t-\cfrac{\lambda^*_1t^2}{1-c^*_1t-\cfrac{\lambda^*_2t^2}{1-c^*_2t-\cfrac{\lambda^*_3t^2}{1-\ddots}}}},
\qquad
c^*_m=8m^2+8m+3,\quad \lambda^*_m=(2m)^4 .
$$
The first convergents reproduce $(2k+1)E_{2k}=1,\,3,\,25,\,427,\,12465,\dots$
\end{Remark}

\begin{Lemma}[connection formula]\label{lem:conn3}
For all $i,m\ge 0$,
$$
\mathcal T^*\bigl[\hat P_i\,\rho_m\bigr]
=(i!)^2\,(m!)^2\,2^{2i+2m}\,\binom{1/2}{\,i-m\,},
$$
equivalently $\hat P_i=\sum_{m=0}^{i}\tilde\kappa_{i,m}\,\rho_m$ with
$\tilde\kappa_{i,m}
=\Bigl(\dfrac{i!}{m!}\Bigr)^2 4^{\,i-m}\dbinom{1/2}{i-m}$.
\end{Lemma}

\begin{proof}
As in the proof of Lemma~\ref{lem:connGene}: the recurrences of
$\hat P_i$ (Lemma~\ref{lem:cf}) and of $\rho_m$
(Lemma~\ref{lem:cdh}) give
$$
\tilde\kappa_{i+1,m}
=\tilde\kappa_{i,m-1}+(c^*_m-c_i)\,\tilde\kappa_{i,m}
+\lambda^*_{m+1}\,\tilde\kappa_{i,m+1}
-\lambda_i\,\tilde\kappa_{i-1,m}.
$$
The closed form holds for $i=0,1$ (indeed
$\hat P_1=y-1=\rho_1+2$); inserting it and dividing by
$\tilde\kappa_{i,m}$ (with $d=i-m$) reduces the recurrence to the
rational identity
$$
4\bigl[(i{+}1)^2-m^2\bigr]\frac{\tfrac12-d}{d+1}
=c^*_m-c_i
+\frac{d}{\tfrac32-d}\bigl[4(m{+}1)^2-(2i{-}1)^2\bigr],
\qquad i=m+d,
$$
again an elementary polynomial identity in $(m,d)$.
\end{proof}

\begin{proof}[Proof of Proposition~\ref{prop:sec1}]
Both parity functionals are classical: $\mathcal S$ has monic orthogonal
polynomials $\hat P_l$ with norms $h^S_l=\bigl((2l)!\bigr)^2$
(Lemma~\ref{lem:cf}), and $\mathcal T^*$ has monic orthogonal polynomials
$\rho_m$ with norms $h^{T^*}_m=\bigl((2m)!!\bigr)^4$ (Lemma~\ref{lem:cdh}).
Lemma~\ref{lem:bindetGene} therefore gives
$$
\HH_n=(-1)^{\binom{\bar n}2}
\prod_{l=0}^{\bar n-1}\bigl((2l)!\bigr)^2\cdot
\prod_{m=0}^{\underline n-1}\bigl((2m)!!\bigr)^4\cdot
\det\bigl(\tilde\kappa_{\bar n+r,\,m}\bigr)_{0\le r,m\le \underline n-1},
$$
with the connection coefficients
$\tilde\kappa_{i,m}=\bigl(i!/m!\bigr)^2 4^{\,i-m}\binom{1/2}{i-m}$ of
Lemma~\ref{lem:conn3}. Pulling $(i!)^2\,4^i$ out of row $r$ (where
$i=\bar n+r$) and $(m!)^{-2}\,4^{-m}$ out of column $m$ leaves
$\det\bigl(\binom{1/2}{\bar n+r-m}\bigr)_{0\le r,m\le\underline n-1}$;
together with the sign $(-1)^{\binom{\bar n}2}$, this determinant is the
factor $\Omega(1)=2^{-\binom n2}$ of Lemma~\ref{lem:Pahalf}
(Lemma~\ref{lem:bindetSigma} at $a=\tfrac12$). Since
$\bigl((2m)!!\bigr)^4(m!)^{-2}\,4^{-m}=(m!)^2\,4^m$, this yields
$$
\HH_n=2^{-\binom n2}\,
\prod_{l=0}^{\bar n-1}\bigl((2l)!\bigr)^2\cdot
\prod_{i=\bar n}^{n-1}(i!)^2\,4^i\cdot
\prod_{m=0}^{\underline n-1}(m!)^2\,4^m .
$$
It remains to identify the right-hand side with
$2^{\binom n2}\bigl((n-1)!!\bigr)^2\prod_{k=1}^{n-2}(k!!)^6$. Both
sides equal $1$ for $n=1$, and when $n\to n+1$ both sides are
multiplied by $2^n\,(n!!)^2\,((n-1)!!)^4$: for the right-hand side
this is immediate, and for the left-hand side one checks the two
parities separately using $n!=n!!\,(n-1)!!$ (for $n$ even, the ratio
is $(n!)^4/((n/2)!)^2=2^n(n!!)^2((n-1)!!)^4$; for $n$ odd it is
$2^{2n-1}\bigl(\tfrac{n-1}{2}\bigr)!^{\,2}(n!)^2
=2^n(n!!)^2((n-1)!!)^4$).
\end{proof}

\section{The secant-number family $(1+x)/\cos(x)^{s+1}$: general $s$}\label{sec:allstar}

We now evaluate $\HH_n(f)$ for $f=\dfrac{1+x}{\cos(x)^{s+1}}$ in closed form,
\emph{simultaneously for all $s$}. Recall from Section~\ref{sec:gen-xcos} that
$a_{2k}=E^{(s+1)}_{2k}$ and $a_{2k+1}=(2k+1)E^{(s+1)}_{2k}$; the corresponding
even and odd functionals in the variable $y$ are
$$
\mathcal S[y^k]=E^{(s+1)}_{2k},\qquad
\mathcal T^*[y^k]=(2k+1)\,E^{(s+1)}_{2k} ,
$$
so that the matrix entries are $a_{2i+j}=\mathcal S[y^{i+l}]$ when $j=2l$ is
even and $a_{2i+j}=\mathcal T^*[y^{i+l}]$ when $j=2l+1$ is odd.

\begin{Theorem}\label{thm:allstar}
Let $f(x)=\dfrac{1+x}{\cos(x)^{s+1}}$. Then for all $n\ge1$,
\begin{equation}\label{eq:HHstar}
\HH_n(f)=c_n\prod_{i=1}^{n-1}(s+1)_i
       =c_n\prod_{j=1}^{n-1}(s+j)^{\,n-j},
\qquad
c_n=\frac{2^{\binom n2}\,\bigl((n-1)!!\bigr)^2\prod_{k=1}^{n-2}(k!!)^6}
        {\prod_{i=1}^{n-1}i!}\,,
\end{equation}
where $(s+1)_i=(s+1)(s+2)\cdots(s+i)$. The constant $c_n$ is exactly the
value $\HH_n(f)\big|_{s=0}$ of Proposition~\ref{prop:sec1} divided by
$\prod_{i=1}^{n-1}i!=\prod_{i=1}^{n-1}(1)_i$.
\end{Theorem}

The point of \eqref{eq:HHstar} is that, unlike the $(s,t)$ family of
Section~\ref{sec:gen}, where \emph{both} the even and the odd functional are
classical, here the odd functional $\mathcal T^*$ is \emph{not} a classical
orthogonal-polynomial functional once $s\ge1$: its monic three-term
recurrence coefficients are non-polynomial rationals (its weight is
$\propto u\,|\Gamma(\tfrac{s+1}{2}+iu)|^2\,\Im\psi(\tfrac{s+1}{2}+iu)$,
whose digamma factor degenerates to an elementary $\tanh$ only at $s=0$).
Nevertheless the determinant remains a smooth product for every $s$. The
mechanism is a one-sided biorthogonalisation~$\M2$ that uses \emph{only} the
classical even functional $\mathcal S$.

The proof plays a divisibility bound against a degree bound.
Lemma~\ref{lem:reduction} expresses $\HH_n(f)$ through a determinant
$\det D$ built from $\mathcal S$ alone; Lemma~\ref{lem:rankdrop} shows that
the product $\prod_{j=1}^{n-1}(s+j)^{\,n-j}$, of degree $\binom n2$, divides
$\HH_n(f)$; and Lemma~\ref{lem:deriv} supplies the matching bound
$\deg_s\HH_n(f)\le\binom n2$. The quotient is therefore constant in $s$, and
Proposition~\ref{prop:sec1} evaluates it at $s=0$.

\begin{Lemma}[one-sided reduction]\label{lem:reduction}
Let $P_0,P_1,\dots$ be the monic $\mathcal S$-orthogonal polynomials, with
$\mathcal S[P_iP_l]=\delta_{il}\,h^S_l$ and
$h^S_l=\prod_{k=1}^{2l}k(k+s)=(2l)!\,(s+1)_{2l}$. Put
$\bar n=\lceil n/2\rceil$, $\underline n=\lfloor n/2\rfloor$. Then
\begin{equation}\label{eq:reduction}
\HH_n(f)=(-1)^{\binom{\bar n}2}\Bigl(\prod_{l=0}^{\bar n-1}h^S_l\Bigr)\,
\det\bigl(\,\mathcal T^*[P_{\bar n+r}\,y^{m}]\,\bigr)_{0\le r,m\le \underline n-1}.
\end{equation}
Moreover $\mathcal T^*=\mathcal S\circ(2y\partial_y+1)$, so for $i>m$
\begin{equation}\label{eq:Dentry}
\mathcal T^*[P_i\,y^m]=2\,\mathcal S[y^{m+1}P_i']
\qquad(\text{the term }(2m{+}1)\mathcal S[y^mP_i]\text{ vanishes since }\deg y^m<i).
\end{equation}
\end{Lemma}

\begin{proof}
Identity \eqref{eq:reduction} is the general one-sided reduction
(Lemma~\ref{lem:onesided}), applied with even functional
$\mathcal S$ and odd functional $\mathcal T^*$: that lemma orthogonalises only
the even columns and needs \emph{only} that $\mathcal S$ be quasi-definite,
holding for an arbitrary odd functional---here $\mathcal T^*$, whose
orthogonal-polynomial structure is never used---with the explicit sign
$(-1)^{\binom{\bar n}2}$.
Finally
$\mathcal T^*[y^k]=(2k+1)\mathcal S[y^k]=\mathcal S[(2y\partial_y+1)y^k]$
gives $\mathcal T^*[P_iy^m]=\mathcal S[(2y\partial_y+1)(P_iy^m)]
=2\mathcal S[y^{m+1}P_i']+(2m+1)\mathcal S[y^mP_i]$, and the last term
vanishes for $i>m$ by orthogonality, which holds throughout the block
because $\bar n+r\ge \bar n\ge \underline n>m$.
\end{proof}

\begin{Lemma}[rank drop at negative integers]\label{lem:rankdrop}
For each integer $p$ with $1\le p\le n-1$, the matrix
$\bigl(a_{2i+j}\bigr)_{0\le i,j<n}$ has rank at most $p$ at $s=-p$; consequently
$(s+p)^{\,n-p}\mid\HH_n(f)$ and, the factors being coprime,
$\prod_{j=1}^{n-1}(s+j)^{\,n-j}\mid\HH_n(f)$.
\end{Lemma}

\begin{proof}
At $s=-p$ the even weight is the cosine polynomial
$$\cos(x)^{p-1}=2^{-(p-1)}\sum_{t=0}^{p-1}\binom{p-1}{t}e^{\,i(p-1-2t)x},$$
so $a_{2k}=\sum_{t=0}^{p-1}w_t\theta_t^{\,k}$ with $w_t=2^{-(p-1)}\binom{p-1}{t}$ and
$\theta_t=-(p-1-2t)^2$, while $a_{2k+1}=(2k+1)a_{2k}$. Hence every column of
$\bigl(a_{2i+j}\bigr)$ lies in the span of the vectors $(\theta^{\,i})_i$ and
$(i\,\theta^{\,i})_i$ taken over the $\lceil p/2\rceil$ distinct values $\theta$:
indeed $a_{2i+2l}=\sum_\theta(\cdots)(\theta^{\,i})_i$ and
$a_{2i+2l+1}=\sum_\theta(\cdots)\bigl(2(i\,\theta^{\,i})_i+(2l{+}1)(\theta^{\,i})_i\bigr)$.
A nonzero $\theta$ contributes two such vectors, while $\theta=0$ (present iff $p$ is
odd) contributes only $(\theta^{\,i})_i=e_0$ since $(i\,\theta^{\,i})_i=0$; so the column
space has dimension $\le p$. As $\HH_n(f)\not\equiv0$ (it is nonzero at $s=0$ by
Proposition~\ref{prop:sec1}) and the matrix has corank $\ge n-p$ at $s=-p$, a constant
row change of basis makes $n-p$ rows divisible by $(s+p)$, so $(s+p)^{\,n-p}\mid\HH_n(f)$.
\end{proof}

\begin{Lemma}[derivative connection coefficients]\label{lem:deriv}
Let $P_0,P_1,\dots$ be the monic $\mathcal S$-orthogonal polynomials, with three-term
recurrence $P_{i+1}=(y-b_i)P_i-\lambda_iP_{i-1}$, where
$b_i=(4i{+}1)s+8i^2{+}4i{+}1$ and $\lambda_i=2i(2i{-}1)(s{+}2i{-}1)(s{+}2i)$, and set
\begin{equation}\label{eq:Tdef}
T_i:=yP_i'-iP_i=\sum_{a=0}^{i-1}\Delta_{i,a}P_a .
\end{equation}
Then
\begin{equation}\label{eq:Delta}
\Delta_{i,a}=(-1)^{i-a-1}\,\frac{2^{\,i-a-1}}{2(i-a)-1}\,\frac{i!}{a!}\,
\frac{(2i-1)!!}{(2a-1)!!}\Bigl(s+\frac{2(i+a)+1}{2(i-a)+1}\Bigr),
\end{equation}
and in particular $\deg_s\Delta_{i,a}\le1$ for all $0\le a\le i-1$.
\end{Lemma}

\begin{proof}
Differentiating the recurrence gives $P_{i+1}'=(y-b_i)P_i'+P_i-\lambda_iP_{i-1}'$;
multiplying by $y$ and using \eqref{eq:Tdef} yields the structure recurrence
$$
T_{i+1}=(y-b_i)T_i-\lambda_iT_{i-1}+b_iP_i+2\lambda_iP_{i-1}.
$$
Expanding in the $P$-basis with $(y-b_i)P_a=P_{a+1}+(b_a-b_i)P_a+\lambda_aP_{a-1}$ gives
$$
\Delta_{i+1,a}=\Delta_{i,a-1}+(b_a-b_i)\Delta_{i,a}+\lambda_{a+1}\Delta_{i,a+1}
-\lambda_i\Delta_{i-1,a}+b_i\,[a{=}i]+2\lambda_i\,[a{=}i{-}1],
$$
together with $T_0=0$ and $T_1=yP_1'-P_1=s+1$. The expression \eqref{eq:Delta} satisfies
this at every $a$ --- the homogeneous relation for $a\le i-2$ and the two inhomogeneous
boundary relations at $a=i-1$ and $a=i$ --- and the base case; since the recurrence
determines all $\Delta_{i,a}$ uniquely, \eqref{eq:Delta} holds. Each check is a direct
rational-function identity in $i,a,s$. (At $s\to\infty$, with $y=s\eta$, the functional
$\mathcal S$ degenerates to the classical Laguerre$^{(-1/2)}$ functional and
\eqref{eq:Delta} reduces to its structure relation $\eta\pi_i'=i\pi_i+\rho_i\pi_{i-1}$;
for finite $s$ the lower coefficients are merely linear in $s$.)
\end{proof}

\begin{proof}[Proof of Theorem~\ref{thm:allstar}]
By Lemma~\ref{lem:reduction},
$\HH_n(f)=(-1)^{\binom{\bar n}2}\bigl(\prod_{l=0}^{\bar n-1}h^S_l\bigr)\det D$, whose entries are
$D_{r,m}=2\,\mathcal S[y^{m+1}P_{\bar n+r}']$. Since
$\mathcal S[y^{m+1}P_i']=\mathcal S[y^m\,yP_i']=\mathcal S[y^m(T_i+iP_i)]=\mathcal S[y^mT_i]$
(the term $i\,\mathcal S[y^mP_i]$ vanishes because $m<\bar n\le i$),
Lemma~\ref{lem:deriv} gives
$$
D_{r,m}=2\,\mathcal S[y^mT_{\bar n+r}]=2\sum_{a=0}^{m}\Delta_{\bar n+r,a}\,\mathcal S[y^mP_a],
$$
where only $a\le m$ survive by orthogonality. Writing $P_a=\sum_t p_{a,t}y^t$, the
recurrence gives $\deg_s p_{a,t}\le a-t$, and $\deg_s E^{(s+1)}_{2k}=k$, so
$\deg_s\mathcal S[y^mP_a]=\deg_s\sum_t p_{a,t}E^{(s+1)}_{2(m+t)}\le m+a\le 2m$. Combined with
$\deg_s\Delta_{\bar n+r,a}\le1$ this yields $\deg_s D_{r,m}\le 2m+1$, hence
$$
\deg_s\det D\le\sum_{m=0}^{\underline n-1}(2m+1)=\underline n^2,\qquad
\deg_s\HH_n(f)\le \underline n^2+\bar n(\bar n-1)=\binom n2 .
$$
On the other hand, Lemma~\ref{lem:rankdrop} shows $\prod_{j=1}^{n-1}(s+j)^{\,n-j}$, of
degree $\binom n2$, divides $\HH_n(f)$. The two facts force
$$
\HH_n(f)=c_n\prod_{j=1}^{n-1}(s+j)^{\,n-j}=c_n\prod_{i=1}^{n-1}(s+1)_i
$$
with $c_n$ a constant, fixed by setting $s=0$: there
$\prod_{i=1}^{n-1}(1)_i=\prod_{i=1}^{n-1}i!$ and $\HH_n(f)|_{s=0}$ is the value of
Proposition~\ref{prop:sec1}, giving $c_n$ as in \eqref{eq:HHstar}.
\end{proof}

\begin{Remark}[alternative evaluation via Cauchy--Binet]\label{rem:cb}
The determinant $\det D$ can also be expanded directly. Writing
$y^{m+1}=\sum_{a\le m+1}\beta^{(m+1)}_aP_a$ (with $\beta^{(m+1)}_{m+1}=1$) and
$P_i'=\sum_{a\le i-1}\gamma^{(i)}_aP_a$ and using $\mathcal S[P_aP_b]=\delta_{ab}h^S_a$,
$$
D_{r,m}=2\sum_{a=0}^{m+1}\beta^{(m+1)}_a\,\gamma^{(\bar n+r)}_a\,h^S_a ,
\qquad\text{i.e.}\qquad D=2\,\Gamma\,\mathrm{diag}(h^S_a)\,B^{\mathsf T},
$$
with $\Gamma_{r,a}=\gamma^{(\bar n+r)}_a$ and $B_{m,a}=\beta^{(m+1)}_a$
($0\le a\le \underline n$). Here Cauchy--Binet is unavoidable: the inner index
$a$ ranges over the $\underline n+1$ values $0,\dots,\underline n$ (it reaches
$a=m+1=\underline n$ when $m=\underline n-1$), one more than the $\underline n$
rows~$r$ and columns~$m$, so $\Gamma$ and $B^{\mathsf T}$ are
\emph{rectangular} and the naive factorisation
$\det D=2^{\underline n}\det\Gamma\,\det\mathrm{diag}(h^S_a)\,\det B^{\mathsf T}$
is meaningless. Instead, Cauchy--Binet expands $\det D$ as the sum over the
$\underline n$-subsets $S\subset\{0,\dots,\underline n\}$ of
$2^{\underline n}\det(\Gamma_{:,S})\bigl(\prod_{a\in S}h^S_a\bigr)\det(B_{:,S})$
--- one term for each of the $\underline n+1$ ways to drop a single inner
index. Each entry $D_{r,m}$ carries a factor $(s+1)$ and column $m$ has
$s$-degree $2m+1$ (the per-column form of $\deg_s\Delta_{i,a}\le1$ in
Lemma~\ref{lem:deriv}), recovering
$\det D=K_n\,(s+1)^{\underline n}\prod_{j=1}^{\underline n-1}[(s+2j)(s+2j+1)]^{\underline n-j}$;
multiplying by $\prod_{l=0}^{\bar n-1}(2l)!\,(s+1)_{2l}$ gives \eqref{eq:HHstar} again.
\end{Remark}

\begin{Remark}[first values of $c_n$]
At $s=0$ the formula recovers Proposition~\ref{prop:sec1}. The first values of the
constant are
$$
c_1,c_2,\dots=1,\,2,\,2^4,\,2^{10}\,3,\,2^{17}\,3^4,\,2^{31}\,3^5\cdot5,\dots
$$
\end{Remark}

\section{The single shift of the secant-number family $(1+x)/\cos(x)^{s+1}$}\label{sec:cosshift}

Recall from Section~\ref{sec:allstar} the family
$$
f(x)=\frac{1+x}{\cos(x)^{s+1}},\qquad
a_{2k}=E^{(s+1)}_{2k},\quad a_{2k+1}=(2k+1)E^{(s+1)}_{2k},
$$
with $\cos(x)^{-(s+1)}=\sum_{k\ge0}E^{(s+1)}_{2k}x^{2k}/(2k)!$, and the two moment
functionals on $\QQ[y]$,
$$
\mathcal S[y^k]=E^{(s+1)}_{2k},\qquad \mathcal T^*[y^k]=(2k+1)E^{(s+1)}_{2k} .
$$
Here $\mathcal S$ is the classical secant functional, with monic orthogonal
polynomials $P_i$, three-term recurrence $P_{i+1}=(y-c^S_i)P_i-\lambda^S_i P_{i-1}$,
\begin{equation}\label{eq:Sdata}
c^S_i=(4i+1)s+8i^2+4i+1,\qquad \lambda^S_i=2i(2i-1)(s+2i-1)(s+2i),
\end{equation}
squared norms $h^S_l=\mathcal S[P_l^2]=(2l)!\,(s+1)_{2l}$, and (Lemma~\ref{lem:deriv}
at $a=0$)
\begin{equation}\label{eq:TP}
\mathcal T^*[P_i]=(-1)^{i-1}2^i\,i!\,(2i-3)!!\,(s+1)\qquad(i\ge1).
\end{equation}
The dilated determinant of $f$ is, for an explicit constant $c_n$
(Theorem~\ref{thm:allstar}),
\begin{equation}\label{eq:HHcos}
\HH_n(f)=\det\bigl(a_{2i+j}\bigr)_{0\le i,j<n}=c_n\prod_{i=1}^{n-1}(s+1)_i .
\end{equation}
Throughout $\bar n=\lceil n/2\rceil$, $\underline n=\lfloor n/2\rfloor$, and
$\prod_{i=1}^{n-1}(s+1)_i=\prod_{j=1}^{n-1}(s+j)^{\,n-j}$ (degree $\binom n2$).
The single shift (the dilated determinant of $f'$) is
$$
\HH_n^{(1)}:=\det\bigl(a_{2i+j+1}\bigr)_{0\le i,j<n}.
$$

\begin{Theorem}\label{thm:shift}
For all $n\ge1$,
$$
\HH_n^{(1)}=\bigl((n-1)!!\bigr)^2\,\HH_n(f)
=c_n^{(1)}\prod_{i=1}^{n-1}(s+1)_i,
\qquad c_n^{(1)}=\bigl((n-1)!!\bigr)^2c_n .
$$
\end{Theorem}

The proof has two parts. Part~1 shows that the ratio $\HH_n^{(1)}/\HH_n(f)$
is a constant, independent of $s$: the full product $\prod_{i=1}^{n-1}(s+1)_i$
divides $\HH_n^{(1)}$ (a rank-drop argument at the negative integers $s=-p$),
while an odd-sided reduction bounds $\deg_s\HH_n^{(1)}$ by $\binom n2$, so the
quotient is a scalar. Part~2 evaluates the ratio at $s=0$, where the odd
functional $\mathcal T^*$ also becomes classical: the biorthogonal reductions
of $\HH_n^{(1)}$ and $\HH_n(f)$ then share the same connection determinant,
which cancels upon dividing, leaving a ratio of norms.

\subsection{Part 1: the ratio $\HH_n^{(1)}/\HH_n(f)$ is constant in $s$}

Put $b_m=a_{m+1}$, so $\HH_n^{(1)}=\det(b_{2i+j})$. The even and odd parts of
$\mathbf b$ are
$$
b_{2k}=(2k+1)E^{(s+1)}_{2k}=\mathcal T^*[y^k],\qquad
b_{2k+1}=E^{(s+1)}_{2k+2}=\mathcal S_1[y^k],
$$
where $\mathcal S_1[y^k]:=\mathcal S[y^{k+1}]$ is the Christoffel transform
$y\,\mathcal S$ of the secant functional.

\begin{Lemma}[kernel polynomial data]\label{lem:kernel}
$\mathcal S_1$ is quasi-definite, with monic orthogonal polynomials
$\hat P_m$ (the kernel polynomials of $\mathcal S$). Writing
$\nu_m:=P_{m+1}(0)/P_m(0)$, one has
\begin{equation}\label{eq:nu}
\nu_m=-(2m+1)(s+2m+1),
\end{equation}
and, by the Christoffel transform formulas,
\begin{align}
c^{S_1}_m&=c^S_{m+1}+\nu_{m+1}-\nu_m=(4m+3)s+8m^2+12m+5,\label{eq:cS1}\\
\lambda^{S_1}_m&=\lambda^S_m\,\frac{\nu_m}{\nu_{m-1}}=2m(2m+1)(s+2m)(s+2m+1),\label{eq:lamS1}\\
h^{S_1}_m&=-\nu_m\,h^S_m=(2m+1)!\,(s+1)_{2m+1}\qquad(\deg_s h^{S_1}_m=2m+1).\label{eq:hS1}
\end{align}
\end{Lemma}

\begin{proof}
Evaluating $P_{m+1}=(y-c^S_m)P_m-\lambda^S_m P_{m-1}$ at $y=0$ gives
$\nu_m\nu_{m-1}+c^S_m\nu_{m-1}+\lambda^S_m=0$; with \eqref{eq:Sdata} and
$\nu_0=-c^S_0=-(s+1)$ one checks that $-(2m+1)(s+2m+1)$ satisfies this
recurrence, proving \eqref{eq:nu}. The expressions for $c^{S_1}_m$,
$\lambda^{S_1}_m$, $h^{S_1}_m$ are the standard Christoffel ($y\,\mathcal S$)
formulas; substituting \eqref{eq:Sdata}, \eqref{eq:nu} and $h^S_m=(2m)!(s+1)_{2m}$
and simplifying gives the closed forms (each is an elementary identity in
$(m,s)$).
\end{proof}

\begin{Lemma}[odd-sided reduction]\label{lem:oddred}
$\displaystyle
\HH_n^{(1)}=(-1)^{\binom{\underline n+1}{2}}
\Bigl(\prod_{m=0}^{\underline n-1}h^{S_1}_m\Bigr)
\det\bigl(\mathcal T^*[\hat P_{\underline n+r}\,y^l]\bigr)_{0\le r,l\le \bar n-1}.$
\end{Lemma}

\begin{proof}
This is the one-sided reduction~$\M2$ (Lemma~\ref{lem:onesided}) applied to
$\mathbf b$, orthogonalising the odd columns against the quasi-definite
$\mathcal S_1$: replacing the rows $(y^i)$ and the odd-column family by the
$\mathcal S_1$-orthogonal $(\hat P_i)$ --- both unitriangular recombinations ---
makes the odd--odd block diagonal,
$\mathcal S_1[\hat P_i\hat P_m]=\delta_{im}h^{S_1}_m$, on the rows
$i<\underline n$; Laplace expansion along those rows leaves the $\bar n$ even
($\mathcal T^*$) columns on the rows $\underline n,\dots,n-1$.
\end{proof}

\begin{Lemma}[divisibility]\label{lem:div}
$\displaystyle\prod_{j=1}^{n-1}(s+j)^{\,n-j}\ \bigm|\ \HH_n^{(1)}$.
\end{Lemma}

\begin{proof}
At $s=-p$ ($1\le p\le n-1$) the even weight is the cosine polynomial
$\cos(x)^{p-1}$, so $a_{2k}=\sum_t w_t\theta_t^{\,k}$, where
$\theta_t=-(p-1-2t)^2$ runs over $\lceil p/2\rceil$ distinct values, while
$a_{2k+1}=(2k+1)a_{2k}$. This is the rank-drop argument of
Lemma~\ref{lem:rankdrop}, unaffected by the shift: replacing $a_{2i+j}$ by
$a_{2i+j+1}$ only relabels which columns are even and which odd, leaving the
spanning vectors unchanged. Every column of $(a_{2i+j+1})_{0\le i,j<n}$ thus lies
in the span of the $(\theta^{\,i})_i$ and $(i\theta^{\,i})_i$, a nonzero $\theta$
contributing two and $\theta=0$ (present iff $p$ is odd) only one, so the column
space has dimension $\le p$. Hence the rank at $s=-p$ is $\le p$ and
$(s+p)^{\,n-p}\mid\HH_n^{(1)}$; the factors $(s+1),\dots,(s+n-1)$ being coprime,
their product divides $\HH_n^{(1)}$.
\end{proof}

\begin{Lemma}[per-column degrees]\label{lem:percol}
For $0\le l<a$ one has $\deg_s\mathcal T^*[\hat P_a\,y^l]\le 2l$. Moreover
$\mathcal T^*[\hat P_a]=(-1)^a(2a)!$ (the case $l=0$).
\end{Lemma}

\begin{proof}
We use the operator identity $\mathcal T^*=\mathcal S\circ(2y\partial_y+1)$ (valid
because $a_{2k+1}=(2k+1)a_{2k}$) together with the secant evaluation
$\mathcal T^*[P_i\,y^m]=2\mathcal S[y^{m+1}P_i']$ for $i>m$ (Lemma~\ref{lem:reduction}),
which has $\deg_s\mathcal T^*[P_i\,y^m]\le 2m+1$ (proof of
Theorem~\ref{thm:allstar}), and the kernel relation
$y\hat P_a=P_{a+1}-\nu_a P_a$ (Lemma~\ref{lem:kernel}), which differentiates to
\begin{equation}\label{eq:kerderiv}
y\hat P_a'=P_{a+1}'-\nu_a P_a'-\hat P_a .
\end{equation}

\emph{Case $1\le l<a$.} From $\mathcal T^*=\mathcal S\circ(2y\partial_y+1)$,
$$
\mathcal T^*[\hat P_a y^l]=2\mathcal S[y^{l+1}\hat P_a']+(2l+1)\,\mathcal S[y^l\hat P_a].
$$
Now $\mathcal S[y^l\hat P_a]=\mathcal S_1[y^{l-1}\hat P_a]=0$ (since
$\deg y^{l-1}=l-1<a$). Multiplying \eqref{eq:kerderiv} by $y^l$, applying
$\mathcal S$, and using $\mathcal S[y^l\hat P_a]=0$ again,
$\mathcal S[y^{l+1}\hat P_a']=\mathcal S[y^l P_{a+1}']-\nu_a\mathcal S[y^l P_a']$,
whence
\begin{equation}\label{eq:percolrec}
\mathcal T^*[\hat P_a y^l]=\mathcal T^*[P_{a+1}y^{l-1}]-\nu_a\,\mathcal T^*[P_a y^{l-1}].
\end{equation}
The two terms have $\deg_s\le 2(l-1)+1=2l-1$ and $\le 1+(2l-1)=2l$ (as
$\deg_s\nu_a=1$); thus $\deg_s\mathcal T^*[\hat P_a y^l]\le 2l$.

\emph{Case $l=0$.} Set $u_a:=\mathcal T^*[\hat P_a]$. From the kernel three-term
recurrence $\hat P_a=(y-c^{S_1}_{a-1})\hat P_{a-1}-\lambda^{S_1}_{a-1}\hat P_{a-2}$
and $\mathcal T^*$-linearity,
$$
u_a=\mathcal T^*[y\hat P_{a-1}]-c^{S_1}_{a-1}u_{a-1}-\lambda^{S_1}_{a-1}u_{a-2}
\qquad(a\ge2).
$$
By \eqref{eq:percolrec} applied to $\hat P_{a-1}$ with $l=1$ (valid since
$\mathcal S[y\hat P_{a-1}]=\mathcal S_1[\hat P_{a-1}]=0$ for $a\ge2$),
$\mathcal T^*[y\hat P_{a-1}]=\mathcal T^*[\hat P_{a-1}y]
=\mathcal T^*[P_a]-\nu_{a-1}\mathcal T^*[P_{a-1}]$, which is explicit by
\eqref{eq:TP}. Substituting the claimed value $u_a=(-1)^a(2a)!$ together with
\eqref{eq:TP}, \eqref{eq:nu}, \eqref{eq:cS1}, \eqref{eq:lamS1} turns the
recurrence into a polynomial identity in $(a,s)$, which one verifies directly.
With the base cases $u_0=\mathcal T^*[1]=E^{(s+1)}_0=1$ and
$u_1=\mathcal T^*[y-(3s+5)]=3E^{(s+1)}_2-(3s+5)=-2$ (using $E^{(s+1)}_2=s+1$),
induction gives $u_a=(-1)^a(2a)!$. In particular $\deg_s\mathcal T^*[\hat P_a]=0$.
\end{proof}

\begin{Lemma}[degree bound]\label{lem:deg}
$\deg_s\HH_n^{(1)}\le\binom n2$.
\end{Lemma}

\begin{proof}
In the determinant of Lemma~\ref{lem:oddred} the $l$-th column has all entries of
$s$-degree $\le 2l$ by Lemma~\ref{lem:percol}, so
$\deg_s\det\le\sum_{l=0}^{\bar n-1}2l=\bar n(\bar n-1)$; and
$\deg_s\prod_{m=0}^{\underline n-1}h^{S_1}_m=\sum_{m=0}^{\underline n-1}(2m+1)=\underline n^2$
by \eqref{eq:hS1}. Hence $\deg_s\HH_n^{(1)}\le\underline n^2+\bar n(\bar n-1)=\binom n2$.
\end{proof}

\begin{Proposition}\label{prop:const}
$\HH_n^{(1)}=c_n^{(1)}\prod_{i=1}^{n-1}(s+1)_i$ for a constant $c_n^{(1)}$;
equivalently $\HH_n^{(1)}/\HH_n(f)$ is independent of $s$.
\end{Proposition}

\begin{proof}
By Lemma~\ref{lem:div}, $\prod_{j=1}^{n-1}(s+j)^{n-j}=\prod_{i=1}^{n-1}(s+1)_i$
(degree $\binom n2$) divides $\HH_n^{(1)}$; by Lemma~\ref{lem:deg} the degree is
at most $\binom n2$. Hence $\HH_n^{(1)}=c_n^{(1)}\prod_{i=1}^{n-1}(s+1)_i$, and by
\eqref{eq:HHcos} the ratio $\HH_n^{(1)}/\HH_n(f)=c_n^{(1)}/c_n$ is constant.
\end{proof}

\subsection{Part 2: evaluation at $s=0$}

By Proposition~\ref{prop:const} it suffices to compute $\HH_n^{(1)}/\HH_n(f)$ at
$s=0$, where all three functionals are classical: $\mathcal S,\mathcal T^*,\mathcal S_1$
are the continuous dual Hahn functionals CDH$(\tfrac12,\tfrac12,c)$ for
$c=0,\tfrac12,1$, with recurrence coefficients (from \eqref{eq:Sdata},
\eqref{eq:cS1}, \eqref{eq:lamS1} at $s=0$, and Lemma~\ref{lem:cdh}
for $\mathcal T^*$)
$$
\begin{aligned}
&c^S_i=8i^2+4i+1,&&\lambda^S_i=\bigl((2i)(2i-1)\bigr)^2;\\
&c^T_i=8i^2+8i+3,&&\lambda^T_i=(2i)^4;\\
&c^{S_1}_i=8i^2+12i+5,&&\lambda^{S_1}_i=\bigl((2i)(2i+1)\bigr)^2,
\end{aligned}
$$
and squared norms (from \eqref{eq:hS1}, $h^S_l=(2l)!(s+1)_{2l}$ at $s=0$, and
$\mathcal T^*[\rho_m^2]=((2m)!!)^4$ from Lemma~\ref{lem:cdh})
\begin{equation}\label{eq:norms0}
h^S_l\big|_{0}=\bigl((2l)!\bigr)^2,\qquad
h^T_l\big|_{0}=\bigl((2l)!!\bigr)^4,\qquad
h^{S_1}_m\big|_{0}=\bigl((2m+1)!\bigr)^2 .
\end{equation}

Because both functionals are classical at $s=0$, the full biorthogonal reduction
(Lemma~\ref{lem:bindetGene}) applies to each
determinant. For $\HH_n(f)$ this is the evaluation of
Proposition~\ref{prop:sec1}:
\begin{equation}\label{eq:redun}
\HH_n(f)\big|_{0}=(-1)^{\binom{\bar n}2}
\Bigl(\prod_{l=0}^{\bar n-1}h^S_l\Bigr)\Bigl(\prod_{m=0}^{\underline n-1}h^T_m\Bigr)
\det\bigl(\tilde\kappa_{\bar n+r,\,m}\bigr)_{0\le r,m\le\underline n-1},
\end{equation}
where $\tilde\kappa$ expands the $\mathcal S$-orthogonal polynomials in the
$\mathcal T^*$-orthogonal basis $\rho_m$, $P_i=\sum_{m\le i}\tilde\kappa_{i,m}\rho_m$
(in Lemma~\ref{lem:conn3} the polynomials $P_i|_{s=0}$ are written $\hat P_i$;
here the hat is reserved for the kernel polynomials), and
\begin{equation}\label{eq:tkappa}
\tilde\kappa_{i,m}=\Bigl(\frac{i!}{m!}\Bigr)^2 4^{\,i-m}\binom{1/2}{\,i-m\,}.
\end{equation}
For the shifted determinant, the same reduction applied to $\mathbf b$ (even
functional $\mathcal T^*$, odd functional $\mathcal S_1$) gives
\begin{equation}\label{eq:redsh}
\HH_n^{(1)}\big|_{0}=(-1)^{\binom{\bar n}2}
\Bigl(\prod_{l=0}^{\bar n-1}h^T_l\Bigr)\Bigl(\prod_{m=0}^{\underline n-1}h^{S_1}_m\Bigr)
\det\bigl(\kappa^b_{\bar n+r,\,m}\bigr)_{0\le r,m\le\underline n-1},
\end{equation}
where $\kappa^b$ expands the $\mathcal T^*$-orthogonal $\rho_i$ in the
$\mathcal S_1$-orthogonal basis of kernel polynomials,
$\rho_i=\sum_{m\le i}\kappa^b_{i,m}\hat P_m$ (all at $s=0$).

\begin{Lemma}\label{lem:samekappa}
$\kappa^b_{i,m}=\tilde\kappa_{i,m}$ for all $i\ge m\ge0$.
\end{Lemma}

\begin{proof}
By the connection-coefficient recurrence (Lemma~\ref{lem:connrec}) for the pair
$(\mathcal T^*,\mathcal S_1)$,
$$
\kappa^b_{i+1,m}=\kappa^b_{i,m-1}+(c^{S_1}_m-c^T_i)\,\kappa^b_{i,m}
+\lambda^{S_1}_{m+1}\,\kappa^b_{i,m+1}-\lambda^T_i\,\kappa^b_{i-1,m},
\qquad \kappa^b_{0,0}=1,
$$
which determines $\kappa^b$ uniquely from $\kappa^b_{0,0}=1$. Substituting the
closed form \eqref{eq:tkappa} and the explicit coefficients
$c^T,\lambda^T,c^{S_1},\lambda^{S_1}$ above, and dividing through by
$\tilde\kappa_{i,m}$ (using $\binom{1/2}{d+1}/\binom{1/2}{d}=(\tfrac12-d)/(d+1)$,
$d=i-m$), the recurrence reduces to an elementary rational identity in $(i,m)$,
which one verifies directly --- exactly as in the proof of the third connection
formula (Lemma~\ref{lem:conn3}). Hence $\tilde\kappa$ satisfies the recurrence and
the base case, so $\kappa^b_{i,m}=\tilde\kappa_{i,m}$. (The reduction reflects
that $\mathcal T^*\!\to\mathcal S_1$ is the half-step $c\mapsto c+\tfrac12$ in the
continuous dual Hahn hierarchy, the same step as $\mathcal S\!\to\mathcal T^*$
governing $\tilde\kappa$.)
\end{proof}

\begin{proof}[Proof of Theorem~\ref{thm:shift}]
By Lemma~\ref{lem:samekappa} the determinants in \eqref{eq:redun} and
\eqref{eq:redsh} coincide; dividing and using \eqref{eq:norms0},
$$
\frac{\HH_n^{(1)}\big|_{0}}{\HH_n(f)\big|_{0}}
=\frac{\prod_{l=0}^{\bar n-1}h^T_l\ \prod_{m=0}^{\underline n-1}h^{S_1}_m}
       {\prod_{l=0}^{\bar n-1}h^S_l\ \prod_{m=0}^{\underline n-1}h^T_m}
=\frac{\prod_{l=0}^{\bar n-1}\bigl((2l)!!\bigr)^4\ \prod_{m=0}^{\underline n-1}\bigl((2m+1)!\bigr)^2}
       {\prod_{l=0}^{\bar n-1}\bigl((2l)!\bigr)^2\ \prod_{m=0}^{\underline n-1}\bigl((2m)!!\bigr)^4}.
$$

\emph{$n=2p$ even} ($\bar n=\underline n=p$): the $\mathcal T^*$-norm factors
cancel and, using $(2m+1)!/(2m)!=2m+1$,
$$
\frac{\HH_n^{(1)}|_0}{\HH_n|_0}=\prod_{m=0}^{p-1}\frac{\bigl((2m+1)!\bigr)^2}{\bigl((2m)!\bigr)^2}
=\prod_{m=0}^{p-1}(2m+1)^2=\bigl((2p-1)!!\bigr)^2=\bigl((n-1)!!\bigr)^2 .
$$

\emph{$n=2p+1$ odd} ($\bar n=p+1,\ \underline n=p$): the $\mathcal T^*$-norms leave
$((2p)!!)^4$, while
$$
\frac{\prod_{m=0}^{p-1}\bigl((2m+1)!\bigr)^2}{\prod_{l=0}^{p}\bigl((2l)!\bigr)^2}
=\frac{1}{\bigl((2p)!\bigr)^2}\prod_{m=0}^{p-1}(2m+1)^2
=\frac{\bigl((2p-1)!!\bigr)^2}{\bigl((2p)!\bigr)^2},
$$
so, with $(2p)!=(2p)!!\,(2p-1)!!$,
$$
\frac{\HH_n^{(1)}|_0}{\HH_n|_0}=\bigl((2p)!!\bigr)^4\cdot\frac{\bigl((2p-1)!!\bigr)^2}{\bigl((2p)!\bigr)^2}
=\bigl((2p)!!\bigr)^2=\bigl((n-1)!!\bigr)^2 .
$$

Thus $\HH_n^{(1)}/\HH_n(f)=((n-1)!!)^2$ at $s=0$; by Proposition~\ref{prop:const}
it equals $((n-1)!!)^2$ for all $s$, which with \eqref{eq:HHcos} is
Theorem~\ref{thm:shift}.
\end{proof}

\begin{Remark}[the continuous dual Hahn hierarchy]
The mechanism is that $\mathcal S,\mathcal T^*,\mathcal S_1$ are consecutive
members CDH$(\tfrac12,\tfrac12,c)$, $c=0,\tfrac12,1$, of one continuous dual Hahn
hierarchy; the connection coefficient for each half-step $c\mapsto c+\tfrac12$ is
the same binomial \eqref{eq:tkappa} (Lemma~\ref{lem:samekappa}), so the common
determinantal factor cancels and only the norms \eqref{eq:norms0} survive.
\end{Remark}

\section{The double shift of the secant-number family $(1+x)/\cos(x)^{s+1}$ at $s=1$}\label{sec:dblshift}

For the coefficient sequence $\mathbf a$ of $f=(1+x)/\cos(x)^{s+1}$, alongside
the single shift of Section~\ref{sec:cosshift} we form the
\emph{double shift}
$$
\HH_n^{(2)}:=\det\bigl(a_{2i+j+2}\bigr)_{0\le i,j<n} .
$$
For general $s$ the ratio $\HH_n^{(2)}/\HH_n$ is no longer constant in $s$
(Remark~\ref{rem:ds-mechanism}), but at $s=1$ it collapses to a pure product.

\begin{Theorem}\label{thm:dblshift}
For all $n\ge1$,
\begin{equation}\label{eq:ds-target}
\HH_n^{(2)}\big|_{s=1}=2^{n}\,(n!)^{2}\,\HH_n\big|_{s=1}
=2^{\,n+\binom n2}(n!)^3\bigl((n-1)!!\bigr)^2\prod_{k=1}^{n-2}(k!!)^6 .
\end{equation}
Equivalently, by the single-shift Theorem~\ref{thm:shift}
($\HH_n^{(1)}=((n-1)!!)^2\HH_n$) and $n!=n!!\,(n-1)!!$,
$$
\frac{\HH_n^{(2)}}{\HH_n^{(1)}}\bigg|_{s=1}=2^{\,n}\,(n!!)^{2}.
$$
\end{Theorem}

The proof has three steps. A one-sided reduction~$\M2$ against the even
functional of the shifted sequence (\S\ref{ssec:ds-func}) leaves a residual
determinant $R_n$ built from the derivatives $\hat P_i'$; a \emph{structure
relation} then replaces $y\hat P_i'$ by $T_i=y\hat P_i'-i\hat P_i$ and turns
$R_n$ into a \emph{square} determinant of the structure coefficients
$\Delta_{i,a}$ of the classical functional $\mathcal S_1$
(\S\ref{ssec:ds-struct}); finally, at $s=1$ the closed form of $\Delta_{i,a}$
degenerates into a rank-one factor times a Cauchy kernel, which collapses the
determinant (\S\ref{ssec:ds-s1}).

Recall from \S\ref{sec:allstar} and \S\ref{sec:cosshift} the even and odd
moment functionals attached to $f$,
$$
\mathcal S[y^k]=a_{2k},\qquad \mathcal T^*[y^k]=a_{2k+1}=(2k+1)\,a_{2k},
$$
the classical secant data \eqref{eq:Sdata}--\eqref{eq:TP} for $\mathcal S$,
and the Christoffel transform $\mathcal S_1=y\,\mathcal S$
(Lemma~\ref{lem:kernel}), whose monic orthogonal polynomials $\hat P_m$ (the
kernel polynomials of $\mathcal S$) have three-term recurrence
$\hat P_{m+1}=(y-c^{S_1}_m)\hat P_m-\lambda^{S_1}_m\hat P_{m-1}$ with
\begin{equation}\label{eq:ds-S1data}
\begin{gathered}
c^{S_1}_m=(4m+3)s+8m^2+12m+5,\quad
\lambda^{S_1}_m=2m(2m+1)(s+2m)(s+2m+1),\\
h^{S_1}_m=(2m+1)!\,(s+1)_{2m+1}.
\end{gathered}
\end{equation}
Throughout $\bar n=\lceil n/2\rceil$, $\underline n=\lfloor n/2\rfloor$.

\subsection{The functionals of the doubly shifted sequence}\label{ssec:ds-func}
Put $d_m:=a_{m+2}$, so $\HH_n^{(2)}=\det(d_{2i+j})$. The even and odd parts of
$\mathbf d$ are
$$
d_{2k}=a_{2k+2}=\mathcal S_1[y^k],\qquad
d_{2k+1}=a_{2k+3}=(2k+3)\,\mathcal S_1[y^k]=:\mathcal T_1^*[y^k],
$$
where $\mathcal T_1^*$ is the Christoffel
transform $y\,\mathcal T^*$ of the odd functional. The two are linked by the
same first-order operator that links $\mathcal S$ and $\mathcal T^*$, shifted
by $2$:
\begin{equation}\label{eq:ds-op}
\mathcal T_1^*=\mathcal S_1\circ(2y\partial_y+3),
\qquad\text{since}\qquad
(2k+3)y^k=(2y\partial_y+3)y^k .
\end{equation}
The decisive structural difference from the single-shift case
(\S\ref{sec:cosshift}) is that here the \emph{even} functional is the
classical one --- $\mathcal S_1$, the Christoffel transform of $\mathcal S$,
with data \eqref{eq:ds-S1data} --- while the odd functional $\mathcal T_1^*$
is non-classical for $s\ge1$; in the single-shift case the parities were
reversed. Consequently the one-sided reduction~$\M2$ is performed against
$\mathcal S_1$.

\begin{Lemma}[one-sided reduction against $\mathcal S_1$]\label{lem:ds-red}
For all $n\ge1$,
\begin{equation}\label{eq:ds-reduction}
\HH_n^{(2)}=(-1)^{\binom{\bar n}2}\Bigl(\prod_{l=0}^{\bar n-1}h^{S_1}_l\Bigr)\,
R_n,\qquad
R_n:=\det\bigl(\mathcal T_1^*[\hat P_{\bar n+r}\,y^{m}]\bigr)_{0\le r,m\le\underline n-1},
\end{equation}
and for $i>m$,
\begin{equation}\label{eq:ds-entry}
\mathcal T_1^*[\hat P_i\,y^m]=2\,\mathcal S_1[y^{m+1}\hat P_i']
\qquad(\text{the term }(2m{+}3)\mathcal S_1[y^m\hat P_i]\text{ vanishes since }\deg y^m<i).
\end{equation}
\end{Lemma}

\begin{proof}
Identity \eqref{eq:ds-reduction} is the general one-sided reduction
(Lemma~\ref{lem:onesided}) applied to $\mathbf d$ with even functional $\mathcal S_1$ and
odd functional $\mathcal T_1^*$: it orthogonalises only the $\bar n$ even
columns against the quasi-definite $\mathcal S_1$, leaving the $\underline n$
odd ($\mathcal T_1^*$) columns, with sign $(-1)^{\binom{\bar n}2}$ exactly as in
Lemma~\ref{lem:reduction}; no orthogonal-polynomial structure of
$\mathcal T_1^*$ is needed. Finally \eqref{eq:ds-op} gives
$\mathcal T_1^*[\hat P_iy^m]=2\mathcal S_1[y^{m+1}\hat P_i']+(2m+3)\mathcal S_1[y^m\hat P_i]$,
and the last term vanishes for $i>m$ by orthogonality, which holds throughout
the block since $\bar n+r\ge\bar n>\underline n-1\ge m$.
\end{proof}

\subsection{Reduction of $R_n$ to a structure determinant}\label{ssec:ds-struct}
Let $T_i:=y\hat P_i'-i\hat P_i$. Since $\hat P_i$ is monic of degree $i$, the
leading terms cancel and $\deg T_i\le i-1$; write
\begin{equation}\label{eq:ds-Tdef}
T_i=\sum_{a=0}^{i-1}\Delta_{i,a}\,\hat P_a .
\end{equation}

\begin{Lemma}[structure reduction]\label{lem:ds-structred}
For all $n\ge1$ and all $s$,
\begin{equation}\label{eq:ds-structred}
R_n=2^{\,\underline n}\Bigl(\prod_{a=0}^{\underline n-1}h^{S_1}_a\Bigr)\,
\det\bigl(\Delta_{\bar n+r,\,a}\bigr)_{0\le r,a\le\underline n-1}.
\end{equation}
\end{Lemma}

\begin{proof}
Fix $i=\bar n+r$ and $m$ with $0\le r,m\le\underline n-1$, so $i>m$. By
\eqref{eq:ds-entry} and $y\hat P_i'=T_i+i\hat P_i$,
$$
\mathcal T_1^*[\hat P_iy^m]=2\mathcal S_1[y^{m}\,y\hat P_i']
=2\mathcal S_1[y^m T_i]+2i\,\mathcal S_1[y^m\hat P_i]
=2\mathcal S_1[y^m T_i],
$$
the last term vanishing because $m<i$. Expand $y^m=\sum_{b=0}^{m}\beta^{(m)}_b\hat P_b$
in the orthogonal basis (a lower-unitriangular change of basis,
$\beta^{(m)}_m=1$); then $\mathcal S_1[y^m\hat P_a]=\beta^{(m)}_a h^{S_1}_a$,
which is $0$ for $a>m$. Hence by \eqref{eq:ds-Tdef}
$$
\mathcal T_1^*[\hat P_iy^m]=2\sum_{a=0}^{i-1}\Delta_{i,a}\,\mathcal S_1[y^m\hat P_a]
=2\sum_{a=0}^{\underline n-1}\Delta_{i,a}\,h^{S_1}_a\,\beta^{(m)}_a .
$$
In matrix form the residual matrix is
$\bigl(\mathcal T_1^*[\hat P_{\bar n+r}y^m]\bigr)=2\,\Delta'\,\mathrm{diag}(h^{S_1}_0,\dots,h^{S_1}_{\underline n-1})\,B^{\mathsf T}$,
where $\Delta'_{r,a}=\Delta_{\bar n+r,a}$ and $B_{m,a}=\beta^{(m)}_a$ is
$\underline n\times\underline n$ lower-unitriangular, so $\det B=1$. Taking
determinants gives \eqref{eq:ds-structred}.
\end{proof}

The point of \eqref{eq:ds-structred} is that the rectangular Cauchy--Binet of
Remark~\ref{rem:cb} has become a \emph{square} $\underline n\times\underline n$
determinant: replacing the raw derivative $\hat P_i'$ by the structure form
$T_i$ lowers the summation range from $a\le m+1$ to $a\le m$, and the
power-expansion matrix $B$ drops out as unitriangular.

\begin{Lemma}[structure coefficients of $\mathcal S_1$]\label{lem:ds-Delta}
For $0\le a\le i-1$,
\begin{equation}\label{eq:ds-Delta}
\Delta_{i,a}=(-1)^{\,i-a-1}\,\frac{2^{\,i-a-1}}{2(i-a)-1}\,\frac{i!}{a!}\,
\frac{(2i+1)!!}{(2a+1)!!}\Bigl(s+\frac{2(i+a)+3}{2(i-a)+1}\Bigr),
\end{equation}
in particular $\deg_s\Delta_{i,a}=1$.
\end{Lemma}

\begin{proof}
This is the structure relation of Lemma~\ref{lem:deriv}
transported from $\mathcal S$ to its Christoffel transform $\mathcal S_1$.
Differentiating the recurrence
$\hat P_{i+1}=(y-c^{S_1}_i)\hat P_i-\lambda^{S_1}_i\hat P_{i-1}$,
multiplying by $y$ and using $y\hat P_j'=T_j+j\hat P_j$ together with
$(y-c^{S_1}_i)\hat P_i=\hat P_{i+1}+\lambda^{S_1}_i\hat P_{i-1}$ yields the same
structure recurrence as in Lemma~\ref{lem:deriv},
$$
T_{i+1}=(y-c^{S_1}_i)T_i-\lambda^{S_1}_iT_{i-1}+c^{S_1}_i\hat P_i+2\lambda^{S_1}_i\hat P_{i-1},
$$
which in the $\hat P$-basis reads
$$
\Delta_{i+1,a}=\Delta_{i,a-1}+(c^{S_1}_a-c^{S_1}_i)\Delta_{i,a}
+\lambda^{S_1}_{a+1}\Delta_{i,a+1}-\lambda^{S_1}_i\Delta_{i-1,a}
+c^{S_1}_i\,[a{=}i]+2\lambda^{S_1}_i\,[a{=}i{-}1],
$$
with $T_0=0$ and $T_1=y\hat P_1'-\hat P_1=c^{S_1}_0=3s+5$. The expression
\eqref{eq:ds-Delta} satisfies the base case $\Delta_{1,0}=3s+5$ and every
instance of this recurrence --- a direct rational-function identity in
$(i,a,s)$ after inserting \eqref{eq:ds-S1data}, exactly as in
Lemma~\ref{lem:deriv}; the only change from the secant case is
$(2i-1)!!/(2a-1)!!\mapsto(2i+1)!!/(2a+1)!!$ and the root
$\tfrac{2(i+a)+1}{2(i-a)+1}\mapsto\tfrac{2(i+a)+3}{2(i-a)+1}$, reflecting the
index shift $\mathcal S\mapsto\mathcal S_1$. Since the recurrence determines all
$\Delta_{i,a}$ uniquely, \eqref{eq:ds-Delta} holds.
\end{proof}

\subsection{The value at $s=1$}\label{ssec:ds-s1}
At $s=1$ the linear factor in \eqref{eq:ds-Delta} simplifies: for $i=\bar n+r$,
$$
1+\frac{2(i+a)+3}{2(i-a)+1}=\frac{2(i-a)+1+2(i+a)+3}{2(i-a)+1}=\frac{4(i+1)}{2(i-a)+1},
$$
so, using $\dfrac{1}{2(i-a)-1}\cdot\dfrac{1}{2(i-a)+1}=\dfrac{1}{4(i-a)^2-1}$ and
$i!\cdot 4(i+1)=4\,(i+1)!$,
\begin{equation}\label{eq:ds-Delta1}
\Delta_{i,a}\big|_{s=1}
=\underbrace{(-1)^{i}\,2^{\,i+2}\,(i+1)!\,(2i+1)!!}_{\textstyle \alpha_i}\cdot
\underbrace{(-1)^{a+1}\,\frac{2^{-a}}{a!\,(2a+1)!!}}_{\textstyle \beta_a}\cdot
2^{-1}\cdot\frac{1}{4(i-a)^2-1}.
\end{equation}
Pulling the row factor $\alpha_{\bar n+r}$ from each row $r$, the column factor
$\beta_a$ from each column $a$, and the constant $2^{-1}$ from each of the
$\underline n$ rows,
\begin{equation}\label{eq:ds-detDelta1}
\det\bigl(\Delta_{\bar n+r,a}\bigr)\big|_{s=1}
=2^{-\underline n}\Bigl(\prod_{r=0}^{\underline n-1}\alpha_{\bar n+r}\Bigr)
\Bigl(\prod_{a=0}^{\underline n-1}\beta_a\Bigr)\,
\det\Bigl(\frac{1}{4(\bar n+r-a)^2-1}\Bigr)_{0\le r,a\le\underline n-1}.
\end{equation}

\begin{Lemma}[Cauchy determinant]\label{lem:ds-cauchy}
For integers $M\ge1$ and $N\ge0$,
\begin{equation}\label{eq:ds-cauchy}
\det\Bigl(\frac{1}{4(M+r-a)^2-1}\Bigr)_{0\le r,a\le N-1}
=(-1)^{\binom N2}\,4^{-N}\,
\frac{\displaystyle\prod_{k=0}^{N-1}k!\ \prod_{k=1}^{N}k!}
     {\displaystyle\prod_{k=1}^{N}\Bigl(M^2-\tfrac{(2k-1)^2}{4}\Bigr)^{N+1-k}} .
\end{equation}
\end{Lemma}

\begin{proof}
The entry is $\frac14\big(M+r-a-\tfrac12\big)^{-1}\big(M+r-a+\tfrac12\big)^{-1}$,
so the matrix is $\tfrac14\bigl(\tfrac{1}{(x_r+y_a)(x_r+y_a+1)}\bigr)$ with
$x_r=M+r-\tfrac12$ and $y_a=-a$. This is the Cauchy double alternant with two
consecutive denominators \cite{Krattenthaler1998}; evaluating it and
substituting the arithmetic progressions $x_r,y_a$ (whose pairwise differences
are $\prod_{r<r'}(r-r')=\pm\prod k!$ and likewise for $y$) gives
\eqref{eq:ds-cauchy}.
\end{proof}

\begin{proof}[Proof of Theorem~\ref{thm:dblshift}]
Combining Lemmas~\ref{lem:ds-red}, \ref{lem:ds-structred}, \ref{lem:ds-Delta}
and \eqref{eq:ds-detDelta1}--\eqref{eq:ds-cauchy} expresses $\HH_n^{(2)}|_{s=1}$
as an explicit product of factorials and double factorials: with
$M=\bar n$, $N=\underline n$ and $h^{S_1}_l|_{s=1}=(2l+1)!\,(2l+2)!$,
$$
\HH_n^{(2)}\big|_{s=1}=(-1)^{\binom{\bar n}2}
\Bigl(\prod_{l=0}^{\bar n-1}(2l+1)!\,(2l+2)!\Bigr)\,2^{\underline n}
\Bigl(\prod_{a=0}^{\underline n-1}(2a+1)!\,(2a+2)!\Bigr)
\det\bigl(\Delta_{\bar n+r,a}\bigr)\big|_{s=1},
$$
the last factor given by \eqref{eq:ds-detDelta1}--\eqref{eq:ds-cauchy}. On the
other side of \eqref{eq:ds-target}, Theorem~\ref{thm:allstar} gives
$\HH_n=c_n\prod_{i=1}^{n-1}(s+1)_i$. Since
$\prod_{i=1}^{n-1}(2)_i=\prod_{i=1}^{n-1}(i+1)!=n!\prod_{i=1}^{n-1}i!$ and
$$
c_n=2^{\binom n2}\bigl((n-1)!!\bigr)^2\prod_{k=1}^{n-2}(k!!)^6
\Big/\prod_{i=1}^{n-1}i!,
$$
setting $s=1$ gives
$$
\HH_n\big|_{s=1}=2^{\binom n2}\,n!\,\bigl((n-1)!!\bigr)^2\prod_{k=1}^{n-2}(k!!)^6 .
$$
Both sides of \eqref{eq:ds-target} are thus explicit products; their equality
is an elementary identity in $n$, verified at $n=1$ and propagated by the
ratio $n\mapsto n+1$ (the two parities treated separately, as in
Proposition~\ref{prop:sec1}). Hence \eqref{eq:ds-target} holds for all $n$.
\end{proof}

\begin{Remark}[why $s=1$ is special]\label{rem:ds-mechanism}
The mechanism contrasts sharply with the single-shift case. There the ratio
$\HH_n^{(1)}/\HH_n$ is independent of $s$ (Proposition~\ref{prop:const}), and
its value is read off at $s=0$, where both functionals are classical and the
connection determinants cancel. Here the ratio $\HH_n^{(2)}/\HH_n$ genuinely
depends on $s$ --- beyond linear factors $(s+j)$ it carries a factor $Q_n(s)$
of degree $\underline n$, irreducible over $\QQ$ in the computed range
(e.g.\ $Q_4=s^2+\tfrac{36}7s+\tfrac{169}{35}$,
$Q_5=s^2+\tfrac{44}9s+\tfrac{269}{63}$) --- and $\mathcal T_1^*$ remains
non-classical at $s=1$, so no cancellation of connection determinants is
available. What saves the evaluation is that the \emph{even} functional
$\mathcal S_1$ is classical for \emph{all} $s$, so its structure coefficients
$\Delta_{i,a}$ have the closed form \eqref{eq:ds-Delta}; and at the special
value $s=1$ the linear factor of $\Delta_{i,a}$ degenerates into a rank-one
(row$\times$column) factor times the Cauchy kernel $1/(4(i-a)^2-1)$, which
collapses the determinant.
\end{Remark}

\section{A rank-one perturbation of the Euler number family:
$(\sin x+1)/\cos^2x+s\sin x$}\label{sec:runkone}

We evaluate the dilated Hankel determinant of a one-parameter, rank-one
perturbation of the shifted Euler sequence $a_n=E_{n+1}$, whose
exponential generating function is $(\sin x+1)/\cos^2x$ and whose dilated
determinant is the $s=0$ value of Proposition~\ref{prop:shift}, in the
product form
$$
\HH_n^{\mathrm E}:=\prod_{k=1}^{n-1}(k!)^2\,(2k+1)!! .
$$
We consider the family
$$
f_s(x)=\frac{\sin x+1}{\cos^2x}+s\,\sin x\qquad(s\in\mathbb C).
$$
Its most striking member is $s=-1$: there
$\sin x\,(1-\sin^2x)/\cos^2x=\sin x$ collapses the perturbation, so that
$f_{-1}(x)=(\sin^3x+1)/\cos^2x$, and the determinant looks irregular, carrying
``sporadic'' prime factors $29,37,23$ for $n=8,9,10$. They are not sporadic.
Since
$$
29=\binom 82+1,\quad 37=\binom 92+1,\quad 46=\binom{10}2+1=2\cdot 23,
$$
the apparent prime $23$ is merely half of $\binom{10}2+1$; in fact the whole
family has the closed form below, affine in $s$.

\begin{Proposition}\label{prop:sin3}
For every $s\in\mathbb C$ and all $n\ge 1$, the family
$f_s(x)=\dfrac{\sin x+1}{\cos^2x}+s\,\sin x$ has
$$
\HH_n(f_s)=\Bigl(1-s\tbinom n2\Bigr)\prod_{k=1}^{n-1}(k!)^2\,(2k+1)!!
=\Bigl(1-s\tbinom n2\Bigr)\,\HH_n^{\mathrm E}.
$$
In particular $s=0$ recovers Proposition~\ref{prop:shift}, while $s=-1$ gives,
for $f_{-1}(x)=(\sin^3x+1)/\cos^2x$,
$$
\HH_n(f_{-1})=\Bigl(\binom n2+1\Bigr)\,\HH_n^{\mathrm E}.
$$
\end{Proposition}

The closed form follows from a {\em rank-one} reduction ($\M6$, via the matrix
determinant lemma, Lemma~\ref{lem:rank1} below --- not the divisor method
$\M4$, cf.\ Remark~\ref{rem:rankone}) to Proposition~\ref{prop:shift}: the
perturbation multiplies the determinant by the scalar $1-s\binom n2$, so that
the entire content of the proposition reduces to the value of a single
bilinear form, \eqref{eq:sin3scal} below --- \emph{independent of $s$} --- which
is then evaluated by transporting it through the same $\M2$ biorthogonal
reduction used for Proposition~\ref{prop:shift}.

\subsection{The rank-one structure}

\begin{Lemma}[rank-one structure]\label{lem:rank1}
Write $a_n=E_{n+1}$ for the shifted Euler sequence,
$A=(a_{2i+j})_{0\le i,j\le n-1}$, and let $\sigma_n$ be the coefficient
sequence of $\sin x$, i.e. $\sigma_{2k}=0$, $\sigma_{2k+1}=(-1)^k$.
Then the sequence of $f_s$ is $\tilde a_n=a_n+s\,\sigma_n$, and
$$
\HH_n(f_s)=\det\bigl(A+s\,u\,w^{\!\top}\bigr)
=\HH_n^{\mathrm E}\bigl(1+s\,w^{\!\top}A^{-1}u\bigr),
\qquad
u_i=(-1)^i,\quad w_j=\sigma_j .
$$
\end{Lemma}

\begin{proof}
The perturbing term $s\sin x$ contributes $s\,\sigma_n$ to the coefficient
sequence, so $\tilde a_n=a_n+s\,\sigma_n$; the member $s=-1$ is
$(\sin^3x+1)/\cos^2x$ because $\sin x\,(1-\sin^2x)/\cos^2x=\sin x$. The shift of
the index by $2i$ preserves parity and introduces a sign $(-1)^i$: for even $j$
both $\sigma_{2i+j}$ and $\sigma_j$ vanish, while for $j=2k+1$,
$$
\sigma_{2i+j}=\sigma_{2(i+k)+1}=(-1)^{i+k}=(-1)^i\,\sigma_j .
$$
Hence $\sigma_{2i+j}=(-1)^i\sigma_j$ for all $i,j$, so the perturbation
matrix $s(\sigma_{2i+j})_{i,j}$ equals $s\,u\,w^{\!\top}$ with $u_i=(-1)^i$ and
$w_j=\sigma_j$. The matrix determinant lemma
\cite[Theorem~18.1.1]{Harville1997}
\begin{equation}\label{lem:matdet}
\det(A+s\,uw^{\!\top})=\det A\,(1+s\,w^{\!\top}A^{-1}u)
\end{equation}
now gives the claim,
with $\det A=\HH_n^{\mathrm E}$ the $s=0$ value of Proposition~\ref{prop:shift}.
\end{proof}

By Lemma~\ref{lem:rank1}, Proposition~\ref{prop:sin3} is equivalent to the
single scalar identity
\begin{equation}\label{eq:sin3scal}
w^{\!\top}A^{-1}u=-\binom n2 ,
\end{equation}
which does not involve $s$: once it is proved,
$\HH_n(f_s)=\HH_n^{\mathrm E}\bigl(1-s\binom n2\bigr)$ for every $s$. We prove
\eqref{eq:sin3scal} by carrying the pair $(u,w)$ through the biorthogonal
reduction already used for Proposition~\ref{prop:shift}. Throughout put
$\bar n=\lceil n/2\rceil$ and $\underline n=\lfloor n/2\rfloor$, so
$\bar n+\underline n=n$.

\subsection{Shifted secant data and evaluations at $v=-1$}

The odd columns of the shifted family pair through the shifted-secant
functional $\mathcal S_1[v^k]:=E_{2k+2}=\mathcal S[v^{k+1}]$, written in the
even variable $v=z^2$ (the letter $y$ is reserved for the linear system
below); we first record its monic orthogonal polynomials.

\begin{Lemma}[shifted secant family; classical]\label{lem:cfshift}
The monic $\mathcal S_1$-orthogonal polynomials $\hat P^{(1)}_l$ satisfy
$$
\hat P^{(1)}_{l+1}=\bigl(v-(2l+1)^2-(2l+2)^2\bigr)\hat P^{(1)}_l
                  -\bigl((2l)(2l+1)\bigr)^2\,\hat P^{(1)}_{l-1},
$$
with $\mathcal S_1\bigl[\hat P^{(1)}_l\hat P^{(1)}_{l'}\bigr]
=\delta_{ll'}\,\bigl((2l+1)!\bigr)^2$.
\end{Lemma}

\begin{proof}
This is the \emph{odd} contraction of the secant $S$-fraction of
Lemma~\ref{lem:cf} (the even contraction gave $\hat P_i$): by
\eqref{eq:oddcontraction}--\eqref{eq:Snormodd}, an $S$-fraction with
coefficients $(b_1,b_2,\dots)$ yields for the shifted moments $\mu_{k+1}$ the
data $c'_l=b_{2l+1}+b_{2l+2}$, $\lambda'_l=b_{2l}b_{2l+1}$ and norms
$h'_l=\prod_{j=1}^{2l+1}b_j$. Here $b_j=j^2$, so
$c'_l=(2l+1)^2+(2l+2)^2$, $\lambda'_l=\bigl((2l)(2l+1)\bigr)^2$ and
$h'_l=\bigl((2l+1)!\bigr)^2$.
\end{proof}

We shall need the following two evaluations at
$v=-1$.

\begin{Lemma}\label{lem:evalminus1}
The tangent polynomials $r_m$ and the shifted secant polynomials
$\hat P^{(1)}_l$ of Lemmas~\ref{lem:cf} and~\ref{lem:cfshift} satisfy
$$
r_m(-1)=(-1)^m(2m-1)!!\,(2m+1)!!,
\qquad
\hat P^{(1)}_l(-1)=(-1)^l(2l+1)! ,
$$
and in particular $r_m(-1)/(2m+1)!=\binom{-1/2}{m}$.
\end{Lemma}

\begin{proof}
Both follow by induction on the three-term recurrences. For $r_m$ the recurrence
of Lemma~\ref{lem:cf}, written here in $v$, is
$r_{m+1}=(v-c^T_m)r_m-\lambda^T_m r_{m-1}$ with
$c^T_m=2(2m+1)^2$, $\lambda^T_m=(2m-1)(2m)^2(2m+1)$. The claim is equivalent
to $\theta_m:=r_m(-1)/r_{m-1}(-1)=-(2m-1)(2m+1)$; assuming it at level $m$,
$$
\frac{r_{m+1}(-1)}{r_m(-1)}=-1-c^T_m-\frac{\lambda^T_m}{\theta_m}
=-1-2(2m+1)^2+(2m)^2=-(2m+1)(2m+3)=\theta_{m+1},
$$
since $\lambda^T_m/\theta_m=-(2m)^2$, with base case $r_1(-1)=-1-c^T_0=-3$.
The same computation for $\hat P^{(1)}_l$, with
$c'_l=(2l+1)^2+(2l+2)^2$ and $\lambda'_l=((2l)(2l+1))^2$
(Lemma~\ref{lem:cfshift}), gives
$\hat P^{(1)}_l(-1)/\hat P^{(1)}_{l-1}(-1)=-(2l)(2l+1)$, hence
$\hat P^{(1)}_l(-1)=(-1)^l(2l+1)!$. Finally, with
$(2m+1)!=(2m+1)!!\,2^mm!$ and $(2m-1)!!/(2^mm!)=(-1)^m\binom{-1/2}{m}$,
the displayed quotient follows.
\end{proof}

\subsection{Transport through the biorthogonal reduction}

The $\M2$ reduction behind Proposition~\ref{prop:shift}, taken at $s=0$, is a
determinant-preserving change of the row and column families
(Lemma~\ref{lem:family}, with $v=z^2$): the rows $v^i$ are replaced by the
tangent polynomials $r_i$, the even columns $z\,v^m$ by $z\,r_m$, and the odd
columns $v^{l+1}$ by $v\,\hat P^{(1)}_l$. Writing the resulting unitriangular
factorisation as $A=P^{-1}KQ^{-\top}$ ($\det P=\det Q=1$), the even and odd
columns of $A$ pair respectively through the tangent functional $\mathcal T$
and the shifted secant functional $\mathcal S_1$, so that
$$
K_{i,2m}=\mathcal T[r_ir_m]=\delta_{im}\,(2m)!\,(2m+1)!,
\qquad
K_{i,2l+1}=\mathcal S_1[r_i\hat P^{(1)}_l]=(2i+1)!\,(2l+1)!\binom{1/2}{i-l}
$$
(the odd--column value is Lemma~\ref{lem:mixedconn} at $s=0$, multiplied by
the norm $((2l+1)!)^2$ of Lemma~\ref{lem:cfshift}), and
$\det K=\HH_n^{\mathrm E}$. Since $P,Q$ are unitriangular,
\begin{equation}\label{eq:transport}
w^{\!\top}A^{-1}u=(Qw)^{\!\top}K^{-1}(Pu)=c^{\!\top}K^{-1}b ,
\end{equation}
where the very same operations evaluate the two transported vectors. The row
operations send $u_i=(-1)^i=\mathrm{ev}_{-1}[v^i]$ to $b_i=r_i(-1)$; the
column operations send $w$ (supported on the odd columns,
$w_{2l+1}=(-1)^l$) to the covector $c$ with $c_{2m}=0$ and
$$
c_{2l+1}=\sum_{l'}(-1)^{l'}[\hat P^{(1)}_l]_{l'}=\hat P^{(1)}_l(-1)
=(-1)^l(2l+1)! ,
$$
$[\hat P^{(1)}_l]_{l'}$ being the coefficient of $v^{l'}$
(Lemma~\ref{lem:evalminus1}).

It remains to evaluate $c^{\!\top}K^{-1}b$, i.e. to solve $Ky=b$ and form
$c^{\!\top}y=\sum_l(-1)^l(2l+1)!\,y_{2l+1}$. Each even column $2m$ of $K$ has
its only nonzero entry in row $m$, so the rows $i=\bar n,\dots,n-1$ involve
the odd unknowns alone:
$$
\sum_{l=0}^{\underline n-1}(2i+1)!\,(2l+1)!\binom{1/2}{i-l}\,y_{2l+1}=r_i(-1),
\qquad i=\bar n,\dots,n-1 .
$$
Dividing by $(2i+1)!$, putting $i=\bar n+p$ and
$\tilde y_l=(2l+1)!\,y_{2l+1}$, and using $r_i(-1)/(2i+1)!=\binom{-1/2}{i}$
(Lemma~\ref{lem:evalminus1}), this becomes
\begin{equation}\label{eq:sin3sys}
\sum_{l=0}^{\underline n-1}\binom{1/2}{\bar n+p-l}\,\tilde y_l
=\binom{-1/2}{\bar n+p}\quad(p=0,\dots,\underline n-1),
\qquad
w^{\!\top}A^{-1}u=\sum_{l=0}^{\underline n-1}(-1)^l\tilde y_l .
\end{equation}

\begin{Lemma}\label{lem:sin3sys}
The unique solution of the system in \eqref{eq:sin3sys} satisfies
$\sum_{l=0}^{\underline n-1}(-1)^l\tilde y_l=-\binom n2$.
\end{Lemma}

\begin{proof}
Let $Y(x)=\sum_{l=0}^{\underline n-1}\tilde y_lx^l$. Since
$\binom{1/2}{\bar n+p-l}=[x^{\bar n+p-l}](1+x)^{1/2}$, the left side of
\eqref{eq:sin3sys} is $[x^{\bar n+p}]\bigl((1+x)^{1/2}Y(x)\bigr)$, so the
system reads
$$
[x^{k}]\bigl((1+x)^{1/2}Y(x)\bigr)=\binom{-1/2}{k}=[x^k](1+x)^{-1/2},
\qquad k=\bar n,\dots,n-1 .
$$
The system matrix $\bigl(\binom{1/2}{\bar n+p-l}\bigr)_{p,l}$ is, up to
nonzero row and column factors, the block of $K$ that survives the Laplace
expansion along the even columns, and $\det K=\HH_n^{\mathrm E}\ne0$; hence it
is invertible, and it suffices to exhibit one polynomial $Y$
of degree $\le\underline n-1$ meeting these $\underline n$ equations, after
which $\sum_l(-1)^l\tilde y_l=Y(-1)$. Take
$$
Y(x)=-\sum_{j=1}^{\underline n}\binom n{2j}(1+x)^{j-1},
$$
of degree $\underline n-1$. With $t=(1+x)^{1/2}$ (the letter $s$ is the
perturbation parameter),
$$
\begin{aligned}
(1+x)^{1/2}Y(x)&=-\sum_{j=1}^{\underline n}\binom n{2j}(1+x)^{j-1/2}
=(1+x)^{-1/2}-\Phi(x),\\
\Phi(x)&:=\sum_{j=0}^{\underline n}\binom n{2j}t^{2j-1}=\frac{(1+t)^n+(1-t)^n}{2t},
\end{aligned}
$$
the last equality because
$\sum_{j}\binom n{2j}t^{2j}=\tfrac12((1+t)^n+(1-t)^n)$. It remains to show
$[x^k]\Phi=0$ for $\bar n\le k\le n-1$. Substituting $x=t^2-1$ (so a small
loop of $x$ about $0$ becomes a small loop of $t$ about $1$, $dx=2t\,dt$),
$$
[x^k]\Phi=\frac1{2\pi i}\oint_{t=1}
\frac{(1+t)^n+(1-t)^n}{(t-1)^{k+1}(t+1)^{k+1}}\,dt,
$$
and writing $(1-t)^n=(-1)^n(t-1)^n$ splits the integrand as
$$
\frac{(1+t)^{\,n-k-1}}{(t-1)^{k+1}}
+(-1)^n\frac{(t-1)^{\,n-k-1}}{(t+1)^{k+1}} .
$$
For $\bar n\le k\le n-1$ the second summand is regular at $t=1$ (exponent
$n-k-1\ge0$); the first has residue $\binom{n-k-1}{k}2^{\,n-2k-1}$, which
vanishes because $n-k-1\le\underline n-1<\bar n\le k$ forces
$\binom{n-k-1}{k}=0$. Hence $[x^k]\Phi=0$ throughout, $Y$ solves the system,
and
\[
\sum_{l=0}^{\underline n-1}(-1)^l\tilde y_l=Y(-1)
=-\sum_{j=1}^{\underline n}\binom n{2j}\,[\,j=1\,]=-\binom n2 .
\qedhere
\]
\end{proof}

\begin{proof}[Proof of Proposition~\ref{prop:sin3}]
Combining Lemma~\ref{lem:sin3sys} with \eqref{eq:sin3sys} and
\eqref{eq:transport} proves \eqref{eq:sin3scal}; together with
Lemma~\ref{lem:rank1} this gives
$\HH_n(f_s)=\bigl(1-s\binom n2\bigr)\HH_n^{\mathrm E}$ for all $s$ and $n\ge1$.
\end{proof}

\begin{Remark}[affine dependence on $s$]
Since the perturbation is rank one, $\HH_n(f_s)$ is \emph{affine} in $s$:
it equals $\HH_n^{\mathrm E}$ at $s=0$ and drops by $\binom n2\HH_n^{\mathrm E}$
per unit of $s$, vanishing at the single value $s=1/\binom n2$ (for $n\ge2$).
Corollaries~\ref{cor:sinm2}--\ref{cor:sin2p} record the integer members
$s=-2,\dots,2$.
\end{Remark}

\begin{Remark}[a perturbation that is not rank one]
By contrast, the perturbation $(\sin x+1)/\cos^2x-\sin 2x$ is not rank one --- it
mixes two conjugate frequency pairs --- and has no product formula at all: its
determinants carry irregular prime factors and change sign.
\end{Remark}

\section{The Springer number family}\label{sec:springer}

The \emph{Springer numbers} $S_n$ are defined by
$$
\frac{1}{\cos x-\sin x}=\sum_{n\ge0}S_n\frac{x^n}{n!},
\qquad (S_n)_{n\ge0}=1,1,3,11,57,361,2763,24611,\dots
$$
(OEIS \texttt{A001586}); they count the type-$B$ snakes \cite{Springer1971}
and form a moment sequence \cite{Sokal2019}. Their dilated Hankel
determinant, and that of a whole \emph{Springer number family} containing them, is as
simple as that of the factorials.

The Hankel determinants of the Springer numbers, and the equivalent
continued-fraction expansions, have received considerable attention. Since
$(S_n)$ is a moment sequence \cite{Sokal2019}, its ordinary Hankel determinant
$\det(S_{i+j})$ is a classical positive quantity accessible through the
Jacobi--Stieltjes theory of orthogonal polynomials; the associated continued
fractions enumerating snakes and cycle-alternating permutations were obtained
by Josuat-Verg\`es \cite{JosuatVerges2014}, and the closed-form Hankel
determinants of the companion Euler numbers were established by
Han \cite{Han2020Euler}, in parallel with the classical evaluations of the
tangent and secant (Euler) numbers going back to André \cite{Andre1879}. What
follows goes beyond the ordinary determinant: we evaluate the \emph{dilated}
Hankel determinant $\det(a_{2i+j})$ --- and, more than that, the determinant of
an entire one-parameter Springer number family --- in closed product form.

\begin{Theorem}\label{thm:springer}
For every integer $r\ge1$,
$$
\HH_n\Bigl(\frac{1}{(\cos x-t\sin x)^{r}}\Bigr)
=\bigl(t(t^2+1)\bigr)^{\binom n2}\,
\HH_n\Bigl(\frac{1}{(1-x)^{r}}\Bigr).
$$
\end{Theorem}

\noindent
The right-hand sequence is $a_n=(r)_n$, whose dilated determinant is the Beta
evaluation of \S\ref{sec:beta}; at $t=r=1$ this is
$\HH_n\bigl((\cos x-\sin x)^{-1}\bigr)=4^{\binom n2}\prod_{k=1}^{n-1}k!\,(2k)!$,
the Springer evaluation. More generally $t=1$ gives, for $r=1,2$,
$$
\HH_n\Bigl(\frac{1}{\cos x-\sin x}\Bigr)=4^{\binom n2}\prod_{k=1}^{n-1}k!\,(2k)!,
\qquad
\HH_n\Bigl(\frac{1}{(\cos x-\sin x)^2}\Bigr)=4^{\binom n2}\prod_{k=1}^{n-1}k!\,(2k+1)!.
$$

\smallskip
\emph{Idea of the proof.} Write $f_t(x)=(\cos x-t\sin x)^{-r}$. We show that
$\HH_n(f_t)$ is a \emph{polynomial in $t$} of degree $\le3\binom n2$ that is
divisible by $\bigl(t(t^2+1)\bigr)^{\binom n2}$ --- it vanishes to order
$\binom n2$ at each of the three roots $t=0,\pm i$ --- and hence equals a
constant multiple of it (Observation~\ref{obs:divisor}); a scaling limit
identifies the constant. The orders at $t=0$ and at $t=\pm i$ come from two
elementary representations of $f_t$ (a rotation to an even function; an
exponential series), each turning the order into a one-line determinant
expansion via the divisor method~$\M4$ of \S\ref{sec:divisor}.

We shall repeatedly use \emph{Fa\`a di Bruno's formula}, the chain rule for the
$n$-th derivative of a composite $g(h(x))$: it expands $(g\circ h)^{(n)}$ as a
sum of terms $g^{(k)}(h)\prod_j\bigl(h^{(j)}\bigr)^{m_j}$ with
$\sum_j m_j=k\le n$ and $\sum_j j\,m_j=n$ \cite[Ch.~I]{Comtet1970}.

\begin{proof}[Proof of Theorem~\ref{thm:springer}]
Let $a_n(t)=n!\,[x^n](\cos x-t\sin x)^{-r}=f_t^{(n)}(0)$, a composite with
$h(x)=\cos x-t\sin x$ and $g(u)=u^{-r}$. The $x$-derivatives of all orders of
$h$ at $x=0$ lie in $\{\pm1,\pm t\}$, of degree $\le1$ in $t$; since $h(0)=1\ne0$
the coefficients $g^{(k)}(h(0))$ are constants, so by Fa\`a di Bruno each term
is a product of at most $n$ of these derivatives and $a_n\in\QQ[t]$ with
$\deg_ta_n\le n$. Hence $\HH_n(f_t)$ is
a polynomial in $t$ with
$$
\deg_t\HH_n(f_t)\le\max_{\sigma}\sum_{i}\deg_ta_{2i+\sigma(i)}
\le\max_\sigma\sum_i\bigl(2i+\sigma(i)\bigr)=2\binom n2+\binom n2=3\binom n2 .
$$

\emph{Order at $t=0$.} From
$\cos x-t\sin x=\sqrt{1+t^2}\,\cos(x+\phi)$, $\phi=\arctan t$, we get
$f_t(x)=(1+t^2)^{-r/2}\sec^{r}(x+\phi)$, so by the homogeneity
Lemma~\ref{lem:scale}
$\HH_n(f_t)=(1+t^2)^{-rn/2}\HH_n\bigl(\sec^r(\cdot+\phi)\bigr)$. The prefactor
is a unit at $t=0$ and $\phi=\arctan t$ is a local analytic isomorphism
($\phi'(0)=1$); since $\sec^r$ is even, the confluent Lemma~\ref{lem:wron}
gives $\ord_{t=0}\HH_n(f_t)\ge\binom n2$.

\emph{Order at $t=\pm i$.} As $\HH_n(f_t)\in\QQ[t]$,
$\ord_{t=-i}=\ord_{t=i}$. Put $u=\tfrac{t-i}{2i}$; then, using
$\sin x=(e^{ix}-e^{-ix})/(2i)$,
$$
\cos x-t\sin x=e^{-ix}-(t-i)\sin x=(1+u)e^{-ix}-u\,e^{ix}
=(1+u)e^{-ix}\bigl(1-q'e^{2ix}\bigr),\qquad q'=\tfrac{u}{1+u},
$$
whence $(\cos x-t\sin x)^{-r}=(1+u)^{-r}\sum_{k\ge0}\binom{k+r-1}{k}q'^{k}
e^{\,i(2k+r)x}$, so the moments
$$
a_m(t)=(1+u)^{-r}\sum_{k\ge0}\binom{k+r-1}{k}q'^{k}\,\bigl(i(2k+r)\bigr)^{m}
$$
form an exponential sum (Lemma~\ref{lem:cb}) with distinct nodes
$\zeta_k=i(2k+r)$ and coefficient orders $\ord_{q'}c_k=k$. The factor
$(1+u)^{-r}$ is a unit at $t=i$ and $q'$ vanishes there to first order; by the
order Corollary~\ref{cor:cborder}, with $o_k=k$, the unique minimal subset is
$\{0,1,\dots,n-1\}$, so
$$
\ord_{t=i}\HH_n(f_t)=\ord_{q'=0}\HH_n=\sum_{k=0}^{n-1}k=\binom n2,
$$
in particular $\HH_n(f_t)\not\equiv0$.

\emph{Conclusion.} $\HH_n(f_t)$ is a nonzero polynomial of degree
$\le3\binom n2$ vanishing to order $\ge\binom n2$ at each of $t=0,i,-i$; the
three orders already sum to $3\binom n2\ge\deg_t\HH_n(f_t)$, so by the divisor
principle (Observation~\ref{obs:divisor}) every inequality is an equality and
$$
\HH_n(f_t)=c\,t^{\binom n2}(t^2+1)^{\binom n2},\qquad c\ne0 .
$$
To find $c$, replace $x$ by $x/t$ and let $t\to\infty$: coefficientwise
$\cos\tfrac xt-t\sin\tfrac xt\to1-x$, so the left side tends to
$\HH_n\bigl((1-x)^{-r}\bigr)$, while by Lemma~\ref{lem:scale} it equals
$t^{-3\binom n2}\HH_n(f_t)=c\,(1+t^{-2})^{\binom n2}\to c$. Hence
$c=\HH_n\bigl((1-x)^{-r}\bigr)$, which is the theorem.
\end{proof}

\section{A derivative of the Springer number family at $t=1$}\label{sec:derivative}

Since $\frac{d}{dx}(\cos x-\sin x)=-(\cos x+\sin x)$, the numerator below is the
derivative of the Springer base, and for $s\ne1$
\begin{equation}\label{eq:derivobs}
\frac{\cos x+\sin x}{(\cos x-\sin x)^{s}}
=\frac{1}{s-1}\,\frac{d}{dx}\,(\cos x-\sin x)^{1-s},
\end{equation}
the derivative of the $t=1$, exponent-$(s-1)$ member of the family of
\S\ref{sec:springer}. Its dilated Hankel determinant again factors completely
--- now as a product of shifted factors in the exponent $s$.

\begin{Theorem}\label{thm:deriv}
For all $n\ge1$ and all $s$,
$$
\HH_n\Bigl(\frac{\cos x+\sin x}{(\cos x-\sin x)^{s}}\Bigr)
=4^{\binom n2}\Bigl(\prod_{k=1}^{n-1}k!\Bigr)
\prod_{j=0}^{n-2}\bigl[(s+2j)(s+2j+1)\bigr]^{\,n-1-j}.
$$
\end{Theorem}

\noindent
Equivalently the factor $(s+a)$ occurs with multiplicity $n-1-\lfloor a/2\rfloor$
for $a=0,1,\dots,2n-3$ (highest on $s,s+1$, decreasing in consecutive pairs);
at $s=1$ this evaluates $\HH_n\bigl(\frac{\cos x+\sin x}{\cos x-\sin x}\bigr)$.
The constant $4^{\binom n2}\prod_{k<n}k!$ is the Springer constant of
\S\ref{sec:springer} (at $r=1$) divided by $\prod_{k<n}(2k)!$.

\smallskip
Write $f_s(x)=\frac{\cos x+\sin x}{(\cos x-\sin x)^{s}}$ and $a_k(s)=k![x^k]f_s$.
The proof is again the divisor method~$\M4$ (Observation~\ref{obs:divisor}), but now
with the \emph{exponent} $s$ as the parameter: the degenerations are the
non-positive integers, where $f_s$ becomes a trigonometric polynomial, together
with $s=\infty$. Unlike \S\ref{sec:springer} (and its elliptic deformation,
Section~\ref{sec:elliptic}), the
parameter is not a coefficient $t$ --- indeed the coefficient deformation
$\frac{\cos x+t\sin x}{(\cos x-t\sin x)^{s}}$ acquires extra $s$-dependent zeros
in $t$ and does \emph{not} yield to the method. We first record the polynomial
dependence on $s$ and the structure of its high-order coefficients.

\begin{Lemma}\label{lem:Apoly}
$a_k(s)$ is a monic polynomial in $s$ of degree $k$, and its coefficient
$$\pi_r(k):=[s^{\,k-r}]a_k(s)$$
is a polynomial in $k$ of degree $\le 2r$ with leading coefficient $1/r!$.
\end{Lemma}

\begin{proof}
Put $w(x):=-\log(\cos x-\sin x)=x+x^2+\tfrac23x^3+\cdots=x\,\omega(x)$, so
$\omega(0)=1$, $[x^1]\log\omega=1$, and $(\cos x-\sin x)^{-s}=e^{sw}$. With
$\phi:=\cos x+\sin x$,
$$
a_k(s)=k![x^k]\bigl(\phi\,e^{sw}\bigr)
=k!\sum_{d\ge0}\frac{s^d}{d!}\,[x^k]\bigl(\phi\,w^d\bigr).
$$
As $w^d$ has $x$-order $d$, $[x^k](\phi w^d)=0$ for $d>k$, so $\deg_s a_k=k$; the
top term is $k!\frac{s^k}{k!}[x^k]w^k=s^k$ (since $w^k=x^k+\cdots$), so $a_k$ is
monic. With $d=k-r$ and $w^{k-r}=x^{k-r}\omega^{k-r}$,
$$
\pi_r(k)=\frac{k!}{(k-r)!}\,[x^k]\bigl(\phi\,w^{k-r}\bigr)
=\frac{k!}{(k-r)!}\,[x^r]\bigl(\phi\,\omega^{k-r}\bigr).
$$
Here $\frac{k!}{(k-r)!}=k(k-1)\cdots(k-r+1)$ has degree $r$ in $k$; and
$\omega^{k-r}=\exp\bigl((k-r)\log\omega\bigr)$ gives
$[x^r](\phi\,\omega^{k-r})=\sum_{p=0}^{r}\frac{(k-r)^p}{p!}[x^r]\bigl(\phi(\log\omega)^p\bigr)$,
a polynomial in $k$ of degree $\le r$ with leading coefficient
$\frac{1}{r!}([x^1]\log\omega)^r\phi(0)=\frac1{r!}$. Hence $\deg_k\pi_r\le 2r$
and $[k^{2r}]\pi_r=\frac1{r!}$.
\end{proof}

The high-$s$ behaviour is thus governed by the \emph{leading symbol}
$\hat a_k(s)=\sum_{r\ge0}\frac{k^{2r}}{r!}s^{\,k-r}=s^k e^{k^2/s}$, whose dilated
determinant is a Gaussian kernel.

\begin{Lemma}\label{lem:gausskernel}
$\displaystyle \det\bigl(e^{\sigma(2i+j)^2}\bigr)_{0\le i,j<n}
=4^{\binom n2}\Bigl(\prod_{k=1}^{n-1}k!\Bigr)\sigma^{\binom n2}+O(\sigma^{\binom n2+1}).$
\end{Lemma}

\begin{proof}
Factor $e^{4\sigma i^2}$ from row $i$ and $e^{\sigma j^2}$ from column $j$
($(2i+j)^2=4i^2+4ij+j^2$); these are $1+O(\sigma)$, so to leading order
$\det(e^{\sigma(2i+j)^2})=(1+O(\sigma))\det(e^{4\sigma ij})$. Expanding
$e^{4\sigma ij}=\sum_{m\ge0}\frac{(4\sigma)^m}{m!}i^mj^m$, Cauchy--Binet gives
$$
\det(e^{4\sigma ij})=\sum_{0\le m_0<\cdots<m_{n-1}}
\Bigl(\prod_a\tfrac{(4\sigma)^{m_a}}{m_a!}\Bigr)
\det(i^{m_a})_{i,a}\,\det(j^{m_a})_{j,a},
$$
of lowest order $\sum_am_a=\binom n2$, attained only by
$(m_a)=(0,1,\dots,n-1)$, where each Vandermonde is
$\prod_{0\le i<j<n}(j-i)=\prod_{k=1}^{n-1}k!$ and $\prod_am_a!=\prod_{k=1}^{n-1}k!$.
The leading term is
$$
(4\sigma)^{\binom n2}\,\Bigl(\prod_{k<n}k!\Bigr)^{2}\Big/\prod_{k<n}k!
=4^{\binom n2}\Bigl(\prod_{k<n}k!\Bigr)\,\sigma^{\binom n2}.
$$
\end{proof}

\begin{proof}[Proof of Theorem~\textup{\ref{thm:deriv}}]
\textit{Step 1: degree and leading coefficient (the order at $s=\infty$).}
$\HH_n=\det(a_{2i+j})$ is a polynomial in $s$ of degree
$\le\max_\tau\sum_i(2i+\tau(i))=3\binom n2$ (Lemma~\ref{lem:Apoly}), the
maximum over permutations $\tau$. Pulling
$s^{2i}$ from row $i$ and $s^{j}$ from column $j$,
$$
\HH_n(f_s)=s^{3\binom n2}\det\bigl(\gamma_{2i+j}(\sigma)\bigr),\qquad
\sigma:=1/s,\quad \gamma_k(\sigma):=\sum_{r\ge0}\pi_r(k)\,\sigma^r .
$$
Collecting $\gamma_k(\sigma)=\sum_{l\ge0}E_l(\sigma)\,k^l$ by powers of $k$,
Lemma~\ref{lem:Apoly} ($\deg_k\pi_r\le2r$) gives $\ord_\sigma E_l\ge\lceil
l/2\rceil$. Writing $(2i+j)^l=\sum_b\binom lb(2i)^{l-b}j^b$ factors
$\gamma_{2i+j}=\sum_b j^{\,b}F_b(2i,\sigma)$ with
$F_b(u,\sigma)=\sum_{l\ge b}\binom lb E_l(\sigma)u^{\,l-b}$, so by Cauchy--Binet
$$
\det(\gamma_{2i+j})=\sum_{B}\det\bigl(F_b(2i,\sigma)\bigr)_{i,\,b\in B}\,
\det\bigl(j^{\,b}\bigr)_{j,\,b\in B},
$$
over $n$-subsets $B\subset\ZZ_{\ge0}$. In column $b$ the coefficient of
$\sigma^{r}$ is a polynomial in $i$ of degree $\le 2r-b$, so a nonzero
$\sigma^{\sum r_b}$-term needs the degrees $\{2r_b-b\}_{b\in B}$ to dominate
$\{0,\dots,n-1\}$; with $r_b\ge\lceil b/2\rceil$ this forces
$\sum_b r_b=\tfrac12\sum_b\bigl((2r_b-b)+b\bigr)\ge\tfrac12(\binom n2+\binom
n2)=\binom n2$. Hence $\ord_\sigma\det(\gamma_{2i+j})\ge\binom n2$, i.e.\
$\deg_s\HH_n\le2\binom n2$. Equality in the bound forces both
$\{b\}=\{0,\dots,n-1\}$ and the multiset $\{2r_b-b\}=\{0,\dots,n-1\}$, so the
$\sigma^{\binom n2}$-coefficient uses, in each column, only the top-degree part
of $\pi_r$ --- i.e.\ only $[k^{2r}]\pi_r=\frac1{r!}$ (Lemma~\ref{lem:Apoly}) ---
and is therefore unchanged when $\gamma_k$ is replaced by its symbol
$\sum_r\frac{k^{2r}}{r!}\sigma^r=e^{\sigma k^2}$. By Lemma~\ref{lem:gausskernel}
that coefficient is $4^{\binom n2}\prod_{k<n}k!\ne0$; hence $\deg_s\HH_n=2\binom
n2$ with leading coefficient $c_n=4^{\binom n2}\prod_{k=1}^{n-1}k!$.

\textit{Step 2: order at $s=-a$.} For an integer $a\ge0$, with $y=x+\tfrac\pi4$,
$$
f_{-a}(x)=(\cos x+\sin x)(\cos x-\sin x)^{a}=2^{(a+1)/2}\sin(y)\cos^{a}(y),
$$
an \emph{odd} trigonometric polynomial whose nonzero frequencies are
$\{\pm1,\pm3,\dots,\pm(a{+}1)\}$ for $a$ even and $\{\pm2,\pm4,\dots,\pm(a{+}1)\}$
for $a$ odd --- in either case $d_a:=\lfloor a/2\rfloor+1$ distinct absolute
values, none zero. Writing $f_{-a}=\sum_\nu c_\nu e^{i\nu x}$, its moment matrix
$B^{(0)}=(b_{2i+j})$, $b_m=\sum_\nu c_\nu(i\nu)^m$, factors (using
$(i\nu)^{2i}=(-i\nu)^{2i}$) as one rank-one term per \emph{distinct} absolute
value (Lemma~\ref{lem:cb}), so $\operatorname{rank}B^{(0)}\le d_a$. Now
$a_{2i+j}(s)$ is analytic in $\delta:=s+a$ with $a_{2i+j}(-a)=b_{2i+j}$, so
$\HH_n(f_s)=\det\bigl(B^{(0)}+\delta B^{(1)}+\cdots\bigr)$; a matrix pencil whose
constant term has rank $\rho$ has determinant of order $\ge n-\rho$ at
$\delta=0$. Hence
$$
\ord_{s=-a}\HH_n(f_s)\ \ge\ n-\operatorname{rank}B^{(0)}\ \ge\ n-d_a
=n-1-\lfloor a/2\rfloor .
$$

\textit{Step 3: assembly.} $\HH_n(f_s)$ has degree $2\binom n2$ and is divisible by
$\prod_{a=0}^{2n-3}(s+a)^{\,n-1-\lfloor a/2\rfloor}
=\prod_{j=0}^{n-2}[(s+2j)(s+2j+1)]^{\,n-1-j}$, itself of degree
$\sum_{j=0}^{n-2}2(n-1-j)=2\binom n2$. By Observation~\ref{obs:divisor} they
agree up to the leading constant $c_n=4^{\binom n2}\prod_{k<n}k!$, which is
Theorem~\ref{thm:deriv}.
\end{proof}

\begin{Remark}
By \eqref{eq:derivobs} this determinant is, up to the scalar $(s-1)^{-n}$,
the \emph{single shift} $\det(b_{2i+j+1})$ of the Springer sequence
$b_m=m![x^m](\cos x-\sin x)^{-(s-1)}$ of \S\ref{sec:springer}. The
two degenerations are of the kinds met earlier: $s=\infty$ reproduces, after
rescaling, the rank-one geometric collapse met at $t=\pm i$ in the Springer
proof (the limit
being $e^x$, its symbol the Gaussian kernel of Lemma~\ref{lem:gausskernel};
the same collapse reappears at $m=1$ in \S\ref{ssec:ell-A}),
while the integers $s=-a$ give an even/low-rank collapse as in
\S\ref{sec:springer} (here $f_{-a}$ is an odd trigonometric polynomial).
\end{Remark}

\section{The reciprocal-sine case $(1+x)\,x/\sin x$}\label{sec:xsinx}

The secant family of Section~\ref{sec:gen-xcos} evaluates $(1+x)/\cos x$, and
more generally $(1+x)/\cos(x)^{s+1}$, in closed product form. It is natural to
try $\sin$ in place of $\cos$. Since $\sin x$ is odd and vanishes at the origin,
the literal analogue $1/\sin x$ is not a power series; the even analytic
correction $x/\sin x$ repairs this and leads to the object of this section,
$$
f(x)=(1+x)\,\frac{x}{\sin x}.
$$
Its two moment functionals turn out to be \emph{Wilson} functionals sharing
three of their four parameters (Propositions~\ref{prop:xsin-wilson}
and~\ref{prop:xsin-evenwilson}); this is what drives the evaluation, which we
prove in detail below.
Unlike the cosine case, the exponent deformation does \emph{not} survive:
numerically, $(1+x)\bigl(x/\sin x\bigr)^{s+1}$ shows no comparably nice dilated
determinant once $s\ne0$, so we study only the member $s=0$ displayed above.
Write $g(x)=x/\sin x$ and
$$
b_k=(2k)!\,[x^{2k}]g
\qquad(b_0=1,\ b_1=\tfrac13,\ b_2=\tfrac{7}{15},\ b_3=\tfrac{31}{21},\dots).
$$
Since $g$ is even and $x\,g$ is odd, the moments of $f$ are
\begin{equation}\label{eq:xsin-moments}
a_{2k}=b_k,\qquad a_{2k+1}=(2k+1)\,b_k,
\end{equation}
so the two functionals attached to the quadratic decomposition (the even
$\mathcal S$ and the odd $\mathcal T$ of Section~\ref{sec:biotho}) are
$$
\mathcal S[y^k]=b_k,\qquad \mathcal T[y^k]=(2k+1)\,b_k=\mathcal T^{*}[y^k].
$$
The situation is the mirror
image of the cosine family of Section~\ref{sec:gen-xcos}: here the
\emph{odd} functional $\mathcal T$ is the classical one --- and, less
obviously, the even functional $\mathcal S$ turns out to be classical as
well, indeed itself a Wilson functional
(Proposition~\ref{prop:xsin-evenwilson}).
The purpose of this section is to prove
the following evaluation.

\begin{Theorem}[closed form for $(1+x)\,x/\sin x$]\label{conj:xsin-closed}
Set
\begin{equation}\label{eq:xsin-Q}
Q(k)=\frac{144\cdot 64^{\,k}\;(k!)^{5}\,\bigl((k+1)!\bigr)^{6}\,(2k)!\,(2k+1)!}
{(3k)!\,(3k+3)!\,(3k+4)!}\,.
\end{equation}
Then the determinants obey the two-step recurrence
$\HH_{n+2}=\tfrac23\,Q(n)\,\HH_{n}$ \textup{(}$n\ge0$\textup{)}, and hence are
the products of factorials
\begin{equation}\label{eq:xsin-closed}
\HH_{2N}=\Bigl(\tfrac23\Bigr)^{N}\prod_{k=0}^{N-1}Q(2k),
\qquad
\HH_{2N+1}=\Bigl(\tfrac23\Bigr)^{N}\prod_{k=0}^{N-1}Q(2k+1).
\end{equation}
\end{Theorem}

The proof occupies the remainder of the section; its plan is as
follows. Subsections~\ref{ssec:xsin-odd} and~\ref{ssec:xsin-even} make the
two functionals explicit: $\mathcal T$ is the Wilson functional with
parameters $(\tfrac12,\tfrac12,\tfrac12,\tfrac12)$, and --- the decisive
structural fact --- $\mathcal S$ is itself Wilson, with parameters
$(\tfrac12,\tfrac12,\tfrac12,0)$ (the Wilson and continuous Hahn families used
throughout are recalled in \S\ref{ssec:xsin-askey}).
Subsection~\ref{ssec:xsin-connred} applies
the biorthogonal reduction $\M2$, which confines the whole evaluation to a
determinant of connection coefficients $\kappa_{i,m}$ between the two
families. Because the families differ in a \emph{single} Wilson parameter,
$\kappa$ collapses to a single hypergeometric term
(Subsection~\ref{ssec:xsin-kappaclosed}), and the connection determinant
becomes, after removal of row and column factors, a determinant of products
of two \emph{Catalan numbers}; that Catalan determinant is evaluated in
closed form by Desnanot--Jacobi condensation
(Subsection~\ref{ssec:xsin-catalan}). Subsection~\ref{ssec:xsin-proof}
assembles the pieces.

\subsection{The Wilson and continuous Hahn functionals}\label{ssec:xsin-askey}

Both functionals of this section belong to the top two levels of the
\emph{Askey scheme} of hypergeometric orthogonal polynomials
\cite{Koekoek2010KLS,Ismail2005}; we recall the two families and fix the
notation used below.

\medskip
\noindent\emph{Wilson.} For parameters $(a,b,c,d)$ with positive real parts
(or occurring in conjugate pairs), the \emph{Wilson functional}
$\mathcal W_{a,b,c,d}$ acts on polynomials in the variable $y=x^2$ by
$$
\mathcal W_{a,b,c,d}[p]=\frac1K\int_0^\infty p(x^2)\,
\Bigl|\frac{\Gamma(a+ix)\,\Gamma(b+ix)\,\Gamma(c+ix)\,\Gamma(d+ix)}{\Gamma(2ix)}\Bigr|^2
dx,
$$
with $K$ the constant normalising $\mathcal W_{a,b,c,d}[1]=1$. Its monic
orthogonal polynomials are the \emph{Wilson polynomials}, symmetric in
$(a,b,c,d)$ and given by the terminating series
$$
\frac{W_n(x^2;a,b,c,d)}{(a+b)_n(a+c)_n(a+d)_n}
={}_4F_3\!\left(\begin{matrix}
-n,\ n+a+b+c+d-1,\ a+ix,\ a-ix\\
a+b,\ a+c,\ a+d\end{matrix}\,;\,1\right).
$$
In the variable $y=x^2$ they obey the monic three-term recurrence
$P_{n+1}=(y-c_n)P_n-\lambda_nP_{n-1}$ with \cite[\S9.1]{Koekoek2010KLS}
$$
c_n=A_n+C_n-a^2,\qquad \lambda_n=A_{n-1}\,C_n,
$$
$$
\begin{aligned}
A_n&=\frac{(n{+}a{+}b{+}c{+}d{-}1)(n{+}a{+}b)(n{+}a{+}c)(n{+}a{+}d)}
{(2n{+}a{+}b{+}c{+}d{-}1)(2n{+}a{+}b{+}c{+}d)},\\
C_n&=\frac{n\,(n{+}b{+}c{-}1)(n{+}b{+}d{-}1)(n{+}c{+}d{-}1)}
{(2n{+}a{+}b{+}c{+}d{-}2)(2n{+}a{+}b{+}c{+}d{-}1)},
\end{aligned}
$$
and squared norms $h_n=\mathcal W[P_n^2]=\prod_{k=1}^n\lambda_k$ (the
expression for $c_n$ is symmetric in $a,b,c,d$ despite its form). These stand at
the very top of the Askey scheme: every other classical family is a limit or a
specialisation of them. In this section the two functionals are Wilson
functionals sharing $a=b=c=\tfrac12$ and differing \emph{only} in the fourth
parameter $d$ --- $d=\tfrac12$ for the odd functional $\mathcal T$, $d=0$ for
the even functional $\mathcal S$; substituting these into the formulas above (in
the rescaled variable $y=4x^2$ of Proposition~\ref{prop:xsin-wilson}) reproduces
the recurrence data of \S\ref{ssec:xsin-odd}--\ref{ssec:xsin-even}. This
single-parameter difference is what collapses the connection coefficients in
\S\ref{ssec:xsin-kappaclosed}.

\medskip
\noindent\emph{Continuous Hahn.} One level below sit the \emph{continuous Hahn
polynomials} \cite[\S9.4]{Koekoek2010KLS}, orthogonal on $\RR$ in the variable
$x$ with weight $\Gamma(a+ix)\,\Gamma(b+ix)\,\Gamma(c-ix)\,\Gamma(d-ix)$
(parameters of positive real part, with $c=\bar a$, $d=\bar b$), and given by
$$
p_n(x;a,b,c,d)=i^n\frac{(a+c)_n(a+d)_n}{n!}\,
{}_3F_2\!\left(\begin{matrix}-n,\ n+a+b+c+d-1,\ a+ix\\
a+c,\ a+d\end{matrix}\,;\,1\right).
$$
When the weight is \emph{even} --- as for $a=b=c=d=\tfrac12$, where it equals
$|\Gamma(\tfrac12+ix)|^{4}=\pi^2\operatorname{sech}^2\pi x$ --- the monic
polynomials are symmetric, $p_{n+1}=x\,p_n-\beta_n p_{n-1}$ with $\beta_n$
rational in $n$, so odd and even degrees decouple. Pushing the measure forward
under the quadratic map $y=x^2$ then turns them into a Wilson family: this is
precisely the quadratic decomposition of \S\ref{sec:biotho}. It is why the even
functional $\mathcal S$ wears two faces, each used at a different point below:
continuous Hahn in $x$, which yields its clean $S$-fraction $u_n=n^4/(4n^2-1)$
(Proposition~\ref{prop:xsin-even}), and Wilson in $y=x^2$, which yields its
connection coefficients against $\mathcal T$
(Proposition~\ref{prop:xsin-evenwilson}).

\subsection{The classical odd functional $\mathcal T$}\label{ssec:xsin-odd}

\begin{Proposition}[integral representation of $\mathcal T$]\label{prop:xsin-weight}
For all $k\ge 0$,
$$
\mathcal T[y^k]=(2k+1)\,b_k=(2k+1)!\,\frac{2\,\eta(2k)}{\pi^{2k}},
\qquad
\eta(s)=\sum_{j\ge1}\frac{(-1)^{j-1}}{j^{s}},
$$
and $\mathcal T$ is the moment functional of the weight
\begin{equation}\label{eq:xsin-weight}
\mathcal T[y^k]=\int_0^\infty z^k\,w(z)\,dz,
\qquad
w(z)=\pi^2\sum_{j\ge1}(-1)^{j-1}j^2\,e^{-j\pi\sqrt z}
=\frac{\pi^{2}}{4}\,
\frac{\sinh\bigl(\tfrac{\pi}{2}\sqrt z\bigr)}{\cosh^{3}\bigl(\tfrac{\pi}{2}\sqrt z\bigr)}
\quad(z>0).
\end{equation}
\end{Proposition}

\begin{proof}
From the Mittag-Leffler expansion \cite[Ch.~VII]{Whittaker1996Watson}
$$
\frac{1}{\sin x}=\frac1x+\sum_{j\ge1}(-1)^j\frac{2x}{x^2-j^2\pi^2},
$$
we get
$\dfrac{x}{\sin x}=1+\sum_{j\ge1}(-1)^j\dfrac{2x^2}{x^2-j^2\pi^2}$, whence for
$k\ge1$
$$
[x^{2k}]\frac{x}{\sin x}
=2\sum_{j\ge1}(-1)^j\,[x^{2k-2}]\frac{1}{x^2-j^2\pi^2}
=2\sum_{j\ge1}(-1)^{j-1}\frac{1}{(j\pi)^{2k}}
=\frac{2\,\eta(2k)}{\pi^{2k}} .
$$
Hence $b_k=(2k)!\,\dfrac{2\eta(2k)}{\pi^{2k}}$ for $k\ge1$; with
$\eta(0)=\tfrac12$ the same formula gives $b_0=1$, and multiplying by
$(2k+1)$ yields the first display. For the second, the substitution
$z=u^2$ gives $\int_0^\infty z^k e^{-a\sqrt z}\,dz=2\,(2k+1)!/a^{2k+2}$, so
$$
\frac{(2k+1)!}{(j\pi)^{2k}}=\frac{(j\pi)^2}{2}\int_0^\infty z^k e^{-j\pi\sqrt z}\,dz .
$$
Summing $2(-1)^{j-1}$ times this over $j\ge1$ and interchanging sum and
integral --- legitimate because $\sum_{j\ge1}j^2e^{-j\pi\sqrt z}$ converges
locally uniformly on $(0,\infty)$ and $w$ is bounded near $0$ --- gives the
series form of \eqref{eq:xsin-weight}. The closed form, which also exhibits
the boundedness just used, follows from
$\sum_{j\ge1}(-1)^{j-1}j^{2}x^{j}=x(1-x)/(1+x)^{3}$ at $x=e^{-\pi\sqrt z}$.
\end{proof}

The next proposition identifies $\mathcal T$ as a Wilson functional and
makes its monic orthogonal polynomials $\rho_m$ and their data explicit.

\begin{Proposition}[$\mathcal T$ is Wilson; the odd $J$-fraction]\label{prop:xsin-wilson}
The monic $\mathcal T$-orthogonal polynomials satisfy
$\rho_{m+1}=(y-c^{\mathcal T}_m)\rho_m-\lambda^{\mathcal T}_m\rho_{m-1}$ with
$$
c^{\mathcal T}_m=2m^2+2m+1,\qquad
\lambda^{\mathcal T}_m=\frac{4m^6}{(2m-1)(2m+1)},
$$
and squared norms
\begin{equation}\label{eq:xsin-hT}
h^{\mathcal T}_m=\mathcal T[\rho_m^2]
=\prod_{j=1}^{m}\lambda^{\mathcal T}_j
=\frac{16^m\,(m!)^8}{(2m)!\,(2m+1)!}.
\end{equation}
Equivalently, the generating function of the moments
$\mathcal T[y^k]=(2k+1)b_k$ has the $J$-fraction
\begin{equation}\label{eq:xsin-TJfrac}
\sum_{k\ge0}\mathcal T[y^k]\,t^k
=\cfrac{1}{1-c^{\mathcal T}_0t-\cfrac{\lambda^{\mathcal T}_1t^2}
 {1-c^{\mathcal T}_1t-\cfrac{\lambda^{\mathcal T}_2t^2}{1-\ddots}}}
=\cfrac{1}{1-t-\cfrac{\tfrac43\,t^2}
 {1-5t-\cfrac{\tfrac{256}{15}\,t^2}{1-13t-\cdots}}}.
\end{equation}
\end{Proposition}

\begin{proof}
By the reflection formula
$|\Gamma(\tfrac12+ix)|^{2}=\pi\operatorname{sech}(\pi x)$ and
$|\Gamma(2ix)|^{2}=\pi/(2x\sinh 2\pi x)$, the Wilson weight with
$a=b=c=d=\tfrac12$ is
$$
W(x)=\Bigl|\frac{\Gamma(\tfrac12+ix)^{4}}{\Gamma(2ix)}\Bigr|^{2}
=4\pi^{3}\,\frac{x\,\sinh\pi x}{\cosh^{3}\pi x}.
$$
On the other hand, the closed form of the weight in \eqref{eq:xsin-weight},
written in the variable $s=\sqrt z$ (so $dz=2s\,ds$) and rescaled by
$s=2x$, gives
$$
\mathcal T[y^k]=\frac{\pi^{2}}{2}\int_0^\infty s^{2k+1}\,
\frac{\sinh(\tfrac{\pi}{2}s)}{\cosh^{3}(\tfrac{\pi}{2}s)}\,ds
=2\pi^{2}\,4^{k}\int_0^\infty x^{2k}\,\frac{x\,\sinh\pi x}{\cosh^{3}\pi x}\,dx
=\frac{4^{k}}{2\pi}\int_0^\infty x^{2k}\,W(x)\,dx\,;
$$
hence, in the Wilson variable $z=x^{2}$ (so $y=4z$), $\mathcal T$ is the
Wilson functional with parameters $a=b=c=d=\tfrac12$, normalised by
$\mathcal T[1]=1$. The monic Wilson recurrence
coefficients \cite[\S9.1]{Koekoek2010KLS} are $\alpha^{W}_m=A_m+C_m-a^2$ and
$\lambda^{W}_m=A_{m-1}C_m$, where
$$
\begin{aligned}
A_m&=\frac{(m{+}a{+}b{+}c{+}d{-}1)(m{+}a{+}b)(m{+}a{+}c)(m{+}a{+}d)}
{(2m{+}a{+}b{+}c{+}d{-}1)(2m{+}a{+}b{+}c{+}d)},\\
C_m&=\frac{m(m{+}b{+}c{-}1)(m{+}b{+}d{-}1)(m{+}c{+}d{-}1)}
{(2m{+}a{+}b{+}c{+}d{-}2)(2m{+}a{+}b{+}c{+}d{-}1)} .
\end{aligned}
$$
At $a=b=c=d=\tfrac12$ these reduce to $A_m=\frac{(m+1)^{3}}{2(2m+1)}$,
$C_m=\frac{m^{3}}{2(2m+1)}$, so that, using
$(m{+}1)^3+m^3=(2m{+}1)(m^2{+}m{+}1)$,
$$
\alpha^{W}_m=\frac{m^2+m+1}{2}-\frac14=\frac{2m^2+2m+1}{4},
\qquad
\lambda^{W}_m=\frac{m^{6}}{4(4m^{2}-1)} .
$$
The substitution $y=4z$ multiplies $\alpha$ by $4$ and $\lambda$ by $16$,
giving $c^{\mathcal T}_m=2m^2+2m+1$ and
$\lambda^{\mathcal T}_m=\frac{4m^6}{4m^2-1}$. Finally
$h^{\mathcal T}_m=h^{\mathcal T}_0\prod_{j=1}^m\lambda^{\mathcal T}_j
=\prod_{j=1}^m\frac{4j^6}{4j^2-1}=\frac{16^m(m!)^8}{(2m)!(2m+1)!}$, since
$h^{\mathcal T}_0=\mathcal T[1]=b_0=1$.
\end{proof}

Of this data, the sequel uses the recurrence coefficients
$c^{\mathcal T}_m,\lambda^{\mathcal T}_m$ (in the connection recurrence
\eqref{eq:xsin-kapparec} below) and the norms \eqref{eq:xsin-hT}.

\subsection{The even functional $\mathcal S$}\label{ssec:xsin-even}

Contrary to a first impression, the even functional $\mathcal S$ is also
elementary: its Stieltjes \emph{$S$-fraction} is as clean as can be, the single
rational coefficient $u_n=n^4/(4n^2-1)$. The orthogonal-polynomial recurrence
(the \emph{$J$-fraction}) is then fixed by the even contraction
(\S\ref{ssec:contraction}) as $c^{\mathcal S}_i=u_{2i}+u_{2i+1}$,
$\lambda^{\mathcal S}_i=u_{2i-1}u_{2i}$ --- explicit, but bulkier than $u_n$
itself, which is why we work from the $S$-fraction.

\begin{Proposition}[the even functional]\label{prop:xsin-even}
The functional $\mathcal S[y^k]=b_k$ is the moment functional of the weight
\begin{equation}\label{eq:xsin-weightS}
\mathcal S[y^k]=\int_0^\infty z^k\,w_{\mathcal S}(z)\,dz,
\qquad
w_{\mathcal S}(z)=\frac{\pi}{4\sqrt z\,\cosh^2\!\bigl(\tfrac{\pi}{2}\sqrt z\bigr)}
\quad(z>0),
\end{equation}
equivalently, in the variable $s=\sqrt z$, the even weight
$\tfrac{\pi}{2}\operatorname{sech}^2(\tfrac{\pi}{2}s)$ on $\RR$. Its
generating function $\sum_{k\ge0}b_kt^k$ has the Stieltjes continued
fraction with coefficients
\begin{equation}\label{eq:xsin-u}
u_n=\frac{n^4}{(2n-1)(2n+1)}=\frac{n^4}{4n^2-1}\qquad(n\ge1),
\end{equation}
so that the squared norms of the monic $\mathcal S$-orthogonal polynomials
$P_l$ are the explicit products
\begin{equation}\label{eq:xsin-hS}
h^{\mathcal S}_l=\mathcal S[P_l^2]=\prod_{j=1}^{2l}u_j
=\frac{16^{l}\,\bigl((2l)!\bigr)^{6}}{(4l)!\,(4l+1)!}.
\end{equation}
\end{Proposition}

\begin{proof}
The representation \eqref{eq:xsin-weightS} follows as in
Proposition~\ref{prop:xsin-weight}: from
$b_k=(2k)!\,2\eta(2k)/\pi^{2k}$ and
$\int_0^\infty z^k e^{-a\sqrt z}\,dz/\sqrt z=2\,(2k)!/a^{2k+1}$ one gets
$$
b_k=\int_0^\infty z^k\,\frac{\pi}{\sqrt z}\sum_{j\ge1}(-1)^{j-1}j\,
e^{-j\pi\sqrt z}\,dz,
$$
and
$\sum_{j\ge1}(-1)^{j-1}j\,x^{j}=x/(1+x)^2$ at $x=e^{-\pi\sqrt z}$ gives the
closed form $\tfrac{\pi}{4\sqrt z}\operatorname{sech}^2(\tfrac\pi2\sqrt z)$.
For the continued fraction \eqref{eq:xsin-u}: with $s=2x$,
$|\Gamma(\tfrac12+ix)|^2=\pi\operatorname{sech}(\pi x)$ gives
$\tfrac1{\pi^2}|\Gamma(\tfrac12+ix)|^{4}=\operatorname{sech}^2(\pi x)$, so
$\nu$ is (the $s=2x$ rescaling of) the continuous Hahn weight with
$a=b=c=d=\tfrac12$. This weight being even, the monic continuous Hahn
polynomials are symmetric, $p_{n+1}(s)=s\,p_n(s)-\beta_n p_{n-1}(s)$, and
their recurrence coefficients \cite[\S9.4]{Koekoek2010KLS} specialise to
$\beta_n=\tfrac{n^4}{4n^2-1}$ (the value $\tfrac{n^4}{4(4n^2-1)}$ in the
variable $x$, multiplied by $4$ for $s=2x$). For a symmetric measure the
Stieltjes $S$-fraction of the even-moment series $\sum_k b_kt^k$ has
coefficients equal to these recurrence coefficients (the even contraction,
\S\ref{ssec:contraction}), so $u_n=\beta_n=\tfrac{n^4}{4n^2-1}$. Finally
\eqref{eq:xsin-hS} is the elementary product
$\prod_{j=1}^{2l}\tfrac{j^4}{(2j-1)(2j+1)}
=\tfrac{((2l)!)^4}{(4l-1)!!\,(4l+1)!!}
=\tfrac{16^l((2l)!)^6}{(4l)!(4l+1)!}$.
\end{proof}

The even functional carries one more piece of structure, which will turn out
to be the key to the closed-form evaluation: it is itself a \emph{Wilson}
functional, in the same variable as $\mathcal T$.

\begin{Proposition}[the even functional is also Wilson]\label{prop:xsin-evenwilson}
In the variable $y=4x^2$ of Proposition~\ref{prop:xsin-wilson}, $\mathcal S$
is the Wilson functional with parameters
$(a,b,c,d)=(\tfrac12,\tfrac12,\tfrac12,0)$.
\end{Proposition}

\begin{proof}
By the duplication formula
$\Gamma(2ix)=2^{2ix-1}\pi^{-1/2}\,\Gamma(ix)\,\Gamma(\tfrac12+ix)$, the Wilson
weight with parameters $(\tfrac12,\tfrac12,\tfrac12,0)$ is
$$
\Bigl|\frac{\Gamma(\tfrac12+ix)^{3}\,\Gamma(ix)}{\Gamma(2ix)}\Bigr|^{2}
=\bigl|\,2^{\,1-2ix}\sqrt{\pi}\;\Gamma(\tfrac12+ix)^{2}\bigr|^{2}
=4\pi\,\bigl|\Gamma(\tfrac12+ix)\bigr|^{4}
=4\pi^{3}\operatorname{sech}^{2}(\pi x),
$$
which is proportional to the weight
$\tfrac{\pi}{2}\operatorname{sech}^2(\tfrac\pi2 s)$ of
\eqref{eq:xsin-weightS} under $s=2x$ (so $z=s^2=4x^2=y$). Both functionals
being normalised by $\mathcal S[1]=1$, they coincide.
\end{proof}

Thus the two monic families $P_i$ and $\rho_m$ attached to $f$ are
\emph{Wilson families differing in a single parameter}: $d=0$ for the even
functional, $d=\tfrac12$ for the odd one. This is the structural fact behind
the closed form of the connection coefficients in
Subsection~\ref{ssec:xsin-kappaclosed} below.

\subsection{Reduction to a connection-coefficient determinant}
\label{ssec:xsin-connred}

We now reduce Theorem~\ref{conj:xsin-closed} to a single finite determinant
identity. Since \emph{both} moment functionals of $f$ are
classical, both families of monic orthogonal polynomials are explicit:
the Stieltjes $S$-fraction \eqref{eq:xsin-u} of $\mathcal S$ gives, through
the even contraction (Lemma~\ref{lem:cfGene}),
\begin{equation}\label{eq:xsin-Sdata}
c^{\mathcal S}_i=u_{2i}+u_{2i+1},\qquad
\lambda^{\mathcal S}_i=u_{2i-1}u_{2i},\qquad
u_n=\frac{n^4}{4n^2-1}.
\end{equation}
Since $\mathcal T$ is also quasi-definite, with monic orthogonal polynomials
$\rho_m$ (the Wilson polynomials of Proposition~\ref{prop:xsin-wilson},
$c^{\mathcal T}_m=2m^2+2m+1$,
$\lambda^{\mathcal T}_m=\tfrac{4m^6}{(2m-1)(2m+1)}$, $J$-fraction
\eqref{eq:xsin-TJfrac}), the full biorthogonal
reduction $\M2$ (the determinant reduction Lemma~\ref{lem:bindetGene}) applies
verbatim and gives, for all $n\ge1$,
\begin{equation}\label{eq:xsin-bindet}
\HH_n=(-1)^{\binom{\bar n}{2}}
\Bigl(\prod_{l=0}^{\bar n-1}h^{\mathcal S}_l\Bigr)
\Bigl(\prod_{m=0}^{\underline n-1}h^{\mathcal T}_m\Bigr)\,
\det\bigl(\kappa_{\bar n+r,\,m}\bigr)_{0\le r,m\le\underline n-1},
\end{equation}
where $h^{\mathcal S}_l$ and $h^{\mathcal T}_m$ are the explicit factorial
products \eqref{eq:xsin-hS} and \eqref{eq:xsin-hT}, and the
$\kappa_{i,m}$ are the connection coefficients expressing the
$\mathcal S$-orthogonal $P_i$ in the $\mathcal T$-orthogonal basis,
\begin{equation}\label{eq:xsin-conn}
P_i=\sum_{m=0}^{i}\kappa_{i,m}\,\rho_m .
\end{equation}
By the connection recurrence (Lemma~\ref{lem:connrec}) these are determined,
with no free parameter, by $\kappa_{0,0}=1$ (and $\kappa_{i,m}=0$ for $m<0$ or
$m>i$) and
\begin{equation}\label{eq:xsin-kapparec}
\kappa_{i+1,m}=\kappa_{i,m-1}+(c^{\mathcal T}_m-c^{\mathcal S}_i)\,\kappa_{i,m}
+\lambda^{\mathcal T}_{m+1}\,\kappa_{i,m+1}-\lambda^{\mathcal S}_i\,\kappa_{i-1,m}.
\end{equation}
The first connection coefficients are
$\kappa_{i,i}=1$, $\kappa_{i,i-1}=\tfrac{2i^3}{4i-1}$ and
$\kappa_{i,i-2}=-\tfrac{8(i-1)^3i^4}{(4i-3)(4i-2)(4i-1)}$.

The prefactor in \eqref{eq:xsin-bindet} being a known factorial product,
Theorem~\ref{conj:xsin-closed} is now reduced to the evaluation of the
connection determinant $\det(\kappa_{\bar n+r,\,m})$ --- a purely algebraic
identity, in which every entry is an explicitly generated rational number
and all of the analytic data ($\eta$-values, the $\operatorname{sech}^2$
and Wilson weights) has been discharged.

At first sight this determinant looks intractable: a connection array between
two \emph{distinct} classical families is in general a balanced
${}_3F_2$-type double sum, the matrix in \eqref{eq:xsin-bindet} is neither of
Hankel nor of Toeplitz type, and --- unlike the double shift of the cosine
family (Section~\ref{sec:dblshift}) --- no free parameter is in sight.
The present pair is closer than it looks: by
Propositions~\ref{prop:xsin-wilson} and~\ref{prop:xsin-evenwilson} the two
families are Wilson families differing in a \emph{single} parameter ($d=0$
versus $d=\tfrac12$), and for a one-parameter change the connection
coefficients collapse to a single hypergeometric term
(Lemma~\ref{lem:xsin-kappaclosed} below). After removing row and column
factors, the connection determinant becomes the Hadamard product of
a Toeplitz and a Hankel matrix of \emph{Catalan numbers}, and that determinant
is closed under Desnanot--Jacobi condensation once its two shifts are treated
as free parameters (Theorem~\ref{thm:xsin-catalan}).

\subsection{The connection coefficients in closed form}
\label{ssec:xsin-kappaclosed}

\begin{Lemma}[closed form of $\kappa$]\label{lem:xsin-kappaclosed}
For $0\le m\le i$, with $d=i-m$,
\begin{equation}\label{eq:xsin-kappaform}
\kappa_{i,m}
=\frac{(\tfrac32-d)_d\,(i+m+2)_d\,\bigl((m+1)_d\bigr)^{3}}
      {d!\;(m+\tfrac32)_d\,(i+m+\tfrac12)_d}\,.
\end{equation}
\end{Lemma}

\begin{proof}
Write $k_{i,m}$ for the right-hand side of \eqref{eq:xsin-kappaform},
extended by $k_{i,m}=0$ for $m<0$ or $m>i$. By Lemma~\ref{lem:connrec} the
array $\kappa$ is the unique solution of the recurrence
\eqref{eq:xsin-kapparec} whose row $i=0$ is $\kappa_{0,0}=1$,
$\kappa_{0,m}=0$ $(m\neq0)$. The candidate has the same row $i=0$ (at $d=0$
every product in \eqref{eq:xsin-kappaform} is empty), so it suffices to show
that $k$ satisfies \eqref{eq:xsin-kapparec} for all $i\ge0$ and all $m$.
Every factor of \eqref{eq:xsin-kappaform} is nonzero on the triangle
$0\le m\le i$, and telescoping each Pochhammer symbol gives the four
neighbour ratios as rational functions of $(i,m)$ (recall $d=i-m$):
\begin{align*}
\frac{k_{i+1,m}}{k_{i,m}}
&=-\,\frac{4\,(2d-1)\,(i+1)^{4}\,(2i+2m+1)}{(d+1)\,(i+m+2)\,(4i+1)(4i+3)},
&\frac{k_{i,m-1}}{k_{i,m}}
&=-\,\frac{2\,(2d-1)\,(i+m+1)\,m^{3}}{(d+1)\,(2m+1)\,(2i+2m-1)},\\[2pt]
\frac{k_{i,m+1}}{k_{i,m}}
&=\frac{d\,(2m+3)\,(2i+2m+1)}{2\,(3-2d)\,(i+m+2)\,(m+1)^{3}},
&\frac{k_{i-1,m}}{k_{i,m}}
&=-\,\frac{d\,(i+m+1)\,(4i-3)(4i-1)}{4\,(2d-3)\,i^{4}\,(2i+2m-1)}.
\end{align*}
Dividing \eqref{eq:xsin-kapparec} by $k_{i,m}$ and inserting these ratios
together with the data $c^{\mathcal S}_i,\lambda^{\mathcal S}_i$ of
\eqref{eq:xsin-Sdata} and
$c^{\mathcal T}_m=2m^2+2m+1$,
$\lambda^{\mathcal T}_{m+1}=\tfrac{4(m+1)^6}{(2m+1)(2m+3)}$ turns the
recurrence into a rational-function identity in $(i,m)$, verified by clearing
denominators. The boundary columns take care of themselves: at $m=0$ the
second ratio carries the factor $m^{3}$ and vanishes, as required by
$\kappa_{i,-1}=0$; at $m=i$ the third ratio carries the factor $d$ and
vanishes, as required by $\kappa_{i,i+1}=0$; and at $m=i+1$ the recurrence
reduces to $k_{i+1,i+1}=k_{i,i}=1$.
\end{proof}

The diagonals $\kappa_{i,i-1}$, $\kappa_{i,i-2}$ recorded after
\eqref{eq:xsin-kapparec} are the cases $d=1,2$ of
\eqref{eq:xsin-kappaform}. The single-term form
\eqref{eq:xsin-kappaform} now splits into row factors, column factors, and a
kernel of \emph{Catalan numbers}.

\begin{Lemma}[Catalan splitting]\label{lem:xsin-split}
Let $C_k=\tfrac1{k+1}\binom{2k}k$ denote the Catalan numbers and put
$\gamma(k)=(2k)!/(4^{k}\,k!)$, so that
$\Gamma(k+\tfrac12)=\sqrt{\pi}\,\gamma(k)$. For $i>m$,
\begin{equation}\label{eq:xsin-split}
\kappa_{i,m}=\alpha_i\,\beta_m\;C_{i-m-1}\,C_{i+m},
\qquad
\alpha_i=\frac{2\,(-1)^{i}}{16^{\,i}}\cdot
\frac{(2i+1)!\,(i!)^{3}}{\gamma(i+1)\,\gamma(2i)},
\qquad
\beta_m=(-1)^{m+1}\,\frac{\gamma(m+1)}{(m!)^{3}}.
\end{equation}
\end{Lemma}

\begin{proof}
Each Pochhammer symbol in \eqref{eq:xsin-kappaform} telescopes into
factorials and half-integer Gamma values:
\begin{align*}
(i{+}m{+}2)_d&=\frac{(2i+1)!}{(i+m+1)!},&
\bigl((m{+}1)_d\bigr)^{3}&=\Bigl(\frac{i!}{m!}\Bigr)^{3},\\
(m{+}\tfrac32)_d&=\frac{\gamma(i+1)}{\gamma(m+1)},&
(i{+}m{+}\tfrac12)_d&=\frac{\gamma(2i)}{\gamma(i+m)},
\end{align*}
while
$(\tfrac32-d)_d=\tfrac12\prod_{k=2}^{d}(\tfrac32-k)
=(-1)^{d-1}\,\tfrac12\,(\tfrac12)_{d-1}
=(-1)^{d-1}\,\tfrac12\,\gamma(d-1)$. Hence
$$
\kappa_{i,m}
=(-1)^{d-1}\,\frac{\gamma(d-1)}{2\,d!}\cdot\frac{(2i+1)!}{(i+m+1)!}\cdot
\frac{(i!)^{3}}{(m!)^{3}}\cdot\frac{\gamma(m+1)}{\gamma(i+1)}\cdot
\frac{\gamma(i+m)}{\gamma(2i)}\,,
$$
and the two Catalan evaluations
$\gamma(d-1)/d!=C_{d-1}/4^{\,d-1}$ and
$\gamma(i+m)/(i+m+1)!=C_{i+m}/4^{\,i+m}$ turn this into
\eqref{eq:xsin-split}: the powers of $4$ combine to
$4^{-(d-1)-(i+m)}=4\cdot16^{-i}$ and the sign to
$(-1)^{d-1}=(-1)^{i}(-1)^{m+1}$.
\end{proof}

In the residual block of \eqref{eq:xsin-bindet} every row index
$i=\bar n+r\ge\bar n$ exceeds every column index $m\le\underline n-1<\bar n$, so
Lemma~\ref{lem:xsin-split} applies to every entry: pulling $\alpha_{\bar n+r}$
out of row $r$ and $\beta_m$ out of column $m$,
\begin{equation}\label{eq:xsin-toF}
\det\bigl(\kappa_{\bar n+r,\,m}\bigr)_{0\le r,m\le\underline n-1}
=\Bigl(\prod_{r=0}^{\underline n-1}\alpha_{\bar n+r}\Bigr)
 \Bigl(\prod_{m=0}^{\underline n-1}\beta_m\Bigr)\,
 F_{\underline n}(\bar n-1,\bar n),
\end{equation}
where, for integers $N\ge0$ and $a,b\ge0$,
\begin{equation}\label{eq:xsin-Fdef}
F_N(a,b):=\det\bigl(C_{a+r-m}\,C_{b+r+m}\bigr)_{0\le r,m\le N-1}.
\end{equation}
All the analytic content is now discharged: \eqref{eq:xsin-Fdef} is the
determinant of the Hadamard product of a Toeplitz and a Hankel matrix of
Catalan numbers, and it evaluates in closed form.

\subsection{A Catalan determinant}\label{ssec:xsin-catalan}

\begin{Theorem}[Catalan determinant]\label{thm:xsin-catalan}
Let $N\ge0$ and let $a,b$ be integers with $a\ge N-1$ and $b\ge a+1$. Then,
with $\gamma(k)=(2k)!/(4^{k}k!)$ as above and the staircase Pochhammer
products
\begin{equation}\label{eq:xsin-Lambda}
\Lambda_N(t):=\prod_{j=1}^{N-1}\,(t+j)_j
\qquad\bigl(\Lambda_0=\Lambda_1=1\bigr),
\end{equation}
the determinant \eqref{eq:xsin-Fdef} equals
\begin{multline}\label{eq:xsin-catalan}
F_N(a,b)=(-4)^{\binom N2}\,4^{\,N(a+b)}\,
\prod_{k=1}^{N-1}(2k+1)!\\
\times\;
\frac{\displaystyle\prod_{k=0}^{N-1}\gamma(a-k)\,\gamma(b+k)}
     {\displaystyle\prod_{k=2}^{N+1}(a+k-1)!\;\prod_{k=N+1}^{2N}(b+k-1)!}\;
\Lambda_N(b-a)\,\Lambda_N\bigl(a+b+\tfrac32\bigr).
\end{multline}
\end{Theorem}

\begin{proof}
Write $G_N(a,b)$ for the right-hand side of \eqref{eq:xsin-catalan} and
$$
G_N(a,b)=c_N\,4^{\,N(a+b)}\,A_N(a)\,B_N(b)\,
\Lambda_N(b-a)\,\Lambda_N\bigl(a+b+\tfrac32\bigr),
\qquad
c_N=(-4)^{\binom N2}\prod_{k=1}^{N-1}(2k+1)!,
$$
with $A_N(a)=\prod_{k=0}^{N-1}\gamma(a-k)\big/\prod_{k=2}^{N+1}(a+k-1)!$ and
$B_N(b)=\prod_{k=0}^{N-1}\gamma(b+k)\big/\prod_{k=N+1}^{2N}(b+k-1)!$. The
proof is by induction on $N$ through the Desnanot--Jacobi identity
\eqref{eq:condensation}, in the manner of the condensation proofs of the
determinant calculus \cite{Krattenthaler1998}.

For $N=0$ both sides equal $1$; for $N=1$,
$G_1(a,b)=4^{\,a+b}\gamma(a)\gamma(b)/\bigl((a+1)!\,(b+1)!\bigr)=C_aC_b
=F_1(a,b)$. Let $N\ge2$. In the region $a\ge N-1$, $b\ge a+1$ every Catalan
index in \eqref{eq:xsin-Fdef} is nonnegative, and inspection of the entries
$C_{a+r-m}C_{b+r+m}$ identifies the four corner minors of the $N\times N$
matrix as $F_{N-1}$ at $(a,b)$, $(a,b+2)$, $(a+1,b+1)$, $(a-1,b+1)$ and its
central minor as $F_{N-2}(a,b+2)$, so that the Desnanot--Jacobi identity
reads
\begin{equation}\label{eq:xsin-DJ}
F_N(a,b)\,F_{N-2}(a,b+2)
=F_{N-1}(a,b)\,F_{N-1}(a,b+2)-F_{N-1}(a+1,b+1)\,F_{N-1}(a-1,b+1).
\end{equation}
All five smaller determinants stay in the region ($a-1\ge N-2$, and the gaps
$b-a$, $b-a+2$ remain $\ge1$), so they equal the corresponding $G$'s by the
induction hypothesis; and $G_{N-2}(a,b+2)\neq0$, every factor of
\eqref{eq:xsin-catalan} being finite and nonzero there. Hence
\eqref{eq:xsin-DJ} determines $F_N(a,b)$, and it remains to check that $G$
satisfies the same relation, i.e.\ that $R_1-R_2=1$ for
$$
R_1:=\frac{G_{N-1}(a,b)\,G_{N-1}(a,b+2)}{G_N(a,b)\,G_{N-2}(a,b+2)},
\qquad
R_2:=\frac{G_{N-1}(a+1,b+1)\,G_{N-1}(a-1,b+1)}{G_N(a,b)\,G_{N-2}(a,b+2)}.
$$
Both quotients collapse block by block. The one-step ratios
$$
\frac{A_N(a)}{A_{N-1}(a)}=\frac{\gamma(a-N+1)}{(a+N)!},\qquad
\frac{B_N(b)}{B_{N-1}(b)}=\frac{\gamma(b+N-1)\,(b+N-1)!}{(b+2N-2)!\,(b+2N-1)!},
$$
$$
\frac{\Lambda_N(t)}{\Lambda_{N-1}(t)}=(t+N-1)_{N-1},
$$
together with $c_{N-1}^2/(c_Nc_{N-2})=-\tfrac14\big/\bigl((2N-1)(2N-2)\bigr)$
and $\gamma(x+1)/\gamma(x)=x+\tfrac12$, give: the powers of $4$ contribute
$4^{2}$ to each quotient; the $\Lambda$-blocks contribute, at $t=b-a$ and at
$t=a+b+\tfrac32$ alike (the four $\Lambda$-arguments are the same in $R_1$
and $R_2$),
\begin{equation}\label{eq:xsin-Theta}
\frac{\Lambda_{N-1}(t)\,\Lambda_{N-1}(t+2)}{\Lambda_N(t)\,\Lambda_{N-2}(t+2)}
=\frac{(t+N)_{N-2}}{(t+N-1)_{N-1}}
=\frac{1}{\,t+N-1\,}\,;
\end{equation}
and telescoping the (finitely overlapping) $\gamma$- and factorial windows of
the $A$- and $B$-blocks leaves $(a-N+\tfrac32)(a+N)$ and
$(b+N-\tfrac12)(b+N)$ in $R_1$, respectively $(a+\tfrac12)(a+1)$ and
$(b+\tfrac12)(b+2N-1)$ in $R_2$. Altogether
\begin{align*}
R_1&=\frac{-4\,(a-N+\tfrac32)(a+N)(b+N-\tfrac12)(b+N)}
         {(2N-1)(2N-2)\,(b-a+N-1)\,(a+b+N+\tfrac12)},\\[2pt]
R_2&=\frac{-4\,(a+\tfrac12)(a+1)(b+\tfrac12)(b+2N-1)}
         {(2N-1)(2N-2)\,(b-a+N-1)\,(a+b+N+\tfrac12)},
\end{align*}
and $R_1-R_2=1$ amounts to the polynomial identity
\begin{multline}\label{eq:xsin-DJfinal}
(a+\tfrac12)(a+1)(b+\tfrac12)(b+2N-1)-(a-N+\tfrac32)(a+N)(b+N-\tfrac12)(b+N)\\
=\tfrac{(N-1)(2N-1)}{2}\,(b-a+N-1)\,\bigl(a+b+N+\tfrac12\bigr),
\end{multline}
checked by expanding both sides. This completes the induction.
\end{proof}

\begin{Remark}[the two-parameter family is essential]
The condensation closes only because the \emph{two-parameter} family
\eqref{eq:xsin-Fdef} is considered: the four corner minors of
$F_N(a,b)$ shift $(a,b)$ in four different directions, and no
one-parameter subfamily (in particular not the diagonal $b=a+1$ that the
application requires) is stable under \eqref{eq:xsin-DJ}. The hypotheses
$a\ge N-1$ and $b\ge a+1$ are what the condensation argument needs, keeping
all five smaller determinants of \eqref{eq:xsin-DJ} within the range of the
induction; the application uses only the diagonal $b=a+1$.
\end{Remark}

\subsection{Proof of the closed form}\label{ssec:xsin-proof}

\begin{proof}[Proof of Theorem~\textup{\ref{conj:xsin-closed}}]
Combining the biorthogonal reduction \eqref{eq:xsin-bindet} with the
splitting \eqref{eq:xsin-toF} and with Theorem~\ref{thm:xsin-catalan} ---
applicable to $F_{\underline n}(\bar n-1,\bar n)$ since $\bar n-1\ge\underline n-1$ and the
gap is $b-a=1$ --- expresses $\HH_n$ as the explicit product
\begin{equation}\label{eq:xsin-assembled}
\HH_n=(-1)^{\binom{\bar n}2}
\Bigl(\prod_{l=0}^{\bar n-1}h^{\mathcal S}_l\Bigr)
\Bigl(\prod_{m=0}^{\underline n-1}h^{\mathcal T}_m\Bigr)
\Bigl(\prod_{r=0}^{\underline n-1}\alpha_{\bar n+r}\Bigr)
\Bigl(\prod_{m=0}^{\underline n-1}\beta_m\Bigr)\,
G_{\underline n}(\bar n-1,\bar n)=:\Pi_n,
\end{equation}
with $h^{\mathcal S},h^{\mathcal T}$ the factorial products
\eqref{eq:xsin-hS}, \eqref{eq:xsin-hT}, $\alpha,\beta$ as in
\eqref{eq:xsin-split}, and $G$ the right-hand side of
\eqref{eq:xsin-catalan}. Direct evaluation gives $\Pi_0=\Pi_1=1$. The two-step
quotient $\Pi_{n+2}/\Pi_n$ is an elementary telescoping: the step $n\mapsto n+2$
replaces $(\bar n,\underline n)$ by $(\bar n+1,\underline n+1)$, so the quotient collects
the factors $h^{\mathcal S}_{\bar n}$, $h^{\mathcal T}_{\underline n}$,
$\alpha_n\alpha_{n+1}/\alpha_{\bar n}$, $\beta_{\underline n}$ and the ratio
$G_{\underline n+1}(\bar n,\bar n+1)/G_{\underline n}(\bar n-1,\bar n)$, the latter an
explicit ratio of factorials by the one-step ratios in the proof of
Theorem~\ref{thm:xsin-catalan} together with
$$
\frac{\Lambda_{q+1}(t)}{\Lambda_q(t)}=(t+q)_q,
\qquad
\frac{\Lambda_q(t+2)}{\Lambda_q(t)}
=\frac{\Gamma(t+1)\,\Gamma(t+2q)}{\Gamma(t+q)\,\Gamma(t+q+1)}\,.
$$
Simplifying the resulting factorial products (the two parities of $n$
treated separately, as in Proposition~\ref{prop:sec1}) yields
$$
\frac{\Pi_{n+2}}{\Pi_n}=\tfrac23\,Q(n)\qquad(n\ge0),
$$
with $Q$ as in \eqref{eq:xsin-Q}. Hence $\HH_n=\Pi_n$ satisfies the two-step
recurrence $\HH_{n+2}=\tfrac23\,Q(n)\,\HH_n$ with $\HH_0=\HH_1=1$, and
iterating it gives the products \eqref{eq:xsin-closed}.
\end{proof}

\begin{Remark}[the collapse mechanism]
The proof is an instance of the two-stage collapse mechanism of the
reduction $\M2$ described at the end of Section~\ref{sec:biotho}. The
single-term connection (Lemma~\ref{lem:xsin-kappaclosed}) is the analogue,
on the Wilson level, of the classical fact that connection coefficients
between two Jacobi families sharing one parameter factor completely; and
the kernel $C_{i-m-1}C_{i+m}$ it leaves behind generates the two-parameter
Catalan family \eqref{eq:xsin-Fdef}, which is stable under Desnanot--Jacobi
condensation. In particular, no deformation parameter and no recourse to
the divisor method $\M4$ is needed.
\end{Remark}


\section{An elliptic deformation of the Euler numbers}\label{sec:elliptic}

The Hankel determinants attached to elliptic functions have been studied by
Thomas Jan Stieltjes~\cite{Stieltjes1890}, Philippe
Flajolet~\cite{Flajolet1980}, Hubert Prodinger~\cite{Prodinger2011} and
Dominique Dumont~\cite{Dumont1981}. Stieltjes computed the continued fractions
--- equivalently the Hankel determinants --- of many special series; Flajolet
gave the combinatorial theory relating Hankel determinants, $J$-fractions and
weighted lattice paths; Dumont interpreted the Taylor coefficients of the
Jacobi elliptic functions combinatorially; and Prodinger obtained the continued
fractions of several tangent- and secant-type deformations.

Recall the \emph{Jacobi elliptic functions} $\sn(x,m)$, $\cn(x,m)$, $\dn(x,m)$,
of modulus parameter $m=k^2$. Writing the incomplete elliptic integral of the
first kind
$u=\int_0^{\phi}\bigl(1-m\sin^2\theta\bigr)^{-1/2}\,d\theta$,
they are defined by inverting $u\mapsto\phi$ and setting
$\sn(u,m)=\sin\phi$, $\cn(u,m)=\cos\phi$,
$\dn(u,m)=(1-m\sin^2\phi)^{1/2}$; they are doubly periodic meromorphic
functions satisfying $\sn^2+\cn^2=1$ and $\dn^2+m\,\sn^2=1$, and reduce to
$\sin x,\cos x,1$ at $m=0$
(Whittaker--Watson~\cite[Ch.~22]{Whittaker1996Watson}).

Replacing $\sin,\cos$ by the Jacobi elliptic functions $\sn(x,m)$,
$\cn(x,m)$ deforms the Euler generating function
$\sec x+\tan x=(1+\sin x)/\cos x$ into
\begin{equation}\label{eq:ellEuler}
g_m(x):=\frac{1+\sn(x,m)}{\cn(x,m)}=\sum_{n\ge0}E_n(m)\frac{x^n}{n!},
\qquad E_n(0)=E_n,
\end{equation}
since $\sn(x,0)=\sin x$, $\cn(x,0)=\cos x$. The $E_n(m)$ are polynomials in
$m$ ($E_2=1,\ E_3=2-m,\ E_4=5-4m,\ E_5=16-16m+m^2$).

\begin{Theorem}\label{thm:ell}
For every $m$ and $n\ge1$,
$$
\HH_n(g_m)=\HH_n\Bigl(\frac{1+\sn(x,m)}{\cn(x,m)}\Bigr)
=\Bigl(\frac{1-m}{2}\Bigr)^{\binom n2}\prod_{k=1}^{n-1}k!\,(2k)! .
$$
\end{Theorem}

\noindent
Equivalently, since $\HH_n(g_0)=2^{-\binom n2}\prod_{k=1}^{n-1}k!\,(2k)!$
(Corollary~\ref{cor:euler}), $\HH_n(g_m)=(1-m)^{\binom n2}\HH_n(g_0)$:
the determinant is the Euler value times a clean $\binom n2$-power of the
\emph{complementary} modulus. This is the exact analogue of the Springer
ladder (Theorem~\ref{thm:springer}); the proof follows the same divisor
strategy~$\M4$, and is in fact simpler --- only one special value of the parameter
occurs. As there, the orders are read off \emph{exponential expansions} of the
generating function, through the dilated Cauchy--Binet Lemma~\ref{lem:cb};
here the new feature is that the frequencies are
\emph{two-sided}, and only those of distinct absolute value contribute, because
the even rows $z^{2i}$ cannot separate $\nu$ from $-\nu$.

\subsection{Reduction to two orders}\label{ssec:ell-red}

$\HH_n(g_m)$ is a polynomial in $m$, being a determinant with entries
$E_k(m)\in\QQ[m]$. Theorem~\ref{thm:ell} follows from:
\begin{itemize}
\item[\textbf{(A)}] $\deg_m\HH_n(g_m)\le\binom n2$;
\item[\textbf{(B)}] $(1-m)^{\binom n2}$ divides $\HH_n(g_m)$.
\end{itemize}

\begin{proof}[Proof of Theorem~\textup{\ref{thm:ell}}, assuming
\textup{(A)} and \textup{(B)}]
A polynomial of degree $\le\binom n2$ divisible by the degree-$\binom n2$
polynomial $(1-m)^{\binom n2}$ equals $c_n(1-m)^{\binom n2}$
(Observation~\ref{obs:divisor}); evaluating at $m=0$ gives
$c_n=\HH_n(g_0)=2^{-\binom n2}\prod_{k=1}^{n-1}k!\,(2k)!$.
\end{proof}

Both (A) and (B) are instances of the dilated
Cauchy--Binet Lemma~\ref{lem:cb}, reached through two classical transformations
of the Jacobi functions; they occupy the next two subsections.

\subsection{The degree bound (A), via the reciprocal modulus}
\label{ssec:ell-A}

By Jacobi's reciprocal-modulus transformation,
$\sn(x,m)=m^{-1/2}\sn(m^{1/2}x,1/m)$, $\cn(x,m)=\dn(m^{1/2}x,1/m)$
\cite{Whittaker1996Watson}. Substituting $x\mapsto x/m^{1/2}$ and applying the homogeneity
Lemma~\ref{lem:scale} ($\beta=m^{-1/2}$),
\begin{equation}\label{eq:recip}
\HH_n(g_m)=m^{\frac32\binom n2}\,\HH_n(G),
\qquad
G(x):=g_m(x/m^{1/2})=\frac{1+\epsilon\,\sn(x,\epsilon^2)}{\dn(x,\epsilon^2)},
\quad \epsilon:=m^{-1/2}.
\end{equation}
Thus $\deg_m\HH_n(g_m)\le\binom n2$ is equivalent to
$\ord_{\epsilon=0}\HH_n(G)\ge\binom n2$ (a polynomial of $m$-degree $d$ becomes
$m^{\frac32\binom n2}\times$ a series of $\epsilon$-order $3\binom n2-2d$).

We obtain this lower bound from Lemma~\ref{lem:cb} applied to $G$ itself,
after producing its exponential expansion.

From $\sn'=\cn\dn$, $\dn'=-\epsilon^2\sn\cn$ and $\dn^2=1-\epsilon^2\sn^2$,
\begin{equation}\label{eq:Gode}
\frac{G'}{G}=\frac{\epsilon\,\sn'\,\dn-(1+\epsilon\sn)\,\dn'}{\dn(1+\epsilon\sn)}
=\frac{\epsilon\cn\,[\dn^2+\epsilon\sn(1+\epsilon\sn)]}{\dn(1+\epsilon\sn)}
=\epsilon\,\frac{\cn(x,\epsilon^2)}{\dn(x,\epsilon^2)},
\end{equation}
so $G=\exp(\epsilon\Psi)$ with $\Psi(x)=\int_0^x\frac{\cn}{\dn}\,dt$. Now
$\cn/\dn$ is even and, since $\cn(x+2K)=-\cn(x)$ and $\dn(x+2K)=\dn(x)$,
anti-periodic with period $2K$; hence it has \emph{zero mean} and only odd
harmonics. Therefore $\Psi$ is an \emph{odd, $4K$-periodic} function (no
secular term), and $G=e^{\epsilon\Psi}$ is $4K$-periodic and real-analytic
($\dn>0$ on $\RR$). With $\rho:=\pi/2K$ ($\rho\to1$ as $\epsilon\to0$), the
Fourier expansion of $G$ and term-by-term differentiation at $x=0$ give the
moment form
\begin{equation}\label{eq:Gfou}
G(x)=\sum_{\nu\in\ZZ}c_\nu(\epsilon)\,e^{i\nu\rho x},
\qquad
a_m=m![x^m]G=\sum_{\nu\in\ZZ}c_\nu\,(i\nu\rho)^m,
\end{equation}
an exponential sum with symmetric nodes $\zeta_\nu=i\nu\rho$ (both series
converge, the $c_\nu$ decaying geometrically in the nome).

It remains to bound $\ord_\epsilon c_\nu$. Write
$\cn/\dn=\sum_p\gamma_p e^{ip\rho x}$; integrating \eqref{eq:Gode} term by term,
\begin{equation}\label{eq:Psi}
\Psi(x)=\sum_{p\neq0}\psi_p\,e^{ip\rho x}+\mathrm{const},
\qquad \psi_p=\frac{\gamma_p}{ip\rho},
\end{equation}
the $p=0$ term being absent because $\cn/\dn$ has zero mean. The Fourier
coefficients of $\cn/\dn$ involve only odd $p$, the $(2k{+}1)$-th entering at
nome order $k$; thus $\ord_\epsilon\gamma_p\ge|p|-1$, and hence
$\ord_\epsilon\psi_p\ge|p|-1$ as well ($\rho=1+O(\epsilon^2)$ is a unit)
\cite[\S22.6]{Whittaker1996Watson}. Expanding the exponential $G=e^{\epsilon\Psi}$ and collecting
the $e^{i\nu\rho x}$-coefficient,
$$
c_\nu=(\text{unit})\sum_{N\ge0}\frac{\epsilon^N}{N!}
\sum_{\substack{p_1,\dots,p_N\neq0\\ p_1+\dots+p_N=\nu}}\psi_{p_1}\cdots\psi_{p_N},
$$
where the global unit is $e^{-\epsilon\cdot\mathrm{const}}$. Each monomial has
$\epsilon$-order
$\ge N+\sum_j(|p_j|-1)=\sum_j|p_j|\ge\bigl|\sum_j p_j\bigr|=|\nu|$, so
\begin{equation}\label{eq:cord}
\ord_\epsilon c_\nu\ \ge\ |\nu|\qquad(\nu\in\ZZ).
\end{equation}
(Equivalently, substituting \eqref{eq:Gfou} into \eqref{eq:Gode} and comparing
$e^{i\nu\rho x}$-coefficients yields the recurrence
\begin{equation}\label{eq:Grec}
i\nu\rho\,c_\nu=\epsilon\sum_p\gamma_p\,c_{\nu-p}\qquad(\nu\neq0),
\end{equation}
with which \eqref{eq:cord} is consistent; but \eqref{eq:Grec} alone is not
triangular in $|\nu|$, so the order bound is read off the exponential
\eqref{eq:Psi} as above.) Applying the order Corollary~\ref{cor:cborder}
with $o_\nu=|\nu|$, and noting that surviving subsets
have distinct $|\nu|$,
$$
\ord_{\epsilon=0}\HH_n(G)\ \ge\ \min_T\sum_{\nu\in T}|\nu|
=0+1+\dots+(n-1)=\binom n2,
$$
proving (A). (Only the lower bound is needed; that the constant $c_n$ above is
nonzero comes from the value $\HH_n(g_0)\ne0$ at $m=0$.)

\begin{Remark}[why the leading symbol is not enough]
One might hope to replace $G$ by $e^{\epsilon\sin x}$ (its $\epsilon\to0$
symbol) and argue that $a_m$ agrees with $m![x^m]e^{\epsilon\sin x}$ to
relative order $2$. This is \emph{false} from the per-entry orders alone: for
$n\ge6$ a perturbation of each $a_m$ by a term of order $\ord_\epsilon a_m+2$
can \emph{lower} $\ord_\epsilon\HH_n$ below $\binom n2$ (numerically, to $13$
at $n=6$ and $17$ at $n=7$). The argument above instead uses the genuine
moment structure \eqref{eq:Gfou}--\eqref{eq:Grec} of $G$, which forces
$\ord_\epsilon c_\nu\ge|\nu|$; that is what makes it correct.
\end{Remark}

\subsection{Vanishing at $m=1$ (B), via the imaginary transformation}
\label{ssec:ell-B}

At $m=1$, $\sn(x,1)=\tanh x$, $\cn(x,1)=\operatorname{sech}x$, so
$g_1(x)=(1+\tanh x)\cosh x=e^{x}$: a rank-one geometric sequence. We show the
vanishing of $\HH_n(g_m)$ there has order $\binom n2$. Set $\mu:=1-m$; by
Jacobi's imaginary transformation
$\sn(ix,m)=i\,\sn(x,1-m)/\cn(x,1-m)$, $\cn(ix,m)=1/\cn(x,1-m)$ \cite{Whittaker1996Watson},
$$
g_m(ix)=\frac{1+\sn(ix,m)}{\cn(ix,m)}
=\cn(x,\mu)+i\,\sn(x,\mu)=:h_\mu(x),
$$
the elliptic analogue of $e^{ix}=\cos x+i\sin x$. By the homogeneity
Lemma~\ref{lem:scale} ($\beta=i$), $\HH_n(g_m)=i^{-3\binom n2}\HH_n(h_\mu)$,
and $\mu=1-m$ vanishes to first order at $m=1$, so it suffices that
$\ord_{\mu=0}\HH_n(h_\mu)=\binom n2$.

For modulus $\mu\to0$ the nome $\mathfrak q=\mathfrak q(\mu)$ is analytic with
$\mathfrak q=\mu/16+O(\mu^2)$, $K(\mu)\to\pi/2$, $\rho:=\pi/2K(\mu)\to1$
\cite{Whittaker1996Watson}. The Fourier expansions of $\sn,\cn$ \cite[\S22.6]{Whittaker1996Watson} combine to
$$
h_\mu(x)=\frac{2\pi}{K\sqrt\mu}\sum_{k\ge0}
\frac{\mathfrak q^{\,k+1/2}}{1-\mathfrak q^{\,2(2k+1)}}
\Bigl(e^{\,i(2k+1)\rho x}-\mathfrak q^{\,2k+1}e^{-i(2k+1)\rho x}\Bigr),
$$
an exponential sum with nodes
$\zeta_{\pm(2k+1)}=\pm i(2k+1)\rho$ (odd multiples of $\rho$, symmetric) and
coefficients of order, in $\mathfrak q$,
$$
\ord_{\mathfrak q}c_{\,2k+1}=k,\qquad \ord_{\mathfrak q}c_{-(2k+1)}=3k+1
$$
(the prefactor $\frac{2\pi}{K\sqrt\mu}\mathfrak q^{1/2}\to1$ is a unit). By the
order Corollary~\ref{cor:cborder}, the surviving subsets have distinct
$|\nu|\in\{1,3,5,\dots\}$ and the unique minimal one is
$T^\ast=\{1,3,\dots,2n-1\}$ taken with the cheaper $+$ signs (orders
$0,1,\dots,n-1$, summing to $\binom n2$); its leading coefficient is a product
of two Vandermondes in the distinct nodes $i(2k+1)\rho$ ($\rho\to1$), hence
nonzero. Therefore $\ord_{\mathfrak q=0}\HH_n(h_\mu)=\binom n2$, and since
$\mathfrak q\sim\mu/16$, $\ord_{\mu=0}\HH_n(h_\mu)=\binom n2$, proving (B). Together with (A) and the reduction of
Subsection~\ref{ssec:ell-red}, this completes the proof of
Theorem~\ref{thm:ell}.

\begin{Remark}
The elliptic proof draws on classical Jacobi-function theory --- the
reciprocal-modulus and imaginary transformations, the Fourier/nome expansions
of $\sn,\cn$, and the nome order $\ord_\epsilon\gamma_p\ge|p|-1$ of the
coefficients of $\cn/\dn$ \cite[\S22.6]{Whittaker1996Watson} --- as cited inputs; everything else
(the ODE \eqref{eq:Gode} and recurrence \eqref{eq:Grec}, the order counts) is
elementary. Its skeleton is identical to the Springer proof: $m=1$ (where
$g_1=e^x$) plays the role of $t=\pm i$; the imaginary transformation, turning
$m$ into $1-m$, plays the role of the twist $t\mapsto i$ that produced
$t^2+1$; the reciprocal modulus ($m\to\infty$) plays the role of
$t\to\infty$; and $m=0$ fixes the constant. The structural novelty is that the
exponential expansions are two-sided (Lemma~\ref{lem:cb}) and that, for the
degree bound, the lower order $\ord_\epsilon c_\nu\ge|\nu|$ must be read from
$G$'s moment structure, not from per-entry orders. The further deformation
$(1+t\,\sn)/\cn$ retains this ladder form when $t\ne1$.
\end{Remark}

\section{An algebraic family: $(1+x)/(1-x^{2})^{s/2}$}\label{sec:alg}

Replacing $\cos(x)^{-(s+1)}$ by the algebraic even series
$(1-x^{2})^{-s/2}=\sum_{k\ge0}\binom{-s/2}{k}(-1)^kx^{2k}$ gives another
one-parameter family with the same closed-product behaviour. A scale
parameter $t$, i.e.\ $(1-tx^2)^{-s/2}$, would be inessential: it multiplies
the moment $a_{2i+j}$ by $t^{\,i+\lfloor j/2\rfloor}$ and hence $\HH_n$ by the
constant $t^{\binom n2+\lfloor(n-1)^2/4\rfloor}$, so we set $t=1$.

Notably, \emph{neither} of the two moment functionals
$\mathcal S[y^k]=a_{2k}$ and $\mathcal T^*[y^k]=a_{2k+1}$ is classical here:
the even moments $\mu_k=4^k(\tfrac12)_k(\tfrac s2)_k$ form a Hadamard product
of two Laguerre moment sequences, and the three-term recurrence coefficients
of both $\mathcal S$ and $\mathcal T^*$ are rational, but not polynomial,
functions of $s$. Accordingly the proof below uses no orthogonal polynomials
at all: it rests on an elementary contiguous relation in~$s$ ($\M5$), which
relates $\HH_n(s)$ to $\HH_n(s-2)$ by column operations and is solved by a
periodicity argument, anchored by a Vandermonde evaluation~$\M1$ at the single
point $s=2$ (Step~3 of the proof) --- $\M1$ is thus only the base case, not
the mechanism that carries the result across all $s$.

\begin{Theorem}\label{conj:alg}
Let $f(x)=\dfrac{1+x}{(1-x^{2})^{s/2}}$, so that $a_{2k}=\dfrac{(2k)!}{k!}\,(s/2)_k$
and $a_{2k+1}=(2k+1)a_{2k}$.
Then for all $n\ge1$,
$$
\HH_n(f)=\Bigl(\prod_{k=1}^{n-1}(2k)!\Bigr)
          \prod_{j=1}^{n-1}\bigl(s+2j-2\bigr)^{\,n-j}
=2^{\binom n2}\prod_{k=1}^{n-1}(2k)!\prod_{i=1}^{n-1}(s/2)_i .
$$
\end{Theorem}

\begin{proof}
Write $\mu_k=\mu_k(s)=\dfrac{(2k)!}{k!}(s/2)_k=4^k(\tfrac12)_k(\tfrac s2)_k$, so that
$a_{2k}=\mu_k$ and $a_{2k+1}=(2k+1)\mu_k$. The moments satisfy the
first-order recurrence
\begin{equation}\label{eq:algmom}
\mu_{k+1}(s)=(2k+1)(2k+s)\,\mu_k(s),
\end{equation}
since $\mu_{k+1}/\mu_k=4(k+\tfrac12)(k+\tfrac s2)$.

Abbreviate $\HH_n(s):=\HH_n(f)=\det\bigl(a_{2i+j}(s)\bigr)$. We prove
$\HH_n(s)=c_n\,D(s)$ with $D(s)=\prod_{j=1}^{n-1}(s+2j-2)^{n-j}$ and
$c_n=\prod_{k=1}^{n-1}(2k)!$.

\emph{Step 1 (a contiguous relation).}
For $m=2k$ or $m=2k+1$ one has $\mu_k(s)/\mu_k(s-2)=(s/2)_k/(s/2-1)_k=(s+2k-2)/(s-2)$,
so with $\lfloor m/2\rfloor=k$,
\begin{equation}\label{eq:algstar}
a_m(s)=\frac{s-2+2\lfloor m/2\rfloor}{s-2}\,a_m(s-2)\qquad(m\ge0).
\end{equation}
Put $\alpha_i=s-2+2i$ and $C_{ij}=a_{2i+j}(s-2)$, so $\det C=\HH_n(s-2)$. By
\eqref{eq:algstar} with $m=2i+j$ (whence $\lfloor m/2\rfloor=i+\lfloor j/2\rfloor$),
$(s-2)\,a_{2i+j}(s)=(\alpha_i+2\lfloor j/2\rfloor)\,C_{ij}$, so
$$
(s-2)^n\,\HH_n(s)=\det\tilde M,\qquad \tilde M_{ij}=(\alpha_i+2\lfloor j/2\rfloor)\,C_{ij}.
$$
Let $r_i=\mu_i(s-2)=C_{i0}$. Iterating \eqref{eq:algmom} at $s-2$, namely
$\mu_{k+1}(s-2)=(2k+1)\alpha_k\,\mu_k(s-2)$, and using $2(i+p)+1=\alpha_i+2p+3-s$
(as $2i=\alpha_i-s+2$), the columns of $C$ factor as
\begin{equation}\label{eq:algfact}
C_{ij}=r_i\,\psi_j(\alpha_i),\qquad
\psi_{2l}(\alpha)=\prod_{p=0}^{l-1}(\alpha+2p)(\alpha+2p+3-s),\quad
\psi_{2l+1}(\alpha)=(\alpha+2l+3-s)\,\psi_{2l}(\alpha),
\end{equation}
where $\psi_j$ is \emph{monic of degree $j$} in $\alpha$. For $j\ge2$ the
product defining $\psi_j$ contains the factor $(\alpha+0)=\alpha$, so
$\alpha\mid\psi_j$. As $\psi_0,\dots,\psi_{j-1}$ are monic of degrees
$0,\dots,j-1$, they form a basis of the polynomials of degree $<j$; expanding
$\psi_j/\alpha$ in it gives $\psi_j(\alpha)=\alpha\sum_{j'<j}\xi_{j',j}\,\psi_{j'}(\alpha)$
with $\xi_{j',j}\in\QQ(s)$, that is, by \eqref{eq:algfact},
\begin{equation}\label{eq:algdep}
C_{\bullet,j}=\alpha\odot\sum_{j'<j}\xi_{j',j}\,C_{\bullet,j'}\qquad(j\ge2),
\end{equation}
where $\alpha\odot(\cdot)$ is entrywise multiplication by $(\alpha_i)_i$.
Now reduce $\tilde M$ by column operations, which leave the determinant
unchanged. Its $j$-th column is
$\tilde M_{\bullet,j}=\alpha\odot C_{\bullet,j}+2\lfloor j/2\rfloor\,C_{\bullet,j}$;
for $j=0,1$ it already equals $\alpha\odot C_{\bullet,j}$. Inductively, once
columns $0,\dots,j-1$ have been turned into $\alpha\odot C_{\bullet,0},\dots,
\alpha\odot C_{\bullet,j-1}$, relation \eqref{eq:algdep} expresses
$2\lfloor j/2\rfloor\,C_{\bullet,j}$ as a combination of them, and subtracting it
turns column $j$ into $\alpha\odot C_{\bullet,j}$. Hence $\tilde M$ reduces to
$\operatorname{diag}(\alpha)\,C$, and
$$
(s-2)^n\HH_n(s)=\Bigl(\prod_{i=0}^{n-1}(s-2+2i)\Bigr)\HH_n(s-2)
=(s-2)\,\Pi_{n-1}(s)\,\HH_n(s-2),\qquad \Pi_{n-1}(s)=\prod_{r=0}^{n-2}(s+2r).
$$
Cancelling one factor $s-2$,
\begin{equation}\label{eq:algcont}
(s-2)^{\,n-1}\HH_n(s)=\Pi_{n-1}(s)\,\HH_n(s-2).
\end{equation}

\emph{Step 2 (solving the relation).}
Collecting exponents base by base gives $D(s)/D(s-2)=\Pi_{n-1}(s)/(s-2)^{n-1}$,
so $D$ satisfies \eqref{eq:algcont} as well. Dividing, the rational function
$g(s)=\HH_n(s)/D(s)$ satisfies $g(s)=g(s-2)$; a rational function of period $2$
is constant (for $a$ avoiding the finitely many poles of $g$, the function
$g-g(a)$ vanishes at the infinitely many points of $a+2\ZZ$ that also avoid
them, hence identically). Hence $\HH_n(s)=c_n\,D(s)$ for a constant $c_n$.

\emph{Step 3 (the constant).}
At $s=2$ we have $(s/2)_k=(1)_k=k!$, so $a_m=m!$ for every $m$: the family
degenerates to the factorial member of the Beta chain, whose dilated
determinant was evaluated by the Vandermonde reduction~$\M1$ in
Corollary~\ref{cor:factorial} (case $r=0$):
$$
\HH_n(2)=\det\bigl((2i+j)!\bigr)
=2^{\binom n2}\prod_{k=1}^{n-1}k!\prod_{k=1}^{n-1}(2k)! .
$$
As $D(2)=\prod_{j=1}^{n-1}(2j)^{\,n-j}=2^{\binom n2}\prod_{k=1}^{n-1}k!$, we obtain
$c_n=\HH_n(2)/D(2)=\prod_{k=1}^{n-1}(2k)!$, the first displayed form. The
equivalent form follows from
$\prod_{i=1}^{n-1}(s/2)_i=2^{-\binom n2}\prod_{j=1}^{n-1}(s+2j-2)^{n-j}$.
\end{proof}


\section{The squared algebraic family $(1+x)^{2}/(1-x^{2})^{s/2}$}\label{sec:algsq}

The prefactor $1+x$ of Theorem~\ref{conj:alg} can be squared. The resulting
family is \emph{not} a reparametrisation of the previous one: writing
$(1+x)^{2}(1-x^{2})^{-s/2}=(1+x)^{2-s/2}(1-x)^{-s/2}$, a rescaling
$x\mapsto\beta x$ must preserve the singular pair $\{\pm1\}$, hence
$\beta=\pm1$, and matching the exponents at $\mp\beta$ against
$(1+x)^{1-\sigma/2}(1-x)^{-\sigma/2}$ forces two incompatible values of
$\sigma$ for either sign. The determinant confirms that the family is new:
the product form survives except in its topmost factor, where $(s/2)_{n-1}$
deforms into a \emph{binomial} expression carrying the denominator $s-3$. The
excluded value $s=3$ is precisely the point where this binomial degenerates
into odd harmonic numbers (Remark~\ref{rem:algsq-h3}). The proof below runs
the contiguous step of $\M5$ \emph{once}, from $s$ to $s-2$, and closes it by
a Cauchy--Binet expansion instead of the periodicity argument; no anchoring
special value is needed. The same machinery evaluates all the shifted
determinants $\HH_n^{(r)}=\det(a_{2i+j+r})$ of both families; this is the
subject of the next section (Section~\ref{sec:algsqshift}), where those of
the algebraic family turn out to be pure products for \emph{every} shift $r$,
while those of the squared algebraic family split by the parity of $r$ --- odd shifts
erase the defect entirely, even shifts carry the same $s-3$ binomial.

Throughout the section, $a_m=a_m(s)$ and
$\mu_k=\mu_k(s)=\frac{(2k)!}{k!}(s/2)_k$ are the moments of
$(1+x)(1-x^{2})^{-s/2}$ as in Theorem~\ref{conj:alg}.

\begin{Theorem}\label{conj:algsq}
Let $f(x)=\dfrac{(1+x)^{2}}{(1-x^{2})^{s/2}}$, with coefficients
$b_m=m!\,[x^m]f$, that is,
$$
b_{2k}=\mu_k+2k(2k-1)\,\mu_{k-1},\qquad
b_{2k+1}=2\,(2k+1)\,\mu_k .
$$
Then for all $n\ge1$ and $s\ne3$,
\begin{equation}\label{eq:algsq}
\HH_n(f)=2^{\binom n2}\prod_{k=1}^{n-1}(2k)!\;
\prod_{i=1}^{n-2}(s/2)_i\;\cdot\;
\frac{(2n-1)!!+(s-4)\,2^{\,n-1}(s/2)_{n-1}}{s-3}.
\end{equation}
The last factor is a monic polynomial in $s$ of degree $n-1$ (its numerator
vanishes at $s=3$), so that, once cleared of the denominator,
\eqref{eq:algsq} is a polynomial identity valid for every $s$.
\end{Theorem}

\begin{proof}
All identities below are between rational functions of $s$. We keep the
notation of the proof of Theorem~\ref{conj:alg}: $\alpha_i=s-2+2i$,
$r_i=\mu_i(s-2)$, $C_{ij}=a_{2i+j}(s-2)$, and $\psi_j$ is the monic
degree-$j$ polynomial of \eqref{eq:algfact}. Write also
$\varphi_n(\sigma):=\sigma(\sigma+2)\cdots(\sigma+2n-2)=2^{\,n}(\sigma/2)_n$.

\emph{Step 1 (one contiguous step).}
From $(1+x)^{2}=2(1+x)-(1-x^{2})$ we get
$f=2\,(1+x)(1-x^{2})^{-s/2}-(1-x^{2})^{-(s-2)/2}$, hence
$b_m(s)=2a_m(s)-e_m(s-2)$, where $e_{2k}(\sigma)=\mu_k(\sigma)$ and
$e_{2k+1}=0$ are the moments of $(1-x^{2})^{-\sigma/2}$. For $j$ even,
$e_{2i+j}(s-2)=\mu_{i+j/2}(s-2)=C_{ij}$; for $j$ odd it vanishes. Combining
with \eqref{eq:algstar}, $(s-2)\,a_{2i+j}(s)=(\alpha_i+2\lfloor
j/2\rfloor)\,C_{ij}$, we obtain the entrywise multiplier form
\begin{equation}\label{eq:algsqmul}
(s-2)\,b_{2i+j}(s)=(2\alpha_i+\tau_j)\,C_{ij},\qquad
\tau_{2l}=4l-(s-2),\quad \tau_{2l+1}=4l .
\end{equation}

\emph{Step 2 (bidiagonality in the $\psi$-basis).}
The polynomials \eqref{eq:algfact} form a multiplicative chain:
$$
\psi_j(\alpha)=\prod_{m=0}^{j-1}(\alpha+\delta_m),\qquad
\delta_{2p}=2p+3-s,\quad \delta_{2p+1}=2p,
$$
(indeed $\psi_{2l+1}/\psi_{2l}=\alpha+2l+3-s$ and
$\psi_{2l+2}/\psi_{2l+1}=\alpha+2l$), whence the two-term relation
$\alpha\,\psi_j=\psi_{j+1}-\delta_j\,\psi_j$. Therefore
$$
(2\alpha+\tau_j)\,\psi_j=2\,\psi_{j+1}+(\tau_j-2\delta_j)\,\psi_j
=2\,\psi_{j+1}+\epsilon_j\,\psi_j,\qquad
\epsilon_{2l}=s-4,\quad \epsilon_{2l+1}=0,
$$
since $\tau_{2l}-2\delta_{2l}=4l-s+2-2(2l+3-s)=s-4$ and
$\tau_{2l+1}-2\delta_{2l+1}=4l-4l=0$. With $C_{ij}=r_i\psi_j(\alpha_i)$,
equation \eqref{eq:algsqmul} becomes
$(s-2)\,b_{2i+j}(s)=r_i\bigl(2\psi_{j+1}(\alpha_i)
+\epsilon_j\psi_j(\alpha_i)\bigr)$, i.e.
\begin{equation}\label{eq:algsqbid}
(s-2)^{n}\,\HH_n(f)=\Bigl(\prod_{i=0}^{n-1}r_i\Bigr)\det(GT),\qquad
G=\bigl(\psi_k(\alpha_i)\bigr)_{\substack{0\le i\le n-1\\0\le k\le n}},\quad
T=\begin{pmatrix}\epsilon_0&&\\ 2&\epsilon_1&\\ &2&\ddots\\ &&&\epsilon_{n-1}\\
&&&2\end{pmatrix},
\end{equation}
$G$ of size $n\times(n+1)$ and $T$ of size $(n+1)\times n$, lower bidiagonal.

\emph{Step 3 (Cauchy--Binet).}
By the Cauchy--Binet formula,
$\det(GT)=\sum_{u=0}^{n}\det G_{\hat u}\,\det T_{\hat u}$, where $G_{\hat u}$
deletes the column $\psi_u$ of $G$ and $T_{\hat u}$ the row $u$ of $T$.
Deleting row $u$ makes $T_{\hat u}$ block-triangular, so
$\det T_{\hat u}=\prod_{j<u}\epsilon_j\cdot 2^{\,n-u}$, which vanishes as soon
as $u\ge2$ because $\epsilon_1=0$. Only two terms survive:
\begin{equation}\label{eq:algsqcb}
\det(GT)=2^{\,n}\det G_{\hat0}+(s-4)\,2^{\,n-1}\det G_{\hat1}.
\end{equation}

\emph{Step 4 (the two gap alternants).}
Set $V:=\prod_{0\le i<j\le n-1}(\alpha_j-\alpha_i)
=2^{\binom n2}\prod_{k=1}^{n-1}k!$ and
$\Lambda(\beta):=\prod_{i=0}^{n-1}(\beta-\alpha_i)$. Border $G$ with the row
$\bigl(\psi_0(\beta),\dots,\psi_n(\beta)\bigr)$: since the $\psi_k$ are monic
of degree $k$, the $(n+1)\times(n+1)$ alternant equals the Vandermonde
determinant of $(\alpha_0,\dots,\alpha_{n-1},\beta)$, namely $V\Lambda(\beta)$.
Laplace expansion along the added row gives
$$
\sum_{u=0}^{n}(-1)^{n+u}\,\psi_u(\beta)\,\det G_{\hat u}=V\,\Lambda(\beta).
$$
Expanding $\Lambda$ in the chain basis (a Newton expansion),
$\Lambda(\beta)=\sum_{u=0}^{n}\gamma_u\,\psi_u(\beta)$, and comparing
coefficients, $\det G_{\hat u}=(-1)^{n+u}\gamma_u\,V$. The first two
coefficients are read off by evaluation. At $\beta=s-3$ every $\psi_u$ with
$u\ge1$ vanishes (factor $\beta+\delta_0$), so
$\gamma_0=\Lambda(s-3)=\prod_{i=0}^{n-1}\bigl(-(2i+1)\bigr)=(-1)^n(2n-1)!!$.
At $\beta=0$ every $\psi_u$ with $u\ge2$ vanishes (factor $\beta+\delta_1$)
and $\psi_1(0)=\delta_0=3-s$, so
$(-1)^n\varphi_n(s-2)=\Lambda(0)=\gamma_0+(3-s)\gamma_1$. Hence
\begin{equation}\label{eq:algsqgap}
\det G_{\hat0}=(2n-1)!!\;V,\qquad
\det G_{\hat1}=\frac{\varphi_n(s-2)-(2n-1)!!}{s-3}\;V .
\end{equation}

\emph{Step 5 (assembly).}
Substituting \eqref{eq:algsqgap} into \eqref{eq:algsqcb} and using
$2(s-3)-(s-4)=s-2$ together with $\varphi_n(s-2)=(s-2)\,\varphi_{n-1}(s)$,
$$
\begin{aligned}
\det(GT)&=2^{\,n-1}\,
\frac{(s-2)(2n-1)!!+(s-4)\,\varphi_n(s-2)}{s-3}\;V\\
&=2^{\,n-1}(s-2)\,
\frac{(2n-1)!!+(s-4)\,\varphi_{n-1}(s)}{s-3}\;V .
\end{aligned}
$$
Moreover $\mu_i(s-2)=\mu_i(s)\,(s-2)/(s+2i-2)$ (as in Step~1 of
Theorem~\ref{conj:alg}) and $\prod_{i=0}^{n-1}(s-2+2i)=(s-2)\varphi_{n-1}(s)$
give $\prod_i r_i=\bigl(\prod_i\mu_i(s)\bigr)(s-2)^{\,n-1}/\varphi_{n-1}(s)$.
Inserting both into \eqref{eq:algsqbid} and cancelling $(s-2)^{n}$,
$$
\HH_n(f)=\Bigl(\prod_{i=0}^{n-1}\mu_i(s)\Bigr)V\cdot
\frac{2^{\,n-1}\bigl((2n-1)!!+(s-4)\varphi_{n-1}(s)\bigr)}
{(s-3)\,\varphi_{n-1}(s)} .
$$
Finally $\prod_{i=0}^{n-1}\mu_i(s)\cdot V
=2^{\binom n2}\prod_{k=1}^{n-1}(2k)!\prod_{i=1}^{n-1}(s/2)_i$ and
$2^{\,n-1}\prod_{i=1}^{n-1}(s/2)_i/\varphi_{n-1}(s)
=\prod_{i=1}^{n-2}(s/2)_i$, which is \eqref{eq:algsq} after writing
$\varphi_{n-1}(s)=2^{\,n-1}(s/2)_{n-1}$ in the numerator.
\end{proof}

\begin{Remark}[the same computation re-proves Theorem~\ref{conj:alg}]
\label{rem:algsq-mech}
For $(1+x)(1-x^{2})^{-s/2}$ the multiplier in \eqref{eq:algsqmul} is
$\alpha_i+2\lfloor j/2\rfloor$, and the same two-term reduction gives
$(\alpha+2l)\psi_{2l}=\psi_{2l+1}+(s-3)\psi_{2l}$ and
$(\alpha+2l)\psi_{2l+1}=\psi_{2l+2}$: the diagonal of $T$ becomes
$\epsilon_{2l}=s-3$, $\epsilon_{2l+1}=0$. Steps~3--5 then yield
$\det(GT)=\bigl((2n-1)!!+(s-3)\cdot\frac{\varphi_n(s-2)-(2n-1)!!}{s-3}\bigr)V
=\varphi_n(s-2)\,V$: the double factorials cancel \emph{exactly}, the
denominator $s-3$ disappears, and the pure product of
Theorem~\ref{conj:alg} drops out --- with no periodicity argument and no
anchoring value $s=2$. Squaring the prefactor shifts the diagonal from $s-3$
to $s-4$: the mismatch with the denominator $s-3$ of $\det G_{\hat1}$ ruins
the cancellation, and the binomial numerator of \eqref{eq:algsq} is the trace
it leaves.
\end{Remark}

\begin{Remark}[the excluded value $s=3$]\label{rem:algsq-h3}
The last factor of \eqref{eq:algsq} is also the terminating sum
$$
\frac{(2n-1)!!+(s-4)\,2^{\,n-1}(s/2)_{n-1}}{s-3}
=(2n-1)!!\,\Bigl(1+(s-4)\sum_{k=0}^{n-2}\frac{2^{k}(s/2)_k}{(2k+3)!!}\Bigr),
$$
valid for every $s$: with $t_k=2^{k}(s/2)_k/(2k+1)!!$ one has
$t_{k+1}-t_k=(s-3)\,2^{k}(s/2)_k/(2k+3)!!$, and the sum telescopes to
$\bigl(2^{\,n-1}(s/2)_{n-1}/(2n-1)!!-1\bigr)/(s-3)$. At $s=3$ the summand is
$(2k+1)!!/(2k+3)!!=1/(2k+3)$, and the factor degenerates into
$(2n-1)!!\,(2-O_n)$ with $O_n=1+\tfrac13+\cdots+\tfrac1{2n-1}$:
$$
\HH_n(f)\big|_{s=3}
=2^{\binom n2}\prod_{k=1}^{n-1}(2k)!\;\prod_{i=1}^{n-2}(3/2)_i\;\cdot\;
(2n-1)!!\,(2-O_n).
$$
The odd harmonic number $O_n$ is why $s=3$ must be excluded from any
product-type statement; it is the value of $s$ at which the telescoping
ratio $2^{k}(s/2)_k/(2k+1)!!$ becomes constant in $k$.
\end{Remark}

\begin{Remark}[degenerate values]\label{rem:algsq-degen}
At $s=4$ one has $f=(1-x)^{-2}$ and $b_m=(m+1)!$; the binomial term drops out
($s-4=0$) and \eqref{eq:algsq} reduces to
$2^{\binom n2}\prod_{k=1}^{n-1}(2k)!\,\prod_{i=1}^{n-2}(i+1)!\,(2n-1)!!
=2^{\binom n2}\prod_{k=1}^{n-1}k!\,\prod_{i=0}^{n-1}(2i+1)!$, the case $r=1$
of Corollary~\ref{cor:factorial}. At $s=0$ one has $f=(1+x)^{2}$ and
\eqref{eq:algsq} correctly gives $\HH_1=1$, $\HH_2=-4$ and $\HH_n=0$ for
$n\ge3$.
\end{Remark}

\section{Shifted determinants of the two algebraic families}
\label{sec:algsqshift}

Continuing with the two algebraic families of
Sections~\ref{sec:alg}--\ref{sec:algsq}, we now evaluate all their shifted
dilated determinants
$$
\HH_n^{(r)}(f):=\det\bigl(a_{2i+j+r}\bigr)_{0\le i,j\le n-1}\qquad(r\ge0).
$$
For the algebraic family the answer is a pure
product for \emph{every} $r$, and the proof is a direct Vandermonde
reduction~$\M1$; for the squared algebraic family the parity of $r$ decides whether
the $s-3$ binomial of Theorem~\ref{conj:algsq} survives. Throughout,
$\rho:=\lfloor r/2\rfloor$.

\begin{Proposition}[all shifts of the algebraic family]\label{prop:algshift}
Let $f_1(x)=(1+x)(1-x^{2})^{-s/2}$, with moments $a_m$ as in
Theorem~\ref{conj:alg}. Then for all $n\ge1$ and $r\ge0$,
\begin{equation}\label{eq:algshift}
\HH_n^{(r)}(f_1)=2^{\binom n2}\prod_{k=1}^{n-1}k!\;
\prod_{i=0}^{n-1}a_{2i+r}(s),
\end{equation}
that is, $\prod_i a_{2i+r}=\prod_{i=0}^{n-1}\mu_{i+\rho}$ for $r$ even and
$\prod_{i=0}^{n-1}(2i+r)\,\mu_{i+\rho}$ for $r$ odd.
\end{Proposition}

\begin{proof}
By \eqref{eq:algmom} the term ratio of $\mathbf a$ has two linear branches,
$$
\frac{a_{m+1}}{a_m}=
\begin{cases} m+1, & m\ \text{even},\\ m+s-1, & m\ \text{odd}.\end{cases}
$$
The row indices $x_i=2i$ are all even, so which branch is taken at offset
$p$ from the anchor depends only on the parity of $r+p$, not on the row:
$$
a_{x+r+j}=a_{x+r}\,Q_j(x),\qquad
Q_j(x)=\prod_{p=0}^{j-1}
\begin{cases} x+r+p+1, & r+p\ \text{even},\\
x+r+p+s-1, & r+p\ \text{odd},\end{cases}
$$
a \emph{monic} polynomial of degree $j$ in $x$, the same for every row.
Lemma~\ref{lem:vdm} with $w(x)=a_{x+r}$ applies, the coefficient matrix is
unit triangular, and
$\HH_n^{(r)}(f_1)=\prod_{i<j}(2j-2i)\cdot\prod_i a_{2i+r}$, which is
\eqref{eq:algshift}.
\end{proof}

\begin{Remark}[a third proof of Theorem~\ref{conj:alg}]\label{rem:algshift-m1}
The case $r=0$ of Proposition~\ref{prop:algshift} is Theorem~\ref{conj:alg}
itself: the algebraic family is, after all, within reach of $\M1$. Its term
ratio has degree one \emph{in each parity class} --- which places it outside
the Beta family of Section~\ref{sec:beta} and defeats the classical Hankel
treatment --- but the dilated rows sit on even indices and never mix the two
classes, so the row factorisation \eqref{eq:rowfact} holds with monic $Q_j$
of degree $j$. The three proofs illuminate different structure: $\M5$ the
contiguous relation in $s$, Remark~\ref{rem:algsq-mech} the exact
double-factorial cancellation, and $\M1$ the two-branch Beta pattern.
\end{Remark}

\begin{Theorem}[all shifts of the squared algebraic family]\label{thm:algsqshift}
Let $f(x)=(1+x)^{2}(1-x^{2})^{-s/2}$, with moments $b_m$ as in
Theorem~\ref{conj:algsq}, and let $n\ge1$.
\begin{enumerate}
\item[\rm(i)] For odd $r$ the defect disappears entirely:
$$
\HH_n^{(r)}(f)=2^{\,n}\,\HH_n^{(r)}(f_1)
=2^{\binom n2+n}\prod_{k=1}^{n-1}k!\;\prod_{i=0}^{n-1}(2i+r)\,\mu_{i+\rho}(s).
$$
\item[\rm(ii)] For even $r=2\rho$ and $s\ne3$,
\begin{equation}\label{eq:algsqshift}
\HH_n^{(r)}(f)=2^{\,n-1}\,\HH_n^{(r)}(f_1)\cdot
\frac{P_{n,\rho}(s)}{\Theta_{n,\rho}(s)},\qquad
\Theta_{n,\rho}(s):=\prod_{i=0}^{n-1}(s-2+2\rho+2i),
\end{equation}
where, with the convention $(-1)!!=1$,
$$
P_{n,\rho}(s):=\frac{\displaystyle
(s-2)\,\frac{(2n+2\rho-1)!!}{(2\rho-1)!!}+(s-4)\,\Theta_{n,\rho}(s)}{s-3}
$$
is a monic polynomial in $s$ of degree $n$ (the numerator vanishes at
$s=3$, since $\Theta_{n,\rho}(3)=(2n+2\rho-1)!!/(2\rho-1)!!$). For
$\rho\ge1$ the quotient in \eqref{eq:algsqshift} clears entrywise,
$\mu_{i+\rho}(s)/(s-2+2\rho+2i)=\tfrac12\,
\frac{(2i+2\rho)!}{(i+\rho)!}\,(s/2)_{i+\rho-1}$, giving the fully
explicit form
$$
\HH_n^{(2\rho)}(f)=2^{\binom n2-1}\prod_{k=1}^{n-1}k!\;
\prod_{i=0}^{n-1}\frac{(2i+2\rho)!}{(i+\rho)!}\,(s/2)_{i+\rho-1}\;\cdot\;
P_{n,\rho}(s)\qquad(\rho\ge1);
$$
the case $\rho=0$ is Theorem~\ref{conj:algsq}.
\end{enumerate}
\end{Theorem}

\begin{proof}
The five steps of Theorem~\ref{conj:algsq} go through with the chain
shifted by $r$; we record the changes.

\emph{Steps 1--2.} With $m=2i+j+r$, the identities of Step~1 read
$(s-2)\,b_{2i+j+r}(s)=(2\alpha_i+\tau_{j+r})\,C^{(r)}_{ij}$, where
$C^{(r)}_{ij}=a_{2i+j+r}(s-2)=r_i\,\psi_{j+r}(\alpha_i)$ (the factorisation
\eqref{eq:algfact} holds for every column index) and $\tau_m$ is as in
\eqref{eq:algsqmul}, indexed by the parity of $m=j+r$. Factor
$\psi_r(\alpha_i)$ --- that is, $a_{2i+r}(s-2)/r_i$ --- out of row $i$; the
remaining columns form the shifted chain
$\chi_j:=\psi_{j+r}/\psi_r=\prod_{m=0}^{j-1}(\alpha+\delta_{r+m})$, and the
computation of Step~2, applied verbatim to the parity of $j+r$, gives
$$
(2\alpha+\tau_{j+r})\,\chi_j=2\,\chi_{j+1}+\epsilon_j\,\chi_j,\qquad
\epsilon_j=\begin{cases} s-4, & j+r\ \text{even},\\ 0, & j+r\ \text{odd}.
\end{cases}
$$
Hence $(s-2)^n\,\HH_n^{(r)}(f)
=\prod_i a_{2i+r}(s-2)\cdot\det\bigl(G^{(r)}T^{(r)}\bigr)$ with
$G^{(r)}=(\chi_k(\alpha_i))$ of size $n\times(n+1)$ and $T^{(r)}$ lower
bidiagonal as in \eqref{eq:algsqbid}, its diagonal now being $\epsilon_j$.

\emph{Steps 3--4, $r$ odd.} Here $\epsilon_0=0$, so in the Cauchy--Binet
expansion only $u=0$ survives and
$\det(G^{(r)}T^{(r)})=2^{\,n}\det G^{(r)}_{\hat0}$. Factoring
$\chi_1(\alpha_i)=\alpha_i+\delta_r=\alpha_i+2\rho$ out of each row of
$G^{(r)}_{\hat0}$ leaves the initial segment of a monic chain, whose
alternant is $V$; hence
$\det(G^{(r)}T^{(r)})=2^{\,n}\,\Theta_{n,\rho}(s)\,V$, since
$\prod_i(\alpha_i+2\rho)=\Theta_{n,\rho}(s)$.

\emph{Steps 3--4, $r=2\rho$ even.} As in Theorem~\ref{conj:algsq}, $u=0,1$
survive. The Newton argument of Step~4 now runs at the nodes
$-\delta_r=s-3-2\rho$ and $-\delta_{r+1}=-2\rho$, and again
$\chi_1(-\delta_{r+1})=\delta_r-\delta_{r+1}=3-s$:
$$
\gamma_0=\Lambda(s-3-2\rho)=(-1)^n\,\frac{(2n+2\rho-1)!!}{(2\rho-1)!!},
\qquad
\Lambda(-2\rho)=(-1)^n\,\Theta_{n,\rho}(s)=\gamma_0+(3-s)\gamma_1,
$$
so that, exactly as in \eqref{eq:algsqgap}--\eqref{eq:algsqcb},
$$
\det(G^{(r)}T^{(r)})
=2^{\,n-1}\,
\frac{(s-2)\,\dfrac{(2n+2\rho-1)!!}{(2\rho-1)!!}
+(s-4)\,\Theta_{n,\rho}(s)}{s-3}\;V
=2^{\,n-1}\,P_{n,\rho}(s)\,V .
$$

\emph{Step 5.} By \eqref{eq:algstar},
$a_{2i+r}(s-2)=a_{2i+r}(s)\,(s-2)/(s-2+2(i+\rho))$, so
$\prod_i a_{2i+r}(s-2)
=\prod_i a_{2i+r}(s)\cdot(s-2)^n/\Theta_{n,\rho}(s)$. For $r$ odd the factor
$\Theta_{n,\rho}(s)$ cancels against Step~3, and
$\HH_n^{(r)}(f)=2^{\,n}\prod_i a_{2i+r}(s)\cdot V
=2^{\,n}\HH_n^{(r)}(f_1)$ by Proposition~\ref{prop:algshift}. For $r$ even
it survives in the denominator, giving \eqref{eq:algsqshift}. The explicit
form for $\rho\ge1$ follows from
$(s/2)_{i+\rho}=(s/2)_{i+\rho-1}\,(s+2i+2\rho-2)/2$.
\end{proof}

\begin{Remark}[the contrast, and $s=3$ again]\label{rem:algsqshift}
An odd shift moves the even-column defect $\epsilon_j=s-4$ off the column
$j=0$; the zero $\epsilon_0$ then truncates the Cauchy--Binet expansion to
its first term, and the squared algebraic family costs only the overall factor
$2^{\,n}$ relative to the algebraic family. For even shifts the binomial
survives, and the odd-harmonic degeneration of Remark~\ref{rem:algsq-h3}
persists: the same telescoping gives
$$
P_{n,\rho}(3)=\frac{(2n+2\rho-1)!!}{(2\rho-1)!!}
\Bigl(2-\sum_{j=\rho+1}^{n+\rho}\frac{1}{2j-1}\Bigr),
$$
so no even shift admits a product form at $s=3$.
\end{Remark}

\section{A Bessel analogue of the Euler number family}\label{sec:besselst}

\subsection*{The Bessel cosine $\mathrm{cosb}_s$}

For $s>-1$ define the \emph{Bessel cosine} of order $s$ as the normalised
Bessel function
\begin{equation}\label{eq:cosbdef}
\mathrm{cosb}_s(x)\;:=\;\Gamma(s+1)\Bigl(\frac2x\Bigr)^{\!s}J_s(x)
\;=\;\sum_{k\ge0}\frac{(-1)^k}{4^k\,k!\,(s+1)_k}\,x^{2k},
\qquad \mathrm{cosb}_s(0)=1.
\end{equation}
(The normalisation is classical --- it is the function written $\Lambda_s$
in the characteristic-function literature --- but the name and symbol
$\mathrm{cosb}$ are ours.) The name is earned by the exact identities
$$
\mathrm{cosb}_{-1/2}(x)=\cos x,
\qquad
\mathrm{cosb}_{1/2}(x)=\frac{\sin x}{x},
\qquad
\mathrm{cosb}_s'(x)=-\frac{x}{2s+2}\,\mathrm{cosb}_{s+1}(x),
$$
so that with $\mathrm{sinb}_s:=-\mathrm{cosb}_s'$ one has
$\mathrm{sinb}_{-1/2}=\sin x$ and the derivative ladder generalises
$\cos'=-\sin$. At half-integer orders $\mathrm{cosb}_s$ is elementary
(trigonometric); at all other orders it is a genuine Bessel function. By
the Poisson integral representation, for $s>-\tfrac12$,
$$
\mathrm{cosb}_s(x)
=\frac{\Gamma(s+1)}{\Gamma(s+\tfrac12)\Gamma(\tfrac12)}
\int_{-1}^{1}e^{ixt}\,(1-t^2)^{s-1/2}\,dt,
$$
i.e.\ $\mathrm{cosb}_s$ is the characteristic function of the symmetric
Beta (Gegenbauer) distribution on $[-1,1]$: arcsine at $s=0$, uniform at
$s=\tfrac12$, Wigner semicircle at $s=1$. Note where the parameter lives:
in the trigonometric family the power sits on the \emph{function} side
($1/\cos(x)^{s+1}$), here it sits on the \emph{weight} side
($(1-t^2)^{s-1/2}$). A literal power form
$\mathrm{cosb}_s=h(x)^{-(s+1)}$ with one fixed $h$ is impossible: it would
force $\mathrm{cosb}_0\,\mathrm{cosb}_2=\mathrm{cosb}_1^{\,2}$, and in fact
$\mathrm{cosb}_0\,\mathrm{cosb}_2-\mathrm{cosb}_1^{\,2}
=-\tfrac{x^2}{12}+\tfrac{5x^4}{384}-\cdots\ne0$.

\subsection*{The family and the closed form}

The trigonometric $(s,t)$ family of Section~\ref{sec:gen} is
$1/\cos(x)^{s+1}+\int_0^x\cos(y)^{-(t+1)}dy$, evaluated for all $(s,t)$ by
Proposition~\ref{prop:family}. Its direct Bessel analogue, with the
exponent ladder replaced by the order ladder, is
\begin{equation}\label{eq:besselfam}
f_{s,t}(x)\;=\;\mathrm{cosb}_s(x)+\int_0^x \mathrm{cosb}_t(y)\,dy .
\end{equation}
Writing $a_k=k!\,[x^k]f_{s,t}$ and using $(2k)!=4^k\,k!\,(\tfrac12)_k$,
the coefficients \eqref{eq:cosbdef} give
\begin{equation}\label{eq:bessel-moments}
a_{2k}=(-1)^k\,\frac{(\tfrac12)_k}{(s+1)_k},
\qquad
a_{2k+1}=(-1)^k\,\frac{(\tfrac12)_k}{(t+1)_k}
\qquad(k\ge0),
\end{equation}
the odd case because integration divides the coefficient of $x^{2k+1}$ by
$2k+1$ and $(2k+1)!/(2k+1)=(2k)!$. Thus the even and odd moment
generating functions are one and the same Gauss series at the two
parameters,
$$
\sum_{k\ge 0}a_{2k}\,x^k={}_2F_1\Bigl(\tfrac12,1;\,s+1;\,-x\Bigr),
\qquad
\sum_{k\ge 0}a_{2k+1}\,x^k={}_2F_1\Bigl(\tfrac12,1;\,t+1;\,-x\Bigr):
$$
the even and odd moment functionals belong to a \emph{single} classical
hierarchy (they are Jacobi--Gegenbauer functionals --- in $y=z^2$, the
moments of the weight $y^{-1/2}(1-y)^{s-1/2}$ carried to $[-1,0]$) and
differ in the single parameter $s\leftrightarrow t$. This is precisely
the collapse hypothesis of the biorthogonal reduction~$\M2$
(Section~\ref{sec:biotho}), and the whole family evaluates in closed
form.

\begin{Theorem}[Bessel $(s,t)$ family; $\M2$]\label{thm:besselst}
Let $f_{s,t}$ be as in \eqref{eq:besselfam}. Then, as an identity of
rational functions of $(s,t)$,
\begin{equation}\label{eq:bessel-closed}
\HH_n(f_{s,t})
=(-1)^{\binom{\bar n}2}\,
\frac{\prod_{i=0}^{n-1}(2i)!}{2^{\,n(n-1)}}\,
\Bigl(\prod_{l=0}^{\bar n-1}\frac{(s+\tfrac12)_l}{(s+1)_{n-1+l}}\Bigr)
\Bigl(\prod_{m=0}^{\underline n-1}\frac{(t+\tfrac12)_m}{(t+1)_{n-1+m}}\Bigr)
\prod_{r=1}^{\bar n}\prod_{c=1}^{\underline n}\frac{t-s+c-r}{n-r-c+1}\,.
\end{equation}
\end{Theorem}

Both sides of \eqref{eq:bessel-closed} lie in $\QQ(s,t)$, so the identity
may be specialised at every $(s,t)$ with $s,t\notin\{-1,-2,\dots\}$ --- in
particular throughout the Poisson range $s,t>-\tfrac12$. The sign together
with the double product is exactly the signed factor of the trigonometric
family at \emph{doubled} argument, $\Omega(2(t-s))$ in the notation of
Section~\ref{ssec:stomega}; we return to this parallel in
Remark~\ref{rem:bessel-parallel}. The contrast observed numerically ---
smooth at half-integer offset, degenerate at integer offset --- is now a
one-line consequence.

\begin{Corollary}[smooth and degenerate offsets]\label{cor:besseldicho}
Let $s,t>-\tfrac12$.
\begin{enumerate}
\item[\rm(i)] If $t-s\in\ZZ$, then $\HH_n(f_{s,t})=0$ precisely for
$n\ge 2(t-s)+1$ when $t>s$, for $n\ge2$ when $t=s$, and for
$n\ge 2(s-t)+2$ when $t<s$.
\item[\rm(ii)] If $t-s\notin\ZZ$, then $\HH_n(f_{s,t})\neq0$ for every
$n$, and \eqref{eq:bessel-closed} is a closed product of linear forms.
For half-integer offset $t-s\in\tfrac12+\ZZ$ the signed factor moreover
evaluates in double factorials by Lemma~\ref{lem:Pahalf}, applied with
the odd argument $\delta=2(t-s)$.
\end{enumerate}
\end{Corollary}

\begin{proof}
For $s,t>-\tfrac12$ every Pochhammer symbol in \eqref{eq:bessel-closed}
is a product of positive reals, so $\HH_n=0$ if and only if the double
product vanishes, i.e.\ if and only if $t-s=r-c$ for some $1\le r\le\bar n$,
$1\le c\le\underline n$. This forces $t-s\in\ZZ$, and conversely, for
$t-s\in\ZZ$, such $(r,c)$ exist if and only if $\bar n\ge t-s+1$ and
$\underline n\ge s-t+1$, i.e.\ $n\ge\max\bigl(2(t-s)+1,\,2(s-t)+2,\,2\bigr)$,
which is the stated threshold in each of the three cases.
\end{proof}

\begin{Example}[$J_0$ plus the sine integral]\label{ex:besselJ0Si}
At $(s,t)=(0,\tfrac12)$ the member \eqref{eq:besselfam} is
$J_0(x)+\mathrm{Si}(x)$, since $\mathrm{cosb}_0=J_0$ and
$\int_0^x\mathrm{cosb}_{1/2}(y)\,dy=\int_0^x\sin(y)/y\,dy$. Here
$2(t-s)=1$, $\Omega(1)=2^{-\binom n2}$ by Lemma~\ref{lem:Pahalf}, and
\eqref{eq:bessel-closed} specialises to the factorial product
$$
\HH_n\bigl(J_0+\mathrm{Si}\bigr)
=\frac{\prod_{i=0}^{n-1}(2i)!}{2^{\,n(n-1)}\,2^{\binom n2}}\,
\prod_{l=0}^{\bar n-1}\frac{(\tfrac12)_l}{(n-1+l)!}\,
\prod_{m=0}^{\underline n-1}\frac{m!}{(\tfrac32)_{n-1+m}}
=1,\ \frac16,\ \frac1{960},\ \frac1{1693440},\ \dots
$$
($n=1,2,3,4$) --- the numerators are $1$, which is the arithmetic
signature that first singled the half-integer lines out.
\end{Example}

\begin{Remark}[the boundary of the ladder]\label{rem:bessel-boundary}
The closed form also quantifies the collapse at the edge of the order
ladder. At $s=-\tfrac12$ the numerators $(s+\tfrac12)_l$ vanish for every
$l\ge1$, so $\HH_n=0$ as soon as $\bar n\ge2$, i.e.\ for all $n\ge3$;
at $t=-\tfrac12$ likewise $\HH_n=0$ for all $n\ge4$ ($\underline n\ge2$). Thus
$\cos x+\int_0^xJ_0$ degenerates from $n=3$ and $J_0(x)+\sin x$ from
$n=4$, although both pairs have half-integer offset: rule (ii) of
Corollary~\ref{cor:besseldicho} genuinely needs $s,t>-\tfrac12$. More
generally $(s+\tfrac12)_l=0$ for $l>k$ at $s=-\tfrac12-k$, giving
$\HH_n=0$ for $n\ge2k+3$. The explanation is visible in the Poisson
formula: at $s=-\tfrac12$ the weight $(1-t^2)^{s-1/2}$ ceases to be a
measure and $\cos x$ is the characteristic function of the \emph{atomic}
two-point measure at $\pm1$ --- a rank collapse, the same degeneracy as
the $(1+x)\cos$-type atomic cases elsewhere in the paper.
\end{Remark}

The proof of Theorem~\ref{thm:besselst} occupies the next three
subsections: the orthogonal data of the family are explicit
(Lemma~\ref{lem:bessel-data}), the connection coefficients collapse to a
single term because the two functionals differ in one parameter
(Lemma~\ref{lem:bessel-kappa}), and the residual connection determinant
--- which, unlike the binomial determinant of the trigonometric family,
carries the parameters $s,t$ in its entries --- is evaluated by
Desnanot--Jacobi condensation (Theorem~\ref{thm:bessel-det}), in the
manner of the Catalan determinant of Section~\ref{ssec:xsin-catalan}.

\subsection{Orthogonal data}\label{ssec:bessel-data}

By \eqref{eq:bessel-moments} the even part of $\mathbf a$ involves only
$s$ and the odd part only $t$, through one and the same one-parameter
moment sequence; we record the data for the even part, the odd part being
identical with $t$ in place of $s$. All computations take place over the
field $\QQ(s,t)$.

\begin{Lemma}[orthogonal data]\label{lem:bessel-data}
The even and odd parts of $\mathbf a$ admit the Stieltjes $S$-fractions
\eqref{eq:Fsfrac} with coefficients
\begin{equation}\label{eq:bessel-u}
u_j=-\,\frac{j\,(j+2s-1)}{4\,(s+j-1)(s+j)},
\qquad
v_j=-\,\frac{j\,(j+2t-1)}{4\,(t+j-1)(t+j)}
\qquad(j\ge1)
\end{equation}
(for $j=1$ the factor $s+j-1=s$ cancels: $u_1=-\tfrac1{2(s+1)}$).
Consequently both functionals are quasi-definite over $\QQ(s,t)$, with
squared norms
\begin{equation}\label{eq:bessel-norms}
h^S_i=\frac{(2i)!\;(s+\tfrac12)_i}{4^i\,(s+i)_i\,(s+1)_{2i}},
\qquad
h^T_m=\frac{(2m)!\;(t+\tfrac12)_m}{4^m\,(t+m)_m\,(t+1)_{2m}}.
\end{equation}
\end{Lemma}

\begin{proof}
The even generating function is ${}_2F_1(\tfrac12,1;s+1;-x)$. Gauss's
continued fraction \cite{Wall1948,Cuyt2008etal} expands the ratio
$F(a,b+1;c+1;z)/F(a,b;c;z)$ of contiguous hypergeometric series as the
formal $S$-fraction $1/(1-\alpha_1z/(1-\alpha_2z/(1-\cdots)))$ with
$$
\alpha_{2k+1}=\frac{(a+k)(c-b+k)}{(c+2k)(c+2k+1)},
\qquad
\alpha_{2k}=\frac{(b+k)(c-a+k)}{(c+2k-1)(c+2k)}.
$$
Taking $(a,b,c)=(\tfrac12,0,s)$, where $F(\tfrac12,0;s;z)=1$, both
parities merge into the single expression
$\alpha_j=\dfrac{j(j+2s-1)}{4(s+j-1)(s+j)}$, and the substitution $z=-x$
turns $1/(1-\alpha_1z/\cdots)$ into the $S$-fraction in $x$ with
coefficients $u_j=-\alpha_j$, which is \eqref{eq:bessel-u}. By
Lemma~\ref{lem:cfGene} the squared norms telescope,
$$
h^S_i=\prod_{j=1}^{2i}u_j
=\frac{(2i)!\,(2s)_{2i}}{16^{\,i}\,(s)_{2i}\,(s+1)_{2i}}
$$
(the $2i$ signs cancel), and the duplication formula in Pochhammer form,
$(2s)_{2i}=4^i\,(s)_i\,(s+\tfrac12)_i$, together with
$(s)_{2i}=(s)_i\,(s+i)_i$, gives \eqref{eq:bessel-norms}. Each $h^S_i$ is
a nonzero element of $\QQ(s,t)$, whence quasi-definiteness.
\end{proof}

\subsection{Connection coefficients}\label{ssec:bessel-kappa}

As in Section~\ref{sec:biotho} let $P_i$ and $Q_m$ be the monic
orthogonal polynomials of $\mathcal S$ and $\mathcal T$, and
$P_i=\sum_{m=0}^{i}\kappa_{i,m}\,Q_m$ the connection expansion. Because
the two functionals differ in a single parameter, the $\kappa_{i,m}$
collapse to a single term. Throughout put
$$
\delta:=t-s,\qquad d:=i-m .
$$

\begin{Lemma}[connection coefficients]\label{lem:bessel-kappa}
For $0\le m\le i$,
\begin{equation}\label{eq:bessel-kappa}
\kappa_{i,m}
=\frac{(2i)!}{(2m)!\;4^{d}\,d!}\;
\frac{\delta(\delta-1)\cdots(\delta-d+1)}{(t+2m+1)_d\;(s+i+m)_d}
=\frac{(2i)!}{(2m)!\;4^{d}}\,\binom{\delta}{d}\,
\frac{1}{(t+2m+1)_d\;(s+i+m)_d}\,.
\end{equation}
\end{Lemma}

\begin{proof}
Write $k_{i,m}$ for the right-hand side, extended by $k_{i,m}=0$ for
$m<0$ or $m>i$. By Lemma~\ref{lem:connrec} the array $\kappa$ is the
unique solution of the recurrence \eqref{eq:kapparec} with row
$\kappa_{0,0}=1$, $\kappa_{0,m}=0$ $(m\ne0)$; the candidate has the same
row $i=0$ (at $d=0$ every product in \eqref{eq:bessel-kappa} is empty),
so it suffices to show that $k$ satisfies \eqref{eq:kapparec}. Every
factor of \eqref{eq:bessel-kappa} is nonzero on the triangle
$0\le m\le i$ over $\QQ(s,t)$, and telescoping each block gives the four
neighbour ratios as rational functions of $(m,d,s,t)$ (recall $i=m+d$):
\begin{align*}
\frac{k_{i+1,m}}{k_{i,m}}
&=\frac{(2i+1)(2i+2)\,(\delta-d)\,(s+i+m)}
       {4\,(d+1)\,(t+2m+d+1)\,(s+2i)(s+2i+1)},\\
\frac{k_{i,m-1}}{k_{i,m}}
&=\frac{(2m-1)(2m)\,(\delta-d)\,(t+2m+d)}
       {4\,(d+1)\,(t+2m-1)(t+2m)\,(s+i+m-1)},\\
\frac{k_{i,m+1}}{k_{i,m}}
&=\frac{4\,d\,(t+2m+1)(t+2m+2)\,(s+i+m)}
       {(2m+1)(2m+2)\,(\delta-d+1)\,(t+2m+d+1)},\\
\frac{k_{i-1,m}}{k_{i,m}}
&=\frac{4\,d\,(t+2m+d)\,(s+2i-2)(s+2i-1)}
       {(2i-1)(2i)\,(\delta-d+1)\,(s+i+m-1)}.
\end{align*}
Dividing \eqref{eq:kapparec} by $k_{i,m}$ and inserting these ratios
together with the recurrence data of Lemma~\ref{lem:bessel-data} --- via
Lemma~\ref{lem:cfGene}, $c^S_i=u_{2i}+u_{2i+1}$,
$\lambda^S_i=u_{2i-1}u_{2i}$, $c^T_m=v_{2m}+v_{2m+1}$,
$\lambda^T_{m+1}=v_{2m+1}v_{2m+2}$ --- turns the recurrence into a
rational-function identity in $(m,d,s,t)$, verified by clearing
denominators. The boundary columns take care of themselves: at $m=0$ the
second ratio carries the factor $2m$ and vanishes, as required by
$\kappa_{i,-1}=0$; at $m=i$ the third and fourth ratios carry the factor
$d$ and vanish, as required by $\kappa_{i,i+1}=0$ and $\kappa_{i-1,i}=0$;
and at $m=i+1$ the recurrence reduces to $k_{i+1,i+1}=k_{i,i}=1$.
\end{proof}

The binomial $\binom{\delta}{d}$ in \eqref{eq:bessel-kappa} is the source
of the integer-offset degeneracy: for $\delta=t-s\in\ZZ_{\ge0}$ the
connection array is \emph{banded}, $\kappa_{i,m}=0$ for $i-m>\delta$, and
the residual block of Lemma~\ref{lem:bindetGene} eventually has a zero
diagonal. The next theorem makes this quantitative --- and evaluates the
determinant for every $(s,t)$ at once.

\subsection{The connection determinant}\label{ssec:bessel-det}

Unlike its trigonometric counterpart (Lemma~\ref{lem:bindetSigma}), the
kernel left over after removing row and column factors from
\eqref{eq:bessel-kappa} depends on $i+m$ as well as on $i-m$, so the dual
Jacobi--Trudi identity does not apply. Instead we evaluate, more
generally, \emph{every} minor of the connection array on consecutive rows
and consecutive columns: the two extra shift parameters make the family
stable under Desnanot--Jacobi condensation, exactly as for the Catalan
determinant of the reciprocal-sine function
(Section~\ref{ssec:xsin-catalan}).

\begin{Theorem}[connection determinant]\label{thm:bessel-det}
Let $p\ge a\ge0$ and $N\ge0$ be integers and put $q=p-a$. Then
\begin{multline}\label{eq:bessel-G}
\det\bigl(\kappa_{p+r,\,a+m}\bigr)_{0\le r,m\le N-1}
=\Bigl(\prod_{r=0}^{N-1}\frac{(2p+2r)!}{(2a+2r)!}\Bigr)\,4^{-qN}\\
\times\prod_{r=1}^{q}\prod_{c=1}^{N}
\frac{\delta+c-r}
{\bigl[(q-r)+(N-c)+1\bigr]\,(s+2p-1+c-r)\,(t+2a+2N-1+r-c)}\,.
\end{multline}
\end{Theorem}

\begin{proof}
Write $D(p,a,N)$ for the determinant and $G(p,a,N)$ for the right-hand
side, $D(p,a,0)=G(p,a,0)=1$. Telescoping the products over $r$ or over
$c$ turns $G$ into Pochhammer blocks,
\begin{equation}\label{eq:bessel-Gblocks}
G(p,a,N)=\Bigl(\prod_{r=0}^{N-1}\frac{(2p+2r)!}{(2a+2r)!}\Bigr)\,
\frac{4^{-qN}\,\prod_{r=1}^{q}(\delta+1-r)_N}
{\prod_{j=1}^{q}(j)_N\;\prod_{c=1}^{N}(s+p+a+c-1)_q\;
 \prod_{r=1}^{q}(t+2a+N+r-1)_N}
\end{equation}
(for the hook lengths, $\prod_{c=1}^{N}\bigl[(q-r)+(N-c)+1\bigr]
=(q-r+1)_N$; for the $s$-block, $2p-q=p+a$). The proof is by induction
on $N$ through the Desnanot--Jacobi identity \eqref{eq:condensation}, in the
manner of the condensation proofs of the determinant calculus
\cite{Krattenthaler1998}.

For $N=0$ both sides equal $1$. For $q=0$ (any $N$) the matrix
$(\kappa_{p+r,\,p+m})$ is unitriangular ($\kappa_{i,m}=0$ for $m>i$,
$\kappa_{i,i}=1$), so $D=1=G$, every product in \eqref{eq:bessel-G} being
empty. For $N=1$,
$$
G(p,a,1)=\frac{(2p)!}{(2a)!}\,
\frac{4^{-q}\,\delta(\delta-1)\cdots(\delta-q+1)}
{q!\,(s+p+a)_q\,(t+2a+1)_q}
=\kappa_{p,a}
$$
by \eqref{eq:bessel-kappa}. Let $N\ge2$ and $q\ge1$. Deleting the first
or last row and the first or last column of the matrix of $D(p,a,N)$
shifts $(p,a)$ and $N$ inside the same family, so its four corner minors
are $D(p+1,a+1,N-1)$, $D(p,a,N-1)$, $D(p+1,a,N-1)$, $D(p,a+1,N-1)$, and
its central minor is $D(p+1,a+1,N-2)$; the Desnanot--Jacobi identity
reads
\begin{multline}\label{eq:bessel-DJ}
D(p,a,N)\,D(p+1,a+1,N-2)\\
=D(p+1,a+1,N-1)\,D(p,a,N-1)-D(p+1,a,N-1)\,D(p,a+1,N-1).
\end{multline}
The five smaller determinants have offsets $q$, $q$, $q+1$, $q-1\ge0$ and
sizes $N-1$, $N-2$, so they equal the corresponding $G$'s by the
induction hypothesis, and $G(p+1,a+1,N-2)\ne0$, every factor of
\eqref{eq:bessel-Gblocks} being a nonzero element of $\QQ(s,t)$. Hence
\eqref{eq:bessel-DJ} determines $D(p,a,N)$, and it remains to check that
$G$ satisfies the same relation, i.e.\ that $R_1-R_2=1$ for
$$
R_1:=\frac{G(p+1,a+1,N-1)\,G(p,a,N-1)}{G(p,a,N)\,G(p+1,a+1,N-2)},
\qquad
R_2:=\frac{G(p+1,a,N-1)\,G(p,a+1,N-1)}{G(p,a,N)\,G(p+1,a+1,N-2)}.
$$

Both quotients collapse block by block. First,
$R_1=\rho(p+1,a+1,N-1)\big/\rho(p,a,N)$ for the one-step ratio
$\rho(p,a,N):=G(p,a,N)/G(p,a,N-1)$, which \eqref{eq:bessel-Gblocks}
telescopes to
$$
\rho(p,a,N)
=\frac{(2p+2N-2)!}{(2a+2N-2)!}\,\frac{(\delta+N-q)_q\,(t+2a+N-1)_q}
{4^{q}\,(N)_q\,(s+p+a+N-1)_q\,(t+2a+2N-2)_q\,(t+2a+2N-1)_q}
$$
(the $t$-block by
$(t+2a+N+r-1)_N/(t+2a+N+r-2)_{N-1}
=\frac{(t+2a+2N+r-3)(t+2a+2N+r-2)}{t+2a+N+r-2}$, whose three factors
telescope over $r=1,\dots,q$). In $R_1$ the factorial and $4$-blocks
cancel outright, the two long $t$-blocks coincide between the two levels,
and each remaining Pochhammer quotient leaves a single linear factor
--- e.g.\ $(\delta+N-1-q)_q/(\delta+N-q)_q=(\delta+N-1-q)/(\delta+N-1)$
--- giving
$$
R_1=\frac{(q+N-1)\,(\delta-q+N-1)\,(s+p+a+N-1)\,(t+2a+q+N-1)}
        {(N-1)\,(\delta+N-1)\,(s+2p+N-1)\,(t+2a+N-1)}\,.
$$
Second, $R_2=R_1\cdot X$ with
$X=\mu_a/\mu_{a+1}$, where $\mu_a:=G(p+1,a,N-1)/G(p,a,N-1)$ is the
offset step $q\mapsto q+1$ at size $N-1$. Blockwise,
\begin{multline*}
\mu_a=\Bigl(\prod_{r=0}^{N-2}(2p+2r+1)(2p+2r+2)\Bigr)\\
\times\frac{4^{-(N-1)}\,(\delta-q)_{N-1}}
{(q+1)_{N-1}\,(t+2a+N+q-1)_{N-1}}\,
\prod_{c=1}^{N-1}\frac{s+p+a+c-1}{(s+2p+c-1)(s+2p+c)}\,,
\end{multline*}
and in the quotient $X$ everything depending on $p$ alone cancels,
leaving
$$
X=\frac{q\,(\delta-q)\,(s+p+a)\,(t+2a+2N+q-2)}
      {(q+N-1)\,(\delta-q+N-1)\,(s+p+a+N-1)\,(t+2a+N+q-1)}\,,
$$
so that
$$
R_2=\frac{q\,(\delta-q)\,(s+p+a)\,(t+2a+q+2N-2)}
        {(N-1)\,(\delta+N-1)\,(s+2p+N-1)\,(t+2a+N-1)}\,.
$$
Finally, in the variables $\sigma=s+p+a$, $\tau=t+2a$ and $M=N-1$ --- for
which $\delta-q=\tau-\sigma$, $\delta+M=\tau-\sigma+q+M$,
$s+2p+N-1=\sigma+q+M$ and $t+2a+N-1=\tau+M$ --- the required identity
$R_1-R_2=1$ becomes
\begin{multline}\label{eq:bessel-DJfinal}
(q+M)(\tau-\sigma+M)(\sigma+M)(\tau+q+M)
-q\,(\tau-\sigma)\,\sigma\,(\tau+q+2M)\\
=M\,(\tau-\sigma+q+M)(\sigma+q+M)(\tau+M),
\end{multline}
checked by expanding both sides. This completes the induction.
\end{proof}

\subsection{Proof of the closed form}\label{ssec:bessel-proof}

\begin{proof}[Proof of Theorem~\textup{\ref{thm:besselst}}]
By Lemma~\ref{lem:bessel-data} the two functionals are quasi-definite
over $\QQ(s,t)$, and Lemma~\ref{lem:bindetGene} applies. The residual
block is the case $(p,a,N)=(\bar n,0,\underline n)$ of
Theorem~\ref{thm:bessel-det}, in which $q=\bar n$ and the hook length
$(q-r)+(N-c)+1$ reads $n-r-c+1$; thus
$$
\HH_n=(-1)^{\binom{\bar n}{2}}
\Bigl(\prod_{l=0}^{\bar n-1}h^S_l\Bigr)
\Bigl(\prod_{m=0}^{\underline n-1}h^T_m\Bigr)\,
\det\bigl(\kappa_{\bar n+r,\,m}\bigr)_{0\le r,m\le \underline n-1}.
$$
Three groups of factors merge. The factorials:
$$
\prod_{l=0}^{\bar n-1}(2l)!\cdot\prod_{m=0}^{\underline n-1}(2m)!\cdot
\prod_{r=0}^{\underline n-1}\frac{(2\bar n+2r)!}{(2r)!}
=\prod_{i=0}^{n-1}(2i)!\,.
$$
The powers of two:
$4^{-\binom{\bar n}2-\binom{\underline n}2-\bar n\underline n}=4^{-\binom n2}
=2^{-n(n-1)}$. The $s$-dependent denominators: writing each Pochhammer
symbol as a quotient of Gamma factors (all of which cancel in the end),
$(s+l)_l\,(s+1)_{2l}
=\frac{\Gamma(s+2l)\,\Gamma(s+2l+1)}{\Gamma(s+l)\,\Gamma(s+1)}$ and
$(s+\bar n+c-1)_{\bar n}
=\frac{\Gamma(s+2\bar n+c-1)}{\Gamma(s+\bar n+c-1)}$, so
\begin{align*}
\prod_{l=0}^{\bar n-1}(s+l)_l\,(s+1)_{2l}\cdot
\prod_{c=1}^{\underline n}(s+\bar n+c-1)_{\bar n}
&=\frac{\prod_{j=0}^{2\bar n-1}\Gamma(s+j)}
       {\Gamma(s+1)^{\bar n}\,\prod_{j=0}^{\bar n-1}\Gamma(s+j)}\cdot
 \frac{\prod_{j=2\bar n}^{2\bar n+\underline n-1}\Gamma(s+j)}
      {\prod_{j=\bar n}^{\bar n+\underline n-1}\Gamma(s+j)}\\
&=\frac{\prod_{j=n}^{n+\bar n-1}\Gamma(s+j)}{\Gamma(s+1)^{\bar n}}
=\prod_{l=0}^{\bar n-1}(s+1)_{n-1+l}\,,
\end{align*}
and the same computation with $(t,\underline n,\bar n)$ in place of
$(s,\bar n,\underline n)$ gives
$\prod_{m<\underline n}(t+m)_m(t+1)_{2m}\cdot
\prod_{r=1}^{\bar n}(t+\underline n+r-1)_{\underline n}
=\prod_{m=0}^{\underline n-1}(t+1)_{n-1+m}$. Inverting these two products
(they stand in the denominators of \eqref{eq:bessel-norms} and
\eqref{eq:bessel-G}) and restoring the numerators $(s+\tfrac12)_l$,
$(t+\tfrac12)_m$ from \eqref{eq:bessel-norms} yields
\eqref{eq:bessel-closed}.
\end{proof}

\begin{Remark}[the parallel with the trigonometric family]
\label{rem:bessel-parallel}
The parallel with Proposition~\ref{prop:family} is exact, with one
systematic change of scale. Both formulas share the skeleton
$(-1)^{\binom{\bar n}2}\prod_{i<n}(2i)!$ times a signed double product;
but the trigonometric signed factor is $\Omega(t-s)$ while the Bessel one
is $\Omega(2(t-s))$: \emph{the Bessel ladder walks in half steps}. The
closed double-factorial evaluations of Lemma~\ref{lem:Pahalf} therefore
apply on the lines $t-s\in\tfrac12+\ZZ$ (where $2(t-s)$ is odd), exactly
as they apply to the trigonometric family on the lines $t-s$ odd; and the
forced roots of $\Omega$ produce the degeneration on $t-s\in\ZZ$, as they
do there on $t-s$ even. The parameter blocks, by contrast, are
reciprocal: $(s+1)_{n-1+l}$ in the \emph{denominator} here versus
$(s+1)_{2l}$ in the numerator there, reflecting moments that are bounded
(characteristic functions of probability measures) rather than
factorially growing. Both proofs are instances of the collapse mechanism
of $\M2$: two functionals of a single classical hierarchy differing in
one parameter --- Jacobi--Gegenbauer here, secant-type there, Wilson for
the reciprocal-sine function. The kernel
$\binom{\delta}{d}/\bigl((t+2m+1)_d(s+i+m)_d\bigr)$ of
\eqref{eq:bessel-kappa} sits between the two known extremes: like the
pure binomial $\binom{(t-s)/2}{d}$ of Lemma~\ref{lem:connGene} it is
binomial in $i-m$, but like the Catalan kernel of the reciprocal-sine
family it also depends on $i+m$ --- which is why the finishing step is
condensation rather than Jacobi--Trudi. Finally, for half-integer $s,t$
the members of \eqref{eq:besselfam} are elementary --- e.g.\
$(s,t)=(\tfrac12,0)$ gives $\sin(x)/x+\int_0^xJ_0$ --- but for general
order they are genuine Bessel functions, so the family lies strictly
outside the scope of Section~\ref{sec:gen}: this is not a
reparametrisation of a known family. Consistently with the collapse
mechanism, the crossbreeds mixing the trigonometric and Bessel
hierarchies --- one functional secant-type, the other Gegenbauer --- were
tested numerically and are never smooth: the single-hierarchy hypothesis
of $\M2$ is not ornamental.
\end{Remark}

\section{The double shift of the Bessel $(s,t)$ family}\label{sec:besseldshift}

As for the trigonometric family, alongside $\HH_n$ the most regular
variant is the double shift $\HH_n^{(2)}=\det(a_{2i+j+2})$ of
\eqref{eq:Hdshiftdef}: shifting the column index by $2$ preserves
parity, and $\HH_n^{(2)}$ is the dilated determinant of the Christoffel
transform $y\,\EE$ of $\EE$, with even and odd parts
$\mathcal S'[p]=\mathcal S[yp]$, $\mathcal T'[p]=\mathcal T[yp]$. The
general reduction of Proposition~\ref{prop:dshiftgen} and the transform
data of Lemma~\ref{lem:christoffel} hold for arbitrary $u_j,v_j$; what
follows are the Bessel counterparts of the special ingredients of
Lemma~\ref{lem:christoffelspecial}.

\begin{Lemma}[Christoffel transform, Bessel family]\label{lem:besselchristoffel}
For the coefficients $u_j,v_j$ of \eqref{eq:bessel-u}, with
$P_i,Q_m,\kappa_{i,m}$ as above:
\begin{enumerate}
\item[\rm(a)] $D^S_n=(-1)^nP_n(0)=(-1)^n\,\dfrac{(2n-1)!!}{2^n\,(s+n)_n}$,
and likewise $D^T_m=(-1)^m\,\dfrac{(2m-1)!!}{2^m\,(t+m)_m}$;
\item[\rm(b)] the connection coefficients
$P'_i=\sum_{m}\kappa'_{i,m}Q'_m$ of the primed pair are
\begin{equation}\label{eq:bessel-kappap}
\kappa'_{i,m}
=\frac{2i+1}{2m+1}\;\kappa_{i,m}\Big|_{(s,t)\to(s+1,t+1)}
=\frac{(2i+1)!}{(2m+1)!\,4^{d}}\,\binom{\delta}{d}\,
\frac{1}{(t+2m+2)_d\,(s+i+m+1)_d}
\end{equation}
(the parameter $\delta=t-s$ is invariant under the shift
$(s,t)\to(s+1,t+1)$).
\end{enumerate}
\end{Lemma}

\begin{proof}
(a) By Lemma~\ref{lem:christoffel}(i),
$D^S_{n+1}=c^S_nD^S_n-\lambda^S_nD^S_{n-1}$ with $D^S_0=1$,
$D^S_1=u_1=-\tfrac1{2(s+1)}$. The displayed product has
$$
\frac{D^S_{n+1}}{D^S_n}
=-\frac{2n+1}{2}\,\frac{(s+n)_n}{(s+n+1)_{n+1}}
=-\frac{(2n+1)(s+n)}{2\,(s+2n)(s+2n+1)}\,,
$$
which equals $u_1$ at $n=0$; inserting this ratio and
$c^S_n=u_{2n}+u_{2n+1}$, $\lambda^S_n=u_{2n-1}u_{2n}$ from
\eqref{eq:bessel-u}, the recurrence divided by $D^S_n$ becomes a
rational-function identity in $(n,s)$, routine to verify. The case of
$D^T$ is identical.

(b) By (a) and Lemma~\ref{lem:christoffel}(iii) the primed norms
$h^{S'}_i=(D^S_{i+1}/D^S_i)\,h^S_i$ and likewise $h^{T'}_m$ are
nonzero in $\QQ(s,t)$, so
$\mathcal S',\mathcal T'$ are quasi-definite and, by
Lemma~\ref{lem:connrec}, $\kappa'$ is the unique solution of
\eqref{eq:kapparec} with the primed recurrence data of
Lemma~\ref{lem:christoffel}(ii), $c^{S'}_i=u_{2i+1}+u_{2i+2}$,
$\lambda^{S'}_i=u_{2i}u_{2i+1}$, and likewise for $T'$. The candidate
\eqref{eq:bessel-kappap} has $\kappa'_{i,i}=1$, and its four neighbour
ratios are those displayed in the proof of
Lemma~\ref{lem:bessel-kappa}, taken at $(s+1,t+1)$ and corrected by the
factors $\tfrac{2i+3}{2i+1}$, $\tfrac{2m+1}{2m-1}$,
$\tfrac{2m+1}{2m+3}$, $\tfrac{2i-1}{2i+1}$ respectively. Dividing
\eqref{eq:kapparec} by $\kappa'_{i,m}$ turns it again into a
rational-function identity in $(m,d,s,t)$, verified by clearing
denominators; the boundary columns take care of themselves exactly as
in Lemma~\ref{lem:bessel-kappa}.
\end{proof}

\begin{Proposition}[double shift, Bessel family; $\M2$]\label{prop:besseldshift}
Let $f_{s,t}$ be as in \eqref{eq:besselfam}. Then, as an identity in
$\QQ(s,t)$,
\begin{equation}\label{eq:besseldshift}
\HH_n^{(2)}(f_{s,t})=\det\bigl(a_{2i+j+2}\bigr)_{0\le i,j\le n-1}
=(-1)^n\,\frac{(2n-1)!!}{2^n\,(s+n)_{\bar n}\,(t+n)_{\underline n}}\;
\HH_n(f_{s,t}),
\end{equation}
where $\HH_n(f_{s,t})$ is the unshifted determinant of
Theorem~\textup{\ref{thm:besselst}}.
\end{Proposition}

\begin{proof}
By \eqref{eq:dshiftgen}, writing
$(-1)^nP_{\bar n}(0)\,Q_{\underline n}(0)=D^S_{\bar n}D^T_{\underline n}$,
$$
\HH_n^{(2)}=D^S_{\bar n}\,D^T_{\underline n}\;
\frac{\det\bigl(\kappa'_{\bar n+r,m}\bigr)}
     {\det\bigl(\kappa_{\bar n+r,m}\bigr)}\;\HH_n
\qquad(0\le r,m\le\underline n-1).
$$
By Lemma~\ref{lem:besselchristoffel}(b), pulling $2(\bar n{+}r){+}1$ out
of row $r$ and $1/(2m{+}1)$ out of column $m$,
$$
\det\bigl(\kappa'_{\bar n+r,m}\bigr)
=\frac{(2n-1)!!}{(2\bar n-1)!!\,(2\underline n-1)!!}\;
\det\bigl(\kappa_{\bar n+r,m}\bigr)\Big|_{(s,t)\to(s+1,t+1)}\,.
$$
Both connection determinants are the case $(p,a,N)=(\bar n,0,\underline n)$ of
Theorem~\ref{thm:bessel-det}, and in the quotient of \eqref{eq:bessel-G}
at $(s+1,t+1)$ by \eqref{eq:bessel-G} at $(s,t)$ everything cancels
except the two parameter blocks --- the factor $\delta+c-r$ is invariant
--- and these telescope: for fixed $r$,
$\prod_{c=1}^{\underline n}\frac{s+2\bar n-1+c-r}{s+2\bar n+c-r}
=\frac{s+2\bar n-r}{s+2\bar n+\underline n-r}$, and the product over
$r=1,\dots,\bar n$ gives, using $2\bar n+\underline n=n+\bar n$,
$$
\frac{\det\bigl(\kappa_{\bar n+r,m}\bigr)\big|_{(s+1,t+1)}}
     {\det\bigl(\kappa_{\bar n+r,m}\bigr)}
=\frac{(s+\bar n)_{\bar n}}{(s+n)_{\bar n}}\cdot
 \frac{(t+\underline n)_{\underline n}}{(t+n)_{\underline n}}\,,
$$
the $t$-block telescoping in the same way, over $r$ first and then over
$c$. By Lemma~\ref{lem:besselchristoffel}(a),
$$
D^S_{\bar n}\,D^T_{\underline n}
=(-1)^{\bar n+\underline n}\,
\frac{(2\bar n-1)!!\,(2\underline n-1)!!}
     {2^{\,\bar n+\underline n}\,(s+\bar n)_{\bar n}\,(t+\underline n)_{\underline n}}
=(-1)^n\,\frac{(2\bar n-1)!!\,(2\underline n-1)!!}
     {2^{\,n}\,(s+\bar n)_{\bar n}\,(t+\underline n)_{\underline n}}\,.
$$
Multiplying the three displays, the double factorials and the products
$(s+\bar n)_{\bar n}$, $(t+\underline n)_{\underline n}$ cancel, leaving
\eqref{eq:besseldshift}.
\end{proof}

\begin{Remark}[the double-shift parallel]\label{rem:besseldshift}
The parallel with the double shift of the trigonometric family
(Proposition~\ref{prop:dshift}) is again exact, with the change of scale
of Remark~\ref{rem:bessel-parallel}. There the double shift
\emph{multiplies} $\HH_n$ by
$(2n-1)!!\,\prod_{k<\bar n}(s+2k+1)\,\prod_{k<\underline n}(t+2k+1)$; here it
\emph{divides} by
$(-1)^n\,2^{n}\,(s+n)_{\bar n}(t+n)_{\underline n}/(2n-1)!!$\,: the universal
factor $(2n-1)!!$ is one and the same, while the parameter blocks are
again reciprocal --- bounded moments --- and the odd-step products
$(s{+}1)(s{+}3)\cdots$ become unit-step Pochhammer blocks, the half-step
ladder once more. The transform data compare in the same way:
$D^S_n=(2n-1)!!\,\prod_{k<n}(s+2k+1)$ there,
$D^S_n=(-1)^n(2n-1)!!/\bigl(2^n(s+n)_n\bigr)$ here. Since the new factor
in \eqref{eq:besseldshift} has constant numerator and its poles lie at
$s+n+l=0$ or $t+n+m=0$ only, for $s,t>-\tfrac12$ it is finite and
nonzero: the vanishing locus is untouched, the contrast of
Corollary~\ref{cor:besseldicho} holds verbatim for $\HH_n^{(2)}$, and on
the offsets $t-s\notin\ZZ$ the double shift remains a closed product of
linear forms. (The \emph{single} shift $\det(a_{2i+j+1})$, which swaps
the two parities, behaves differently: already at $n=4$ it carries an
irreducible quadratic factor over $\QQ(s,t)$, and we do not pursue it.)
\end{Remark}

\section{A multiplicative Bessel family: $(1+x)\,\mathrm{cosb}_\nu^{\,2}$ at even orders}
\label{sec:xbesseleven}

The additive family $\mathrm{cosb}_s+\int_0^x\mathrm{cosb}_t$ of
Section~\ref{sec:besselst} is the Bessel analogue of the Euler $(s,t)$
family. The Bessel cosine introduced there,
\begin{equation}\label{eq:xbe-cosb}
\mathrm{cosb}_\nu(x)\;=\;\Gamma(\nu+1)\Bigl(\frac2x\Bigr)^{\!\nu}J_\nu(x)
\;=\;\sum_{k\ge0}\frac{(-1)^k}{4^k\,k!\,(\nu+1)_k}\,x^{2k},
\qquad \mathrm{cosb}_\nu(0)=1,
\end{equation}
also produces an analogue of
the \emph{secant} family $(1+x)/\cos(x)^{s+1}$ of
Section~\ref{sec:gen-xcos}, with the exponent ladder replaced once more by
the order ladder: the multiplicative family
\begin{equation}\label{eq:xbe-def}
f_\nu(x)\;=\;(1+x)\,\mathrm{cosb}_\nu(x)^2 .
\end{equation}
(The parameter cannot sit as a power on the function side:
as recorded in Section~\ref{sec:besselst}, no fixed $h$ gives
$\mathrm{cosb}_\nu=h^{-(\nu+1)}$; and it is the \emph{square} of the even
factor that carries the structure, through the product formula of
Lemma~\ref{lem:xbe-moments} below.) At $\nu=\tfrac12$ the member
\eqref{eq:xbe-def} is $(1+x)(\sin x/x)^2$; at the boundary order
$\nu=-\tfrac12$ it is $(1+x)\cos^2x$.

The determinants $\HH_n(f_\nu)$ display a sharp parity contrast: at even
orders $n=2N$ they factor completely into linear forms in $\nu$, while at
odd orders an irreducible factor of growing degree appears. This section
proves the even half in closed form (Theorem~\ref{thm:xbesseleven}); the
odd orders, which remain open, are discussed in Remark~\ref{rem:xbe-odd}.

The method deserves a word in advance. The even moment functional of
$f_\nu$ is \emph{not} classical --- it is the self-convolution of the
symmetric Beta law, not an orthogonality weight --- so the biorthogonal collapse $\M2$ that
evaluates the secant family and the additive Bessel family is out of
reach. What replaces it is the most elementary reduction of the paper:
the Vandermonde method $\M1$ of Section~\ref{sec:vandermonde}, applied
here for the only time beyond the Beta family. Watson's classical product
formula makes the even moments a \emph{single} hypergeometric term; the
$(1+x)$-coupling preserves this; and for even $n$ --- and only for even
$n$ --- the resulting row factorisation lands exactly on the Vandermonde
degree bound. No orthogonal polynomials appear, and quasi-definiteness is
never invoked.

\subsection{The moments}\label{ssec:xbe-moments}

Write $a_m=m!\,[x^m]f_\nu$ as usual.

\begin{Lemma}[the moments are a single hypergeometric term]\label{lem:xbe-moments}
For all $k\ge0$,
\begin{equation}\label{eq:xbe-mom}
a_{2k}=(-4)^k\,
\frac{(\tfrac12)_k\,(\nu+\tfrac12)_k}{(\nu+1)_k\,(2\nu+1)_k},
\qquad
a_{2k+1}=(2k+1)\,a_{2k}
=(-4)^k\,
\frac{(\tfrac32)_k\,(\nu+\tfrac12)_k}{(\nu+1)_k\,(2\nu+1)_k}\,.
\end{equation}
\end{Lemma}

\begin{proof}
The odd--even coupling is the general $(1+x)$-mechanism of
Section~\ref{sec:gen-xcos}: since
$x\cdot x^{2k}/(2k)!=(2k+1)\,x^{2k+1}/(2k+1)!$, the odd part of $f_\nu$
contributes $a_{2k+1}=(2k+1)a_{2k}$, and the second expression follows
from the first by $(\tfrac32)_k/(\tfrac12)_k=2k+1$.

For the even part, square the series \eqref{eq:xbe-cosb}:
$$
[x^{2k}]\,\mathrm{cosb}_\nu^2
=\sum_{j=0}^{k}
\frac{(-1)^j}{4^j\,j!\,(\nu+1)_j}\cdot
\frac{(-1)^{k-j}}{4^{k-j}\,(k-j)!\,(\nu+1)_{k-j}}
=\frac{(-1)^k}{4^k\,k!}\,
\sum_{j=0}^{k}\binom kj\frac{1}{(\nu+1)_j\,(\nu+1)_{k-j}}\,.
$$
Put $c=\nu+1$. From $\binom kj=(-1)^j(-k)_j/j!$ and
$(c)_{k-j}=(c)_k/(c+k-j)_j$ with $(c+k-j)_j=(-1)^j(1-c-k)_j$, the sum is
a terminating Gauss series at $1$,
$$
\sum_{j=0}^{k}\binom kj\frac{1}{(c)_j\,(c)_{k-j}}
=\frac1{(c)_k}\sum_{j\ge0}\frac{(-k)_j\,(1-c-k)_j}{j!\,(c)_j}
=\frac{{}_2F_1(-k,\,1-c-k;\,c;\,1)}{(c)_k}
=\frac{(2c+k-1)_k}{(c)_k^{\,2}}
$$
by the Chu--Vandermonde summation
${}_2F_1(-k,b;c;1)=(c-b)_k/(c)_k$. Hence, with $(2k)!=4^k\,k!\,(\tfrac12)_k$,
$$
a_{2k}=(2k)!\,[x^{2k}]\,\mathrm{cosb}_\nu^2
=(-1)^k\,(\tfrac12)_k\,\frac{(2\nu+k+1)_k}{(\nu+1)_k^{\,2}}\,,
$$
and the duplication formula in Pochhammer form,
$(2\nu+1)_{2k}=4^k(\nu+\tfrac12)_k(\nu+1)_k$, turns
$(2\nu+k+1)_k=(2\nu+1)_{2k}/(2\nu+1)_k$ into
$4^k(\nu+\tfrac12)_k(\nu+1)_k/(2\nu+1)_k$, which is \eqref{eq:xbe-mom}.
\end{proof}

\subsection{The Vandermonde reduction at even orders}\label{ssec:xbe-vdm}

Throughout fix $N\ge1$ and put $n=2N$. Define
\begin{equation}\label{eq:xbe-w}
w_i\;=\;\frac{a_{2i}}{(\nu+1+i)_{N-1}\,(2\nu+1+i)_{N-1}}
\qquad(0\le i\le 2N-1)
\end{equation}
and, for $0\le l,m\le N-1$, the polynomials
\begin{equation}\label{eq:xbe-Q}
\begin{aligned}
Q_{2l}(x)&=(-4)^l\,
(x+\tfrac12)_l\,(x+\nu+\tfrac12)_l\,
(x+\nu+1+l)_{N-1-l}\,(x+2\nu+1+l)_{N-1-l},\\
Q_{2m+1}(x)&=(2x+2m+1)\,Q_{2m}(x),
\end{aligned}
\end{equation}
of degrees $\deg Q_{2l}=2N-2$ and $\deg Q_{2m+1}=2N-1$, with coefficients
in $\QQ[\nu]$ (the Pochhammer symbols are taken in the variable $x$).

\begin{Proposition}[row factorisation and reduction]\label{prop:xbe-vdm}
For all $0\le i,j\le 2N-1$,
\begin{equation}\label{eq:xbe-rowfact}
a_{2i+j}\;=\;w_i\;Q_j(i),
\end{equation}
and consequently
\begin{equation}\label{eq:xbe-vdmred}
\HH_{2N}(f_\nu)
\;=\;\det M_N\;\prod_{k=1}^{2N-1}k!\;\prod_{i=0}^{2N-1}w_i\,,
\end{equation}
where $M_N$ is the coefficient matrix of the family \eqref{eq:xbe-Q},
that is, $Q_j(x)=\sum_{k=0}^{2N-1}m_{kj}\,x^k$ with $0\le k,j\le 2N-1$.
\end{Proposition}

\begin{proof}
By Lemma~\ref{lem:xbe-moments} the even moments have the term ratio
$$
\frac{a_{2i+2}}{a_{2i}}
=-4\,\frac{(i+\tfrac12)(i+\nu+\tfrac12)}{(i+\nu+1)(i+2\nu+1)}\,,
$$
which telescopes to
$$
a_{2i+2l}=a_{2i}\,(-4)^l\,
\frac{(i+\tfrac12)_l\,(i+\nu+\tfrac12)_l}
     {(i+\nu+1)_l\,(i+2\nu+1)_l}
\qquad(l\ge0).
$$
Every even column index $j=2l$ of $\HH_{2N}$ has $l\le N-1$, so both
denominators divide the fixed padding of \eqref{eq:xbe-w}:
$(i+\nu+1)_l\,(i+\nu+1+l)_{N-1-l}=(i+\nu+1)_{N-1}$ and likewise for the
second block. Hence $a_{2i+2l}=w_i\,Q_{2l}(i)$; and
$a_{2i+2m+1}=(2i+2m+1)\,a_{2i+2m}=w_i\,(2i+2m+1)\,Q_{2m}(i)
=w_i\,Q_{2m+1}(i)$, which is \eqref{eq:xbe-rowfact}. The degrees in
\eqref{eq:xbe-Q} are $l+l+(N-1-l)+(N-1-l)=2N-2$ and $2N-1$, both
$\le n-1$; so the matrix factors,
$$
\bigl(a_{2i+j}\bigr)_{0\le i,j\le 2N-1}
=\operatorname{diag}(w_i)\cdot V\cdot M_N,
\qquad V=\bigl(i^{\,k}\bigr)_{0\le i,k\le 2N-1},
$$
exactly as in the proof of Lemma~\ref{lem:vdm} (the Vandermonde
reduction $\M1$, applied to the array $a_{2i+j}$ in the row variable
$i$, with nodes $x_i=i$). Taking determinants, and evaluating
$\det V=\prod_{0\le i<i'\le 2N-1}(i'-i)=\prod_{k=1}^{2N-1}k!$, gives
\eqref{eq:xbe-vdmred}.
\end{proof}

It remains to evaluate $\det M_N$, a polynomial in $\nu$.

\subsection{The coefficient determinant}\label{ssec:xbe-coeffdet}

\begin{Lemma}[coefficient determinant]\label{lem:xbe-coeffdet}
$$
\det M_N
=2^{\,N(2N-1)}\,
\prod_{k=0}^{N-1}(\tfrac12)_k\,(\nu+\tfrac12)_k^{\,2}\,(2\nu+\tfrac12)_k\,.
$$
\end{Lemma}

\begin{proof}
All computations take place in $\QQ(\nu)$. For \emph{any} choice of nodes
$u_0,\dots,u_{2N-1}\in\QQ(\nu)$ the same factorisation as above gives
$\bigl(Q_j(u_p)\bigr)_{p,j}=V(u)\,M_N$ with $V(u)=(u_p^{\,k})_{p,k}$,
hence
\begin{equation}\label{eq:xbe-alt}
\det\bigl(Q_j(u_p)\bigr)_{0\le p,j\le 2N-1}
=\det M_N\;\prod_{0\le p<p'\le 2N-1}(u_{p'}-u_p)\,.
\end{equation}
Choose the two interlaced arithmetic strings
\begin{equation}\label{eq:xbe-nodes}
u_{2r}=-\tfrac12-r,
\qquad
u_{2r+1}=-\nu-\tfrac12-r
\qquad(0\le r\le N-1),
\end{equation}
whose $\binom{2N}2$ pairwise differences are all nonzero in $\QQ(\nu)$.
The point of this choice is that the matrix
$T=\bigl(Q_j(u_p)\bigr)$ becomes \emph{lower triangular}.

\emph{Triangularity.} At $u_{2r}$ the block $(x+\tfrac12)_l$ of
\eqref{eq:xbe-Q} evaluates to $(-r)_l$, which vanishes exactly for
$l>r$; so $Q_{2l}(u_{2r})=0$ for $2l>2r$. For the odd columns,
$Q_{2m+1}(u_{2r})=\bigl(2u_{2r}+2m+1\bigr)\,Q_{2m}(u_{2r})
=2(m-r)\,Q_{2m}(u_{2r})$ vanishes for $m>r$ (through $Q_{2m}$) and at
$m=r$ (through the linear factor); so $Q_{2m+1}(u_{2r})=0$ for
$2m+1>2r$. At $u_{2r+1}$ the block $(x+\nu+\tfrac12)_l$ evaluates to
$(-r)_l$, which vanishes for $l>r$; so $Q_{2l}(u_{2r+1})=0$ for
$2l>2r+1$, and $Q_{2m+1}(u_{2r+1})=0$ for $2m+1>2r+1$ through the factor
$Q_{2m}$. In all cases $T_{p,j}=0$ for $j>p$.

\emph{The diagonal.} At $x=u_{2r}$ the four blocks of $Q_{2r}$ evaluate
to $(-r)_r=(-1)^r r!$, $(\nu-r)_r$, $(\nu+\tfrac12)_{N-1-r}$ and
$(2\nu+\tfrac12)_{N-1-r}$, so
$$
T_{2r,2r}=Q_{2r}(u_{2r})
=4^r\,r!\;(\nu-r)_r\,
(\nu+\tfrac12)_{N-1-r}\,(2\nu+\tfrac12)_{N-1-r}\,.
$$
At $x=u_{2r+1}$ the blocks of $Q_{2r}$ evaluate to
$(-\nu-r)_r=(-1)^r(\nu+1)_r$, $(-r)_r=(-1)^r r!$,
$(\tfrac12)_{N-1-r}$ and $(\nu+\tfrac12)_{N-1-r}$, while the linear
factor is $2u_{2r+1}+2r+1=-2\nu$; with $\nu\,(\nu+1)_r=(\nu)_{r+1}$,
$$
T_{2r+1,2r+1}=Q_{2r+1}(u_{2r+1})
=(-1)^{r+1}\,2\cdot4^r\,r!\;(\nu)_{r+1}\,
(\tfrac12)_{N-1-r}\,(\nu+\tfrac12)_{N-1-r}\,.
$$
Multiplying down the diagonal ($\sum_r(r+1)=\binom{N+1}2$, and the tail
blocks reindexed by $k=N-1-r$),
\begin{equation}\label{eq:xbe-diag}
\det T
=(-1)^{\binom{N+1}2}\,2^{\,N(2N-1)}
\Bigl(\prod_{k=1}^{N-1}k!\Bigr)^{\!2}\,
\prod_{r=0}^{N-1}(\nu-r)_r\,(\nu)_{r+1}\;
\prod_{k=0}^{N-1}(\tfrac12)_k\,(\nu+\tfrac12)_k^{\,2}\,(2\nu+\tfrac12)_k\,,
\end{equation}
the power of two being $2^N\cdot 4^{\,2\binom N2}=2^{\,N(2N-1)}$.

\emph{The Vandermonde.} The pairs $p<p'$ of \eqref{eq:xbe-nodes} fall
into four families: two pure strings, $u_{2r'}-u_{2r}=-(r'-r)$ and
$u_{2r'+1}-u_{2r+1}=-(r'-r)$ for $r<r'$, contributing
$\prod_{r<r'}(r'-r)^2=\bigl(\prod_{k=1}^{N-1}k!\bigr)^2$; the mixed pairs
$u_{2r'+1}-u_{2r}=-(\nu+r'-r)$ for $r\le r'$ (that is, $2r<2r'+1$),
where the difference $d=r'-r\in\{0,\dots,N-1\}$ occurs $N-d$ times; and
the mixed pairs $u_{2r'}-u_{2r+1}=\nu-(r'-r)$ for $r<r'$, with
$d=r'-r\in\{1,\dots,N-1\}$ occurring $N-d$ times. Hence
\begin{equation}\label{eq:xbe-vdmnodes}
\prod_{p<p'}(u_{p'}-u_p)
=(-1)^{\binom{N+1}2}
\Bigl(\prod_{k=1}^{N-1}k!\Bigr)^{\!2}\,
\prod_{d=0}^{N-1}(\nu+d)^{N-d}\;
\prod_{d=1}^{N-1}(\nu-d)^{N-d}\,.
\end{equation}

\emph{Division.} Since $(\nu)_{r+1}=\prod_{d=0}^{r}(\nu+d)$ and
$(\nu-r)_r=\prod_{d=1}^{r}(\nu-d)$, one has
$\prod_{r=0}^{N-1}(\nu)_{r+1}=\prod_{d=0}^{N-1}(\nu+d)^{N-d}$ and
$\prod_{r=0}^{N-1}(\nu-r)_r=\prod_{d=1}^{N-1}(\nu-d)^{N-d}$: the two
$\nu$-blocks of \eqref{eq:xbe-diag} are exactly the two mixed blocks of
\eqref{eq:xbe-vdmnodes}, and the signs and squared superfactorials also
match. Dividing \eqref{eq:xbe-diag} by \eqref{eq:xbe-vdmnodes} in
\eqref{eq:xbe-alt} leaves the stated product.
\end{proof}

\subsection{The closed form}\label{ssec:xbe-closed}

\begin{Theorem}[even orders of $(1+x)\,\mathrm{cosb}_\nu^2$; $\M1$]
\label{thm:xbesseleven}
Let $f_\nu=(1+x)\,\mathrm{cosb}_\nu^2$. For every $N\ge1$, as an identity
in $\QQ(\nu)$,
\begin{equation}\label{eq:xbe-closed}
\HH_{2N}(f_\nu)
=(-1)^N\,2^{\binom{2N}2}\,
\prod_{i=0}^{2N-1}(2i)!\;\;
\prod_{k=0}^{N-1}(\tfrac12)_k\,(\nu+\tfrac12)_k^{\,2}\,(2\nu+\tfrac12)_k
\;\;
\prod_{i=0}^{2N-1}
\frac{(\nu+\tfrac12)_i}{(\nu+1)_{N-1+i}\,(2\nu+1)_{N-1+i}}\,.
\end{equation}
The left side is regular at every $\nu\notin\{-1,-2,\dots\}$, and
\eqref{eq:xbe-closed} may be specialised there after the finitely many
cancellations between the $(2\nu+1)$-blocks and the half-integer
numerator factors.
\end{Theorem}

\begin{proof}
Combine Proposition~\ref{prop:xbe-vdm} and
Lemma~\ref{lem:xbe-coeffdet}. In the product of the row factors
\eqref{eq:xbe-w}, insert the moments \eqref{eq:xbe-mom} and telescope the
padding, $(\nu+1)_i\,(\nu+1+i)_{N-1}=(\nu+1)_{N-1+i}$ and likewise for
$(2\nu+1)$:
$$
\prod_{i=0}^{2N-1}w_i
=\prod_{i=0}^{2N-1}
\frac{(-4)^i\,(\tfrac12)_i\,(\nu+\tfrac12)_i}
     {(\nu+1)_{N-1+i}\,(2\nu+1)_{N-1+i}}\,.
$$
Here $\prod_i(-4)^i=(-1)^N\,4^{\binom{2N}2}$, since
$\sum_{i<2N}i=\binom{2N}2=N(2N-1)\equiv N\pmod 2$; and
$(\tfrac12)_i=(2i)!/(4^i\,i!)$ gives
$$
\prod_{k=1}^{2N-1}k!\;\prod_{i=0}^{2N-1}(\tfrac12)_i
=4^{-\binom{2N}2}\prod_{i=0}^{2N-1}(2i)!\,,
$$
the two superfactorials cancelling. The powers of four cancel each other,
and with the factor $2^{N(2N-1)}=2^{\binom{2N}2}$ of
Lemma~\ref{lem:xbe-coeffdet} the product
$\det M_N\,\prod_k k!\,\prod_iw_i$ collapses to \eqref{eq:xbe-closed}.

For the regularity statement: each $a_m$ is a rational function of $\nu$
with poles only at negative integers (visible in the Cauchy square in the
proof of Lemma~\ref{lem:xbe-moments}, whose denominators are products of
$(\nu+1)_j$), and $\HH_{2N}$ is a polynomial in the $a_m$.
\end{proof}

The first values are
$$
\HH_2(f_\nu)=(-1)\cdot2\cdot 2!\cdot
\frac{\nu+\tfrac12}{(\nu+1)(2\nu+1)}=-\frac{2}{\nu+1}\,,
\qquad
\HH_4(f_\nu)=\frac{2160\,(\nu+\tfrac14)(\nu+\tfrac12)(\nu+\tfrac52)}
{(\nu+1)^7(\nu+2)^4(\nu+3)^2(\nu+4)}\,,
$$
the second after cancelling the half-integer factors of the
$(2\nu+1)$-blocks against the numerator.

\begin{Corollary}[complete factorisation and the zero locus]\label{cor:xbe-dicho}
For every $N\ge1$ the determinant $\HH_{2N}(f_\nu)$ is a nonzero rational
function of $\nu$ that factors completely into linear forms over $\QQ$: a
closed product form at every even order. Its numerator zeros lie at
negative half-integers (from $(\nu+\tfrac12)_k$) and at the points
$\nu=-\tfrac14-\tfrac j2$, $j\ge0$ (from $(2\nu+\tfrac12)_k$); its poles
at negative integers. In the Poisson range $\nu>-\tfrac12$:
$$
\HH_{2N}(f_\nu)=0
\iff
\nu=-\tfrac14\ \text{and}\ N\ge2 ,
$$
and $\operatorname{sgn}\HH_{2N}(f_\nu)=(-1)^N$ for $\nu>-\tfrac14$, while
$\HH_{2N}(f_\nu)<0$ throughout $-\tfrac12<\nu<-\tfrac14$.
\end{Corollary}

\begin{proof}
Inspect \eqref{eq:xbe-closed}. For $\nu>-\tfrac12$ every Pochhammer
factor is a product of positive reals except the blocks
$(2\nu+\tfrac12)_k=\prod_{j=0}^{k-1}(2\nu+\tfrac12+j)$, whose only
possibly nonpositive factor is $2\nu+\tfrac12$: these blocks vanish
simultaneously at $\nu=-\tfrac14$ (for every $k\ge1$, hence for $N\ge2$;
for $N=1$ the block range is empty), are positive for $\nu>-\tfrac14$,
and each block with $k\ge1$ is negative for
$-\tfrac12<\nu<-\tfrac14$, giving the extra sign $(-1)^{N-1}$ there.
\end{proof}

\begin{Corollary}[$\nu=\tfrac12$: the squared sinc]\label{cor:xbe-sinc}
$$
\HH_{2N}\Bigl((1+x)\bigl(\tfrac{\sin x}{x}\bigr)^2\Bigr)
=(-1)^N\,2^{\,8N^2-5N}\,
\Bigl(\prod_{k=1}^{2N-1}k!\Bigr)^{\!2}\,
\frac{(N-1)!}{(3N-1)!}\;
\prod_{i=0}^{2N-1}\frac{(2i)!}{(2N+2i-1)!}\,,
$$
with the values
$-\dfrac{2^2}{3}$,\;
$\dfrac{2^{16}}{3^4\,5^3\,7^2}$,\;
$-\dfrac{2^{47}}{3^9\,5^5\,7^5\,11^3\,13^2}$,\;
$\dfrac{2^{88}}{3^{14}\,5^7\,7^7\,11^5\,13^4\,17^3\,19^2}$
at $n=2,4,6,8$.
\end{Corollary}

\begin{proof}
Specialise \eqref{eq:xbe-closed} at $\nu=\tfrac12$, where
$(\nu+\tfrac12)_k=k!$, $(2\nu+\tfrac12)_k=(\tfrac32)_k$ and
$(2\nu+1)_{N-1+i}=(2)_{N-1+i}=(N+i)!$. From
$(\tfrac12)_k=(2k)!/(4^k\,k!)$ and $(\tfrac32)_k=(2k+1)!/(4^k\,k!)$ the
middle block is
$\prod_{k=0}^{N-1}(2k)!\,(2k+1)!\,16^{-k}
=16^{-\binom N2}\prod_{k=1}^{2N-1}k!$, while
$(\tfrac32)_{N-1+i}=(2N+2i-1)!\,/\bigl(4^{N-1+i}(N-1+i)!\bigr)$ turns the
tail into
$$
\prod_{i=0}^{2N-1}
\frac{4^{\,N-1+i}\;i!\;(N-1+i)!}{(2N+2i-1)!\;(N+i)!}
=4^{\,2N(N-1)+\binom{2N}2}\,
\Bigl(\prod_{k=1}^{2N-1}k!\Bigr)
\frac{(N-1)!}{(3N-1)!}\,
\prod_{i=0}^{2N-1}\frac{1}{(2N+2i-1)!}\,,
$$
using $\prod_i(N-1+i)!/(N+i)!=\prod_i(N+i)^{-1}=(N-1)!/(3N-1)!$.
Collecting the powers of two,
$N(2N-1)-2N(N-1)+4N(N-1)+2N(2N-1)=8N^2-5N$, and absorbing
$\prod_i(2i)!$ into the last product gives the stated form.
\end{proof}

The values in Corollary~\ref{cor:xbe-sinc} have numerators that are pure
powers of two and odd denominators. The member $(1+x)(\sin x/x)^2$ sits
one step from the reciprocal-sine function of Section~\ref{sec:xsinx}:
there the sine is downstairs, $(1+x)\,x/\sin x$, the functionals are
Wilson, and the evaluation runs through $\M2$ and condensation; here the
sine is upstairs and squared, the functional is not classical at all, and
the evaluation is pure Vandermonde --- yet only the even orders survive.

\begin{Remark}[odd orders: where $\M1$ stops]\label{rem:xbe-odd}
At an odd order $n=2N+1$ the even column indices reach $l=N$, so the
padding \eqref{eq:xbe-w} must be lengthened to
$(\nu+1+i)_N(2\nu+1+i)_N$, and then
$$
\deg Q_{2l}=2N=n-1,
\qquad
\deg Q_{2m+1}=2N+1=n>n-1 :
$$
the odd columns overflow the Vandermonde degree bound of
Lemma~\ref{lem:vdm} by exactly one, and $\M1$ yields no evaluation ---
for the dilated determinant of \eqref{eq:xbe-def} the method works
precisely at even orders. This failure is structural, not merely
technical. The even
half of the parity contrast is Theorem~\ref{thm:xbesseleven}; the odd
half remains open.
\end{Remark}

\section{Connection with a determinant of Chapoton--Han}\label{sec:root}

The dilated Hankel determinant of this paper was first met by the author
while studying a conjecture from his joint work with Chapoton on the
roots of the Poupard and Kreweras polynomials \cite{Chapoton2020Han}.
This section closes the circle: after a change of basis, that conjecture
\emph{is} the dilated Hankel determinant of $(1+x)/\cos x$ evaluated in
Proposition~\ref{prop:sec1}.

\subsection{The evaluation $\rho$ and the conjecture}\label{ssec:rootdef}

Call the \emph{index} of a nonzero polynomial its degree plus its
valuation, and call $P$ \emph{palindromic} of index $d$ if
$x^{d}P(1/x)=P(x)$ --- for example $P=x^i(1+x)^j$, of index $2i+j$. For
such $P$ the polynomial $(x^d+1)P(1)-2P(x)$ vanishes to second order at
$x=1$ (palindromy gives $P'(1)=\tfrac d2P(1)$), so the
\emph{index-lowering operator} of \cite[\S5]{Chapoton2020Han},
$$
\mathscr N_0(P)=\frac{(x^{d}+1)\,P(1)-2P(x)}{(x-1)^2}\,,
$$
is a polynomial, palindromic of index $d-2$ when nonzero. Iterating
$\lfloor d/2\rfloor$ times lands in index at most $1$, and, following
\cite{Chapoton2020Han}, the evaluation $\rho$ is the final value of this
iteration:
$$
\rho(P):=\text{constant term of }\mathscr N_0^{\lfloor d/2\rfloor}(P);
$$
in particular $\rho(1)=1$ and $\rho\bigl(c(1+x)\bigr)=c$. Since
$P=(x-1)^2Q$ has $P(1)=0$, giving $\mathscr N_0\bigl((x-1)^2Q\bigr)=-2Q$,
and since $\rho(P)=\rho(\mathscr N_0P)$ by definition, the map $\rho$
obeys, for palindromic $Q$,
\begin{equation}\label{eq:rhoprop}
\rho\bigl((x-1)^2Q(x)\bigr)=-2\,\rho\bigl(Q(x)\bigr).
\end{equation}
Note that $\rho$ is linear on each fixed index but is \emph{not} a linear
functional on all of $\QQ[x]$: for example
$\rho\bigl((1+x)x^k\bigr)\neq\rho(x^k)+\rho(x^{k+1})$, the three indices
$2k+1,\,2k,\,2k+2$ being distinct. Every manipulation below stays within a
single index, on which $\rho$ \emph{is} linear. Put
$$
A_k:=\rho(x^k),\qquad B_k:=\rho\bigl((1+x)x^k\bigr)\qquad(k\ge0),
$$
so that $(A_k)=1,1,2,10,104,1816,\dots$ are the generalised Euler
numbers of type $2^n$ (OEIS \texttt{A005799}) and
$(B_k/2^k)=1,1,2,8,56,608,\dots$ are the Genocchi medians
(OEIS \texttt{A005439}). Following
\cite{Chapoton2020Han}, let $M_N$ be the $N\times N$ matrix
\begin{equation}\label{eq:Mdef}
M_N(i,j)=\rho\bigl(x^i(1+x)^j\bigr)\,2^{-\lfloor j/2\rfloor},
\qquad 0\le i,j\le N-1 .
\end{equation}
The conjecture in question \cite[Conj.~5.4]{Chapoton2020Han} asserts that
\begin{equation}\label{eq:conj54}
\det M_N=\prod_{k=1}^{N-1}\bigl((N-k)!\bigr)^{\epsilon(k)},
\qquad
\epsilon(k)=\begin{cases}2&\text{$k$ odd},\\4&\text{$k$ even}.\end{cases}
\end{equation}
(Its first row, $j\mapsto M_N(0,j)$, is exactly the sequence
$1,1,1,3,5,25,61,427,\dots$ of $f=(1+x)/\cos x$, which is what first
suggested the connection.)

\begin{Theorem}\label{thm:rootlink}
For all $N\ge1$,
\begin{equation}\label{eq:Mtos}
\det M_N=2^{-\binom N2}\,\HH_N\bigl((1+x)/\cos x\bigr)
=\bigl((N-1)!!\bigr)^2\prod_{k=1}^{N-2}(k!!)^6
=\prod_{k=1}^{N-1}\bigl((N-k)!\bigr)^{\epsilon(k)} .
\end{equation}
In particular \cite[Conj.~5.4]{Chapoton2020Han} holds.
\end{Theorem}

The proof has three steps. First, a triangular change of column basis,
powered by the reduction \eqref{eq:rhoprop}, strips the powers of $1+x$:
$\det M_N$ is an explicit power of $2$ times the dilated Hankel
determinant of the sequence interleaving $(A_k)$ and $(B_k)$
(Theorem~\ref{thm:conj54double}). Second --- the key step --- we prove a
moment representation of $\rho$ (Lemma~\ref{lem:rhomom}): on each index,
$\rho$ is the secant functional $\mathcal S$ (even index) or the
functional $\mathcal T^*$ (odd index) of Section~\ref{sec:gen-xcos},
composed with an explicit change of variable; in particular
$A_k=\mathcal S[((1+y)/2)^k]$ and $B_k=\mathcal T^*[((1+y)/2)^k]$. This is
where the secant numbers enter; the identification was observed,
conjecturally, in \cite[\S5]{Chapoton2020Han}. Third, the affine
substitution $y=2v-1$ transports the orthogonal data of
Section~\ref{sec:gen-xcos} (Lemma~\ref{lem:rootaffine}), and the
biorthogonal reduction~$\M2$ runs exactly as in the proof of
Proposition~\ref{prop:sec1}. In passing, the first step also proves
Conjecture~5.3 of \cite{Chapoton2020Han} (Remark~\ref{rem:conj53}).

\subsection{From $M_N$ to a dilated Hankel determinant}
\label{ssec:rootHH}

\begin{Lemma}\label{lem:rhored}
For all $i\ge0$ and $m\ge0$,
\begin{align}
\rho\bigl(x^i(1+x)^{2m}\bigr)
&=(-2)^m\sum_{l=0}^m\binom ml(-2)^l\,\rho(x^{i+l})
=(-2)^m\sum_{l=0}^m\binom ml(-2)^l\,A_{i+l},
\label{eq:Re}\\
\rho\bigl(x^i(1+x)^{2m+1}\bigr)
&=(-2)^m\sum_{l=0}^m\binom ml(-2)^l\,\rho\bigl((1+x)x^{i+l}\bigr)
=(-2)^m\sum_{l=0}^m\binom ml(-2)^l\,B_{i+l}.
\label{eq:Ro}
\end{align}
\end{Lemma}

\begin{proof}
Since $(1+x)^2=(x-1)^2+4x$, the binomial theorem gives
$$
x^i(1+x)^{2m}=\sum_{l=0}^m\binom ml4^l\,x^{i+l}(x-1)^{2m-2l},
$$
all of whose terms have index $2i+2m$, so $\rho$ may be applied term by
term. Iterating \eqref{eq:rhoprop} yields
$\rho\bigl(x^{i+l}(x-1)^{2m-2l}\bigr)=(-2)^{m-l}\rho(x^{i+l})$, so
$$
\rho\bigl(x^i(1+x)^{2m}\bigr)
=\sum_{l=0}^m\binom ml4^l(-2)^{m-l}\rho(x^{i+l})
=(-2)^m\sum_{l=0}^m\binom ml(-2)^l\rho(x^{i+l}),
$$
because $4^l(-2)^{m-l}=(-2)^m(-2)^l$. Multiplying the same expansion by
$(1+x)$ and using
$\rho\bigl((1+x)x^{i+l}(x-1)^{2m-2l}\bigr)=(-2)^{m-l}\rho((1+x)x^{i+l})$
gives \eqref{eq:Ro}.
\end{proof}

\begin{Theorem}\label{thm:conj54double}
Let $a$ be the sequence interleaving $A$ and $B$, namely
$a_{2k}=A_k$ and $a_{2k+1}=B_k$. Then
\begin{equation}\label{eq:MtoH}
\det M_N=2^{\tau_N}\,\HH_N(a),
\qquad
\tau_N=\binom{\bar N}{2}+\binom{\underline N}{2},
\end{equation}
where $\bar N=\lceil N/2\rceil$ and $\underline N=\lfloor N/2\rfloor$, as
throughout. In other words, the Chapoton--Han determinant
\eqref{eq:conj54} is, up to an explicit power of two, the dilated Hankel
determinant of the sequence $(1,1,1,2,2,8,10,64,104,\dots)$.
\end{Theorem}

\begin{proof}
By \eqref{eq:Mdef} and Lemma~\ref{lem:rhored} (with
$\lfloor(2m+1)/2\rfloor=m$), the columns of $M_N$ are
$$
M_N(\cdot,2m)=(-1)^m\sum_{l=0}^m\binom ml(-2)^l\,\alpha_l,
\qquad
M_N(\cdot,2m+1)=(-1)^m\sum_{l=0}^m\binom ml(-2)^l\,\beta_l,
$$
where $\alpha_l=(A_{i+l})_{i}$ and $\beta_l=(B_{i+l})_{i}$. Thus
$M_N$ arises from the matrix $W$ with ordered columns
$\alpha_0,\beta_0,\alpha_1,\beta_1,\dots$ by a triangular column
operation: column $2m$ (resp.\ $2m+1$) involves only the $\alpha_l$
(resp.\ $\beta_l$) with $l\le m$, with coefficient
$(-1)^m\binom mm(-2)^m=2^m$ on $\alpha_m$ (resp.\ $\beta_m$). Hence
$$
\det M_N
=\Bigl(\prod_{0\le 2m\le N-1}2^m\Bigr)
\Bigl(\prod_{0\le 2m+1\le N-1}2^m\Bigr)\det W
=2^{\tau_N}\det W,
$$
since $\sum_{2m<N}m+\sum_{2m+1<N}m
=\binom{\bar N}2+\binom{\underline N}2=\tau_N$. Finally
$W(i,2m)=A_{i+m}=a_{2i+2m}$ and $W(i,2m+1)=B_{i+m}=a_{2i+2m+1}$, that is
$W=\bigl(a_{2i+j}\bigr)$, so $\det W=\HH_N(a)$.
\end{proof}

\begin{Remark}[Conjecture~5.3 of Chapoton--Han]\label{rem:conj53}
Lemma~\ref{lem:rhored} also settles \cite[Conj.~5.3]{Chapoton2020Han},
which asserts that $2^{\lfloor j/2\rfloor}$ divides $\rho(x^i(1+x)^j)$
(so that the entries \eqref{eq:Mdef} are integers; the case of odd $j$
was proved in \cite{Chapoton2020Han}). Indeed $\mathscr N_0$ preserves
integer coefficients --- the division by the monic $(x-1)^2$ is exact ---
so $A_k,B_k\in\ZZ$, and \eqref{eq:Re}--\eqref{eq:Ro} display
$\rho\bigl(x^i(1+x)^{2m}\bigr)$ and $\rho\bigl(x^i(1+x)^{2m+1}\bigr)$ as
$(-2)^m$ times an integer.
\end{Remark}

\subsection{The evaluation $\rho$ as secant moments}\label{ssec:rootmom}

It remains to evaluate $\HH_N(a)$, and for that we must first identify
$A_k$ and $B_k$. The following moment representation is the bridge; it
proves, in particular, the identification of the first row of
$M_N$ with the coefficients of $(1+x)/\cos x$ observed in
\cite[\S5]{Chapoton2020Han}.

\begin{Lemma}[moment representation of $\rho$]\label{lem:rhomom}
Let $\mathcal S[y^j]=E_{2j}$ and $\mathcal T^*[y^j]=(2j+1)E_{2j}$ be the
functionals of Section~\ref{sec:gen-xcos}, and set
$\zeta:=x+x^{-1}-2=(x-1)^2/x$. Every palindromic polynomial of even index
$2e$ has the form $P=x^e\,\pi(\zeta)$, and every one of odd index $2e+1$
the form $P=(1+x)\,x^e\,\sigma(\zeta)$, with $\deg\pi,\deg\sigma\le e$;
and
\begin{equation}\label{eq:rhomom}
\rho\bigl(x^e\,\pi(\zeta)\bigr)
=\mathcal S\Bigl[\Bigl(\tfrac{1+y}{2}\Bigr)^{\!e}
\pi\Bigl(\tfrac{-4}{1+y}\Bigr)\Bigr],
\qquad
\rho\bigl((1+x)\,x^e\,\sigma(\zeta)\bigr)
=\mathcal T^*\Bigl[\Bigl(\tfrac{1+y}{2}\Bigr)^{\!e}
\sigma\Bigl(\tfrac{-4}{1+y}\Bigr)\Bigr],
\end{equation}
the arguments on the right being polynomials in $y$ of degree $\le e$.
In particular ($\pi=\sigma=1$, $e=k$),
\begin{equation}\label{eq:ABmom}
A_k=\mathcal S\Bigl[\Bigl(\tfrac{1+y}2\Bigr)^{\!k}\Bigr]
=2^{-k}\sum_{j=0}^k\binom kj E_{2j},
\qquad
B_k=\mathcal T^*\Bigl[\Bigl(\tfrac{1+y}2\Bigr)^{\!k}\Bigr]
=2^{-k}\sum_{j=0}^k\binom kj(2j+1)E_{2j}.
\end{equation}
\end{Lemma}

\begin{proof}
If $P$ is palindromic of index $2e$, the Laurent polynomial $x^{-e}P$ is
invariant under $x\mapsto1/x$, hence a polynomial in $x+x^{-1}=\zeta+2$:
this gives $P=x^e\pi(\zeta)$ with $\deg\pi\le e$. If the index is odd,
then $P(-1)=-P(-1)$, so $(1+x)\mid P$, and $P/(1+x)$ is palindromic of
even index. At fixed $e$ both sides of \eqref{eq:rhomom} are linear in
$\pi$ (resp.\ $\sigma$).

Introduce
$$
G_e(\zeta):=x^e+x^{-e},
\qquad
H_e(\zeta):=\sum_{j=-e}^{e}(-1)^jx^j ,
$$
both invariant under $x\mapsto1/x$, hence polynomials in $\zeta$ of
degree $e$, with $G_e(0)=2$ and $H_e(0)=(-1)^e$ (value at $x=1$). Using
$(x-1)^2=x\zeta$, $x^{2e}+1=x^e\,G_e(\zeta)$ and
$x^{2e+1}+1=(-1)^e(1+x)\,x^e\,H_e(\zeta)$, the operator $\mathscr N_0$
reads, in these coordinates,
$$
\mathscr N_0\bigl(x^e\pi(\zeta)\bigr)=x^{e-1}\,\tilde\pi(\zeta),
\qquad
\tilde\pi=\frac{G_e\,\pi(0)-2\pi}{\zeta}\,,
$$
$$
\mathscr N_0\bigl((1+x)x^e\sigma(\zeta)\bigr)
=(1+x)\,x^{e-1}\,\tilde\sigma(\zeta),
\qquad
\tilde\sigma=\frac{2\bigl((-1)^e\sigma(0)\,H_e-\sigma\bigr)}{\zeta}
$$
(both numerators vanish at $\zeta=0$, so the divisions are exact).

We prove \eqref{eq:rhomom} by induction on $e$. For $e=0$ both sides
equal the constant $\pi$ (resp.\ $\sigma$, since $\rho(c(1+x))=c$ and
$\mathcal T^*[1]=1$). For $e\ge1$, since $\rho(P)=\rho(\mathscr N_0P)$,
it suffices to show that the right-hand sides are unchanged when
$(e,\pi)$ is replaced by $(e-1,\tilde\pi)$, and $(e,\sigma)$ by
$(e-1,\tilde\sigma)$. Writing $\zeta^*:=-4/(1+y)$, so that
$1/\zeta^*=-(1+y)/4$,
$$
\begin{aligned}
\Bigl(\tfrac{1+y}2\Bigr)^{e-1}\tilde\pi(\zeta^*)
&=\Bigl(\tfrac{1+y}2\Bigr)^{e}
\Bigl(\pi(\zeta^*)-\tfrac{\pi(0)}2\,G_e(\zeta^*)\Bigr),\\
\Bigl(\tfrac{1+y}2\Bigr)^{e-1}\tilde\sigma(\zeta^*)
&=\Bigl(\tfrac{1+y}2\Bigr)^{e}
\bigl(\sigma(\zeta^*)-(-1)^e\sigma(0)\,H_e(\zeta^*)\bigr),
\end{aligned}
$$
so the induction step amounts to the two vanishing statements
\begin{equation}\label{eq:GHvanish}
\mathcal S\Bigl[\Bigl(\tfrac{1+y}2\Bigr)^{\!e}G_e(\zeta^*)\Bigr]=0,
\qquad
\mathcal T^*\Bigl[\Bigl(\tfrac{1+y}2\Bigr)^{\!e}H_e(\zeta^*)\Bigr]=0
\qquad(e\ge1).
\end{equation}
These follow from the explicit expansions
\begin{equation}\label{eq:GHexpand}
\Bigl(\tfrac{1+y}2\Bigr)^{\!e}G_e(\zeta^*)
=(-1)^e2^{1-e}\sum_{j=0}^{e}(-1)^j\binom{2e}{2j}\,y^j,
\qquad
\Bigl(\tfrac{1+y}2\Bigr)^{\!e}H_e(\zeta^*)
=2^{-e}\sum_{j=0}^{e}(-1)^j\binom{2e+1}{2j+1}\,y^j .
\end{equation}
To verify \eqref{eq:GHexpand}, parametrise $y=\tan^2\varphi$; then
$\zeta^*=-4\cos^2\varphi$ corresponds to $x=-e^{-2i\varphi}$ (as
$\zeta=x+x^{-1}-2$), whence $G_e(\zeta^*)=(-1)^e\,2\cos2e\varphi$ and
$H_e(\zeta^*)=\sum_{j=-e}^{e}e^{-2ij\varphi}
=\sin\bigl((2e+1)\varphi\bigr)/\sin\varphi$, while
$1+y=\sec^2\varphi$; since
$(\cos\varphi+i\sin\varphi)^{n}=\cos^{n}\!\varphi\,(1+i\tan\varphi)^{n}$
and the right-hand sides of \eqref{eq:GHexpand} are
$(-1)^e2^{1-e}\operatorname{Re}(1+i\tan\varphi)^{2e}$ and
$2^{-e}\operatorname{Im}\bigl((1+i\tan\varphi)^{2e+1}\bigr)/\tan\varphi$,
both equalities reduce to $\cos2e\varphi=\cos2e\varphi$ and
$\sin(2e+1)\varphi=\sin(2e+1)\varphi$. Two polynomials in $y$ agreeing
for all $y=\tan^2\varphi\ge0$ agree identically. Applying the
functionals to \eqref{eq:GHexpand} now gives
$$
\sum_{j=0}^{e}(-1)^j\binom{2e}{2j}E_{2j}
=(-1)^e\,(2e)!\,[x^{2e}]\bigl(\sec x\cdot\cos x\bigr)=0
\qquad(e\ge1),
$$
since $\sec x\cdot\cos x=1$, and, using
$(2j+1)\binom{2e+1}{2j+1}=(2e+1)\binom{2e}{2j}$,
$$
\sum_{j=0}^{e}(-1)^j\binom{2e+1}{2j+1}(2j+1)E_{2j}
=(2e+1)\sum_{j=0}^{e}(-1)^j\binom{2e}{2j}E_{2j}=0 .
$$
This proves \eqref{eq:GHvanish}, hence \eqref{eq:rhomom}; and
\eqref{eq:ABmom} follows by the binomial theorem.
\end{proof}

By \eqref{eq:ABmom} and linearity, the moment functionals
$\mathcal A[v^k]:=A_k$ and $\mathcal B[v^k]:=B_k$ satisfy
$\mathcal A[F]=\mathcal S\bigl[F\bigl(\tfrac{1+y}2\bigr)\bigr]$ and
$\mathcal B[F]=\mathcal T^*\bigl[F\bigl(\tfrac{1+y}2\bigr)\bigr]$ for
every polynomial $F$: the even and odd parts of $a$ are the secant pair
of Section~\ref{sec:gen-xcos}, transported by the affine substitution
$y=2v-1$. All the orthogonal data transport along.

\begin{Lemma}\label{lem:rootaffine}
The functionals $\mathcal A$ and $\mathcal B$ are quasi-definite, with
monic orthogonal polynomials
$$
p_i(v)=2^{-i}\,\hat P_i(2v-1),
\qquad
q_m(v)=2^{-m}\,\varrho_m(2v-1),
$$
where $\hat P_i$ and $\varrho_m$ are the orthogonal polynomials of
$\mathcal S$ and $\mathcal T^*$ (Lemmas~\ref{lem:cf} and~\ref{lem:cdh};
we write $\varrho$ to keep the latter apart from the evaluation $\rho$).
Their recurrence data and norms are
$$
c^{\mathcal A}_i=4i^2+2i+1,\quad
\lambda^{\mathcal A}_i=\bigl(i(2i-1)\bigr)^2,\quad
h^{\mathcal A}_i=\mathcal A[p_i^2]
=(i!)^2\bigl((2i-1)!!\bigr)^2=4^{-i}\bigl((2i)!\bigr)^2;
$$
$$
c^{\mathcal B}_m=4m^2+4m+2,\quad
\lambda^{\mathcal B}_m=4m^4,\quad
h^{\mathcal B}_m=\mathcal B[q_m^2]
=4^m(m!)^4=4^{-m}\bigl((2m)!!\bigr)^4,
$$
and the connection coefficients are
\begin{equation}\label{eq:rootconn}
\mathcal B\bigl[p_i\,q_m\bigr]=2^{i+m}(i!)^2(m!)^2\binom{1/2}{i-m},
\qquad\text{i.e.}\qquad
p_i=\sum_{m=0}^i 2^{i-m}\Bigl(\tfrac{i!}{m!}\Bigr)^2
\binom{1/2}{i-m}\,q_m .
\end{equation}
\end{Lemma}

\begin{proof}
The polynomials $p_i$, $q_m$ are monic of degrees $i$, $m$ in $v$, and
orthogonality and norms transport directly:
$\mathcal A[p_ip_l]=\mathcal S\bigl[2^{-i-l}\hat P_i\hat P_l\bigr]
=\delta_{il}\,4^{-i}\bigl((2i)!\bigr)^2$ by Lemma~\ref{lem:cf}, and
$\mathcal B[q_mq_{m'}]
=\delta_{mm'}\,4^{-m}\bigl((2m)!!\bigr)^4$ by Lemma~\ref{lem:cdh}. The
affine substitution transforms the three-term recurrence data by
$c\mapsto(c+1)/2$ and $\lambda\mapsto\lambda/4$; applied to
$c^{\mathcal S}_i=(2i)^2+(2i+1)^2$,
$\lambda^{\mathcal S}_i=\bigl((2i-1)(2i)\bigr)^2$ and
$c^*_m=8m^2+8m+3$, $\lambda^*_m=(2m)^4$, this yields the displayed
values. Finally
$\mathcal B[p_iq_m]=2^{-i-m}\,\mathcal T^*[\hat P_i\varrho_m]$, so
\eqref{eq:rootconn} is the image of the connection formula
$\mathcal T^*[\hat P_i\varrho_m]
=(i!)^2(m!)^2\,2^{2i+2m}\binom{1/2}{i-m}$ of Lemma~\ref{lem:conn3}.
\end{proof}

\begin{proof}[Proof of Theorem~\ref{thm:rootlink}]
By Theorem~\ref{thm:conj54double} it suffices to evaluate $\HH_N(a)$.
Both parity functionals of $a$ are classical (Lemma~\ref{lem:rootaffine}),
so, exactly as in the proof of Proposition~\ref{prop:sec1},
Lemma~\ref{lem:bindetGene} gives
$$
\HH_N(a)=(-1)^{\binom{\bar N}2}
\prod_{l=0}^{\bar N-1}h^{\mathcal A}_l\cdot
\prod_{m=0}^{\underline N-1}h^{\mathcal B}_m\cdot
\det\bigl(\kappa_{\bar N+r,\,m}\bigr)_{0\le r,m\le\underline N-1},
\qquad
\kappa_{i,m}=2^{\,i-m}\Bigl(\tfrac{i!}{m!}\Bigr)^2\binom{1/2}{i-m},
$$
with the connection coefficients of \eqref{eq:rootconn} --- those of that
proof with the weight $4^{\,i-m}$ replaced by $2^{\,i-m}$. Pulling
$(i!)^2\,2^i$ out of row $r$ (where $i=\bar N+r$) and $(m!)^{-2}\,2^{-m}$
out of column $m$ leaves the same binomial determinant, worth
$\Omega(1)=2^{-\binom N2}$ together with the sign
(Lemma~\ref{lem:Pahalf}), so
$$
\HH_N(a)=2^{-\binom N2}\,
\prod_{l=0}^{\bar N-1}4^{-l}\bigl((2l)!\bigr)^2\cdot
\prod_{i=\bar N}^{N-1}(i!)^2\,2^i\cdot
\prod_{m=0}^{\underline N-1}(m!)^2\,2^m .
$$
Comparing factor by factor with the corresponding display in the proof of
Proposition~\ref{prop:sec1}, which carries $4^i$ and $4^m$ in place of
$2^i$ and $2^m$ and no $4^{-l}$,
$$
\frac{\HH_N(a)}{\HH_N\bigl((1+x)/\cos x\bigr)}
=\prod_{l=0}^{\bar N-1}4^{-l}\cdot
\prod_{i=\bar N}^{N-1}2^{-i}\cdot
\prod_{m=0}^{\underline N-1}2^{-m}
=2^{-\tau_N-\binom N2},
$$
since $2\binom{\bar N}2+\bigl(\binom N2-\binom{\bar N}2\bigr)
+\binom{\underline N}2=\tau_N+\binom N2$, using
$\sum_{i=\bar N}^{N-1}i=\binom N2-\binom{\bar N}2$. Multiplying by
$2^{\tau_N}$ and using Proposition~\ref{prop:sec1},
$$
\det M_N=2^{\tau_N}\,\HH_N(a)
=2^{-\binom N2}\,\HH_N\bigl((1+x)/\cos x\bigr)
=\bigl((N-1)!!\bigr)^2\prod_{k=1}^{N-2}(k!!)^6 .
$$
It remains to identify this with the product of \eqref{eq:conj54}. Both
equal $1$ for $N=1$, and both are multiplied by
$(N!!)^2\bigl((N-1)!!\bigr)^4$ when $N\to N+1$: for the double-factorial
product this is immediate, and for
$\prod_{k=1}^{N-1}\bigl((N-k)!\bigr)^{\epsilon(k)}
=\prod_{j=1}^{N-1}(j!)^{\epsilon(N-j)}$ the shift $N\to N+1$ flips every
exponent (as $\epsilon(k)+\epsilon(k+1)=6$), so the ratio is
$$
(N!)^{2}\,
\prod_{\substack{1\le j\le N-1\\ N-j\ \mathrm{odd}}}(j!)^{2}
\Bigm/
\prod_{\substack{1\le j\le N-1\\ N-j\ \mathrm{even}}}(j!)^{2}
=(N!)^2\bigl((N-1)!!\bigr)^2
=(N!!)^2\bigl((N-1)!!\bigr)^4,
$$
the middle equality because the quotient telescopes over consecutive
factorials:
$\frac{(N-1)!}{(N-2)!}\cdot\frac{(N-3)!}{(N-4)!}\cdots
=(N-1)(N-3)\cdots=(N-1)!!$.
\end{proof}

Thus the root conjecture of Chapoton--Han and the Euler dilated
Hankel determinant are one and the same evaluation: the change of column
basis of Theorem~\ref{thm:conj54double} strips the powers of $1+x$, and
the substitution $y=2v-1$ of Lemma~\ref{lem:rootaffine} carries what
remains onto Proposition~\ref{prop:sec1}.

\section{Corollaries: explicit evaluations}\label{sec:coro}

This section collects the explicit members of the families evaluated in this
paper. No new proofs are needed: each corollary is a substitution of
parameters into the theorem quoted at the head of its subsection, recording
the generating function $f$, the determinant $\HH_n(f)$ and, where a closed
form is available, the single and double shifts $\HH_n^{(1)}=\HH_n(f')$ and
$\HH_n^{(2)}=\HH_n(f'')$.

\subsection{The Beta family $a_m=\rho^m(\alpha)_m/(\beta)_m$}

Every member of the Beta family \eqref{eq:betadef} is a plain sequence
$a_m=\rho^m(\alpha)_m/(\beta)_m$ (we normalise $a_0=1$), with term ratio
$a_{m+1}/a_m=\rho\,(m+\alpha)/(m+\beta)$; when the denominator parameter
$\beta$ is absent the sequence is $\rho^m(\alpha)_m$. By
Corollary~\ref{cor:betadouble} (resp.\ Proposition~\ref{prop:degenerate} when
$\beta$ is absent) the dilated Hankel determinant $\HH_n=\det(a_{2i+j})$ is a
closed product of factorials. The shifted sequence $a_{m+1}$ is again a Beta
member, with $(\rho,\alpha,\beta)\mapsto(\rho,\alpha+1,\beta+1)$, so the
single shift $\HH_n^{(1)}:=\det(a_{2i+j+1})$ is likewise closed. We record
several well-known members with their $(\rho,\alpha,\beta)$, giving both
$\HH_n$ and $\HH_n^{(1)}$.

\begin{Corollary}[Catalan numbers]\label{cor:bcatalan}
$C_m=\frac1{m+1}\binom{2m}m=4^m(\tfrac12)_m/(2)_m$, so
$(\rho,\alpha,\beta)=(4,\tfrac12,2)$;
the shift $C_{m+1}$ is the member $(4,\tfrac32,3)$, and
$$
\begin{aligned}
\HH_n(C)&=2^{\binom n2}\prod_{i=0}^{n-1}\binom{4i}{2i}\frac{i!\,(n+i)!}{(n+2i)!}=1,3,32,1232,\dots,\\
\HH_n^{(1)}&=2^{\binom n2}\prod_{i=0}^{n-1}\frac{i!\,(4i+2)!\,(n+i)!}{(2i)!\,(2i+1)!\,(n+2i+1)!}.
\end{aligned}
$$
\end{Corollary}

\begin{Corollary}[central binomial coefficients]\label{cor:bcbinom}
$\binom{2m}m=4^m(\tfrac12)_m/(1)_m$, so $(\rho,\alpha,\beta)=(4,\tfrac12,1)$;
the shift $\binom{2m+2}{m+1}$ is the member $(4,\tfrac32,2)$, and
$$
\begin{aligned}
\HH_n&=2^{\binom n2+n-1}\prod_{i=0}^{n-1}\binom{4i}{2i}\frac{i!\,(n-1+i)!}{(n-1+2i)!}=1,8,224,\dots,\\
\HH_n^{(1)}&=2^{\binom n2+n-1}\prod_{i=0}^{n-1}\frac{i!\,(4i+2)!\,(n-1+i)!}{(2i)!\,(2i+1)!\,(n+2i)!}.
\end{aligned}
$$
\end{Corollary}

\begin{Corollary}[$\binom{2m+1}m$]\label{cor:b2m1}
$\binom{2m+1}m=4^m(\tfrac32)_m/(2)_m$, so $(\rho,\alpha,\beta)=(4,\tfrac32,2)$;
the shift $\binom{2m+3}{m+1}$ is the member $(4,\tfrac52,3)$, and
$$
\HH_n=2^{\binom n2}\prod_{i=0}^{n-1}\frac{(4i+1)!}{(2i+n)!}=1,5,96,6864,\dots,\qquad
\HH_n^{(1)}=2^{\binom n2}\prod_{i=0}^{n-1}\frac{(4i+3)!}{(2i+1)\,(2i+n+1)!}.
$$
\end{Corollary}

\begin{Corollary}[factorials]\label{cor:bfact}
$m!=(1)_m$, so $(\rho,\alpha,\beta)=(1,1,\text{--})$ ($\beta$ absent); the
shift is $(m+1)!=(2)_m$, and
$$
\HH_n(m!)=2^{\binom n2}\prod_{i=0}^{n-1}i!\,(2i)!,\qquad
\HH_n^{(1)}=2^{\binom n2}\prod_{i=0}^{n-1}i!\,(2i+1)! .
$$
\end{Corollary}

\begin{Corollary}[double factorials]\label{cor:bddf}
$(2m-1)!!=2^m(\tfrac12)_m$ and $(2m+1)!!=2^m(\tfrac32)_m$, so
$(\rho,\alpha,\beta)=(2,\tfrac12,\text{--})$
and $(2,\tfrac32,\text{--})$; here the shift sends
$(2m-1)!!\mapsto(2m+1)!!\mapsto(2m+3)!!$, so
$$
\begin{aligned}
\HH_n\bigl((2m-1)!!\bigr)&=4^{\binom n2}\prod_{i=0}^{n-1}i!\,(4i-1)!!,\\
\HH_n\bigl((2m+1)!!\bigr)&=4^{\binom n2}\prod_{i=0}^{n-1}i!\,(4i+1)!!,\\
\HH_n^{(1)}\bigl((2m+1)!!\bigr)&=4^{\binom n2}\prod_{i=0}^{n-1}i!\,(4i+3)!! .
\end{aligned}
$$
\end{Corollary}

\begin{Corollary}[$(2m)!/m!$]\label{cor:bqf}
$\dfrac{(2m)!}{m!}=4^m(\tfrac12)_m$, so
$(\rho,\alpha,\beta)=(4,\tfrac12,\text{--})$ ($\beta$ absent);
the shift is $\dfrac{(2m+2)!}{(m+1)!}=2\cdot4^m(\tfrac32)_m$, and
$$
\HH_n=8^{\binom n2}\prod_{i=0}^{n-1}\frac{i!\,(4i)!}{(2i)!},\qquad
\HH_n^{(1)}=8^{\binom n2}\prod_{i=0}^{n-1}\frac{i!\,(4i+2)!}{(2i+1)!} .
$$
\end{Corollary}

\subsection{The Gaussian family}

By Theorem~\ref{thm:gauss} the Gaussian generating function $f(x)=e^{cx+x^2/2}$
has a closed product $\HH_n(f)$, and by Theorem~\ref{thm:dergauss}
its derivative $f'=(c+x)\,e^{cx+x^2/2}$, whose dilated Hankel determinant is the
single shift $\HH_n(f')=\det(a_{2i+j+1})$, satisfies $\HH_n(f')=c^n\,\HH_n(f)$. We
record the specialisations $c\in\{1,-1\}$, in each case giving $f$ and its
derivative $f'$.

\begin{Corollary}[$c=1$]\label{cor:g1}
The moments $a_n$ are the involution numbers of $\{1,\dots,n\}$, and
$$
f=e^{x+x^2/2},\qquad \HH_n(f)=2^{\binom n2}\prod_{k=1}^{n-1}k! .
$$
For the derivative,
$$
f'=(1+x)\,e^{x+x^2/2},\qquad \HH_n(f')=\HH_n(f)=2^{\binom n2}\prod_{k=1}^{n-1}k! ,
$$
so the determinant is invariant under differentiation.
\end{Corollary}

\begin{Corollary}[$c=-1$]\label{cor:gm1}
The moments are the signed involution numbers $a_n=(-1)^n I_n$ ($I_n$ the
involution count), and
$$
f=e^{-x+x^2/2},\qquad \HH_n(f)=(-2)^{\binom n2}\prod_{k=1}^{n-1}k! .
$$
For the derivative,
$$
f'=(x-1)\,e^{-x+x^2/2},\qquad \HH_n(f')=(-1)^n(-2)^{\binom n2}\prod_{k=1}^{n-1}k! .
$$
\end{Corollary}

\subsection{The Euler number family}
The cases $s<0$, $t<0$, or $t-s$ even are trivial; assume henceforth
$s\ge0$, $t\ge0$ and $\delta=t-s$ odd.
Throughout $\bar n=\lceil n/2\rceil$, $\underline n=\lfloor n/2\rfloor$. With these,
Proposition~\ref{prop:family} evaluates $\HH_n$ in closed form for every
$(s,t)$; its signed factor $\Omega(\delta)$ ($\delta=t-s$) is evaluated for
$\delta\in\{\pm1,\pm3\}$ by Lemma~\ref{lem:Pahalf}.

Each corollary below records both $f$ and, for the \emph{double shift}, its
second derivative $f''$: the doubly shifted sequence $(F_{k+2})_{k\ge0}$ has
exponential generating function $f''$, and by Proposition~\ref{prop:dshift} the
determinant $\HH_n^{(2)}(f)=\HH_n(f'')$ is the unshifted one times the explicit
positive scalar $(2n-1)!!\,\sigma_{\bar n}(s)\,\sigma_{\underline n}(t)$, where
$\sigma_K(c):=\prod_{k=0}^{K-1}(c+2k+1)$; here
$$
\sigma_K(0)=(2K-1)!!,\quad \sigma_K(1)=2^K K!,\quad
\sigma_K(2)=(2K+1)!!,\quad \sigma_K(3)=2^K (K+1)!\,,
$$
and explicitly
$f''=(s{+}1)\dfrac{(s{+}1)\sin^2x+1}{\cos^{s+3}x}+(t{+}1)\dfrac{\sin x}{\cos^{t+2}x}$.
All the closed forms below were verified for $n\le12$.

Four of the cases ($s\in\{0,2\}$, $t\in\{1,3\}$) additionally carry a
\emph{single shift}: the shifted sequence $(F_{k+1})_{k\ge0}$ has exponential
generating function $f'$, and $\HH_n^{(1)}=\HH_n(f')$. On the line $t=1$ this is a
scalar multiple of $\HH_n(f)$ (Proposition~\ref{prop:shift}); on $t=3$ it carries
the extra linear factor $\Gamma_n$ of Theorem~\ref{thm:t3}. The single
shift is recorded in the four corresponding corollaries below.

We list the eight cases with $s+t$ odd and $s,t\in\{0,1,2,3\}$, beginning with the
secant/tangent (Euler) case $(s,t)=(0,1)$.

\begin{Corollary}[$s=0,\ t=1$]\label{cor:euler}
For $u_j=j^2$, $v_j=j(j+1)$ the moments are the Euler numbers $a_n=E_n$ of
\eqref{eq:eulerdef} (the alternating permutations of $\{1,\dots,n\}$), and
$$
f=\frac{1+\sin x}{\cos x}=\sec x+\tan x,\qquad
\HH_n(f)=2^{-\binom n2}\prod_{k=1}^{n-1}k!\,(2k)!
        =\prod_{k=1}^{n-1}(k!)^2\,(2k-1)!! .
$$
For the single shift,
$$
f'=\frac{1+\sin x}{\cos^2 x},\qquad
\HH_n^{(1)}(f)=\HH_n(f')=(2n-1)!!\,\HH_n(f)=(2n-1)!!\prod_{k=1}^{n-1}(k!)^2(2k-1)!! .
$$
For the double shift,
$$
f''=\frac{(1+\sin x)^2}{\cos^3 x},\qquad
\HH_n^{(2)}(f)=\HH_n(f'')=(2n-1)!!\,(2\bar n-1)!!\,2^{\underline n}\,\underline n!\;\HH_n(f).
$$
\end{Corollary}

\begin{Corollary}[$s=1,\ t=2$]\label{cor:12}
For $u_j=j(j+1)$, $v_j=j(j+2)$,
$$
	f=\frac{\sin x+2}{2\cos^2x} + \frac 12 \log(\sec x + \tan x),
$$
$$
	\HH_n(f)=2^{-\binom n2-\underline n}\prod_{k=1}^{n-1}(2k)!\ \prod_{j=1}^{n} j!\ >0.
$$
For the double shift,
$$
f''=\frac{4\sin^2x+3\sin x+2}{\cos^4 x},\qquad
\HH_n^{(2)}(f)=\HH_n(f'')=(2n-1)!!\,2^{\bar n}\bar n!\,(2\underline n+1)!!\;\HH_n(f).
$$
\end{Corollary}

\begin{Corollary}[$s=2,\ t=3$]\label{cor:23}
For $u_j=j(j+2)$, $v_j=j(j+3)$,
$$
f=\frac{3+\sin x\,(2\cos^2x+1)}{3\cos^3x},
$$
$$
\HH_n(f)=2^{-\binom{n+1}{2}}\,3^{-\underline n}
       \prod_{k=1}^{n-1}(2k)!\ \prod_{j=1}^{n+1} j!\ >0.
$$
For the single shift,
$$
f'=\frac{1+3\sin x}{\cos^4 x},\qquad
\HH_n^{(1)}(f)=\HH_n(f')=\Lambda_n(2)\,\bigl(n(n+1)+1\bigr),
$$
the carrier of Theorem~\ref{thm:t3} at $s=2$ being $n(n+1)+1$, with
parity factor $\rho_n=\tfrac1{2\bar n+1}$ (here the same for both parities)
and smooth part
\begin{equation}\label{eq:sHH23}
\Lambda_n(2)=\frac{2^{-\binom n2}}{2\bar n+1}\,6^{-\bar n}\prod_{l=0}^{\bar n-1}(2l)!\,(2l+3)!\;
2^{-\underline n}\prod_{c=0}^{\underline n-1}(2c+3)(2c+2)!\;\prod_{r=0}^{\underline n-1}(2\bar n+2r+1)! .
\end{equation}
For the double shift,
$$
f''=\frac{9\sin^2x+4\sin x+3}{\cos^5 x},\qquad
\HH_n^{(2)}(f)=\HH_n(f'')=(2n-1)!!\,(2\bar n+1)!!\,2^{\underline n}(\underline n+1)!\;\HH_n(f).
$$
\end{Corollary}

\begin{Corollary}[$s=1,\ t=0$]\label{cor:10}
For $u_j=j(j+1)$, $v_j=j^2$,
$$
f=\frac{1}{\cos^2x}+\log(\sec x+\tan x),
$$
$$
\HH_n(f)=\Omega(-1)
      \prod_{i=0}^{n-1}(2i)!\ \prod_{l=0}^{\bar n-1}(2l+1)!\ \prod_{m=0}^{\underline n-1}(2m)!\,,
$$
with $\Omega(-1)$ as in Lemma~\ref{lem:Pahalf}. For the double shift,
$$
f''=\frac{-\sin^3x+4\sin^2x+\sin x+2}{\cos^4 x},\qquad
\HH_n^{(2)}(f)=\HH_n(f'')=(2n-1)!!\,2^{\bar n}\bar n!\,(2\underline n-1)!!\;\HH_n(f).
$$
\end{Corollary}

\begin{Corollary}[$s=2,\ t=1$]\label{cor:21}
For $u_j=j(j+2)$, $v_j=j(j+1)$,
$$
f=\frac{1+\sin x\cos^2x}{\cos^3x},
$$
$$
\HH_n(f)=2^{-\bar n}\,\Omega(-1)
      \prod_{i=0}^{n-1}(2i)!\ \prod_{l=0}^{\bar n-1}(2l+2)!\ \prod_{m=0}^{\underline n-1}(2m+1)!\,.
$$
For the single shift, $f'=\dfrac{\cos^2 x+3\sin x}{\cos^4 x}$, and specialising
Proposition~\ref{prop:shift} at $s=2$ (no carrier factor arises on the line
$t=1$),
\begin{equation}\label{eq:sHH21}
\HH_n^{(1)}(f)=\HH_n(f')=(-1)^{\bar n-1}\,\frac{\underline n!\,\bigl((2\underline n+1)!!\bigr)^2}{(2\bar n-3)!!}\;
(2n-1)!!\prod_{k=1}^{n-1}(k!)^2(2k-1)!! ,
\end{equation}
with the convention $(-1)!!=1$; equivalently the secant/tangent single shift of
Corollary~\ref{cor:euler} times $(-1)^{\bar n-1}\underline n!\bigl((2\underline n+1)!!\bigr)^2/(2\bar n-3)!!$.
For the double shift,
$$
f''=\frac{-2\sin^3x+9\sin^2x+2\sin x+3}{\cos^5 x},\qquad
\HH_n^{(2)}(f)=\HH_n(f'')=(2n-1)!!\,(2\bar n+1)!!\,2^{\underline n}\underline n!\;\HH_n(f).
$$
\end{Corollary}

\begin{Corollary}[$s=3,\ t=2$]\label{cor:32}
For $u_j=j(j+3)$, $v_j=j(j+2)$,
$$
f=\frac{2+\sin x\cos^2x}{2\cos^4x}+\tfrac12\log(\sec x+\tan x),
$$
$$
\HH_n(f)=2^{-\underline n}\,6^{-\bar n}\,\Omega(-1)
      \prod_{i=0}^{n-1}(2i)!\ \prod_{l=0}^{\bar n-1}(2l+3)!\ \prod_{m=0}^{\underline n-1}(2m+2)!\,.
$$
For the double shift,
$$
\begin{aligned}
f''&=\frac{-3\sin^3x+16\sin^2x+3\sin x+4}{\cos^6 x},\\
\HH_n^{(2)}(f)&=\HH_n(f'')=(2n-1)!!\,2^{\bar n}(\bar n+1)!\,(2\underline n+1)!!\;\HH_n(f).
\end{aligned}
$$
\end{Corollary}

\begin{Corollary}[$s=0,\ t=3$]\label{cor:03}
For $u_j=j^2$, $v_j=j(j+3)$,
$$
f=\frac{3\cos^2x+\sin x\,(2\cos^2x+1)}{3\cos^3x},
$$
$$
\HH_n(f)=6^{-\underline n}\,\Omega(3)
      \prod_{i=0}^{n-1}(2i)!\ \prod_{l=0}^{\bar n-1}(2l)!\ \prod_{m=0}^{\underline n-1}(2m+3)!\,,
$$
with $\Omega(3)$ as in Lemma~\ref{lem:Pahalf}. For the single shift,
$$
f'=\frac{1+\sin x\cos^2 x}{\cos^4 x},\qquad
\HH_n^{(1)}(f)=\HH_n(f')=\Lambda_n(0)\,\Gamma_n(0),\quad \Gamma_n(0)=n(n+1)-1-4\bar n,
$$
the carrier of Theorem~\ref{thm:t3} at $s=0$, with smooth part
\begin{equation}\label{eq:sHH03}
\Lambda_n(0)=(-1)^{\underline n+1}\rho_n\,2^{-\binom n2}\,6^{-\bar n}
\prod_{l=0}^{\bar n-1}(2l)!\,(2l+3)!\;\prod_{c=0}^{\underline n-1}(2c+1)!\;\prod_{r=0}^{\underline n-1}(2\bar n+2r+1)! ,
\end{equation}
with parity factor $\rho_n=\frac{1}{n+1}$ for $n$ even and
$\rho_n=\frac{n}{n+2}$ for $n$ odd.
For the double shift,
$$
f''=\frac{-\sin^4x+4\sin x+1}{\cos^5 x},\qquad
\HH_n^{(2)}(f)=\HH_n(f'')=(2n-1)!!\,(2\bar n-1)!!\,2^{\underline n}(\underline n+1)!\;\HH_n(f).
$$
\end{Corollary}

\begin{Corollary}[$s=3,\ t=0$]\label{cor:30}
For $u_j=j(j+3)$, $v_j=j^2$,
$$
f=\frac{1}{\cos^4x}+\log(\sec x+\tan x),
$$
$$
\HH_n(f)=\Omega(-3)\,6^{-\bar n}
      \prod_{i=0}^{n-1}(2i)!\ \prod_{l=0}^{\bar n-1}(2l+3)!\ \prod_{m=0}^{\underline n-1}(2m)!\,,
$$
with $\Omega(-3)$ as in Lemma~\ref{lem:Pahalf}. For the double shift,
$$
\begin{aligned}
f''&=\frac{\sin^5x-2\sin^3x+16\sin^2x+\sin x+4}{\cos^6 x},\\
\HH_n^{(2)}(f)&=\HH_n(f'')=(2n-1)!!\,2^{\bar n}(\bar n+1)!\,(2\underline n-1)!!\;\HH_n(f).
\end{aligned}
$$
\end{Corollary}

\subsection{The secant-number family $(1+x)/\cos^{s+1}x$}

By Theorem~\ref{thm:allstar} the family $f(x)=(1+x)/\cos^{s+1}x$ has
$\HH_n(f)=c_n\prod_{i=1}^{n-1}(s+1)_i$, with the constant $c_n$ (independent
of $s$) given there.
The single shift is $\bigl((n-1)!!\bigr)^2$ times the unshifted determinant for every
$s$ (Theorem~\ref{thm:shift}), while the double shift admits such a closed
multiple only at $s=1$ (Theorem~\ref{thm:dblshift}). We record the cases
$s=0,1,2,3$.

\begin{Corollary}[$s=0$]\label{cor:xc0}
For $f=\dfrac{1+x}{\cos x}$,
$$
\HH_n(f)=c_n\prod_{i=1}^{n-1}i!=2^{\binom n2}\bigl((n-1)!!\bigr)^2\prod_{k=1}^{n-2}(k!!)^6 .
$$
For the single shift, with $f'=\dfrac{\cos x+(1+x)\sin x}{\cos^2x}$,
$$
\HH_n^{(1)}(f)=\HH_n(f')=\bigl((n-1)!!\bigr)^2\,\HH_n(f).
$$
\end{Corollary}

\begin{Corollary}[$s=1$]\label{cor:xc1}
For $f=\dfrac{1+x}{\cos^2x}$,
$$
\HH_n(f)=c_n\prod_{i=1}^{n-1}(i+1)! .
$$
For the single shift, with $f'=\dfrac{\cos x+2(1+x)\sin x}{\cos^3x}$,
$$
\HH_n^{(1)}(f)=\HH_n(f')=\bigl((n-1)!!\bigr)^2\,\HH_n(f);
$$
and for the double shift (Theorem~\ref{thm:dblshift}), with
$f''=\dfrac{2(1+x)(1+2\sin^2x)+4\sin x\cos x}{\cos^4x}$,
$$
\HH_n^{(2)}(f)=\HH_n(f'')=2^n\,(n!)^2\,\HH_n(f).
$$
\end{Corollary}

\begin{Corollary}[$s=2$]\label{cor:xc2}
For $f=\dfrac{1+x}{\cos^3x}$,
$$
\HH_n(f)=c_n\prod_{i=1}^{n-1}\frac{(i+2)!}{2} .
$$
For the single shift, with $f'=\dfrac{\cos x+3(1+x)\sin x}{\cos^4x}$,
$$
\HH_n^{(1)}(f)=\HH_n(f')=\bigl((n-1)!!\bigr)^2\,\HH_n(f).
$$
\end{Corollary}

\begin{Corollary}[$s=3$]\label{cor:xc3}
For $f=\dfrac{1+x}{\cos^4x}$,
$$
\HH_n(f)=c_n\prod_{i=1}^{n-1}\frac{(i+3)!}{6} .
$$
For the single shift, with $f'=\dfrac{\cos x+4(1+x)\sin x}{\cos^5x}$,
$$
\HH_n^{(1)}(f)=\HH_n(f')=\bigl((n-1)!!\bigr)^2\,\HH_n(f).
$$
\end{Corollary}

\subsection{A rank-one perturbation of the Euler number family}

By Proposition~\ref{prop:sin3} the one-parameter family
$$
f_s(x)=\frac{\sin x+1}{\cos^2x}+s\,\sin x
$$
is a rank-one perturbation of the secant/tangent single shift
(Proposition~\ref{prop:shift}, generating function $f_0$), with
$\HH_n(f_s)=\bigl(1-s\tbinom n2\bigr)\,\HH_n^{\mathrm E}$, where
$\HH_n^{\mathrm E}:=\prod_{k=1}^{n-1}(k!)^2\,(2k+1)!!$.
The determinant is affine in $s$; it is the unperturbed value
$\HH_n^{\mathrm E}>0$ at $s=0$, and at $s=-1$ the perturbation collapses to
$f_{-1}=(\sin^3x+1)/\cos^2x$. We record $s=-2,-1,0,1,2$, writing each $f_s$ as a
single fraction $(\text{polynomial in }\sin x)/\cos^2x$.

\begin{Corollary}[$s=-2$]\label{cor:sinm2}
For $f_{-2}=\dfrac{2\sin^3x-\sin x+1}{\cos^2x}$,
$$
\HH_n(f_{-2})=\bigl(n^2-n+1\bigr)\prod_{k=1}^{n-1}(k!)^2\,(2k+1)!!\ >0 .
$$
\end{Corollary}

\begin{Corollary}[$s=-1$]\label{cor:sin3}
For $f_{-1}=\dfrac{\sin^3 x+1}{\cos^2 x}$,
$$
\HH_n(f_{-1})=\Bigl(\tbinom n2+1\Bigr)\prod_{k=1}^{n-1}(k!)^2\,(2k+1)!!\ >0 .
$$
\end{Corollary}

\begin{Corollary}[$s=0$]\label{cor:sin0}
For $f_{0}=\dfrac{\sin x+1}{\cos^2 x}$, the secant/tangent single shift of
Corollary~\ref{cor:euler} (Proposition~\ref{prop:shift}),
$$
\HH_n(f_{0})=\HH_n^{\mathrm E}=\prod_{k=1}^{n-1}(k!)^2\,(2k+1)!! .
$$
\end{Corollary}

\begin{Corollary}[$s=1$]\label{cor:sin1}
For $f_{1}=\dfrac{-\sin^3x+2\sin x+1}{\cos^2 x}$,
$$
\HH_n(f_{1})=\Bigl(1-\tbinom n2\Bigr)\prod_{k=1}^{n-1}(k!)^2\,(2k+1)!! .
$$
Here $\HH_1=1$, $\HH_2=0$, and $\HH_n<0$ for $n\ge3$: the perturbation
exactly cancels the determinant at $n=2$.
\end{Corollary}

\begin{Corollary}[$s=2$]\label{cor:sin2p}
For $f_{2}=\dfrac{-2\sin^3x+3\sin x+1}{\cos^2 x}$,
$$
\HH_n(f_{2})=\bigl(1-n(n-1)\bigr)\prod_{k=1}^{n-1}(k!)^2\,(2k+1)!! ,
$$
of sign $(-1)$ for all $n\ge2$.
\end{Corollary}

\subsection{The Springer number family $1/(\cos x-t\sin x)^r$}

By Theorem~\ref{thm:springer} the determinant is, for every integer $r\ge1$,
the power $\bigl(t(t^2+1)\bigr)^{\binom n2}$ times the Beta value
$\HH_n\bigl((1-x)^{-r}\bigr)$ for $a_n=(r)_n$ (\S\ref{sec:beta}). The factor
$t(t^2+1)$ shows $t=0$ is degenerate, so we take $t=1$, the Springer line.

\begin{Corollary}[$t=1,\ r=1$; Springer numbers]\label{cor:spr1}
For $f=\dfrac{1}{\cos x-\sin x}$, whose moments are the Springer numbers,
$$
\HH_n(f)=4^{\binom n2}\prod_{k=1}^{n-1}k!\,(2k)! .
$$
\end{Corollary}

\begin{Corollary}[$t=1,\ r=2$]\label{cor:spr2}
For $f=\dfrac{1}{(\cos x-\sin x)^2}$,
$$
\HH_n(f)=4^{\binom n2}\prod_{k=1}^{n-1}k!\,(2k+1)! .
$$
\end{Corollary}

\subsection{A derivative of the Springer number family $(\cos x+\sin x)/(\cos x-\sin x)^s$}

By Theorem~\ref{thm:deriv} the determinant factors completely into linear
forms in $s$ for every $s$; a factor vanishes when $s$ is a non-positive
integer (degenerate), so we take $s=1,2,3$.

\begin{Corollary}[$s=1$]\label{cor:der1}
$$
\HH_n\Bigl(\frac{\cos x+\sin x}{\cos x-\sin x}\Bigr)
=4^{\binom n2}\Bigl(\prod_{k=1}^{n-1}k!\Bigr)\prod_{j=0}^{n-2}\bigl[(2j+1)(2j+2)\bigr]^{\,n-1-j} .
$$
\end{Corollary}

\begin{Corollary}[$s=2$]\label{cor:der2}
$$
\HH_n\Bigl(\frac{\cos x+\sin x}{(\cos x-\sin x)^2}\Bigr)
=4^{\binom n2}\Bigl(\prod_{k=1}^{n-1}k!\Bigr)\prod_{j=0}^{n-2}\bigl[(2j+2)(2j+3)\bigr]^{\,n-1-j} .
$$
\end{Corollary}

\begin{Corollary}[$s=3$]\label{cor:der3}
$$
\HH_n\Bigl(\frac{\cos x+\sin x}{(\cos x-\sin x)^3}\Bigr)
=4^{\binom n2}\Bigl(\prod_{k=1}^{n-1}k!\Bigr)\prod_{j=0}^{n-2}\bigl[(2j+3)(2j+4)\bigr]^{\,n-1-j} .
$$
\end{Corollary}

\subsection{The reciprocal-sine case $(1+x)\,x/\sin x$}

The function $f=(1+x)\,x/\sin x$ carries no free parameter to specialise, so
we record its evaluation directly. By Theorem~\ref{conj:xsin-closed}, with
$Q(k)$ of \eqref{eq:xsin-Q},
$$
\HH_{2m}(f)=\Bigl(\tfrac23\Bigr)^{m}\prod_{k=0}^{m-1}Q(2k),\qquad
\HH_{2m+1}(f)=\Bigl(\tfrac23\Bigr)^{m}\prod_{k=0}^{m-1}Q(2k+1),
$$
a product of factorials --- the deepest evaluation of the paper.

\subsection{An elliptic deformation of the Euler numbers}

The Jacobi-elliptic deformation $g_m=\dfrac{1+\sn(x,m)}{\cn(x,m)}$ of the Euler
generating function (with $g_0=(1+\sin x)/\cos x$) satisfies, by
Theorem~\ref{thm:ell}, $\HH_n(g_m)=(1-m)^{\binom n2}\,\HH_n(g_0)$ --- the Euler
value times a $\binom n2$-power of the complementary modulus; the
modulus $m=1$ is degenerate ($\HH_n=0$ for $n\ge2$). We record
$m=-1,2,\tfrac12$.

\begin{Corollary}[$m=-1$]\label{cor:ellm1}
For $g_{-1}=\dfrac{1+\sn(x,-1)}{\cn(x,-1)}$,
$$
\HH_n(g_{-1})=\prod_{k=1}^{n-1}k!\,(2k)! .
$$
\end{Corollary}

\begin{Corollary}[$m=2$]\label{cor:ell2}
For $g_{2}=\dfrac{1+\sn(x,2)}{\cn(x,2)}$,
$$
\HH_n(g_{2})=\Bigl(-\tfrac12\Bigr)^{\binom n2}\prod_{k=1}^{n-1}k!\,(2k)! .
$$
\end{Corollary}

\begin{Corollary}[$m=\tfrac12$]\label{cor:ellh}
For $g_{1/2}=\dfrac{1+\sn(x,\tfrac12)}{\cn(x,\tfrac12)}$,
$$
\HH_n(g_{1/2})=\Bigl(\tfrac14\Bigr)^{\binom n2}\prod_{k=1}^{n-1}k!\,(2k)! .
$$
\end{Corollary}

\subsection{An algebraic family $(1+x)/(1-x^2)^{s/2}$}

By Theorem~\ref{conj:alg} the family $f(x)=(1+x)/(1-x^2)^{s/2}$, with moments
$a_{2k}=\frac{(2k)!}{k!}(s/2)_k$, $a_{2k+1}=(2k+1)a_{2k}$, has a closed product
evaluation with linear factors $(s+2j-2)^{\,n-j}$.
(A scale parameter $t$, i.e.\ $(1-tx^2)^{-s/2}$, only multiplies $\HH_n$ by
$t^{\binom n2+\lfloor(n-1)^2/4\rfloor}$, so we take $t=1$.)

The determinant is nonzero for all odd $s$ --- including negative ones, where
$f=(1+x)(1-x^2)^{-s/2}$ is a genuine square-root series. For even $s\le0$, by
contrast, $f$ is a polynomial, the moments terminate, and $\HH_n$ vanishes for
$n\ge 2-\tfrac s2$ (the factor $s+2j-2$ being zero at $j=1-\tfrac s2$; thus
$\HH_n=0$ for $n\ge2$ at $s=0$ and for $n\ge3$ at $s=-2$). We record the cases
$s=-3,-1,1,2,3$.

\begin{Corollary}[$s=-3$]\label{cor:algm3}
For $f=(1+x)(1-x^2)^{3/2}$,
$$
\HH_n(f)=\Bigl(\prod_{k=1}^{n-1}(2k)!\Bigr)\prod_{j=1}^{n-1}(2j-5)^{\,n-j},
$$
nonzero, of sign $-1$ for all $n\ge2$.
\end{Corollary}

\begin{Corollary}[$s=-1$]\label{cor:algm1}
For $f=(1+x)\sqrt{1-x^2}$,
$$
\HH_n(f)=\Bigl(\prod_{k=1}^{n-1}(2k)!\Bigr)\prod_{j=1}^{n-1}(2j-3)^{\,n-j},
$$
nonzero, of sign $(-1)^{n-1}$.
\end{Corollary}
\begin{Corollary}[$s=1$]\label{cor:alg1}
For $f=\dfrac{1+x}{\sqrt{1-x^2}}$,
$$
\HH_n(f)=\Bigl(\prod_{k=1}^{n-1}(2k)!\Bigr)\prod_{j=1}^{n-1}(2j-1)^{\,n-j}.
$$
\end{Corollary}

\begin{Corollary}[$s=2$]\label{cor:alg2}
For $f=\dfrac{1+x}{1-x^2}$ (so that $a_n=n!$),
$$
\HH_n(f)=\Bigl(\prod_{k=1}^{n-1}(2k)!\Bigr)\prod_{j=1}^{n-1}(2j)^{\,n-j}
       =2^{\binom n2}\prod_{k=1}^{n-1}(2k)!\,k! .
$$
\end{Corollary}

\begin{Corollary}[$s=3$]\label{cor:alg3}
For $f=\dfrac{1+x}{(1-x^2)^{3/2}}$,
$$
\HH_n(f)=\Bigl(\prod_{k=1}^{n-1}(2k)!\Bigr)\prod_{j=1}^{n-1}(2j+1)^{\,n-j}.
$$
\end{Corollary}

\subsection{The squared algebraic family $(1+x)^{2}/(1-x^2)^{s/2}$}

By Theorem~\ref{conj:algsq} the family $f(x)=(1+x)^{2}/(1-x^2)^{s/2}$ has, for
$s\ne3$, a closed product evaluation whose last factor is a monic polynomial
in $s$ of degree $n-1$ carrying the denominator $s-3$; the excluded value
$s=3$ is where its numerator vanishes (Remark~\ref{rem:algsq-h3}). We record
$s=2,4$.

\begin{Corollary}[$s=2$]\label{cor:calgsq2}
For $f=\dfrac{(1+x)^{2}}{1-x^2}=\dfrac{1+x}{1-x}$ (so $a_0=1$ and $a_m=2\,m!$ for
$m\ge1$), the last factor is $2^{\,n}(n-1)!-(2n-1)!!$ and
$$
\HH_n(f)=2^{\binom n2}\prod_{k=1}^{n-1}(2k)!\;\prod_{i=1}^{n-2}i!\;
\bigl(2^{\,n}(n-1)!-(2n-1)!!\bigr)=1,\,4,\,384,\,-39813120,\,\dots
$$
\end{Corollary}

\begin{Corollary}[$s=4$]\label{cor:calgsq4}
For $f=\dfrac{(1+x)^{2}}{(1-x^2)^{2}}$ the factor $(s-4)$ kills the deformation,
the last factor reducing to $(2n-1)!!$, and
$$
\HH_n(f)=2^{\binom n2}\prod_{k=1}^{n-1}(2k)!\;\prod_{i=1}^{n-2}(2)_i\;(2n-1)!!
=1,\,12,\,11520,\,2786918400,\,\dots
$$
\end{Corollary}

\subsection{The Bessel $(s,t)$ family $\mathrm{cosb}_s+\int_0^x\mathrm{cosb}_t$}

By Theorem~\ref{thm:besselst} the family
$f_{s,t}=\mathrm{cosb}_s(x)+\int_0^x\mathrm{cosb}_t(y)\,dy$ has the closed
evaluation \eqref{eq:bessel-closed}, whose sign and double product together
form the signed factor $\Omega(2(t-s))$ of the trigonometric family
(Section~\ref{ssec:stomega}). The determinant is a nonzero product exactly on
the half-integer offsets $t-s\in\tfrac12+\ZZ$
(Corollary~\ref{cor:besseldicho}); its double shift $\HH_n^{(2)}(f)=\HH_n(f'')$
is a closed multiple of $\HH_n(f)$ (Proposition~\ref{prop:besseldshift}),
where $f''$ is computed from the derivative ladder
$\mathrm{cosb}_\nu'=-\frac{x}{2\nu+2}\,\mathrm{cosb}_{\nu+1}$.
We record the two half-integer members with $s,t\in\{0,\tfrac12\}$.

\begin{Corollary}[$s=0,\ t=\tfrac12$]\label{cor:cbessel01}
Here $\mathrm{cosb}_0=J_0$ and
$\int_0^x\mathrm{cosb}_{1/2}(y)\,dy=\mathrm{Si}(x)$, the sine integral, so
$f=J_0(x)+\mathrm{Si}(x)$. Specialising \eqref{eq:bessel-closed} and converting
the half-integer Pochhammer symbols into factorials, every block merges into
a single product per parity:
$$
\begin{aligned}
\HH_{2m}(f)&=\frac{1}{2^{\binom{2m}2}}\,
\prod_{k=0}^{m-1}\frac{\bigl((2k)!\bigr)^{2}\,(2m+2k)!}{(4m+2k-1)!}\,,\\
\HH_{2m+1}(f)&=\frac{(2m)!\,(4m)!}{2^{\,m(2m+5)}\;m!\,(3m)!}\,
\prod_{k=0}^{m-1}\frac{\bigl((2k)!\bigr)^{2}\,(2m+2k)!}{(4m+2k+1)!}\,.
\end{aligned}
$$
For the double shift,
$$
f''=\tfrac{x^2}{8}\,\mathrm{cosb}_2-\tfrac12\,\mathrm{cosb}_1
-\tfrac{x}{3}\,\mathrm{cosb}_{3/2},\qquad
\HH_n^{(2)}(f)=\HH_n(f'')
=(-1)^n\,\frac{3\,\bigl((2n)!\bigr)^{2}}{2^{\,n+1}\,n!\,(3n)!}\;\HH_n(f).
$$
\end{Corollary}

\begin{Corollary}[$s=\tfrac12,\ t=0$]\label{cor:cbessel10}
Here $f=\dfrac{\sin x}{x}+\int_0^x J_0(y)\,dy$. The even-order determinants
coincide with those of Corollary~\ref{cor:cbessel01} up to sign,
$$
\HH_{2m}(f)=(-1)^m\,\HH_{2m}\bigl(J_0+\mathrm{Si}\bigr)
=\frac{(-1)^m}{2^{\binom{2m}2}}\,
\prod_{k=0}^{m-1}\frac{\bigl((2k)!\bigr)^{2}\,(2m+2k)!}{(4m+2k-1)!}\,,
$$
while at odd order $n=2m+1$
$$
\HH_{2m+1}(f)=(-1)^m\,(2m+1)\,2^{\,m(3-2m)}\,
\frac{m!\,(3m)!\,(4m)!}{(6m+1)!}\,
\prod_{k=0}^{m-1}\frac{\bigl((2k)!\bigr)^{2}\,(2m+2k)!}{(4m+2k+1)!}\,.
$$
For the double shift,
$$
f''=\tfrac{x^2}{15}\,\mathrm{cosb}_{5/2}-\tfrac13\,\mathrm{cosb}_{3/2}
-\tfrac{x}{2}\,\mathrm{cosb}_1,\qquad
\HH_n^{(2)}(f)=\HH_n(f'')
=(-1)^n\,\frac{\lfloor 3n/2\rfloor\,\bigl((2n)!\bigr)^{2}}{2^{\,n}\;n\;n!\,(3n)!}\;\HH_n(f).
$$
\end{Corollary}

\subsection{The multiplicative Bessel family $(1+x)\,\mathrm{cosb}_\nu^{\,2}$}

By Theorem~\ref{thm:xbesseleven} the \emph{even}-order determinants of
$f_\nu=(1+x)\,\mathrm{cosb}_\nu^{\,2}$ factor completely into linear forms
over $\QQ$, by the closed evaluation \eqref{eq:xbe-closed}; at odd orders no
product form exists. We record $\nu=0,\tfrac12,1$,
converting all half-integer Pochhammer symbols into factorials.

\begin{Corollary}[$\nu=0$]\label{cor:cxbe0}
Here $\mathrm{cosb}_0=J_0$, and
$$
f=(1+x)\,J_0(x)^2,
$$
$$
\HH_{2m}(f)=\frac{(-1)^m}{2^{\,m(6m-5)}}\;
\prod_{k=0}^{m-1}\Bigl(\frac{(2k)!}{k!}\Bigr)^{\!4}\;
\prod_{i=0}^{2m-1}\frac{\bigl((2i)!\bigr)^{2}}{i!\,\bigl((m-1+i)!\bigr)^{2}}\,.
$$
\end{Corollary}

\begin{Corollary}[$\nu=\tfrac12$]\label{cor:cxbeh}
Here $\mathrm{cosb}_{1/2}=\sin x/x$, and
$$
f=(1+x)\,\frac{\sin^2x}{x^2},
$$
$$
\HH_{2m}(f)=(-1)^m\,2^{\,m(8m-5)}\;
\prod_{k=0}^{m-1}(2k)!\,(2k+1)!\;
\prod_{i=0}^{2m-1}\frac{(2i)!\;i!\,(m-1+i)!}{(2m+2i-1)!\,(m+i)!}\,.
$$
\end{Corollary}

\begin{Corollary}[$\nu=1$]\label{cor:cxbe1}
Here $\mathrm{cosb}_1=2J_1(x)/x$, and
$$
f=4\,(1+x)\,\frac{J_1(x)^2}{x^{2}},
$$
$$
\HH_{2m}(f)=\frac{(-1)^m}{2^{\,m(6m-5)}\,3^{\,m}}\;
\prod_{k=0}^{m-1}\frac{(2k)!\,\bigl((2k+1)!\bigr)^{2}\,(2k+4)!}{(k!)^{3}\,(k+2)!}\;
\prod_{i=0}^{2m-1}\frac{(2i)!\,(2i+1)!}{i!\,(m+i)!\,(m+i+1)!}\,.
$$
\end{Corollary}

\section{The Lindström--Gessel--Viennot approach}\label{sec:lgv}

The determinant $\HH_n$ counts something, and the natural first attempt at a
combinatorial evaluation is the Lindström--Gessel--Viennot (LGV) lemma
\cite{Lindstrom1973,Gessel1985Viennot,Karlin1959McGregor}, which turns a
determinant of path counts into a signed sum over non-intersecting path families.
We recall the lemma, use it to see \emph{why} the classical Hankel determinant is
a clean product, and then read off the single change that the dilation makes. In
the end this route did not deliver our evaluations --- the algebraic methods
$\M1$--$\M6$ did --- but it locates precisely where the difficulty enters, and we
record it for orientation.

\subsection{Hankel minors as path systems}\label{ssec:lgv-minors}

It is convenient to write, for finite index sets $R=\{r_0<\dots<r_{n-1}\}$ and
$C=\{c_0<\dots<c_{n-1}\}$,
$$
H(R;C):=\det\bigl(a_{r_i+c_j}\bigr)_{0\le i,j\le n-1},
$$
the corresponding minor of the infinite Hankel matrix $(a_{p+q})$. In this
notation the classical and dilated determinants are
$$
H_n=H\bigl(\{0,1,\dots,n-1\};\{0,1,\dots,n-1\}\bigr),
\qquad
\HH_n=H\bigl(\{0,2,\dots,2n-2\};\{0,1,\dots,n-1\}\bigr):
$$
they share the column set $\{0,1,\dots,n-1\}$ and differ only in that the row
(``top'') indices are \emph{dilated} from $0,1,\dots,n-1$ to $0,2,\dots,2n-2$ ---
exactly the even-row selection of Definition~\ref{def:double}.

\begin{Lemma}[Lindström--Gessel--Viennot
\cite{Lindstrom1973,Gessel1985Viennot}]\label{lem:lgv}
Let $\mathcal D$ be a locally finite acyclic directed graph with edge weights in a
commutative ring, and for vertices $u,v$ let
$e(u,v)=\sum_{P\colon u\to v}\ \prod_{\text{edges of }P}w$ be the weight
generating function of directed paths. For source and sink tuples
$A=(A_0,\dots,A_{n-1})$ and $B=(B_0,\dots,B_{n-1})$,
$$
\det\bigl(e(A_i,B_j)\bigr)_{0\le i,j\le n-1}
=\sum_{(P_0,\dots,P_{n-1})}\sgn(\sigma)\prod_{k=0}^{n-1}w(P_k),
$$
the sum over families in which $P_k$ runs from $A_k$ to $B_{\sigma(k)}$ for some
permutation $\sigma$ and the $P_k$ are pairwise vertex-disjoint. If the endpoints
are \emph{nonpermutable} --- every vertex-disjoint family has $\sigma=\mathrm{id}$
--- the determinant is the subtraction-free weight of the non-intersecting
families joining $A_k$ to $B_k$.
\end{Lemma}

Now realise the moments as paths. If $(a_k)$ are the moments of a quasi-definite
functional with $J$-fraction coefficients $(b_h,\lambda_h)$ --- the
$(c_n,\lambda_n)$ of Section~\ref{ssec:cf}, renamed here because $c$ denotes
column indices ---
then by Flajolet's theory \cite{Flajolet1980,Viennot1983} $a_k$ is the weight of
all Motzkin paths of length $k$ from height $0$ to height $0$, a level step at
height $h$ carrying weight $b_h$ and a down-step from height $h$ weight
$\lambda_h$. Place a source $A_i$ on the horizontal axis at abscissa $-r_i$ and a
sink $B_j$ at abscissa $c_j$; a directed path $A_i\to B_j$ is then a Motzkin path
of length $r_i+c_j$ that stays weakly above the axis, so $e(A_i,B_j)=a_{r_i+c_j}$
and
$$
H(R;C)=\det\bigl(e(A_i,B_j)\bigr)_{0\le i,j\le n-1}.
$$
Since all sources lie at abscissa $\le0$ and all sinks at $\ge0$, the endpoints
are nonpermutable: the only vertex-disjoint families are the nested ones joining
$A_i$ to $B_i$ (Figures~\ref{fig:lgvclassic} and~\ref{fig:lgvdilated}). Hence
\emph{every} such minor $H(R;C)$ is a
subtraction-free count of non-intersecting Motzkin paths.

\subsection{What the dilation changes}\label{ssec:lgv-dilation}

For the classical $H_n$ (Figure~\ref{fig:lgvclassic}) the two index sets coincide,
$R=C=\{0,1,\dots,n-1\}$: the sources and sinks are equally spaced and symmetric
about the origin, the nested family fills every level in turn, and its weight
telescopes into the Heilermann product
$H_n=\lambda_1^{\,n-1}\lambda_2^{\,n-2}\cdots\lambda_{n-1}$
\cite{Heilermann1846,Viennot1983} (here $a_0=1$). This is the combinatorial face
of the $J$-fraction, and it is what makes $H_n$ ``simple''.

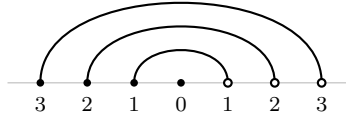
\begin{figure}[ht]
\centering
\begin{tikzpicture}[scale=0.62,line join=round]
\draw[gray!55] (-3.7,0)--(3.7,0);
\draw[thick] (-1,0) .. controls (-1,0.93) and (1,0.93) .. (1,0);
\draw[thick] (-2,0) .. controls (-2,1.60) and (2,1.60) .. (2,0);
\draw[thick] (-3,0) .. controls (-3,2.27) and (3,2.27) .. (3,0);
\foreach \x in {0,-1,-2,-3}{\fill (\x,0) circle(2.3pt);}
\foreach \x in {1,2,3}{\filldraw[fill=white,thick] (\x,0) circle(2.3pt);}
\foreach \x/\t in {0/0,-1/1,-2/2,-3/3}{\node[below=1.5pt] at (\x,0){\scriptsize$\t$};}
\foreach \x in {1,2,3}{\node[below=1.5pt] at (\x,0){\scriptsize$\x$};}
\end{tikzpicture}
\caption{The classical $H_n=H(\{0,1,2,3\};\{0,1,2,3\})$: equally spaced, symmetric
nested arches.}
\label{fig:lgvclassic}
\end{figure}

The dilated determinant is the \emph{same} picture with the sources pulled
apart. In $\HH_n$ the top starting points move from $0,1,\dots,n-1$ to
$0,2,\dots,2n-2$ while the sinks are unchanged
(Figure~\ref{fig:lgvdilated}). The family is still forced and non-intersecting, so
$\HH_n$ too is a subtraction-free count; but the source and sink spacings no
longer match ($2$ against $1$). The nested paths must now span a growing number of
levels between consecutive sinks, the clean level-by-level telescoping that
produced the $J$-fraction breaks down, and the lemma delivers only an unwieldy
positive sum --- not a product. This is the combinatorial shadow of the analytic
obstruction of Section~\ref{sec:intro}: interleaving the even and odd moments
destroys the three-term recurrence on which the $J$-fraction rests. It is exactly
here that the single-graph LGV picture stalls.

\begin{figure}[ht]
\centering
\begin{tikzpicture}[scale=0.62,line join=round]
\draw[gray!55] (-6.7,0)--(3.7,0);
\draw[thick] (-2,0) .. controls (-2,1.00) and (1,1.00) .. (1,0);
\draw[thick] (-4,0) .. controls (-4,1.67) and (2,1.67) .. (2,0);
\draw[thick] (-6,0) .. controls (-6,2.33) and (3,2.33) .. (3,0);
\foreach \x in {0,-2,-4,-6}{\fill (\x,0) circle(2.3pt);}
\foreach \x in {1,2,3}{\filldraw[fill=white,thick] (\x,0) circle(2.3pt);}
\foreach \x/\t in {0/0,-2/2,-4/4,-6/6}{\node[below=1.5pt] at (\x,0){\scriptsize$\t$};}
\foreach \x in {1,2,3}{\node[below=1.5pt] at (\x,0){\scriptsize$\x$};}
\end{tikzpicture}
\caption{The dilated $\HH_n=H(\{0,2,4,6\};\{0,1,2,3\})$: rows dilated, arches
skewed, telescoping lost.}
\label{fig:lgvdilated}
\end{figure}
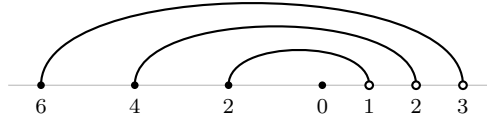

\subsection{The even--odd splitting: the biorthogonal picture}
\label{ssec:lgv-split}

There is nonetheless a way to keep $\HH_n$ subtraction-free, by splitting it into
\emph{two} path systems --- one carried by the even moments, one by the odd
moments. This is the combinatorial face of the biorthogonal reduction~$\M2$. Write
$p=\lceil n/2\rceil$ and $q=\lfloor n/2\rfloor$. In
$\HH_n=(a_{2i+j})_{0\le i,j\le n-1}$ an even column $j=2c$ has entries
$a_{2i+2c}=a_{2(i+c)}$, involving only the even moments $a_{2k}$ (the moments
of the even functional $\mathcal S$); an odd column $j=2c+1$ has entries
$a_{2i+2c+1}=a_{2(i+c)+1}$, involving only the odd moments $a_{2k+1}$ (the
moments of the odd functional $\mathcal T$). Move the $p$ even columns to the
front and the $q$ odd columns to the back; the shuffle has sign
$(-1)^{\binom p2}$ and leaves the block matrix
$[\,M_{\mathcal S}\mid M_{\mathcal T}\,]$ with
$$
M_{\mathcal S}=\bigl(a_{2(i+c)}\bigr)_{0\le i\le n-1,\,0\le c\le p-1}\ (n\times p),
\qquad
M_{\mathcal T}=\bigl(a_{2(i+c)+1}\bigr)_{0\le i\le n-1,\,0\le c\le q-1}\ (n\times q).
$$
Laplace expansion along the first $p$ columns carries the sign
$(-1)^{\sum_{i\in I}i+\binom p2}$ on the row subset $I$; the $\binom p2$
cancels against the shuffle, leaving the subtraction-free-looking
$$
\HH_n=\sum_{\substack{I\subseteq\{0,\dots,n-1\}\\ |I|=p}}
(-1)^{\sum_{i\in I}i}\,
\det\bigl(a_{2(i+c)}\bigr)_{i\in I,\,0\le c\le p-1}\ \cdot\
\det\bigl(a_{2(i+c)+1}\bigr)_{i\in I^{c},\,0\le c\le q-1},
$$
the sum over all $p$-subsets $I$ of the rows, with complementary $q$-subset $I^c$.
Each summand is a product of a $p\times p$ minor of the even Hankel matrix
$(a_{2k+2l})_{k,l}$ and a $q\times q$ minor of the odd Hankel matrix
$(a_{2k+2l+1})_{k,l}$; in the $H(R;C)$ notation above, and writing $H_{\mathcal S}$,
$H_{\mathcal T}$ for Hankel minors of the even and odd moment sequences,
$$
\HH_n=\sum_{I}(-1)^{\sum_{i\in I}i}\,
H_{\mathcal S}\bigl(I;\{0,\dots,p-1\}\bigr)\,
H_{\mathcal T}\bigl(I^{c};\{0,\dots,q-1\}\bigr).
$$
Combinatorially the two factors are non-intersecting Motzkin-path families in the
two \emph{separate} graphs carrying the $J$-fractions of $\mathcal S$ and of
$\mathcal T$ --- each of $(a_{2k})_k$ and $(a_{2k+1})_k$ being again a moment
sequence. Drawing the even family above the axis and the odd family mirrored below
it exhibits a whole term at once: the rows are two-coloured, one colour arching up
and the complementary colour arching down, sharing the horizontal axis
(Figure~\ref{fig:lgvsplit}).
The dilated determinant is thus a signed sum of products of two
classical Hankel minors, one from each parity. The price is that the sum ranges
over all $\binom np$ splittings of the rows between the two systems, and nothing
in the picture explains why it should collapse to a product. That collapse is
captured algebraically by the biorthogonal reduction~$\M2$, which
biorthogonalises against $\mathcal S$ and $\mathcal T$ simultaneously; it is
exactly here that we leave the LGV picture for the methods $\M1$--$\M6$.

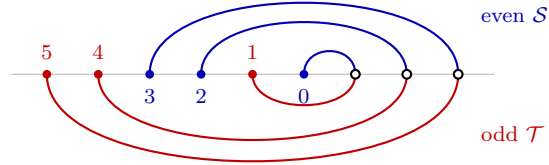
\begin{figure}[ht]
\centering
\begin{tikzpicture}[scale=0.68,line join=round]
  \draw[gray!55] (-5.7,0)--(3.7,0);
  \draw[thick,blue!70!black] (0,0)  .. controls (0,0.60)  and (1,0.60) .. (1,0);
  \draw[thick,blue!70!black] (-2,0) .. controls (-2,1.35) and (2,1.35) .. (2,0);
  \draw[thick,blue!70!black] (-3,0) .. controls (-3,1.85) and (3,1.85) .. (3,0);
  \draw[thick,red!75!black]  (-1,0) .. controls (-1,-0.80) and (1,-0.80) .. (1,0);
  \draw[thick,red!75!black]  (-4,0) .. controls (-4,-1.70) and (2,-1.70) .. (2,0);
  \draw[thick,red!75!black]  (-5,0) .. controls (-5,-2.25) and (3,-2.25) .. (3,0);
  \fill[blue!70!black] (0,0)  circle(2.4pt);
  \fill[red!75!black]  (-1,0) circle(2.4pt);
  \fill[blue!70!black] (-2,0) circle(2.4pt);
  \fill[blue!70!black] (-3,0) circle(2.4pt);
  \fill[red!75!black]  (-4,0) circle(2.4pt);
  \fill[red!75!black]  (-5,0) circle(2.4pt);
  \filldraw[fill=white,thick] (1,0) circle(2.4pt);
  \filldraw[fill=white,thick] (2,0) circle(2.4pt);
  \filldraw[fill=white,thick] (3,0) circle(2.4pt);
  \node[blue!70!black,below=1.5pt] at (0,-0.02) {\scriptsize$0$};
  \node[red!75!black, above=1.5pt] at (-1,0.02) {\scriptsize$1$};
  \node[blue!70!black,below=1.5pt] at (-2,-0.02){\scriptsize$2$};
  \node[blue!70!black,below=1.5pt] at (-3,-0.02){\scriptsize$3$};
  \node[red!75!black, above=1.5pt] at (-4,0.02) {\scriptsize$4$};
  \node[red!75!black, above=1.5pt] at (-5,0.02) {\scriptsize$5$};
  \node[blue!70!black,right] at (3.25,1.15) {\scriptsize even $\mathcal S$};
  \node[red!75!black, right] at (3.25,-1.15){\scriptsize odd $\mathcal T$};
\end{tikzpicture}
\caption{The even--odd splitting of $\HH_n$ ($n=6$, generic): even system
$I=\{0,2,3\}$ above, odd system $I^{c}=\{1,4,5\}$ mirrored below.}
\label{fig:lgvsplit}
\end{figure}

\section{Classical Hankel determinants}\label{sec:hankel-classical}

Alongside the dilated determinant $\HH_n=\det(a_{2i+j})_{0\le i,j\le n-1}$, each
moment sequence produced in this paper also has a \emph{classical} Hankel
determinant
$$
H_n(\mu)=\det\bigl(\mu_{i+j}\bigr)_{0\le i,j\le n-1}.
$$
The evaluation of such determinants is a classical and much-studied subject;
for surveys and general techniques see \cite{Krattenthaler1998,Krattenthaler2005},
and for recent product and near-product evaluations of Hankel determinants of
combinatorial moment sequences see
\cite{Cigler2021Krattenthaler,Krattenthaler2023,Elouafi2015,Egecioglu2008RR}.
Whenever the Jacobi continued fraction of $\sum_k\mu_k t^k$ has been exhibited,
this determinant is immediate: by Heilermann's formula~\eqref{eq:HfromLambda}
\cite{Heilermann1846,Wall1948},
\begin{equation}\label{eq:heilermann57}
H_n(\mu)=\mu_0^{\,n}\prod_{k=1}^{n-1}\lambda_k^{\,n-k},
\end{equation}
the $\lambda_k$ being the denominator coefficients of the J-fraction. We record
here, without proof, the resulting products. Several are classical or already
in the literature, and are included only for completeness; we flag these
below with their sources. The Bessel determinant is the classical Hankel
determinant of Jacobi (Beta) moments \cite{Krattenthaler1998}, and the Springer
determinant (the case $t=r=1$) is known \cite{Barry2012,Sokal2019} (sequence
\texttt{A091804} in \cite{OEIS}). The remaining ones --- for $(2k+1)E_{2k}$, for
the two reciprocal-sine sequences, and for the two-parameter family
$1/(\cos x-t\sin x)^r$ --- appear not to have been recorded in this form. We
begin with the sequence of Remark~\ref{rem:jfraccdh}.

\medskip
\emph{The sequence $(2k+1)E_{2k}$.}
Recall the secant numbers $E_{2k}$, the Taylor coefficients of the secant,
$$
\frac{1}{\cos x}=\sum_{k\ge0}E_{2k}\,\frac{x^{2k}}{(2k)!},
\qquad
E_0,E_2,E_4,E_6,\dots=1,\,1,\,5,\,61,\,1385,\dots,
$$
and form the sequence
$$
\mu_k:=(2k+1)\,E_{2k},
\qquad
\mu_0,\mu_1,\mu_2,\mu_3,\mu_4,\dots=1,\,3,\,25,\,427,\,12465,\dots
$$
By Remark~\ref{rem:jfraccdh} its generating function has the Jacobi continued
fraction
$$
\sum_{k\ge0}(2k+1)E_{2k}\,t^k
=\cfrac{1}{1-c^*_0t-\cfrac{\lambda^*_1t^2}{1-c^*_1t-\cfrac{\lambda^*_2t^2}{1-\cfrac{\lambda^*_3t^2}{1-\ddots}}}},
\qquad
c^*_m=8m^2+8m+3,\quad \lambda^*_m=(2m)^4,
$$
so that $\mu_0=1$ and, explicitly,
$$
c^*_0=3,\ c^*_1=19,\ c^*_2=51,\dots,
\qquad
\lambda^*_1=16,\ \lambda^*_2=256,\ \lambda^*_3=1296,\dots
$$

\begin{Proposition}\label{prop:hankel-secodd}
For the sequence $\mu_k=(2k+1)E_{2k}$ the classical Hankel determinant is,
for all $n\ge1$,
$$
H_n(\mu)=\det\bigl((2(i+j)+1)\,E_{2(i+j)}\bigr)_{0\le i,j\le n-1}
=\prod_{k=1}^{n-1}(2k)^{4(n-k)}
=2^{4\binom n2}\prod_{k=1}^{n-1}(k!)^4 .
$$
\end{Proposition}

This is \eqref{eq:heilermann57} with $\mu_0=1$ and $\lambda^*_k=(2k)^4$. The
first values are
$$
H_1,H_2,H_3,H_4,\dots=1,\,2^4,\,2^{16},\,2^{32}\,3^4,\dots
$$
The sequence $\mu_k$ is \texttt{A009843} in \cite{OEIS}, where this continued
fraction is also recorded; the Hankel determinants of the Euler numbers
themselves are evaluated in \cite{Han2020Euler}, but the determinant of
$(2k+1)E_{2k}$ above appears to be new.

\medskip
\emph{The sequence $(2k+1)b_k$ of the reciprocal sine.}
Recall from Section~\ref{sec:xsinx} the even part $g(x)=x/\sin x$ of the
reciprocal-sine function and its scaled Taylor coefficients $b_k$,
$$
\frac{x}{\sin x}=\sum_{k\ge0}b_k\,\frac{x^{2k}}{(2k)!},
\qquad
b_0,b_1,b_2,b_3,\dots=1,\,\tfrac13,\,\tfrac{7}{15},\,\tfrac{31}{21},\dots,
$$
and form the sequence
$$
\mu_k:=(2k+1)\,b_k,
\qquad
\mu_0,\mu_1,\mu_2,\mu_3,\dots=1,\,1,\,\tfrac73,\,\tfrac{31}{3},\dots
$$
This is the odd moment functional $\mathcal T[y^k]=(2k+1)b_k$ of that
section. By Proposition~\ref{prop:xsin-wilson} its generating function has the
Jacobi continued fraction
$$
\sum_{k\ge0}(2k+1)b_k\,t^k
=\cfrac{1}{1-c^{\mathcal T}_0t-\cfrac{\lambda^{\mathcal T}_1t^2}
 {1-c^{\mathcal T}_1t-\cfrac{\lambda^{\mathcal T}_2t^2}{1-\ddots}}},
\qquad
c^{\mathcal T}_m=2m^2+2m+1,\quad
\lambda^{\mathcal T}_m=\frac{4m^6}{(2m-1)(2m+1)},
$$
so that $\mu_0=1$ and, explicitly,
$$
c^{\mathcal T}_0=1,\ c^{\mathcal T}_1=5,\ c^{\mathcal T}_2=13,\dots,
\qquad
\lambda^{\mathcal T}_1=\tfrac43,\ \lambda^{\mathcal T}_2=\tfrac{256}{15},\
\lambda^{\mathcal T}_3=\tfrac{2916}{35},\dots
$$

\begin{Proposition}\label{prop:hankel-sineodd}
For the sequence $\mu_k=(2k+1)b_k$ the classical Hankel determinant is,
for all $n\ge1$,
$$
H_n(\mu)=\det\bigl((2(i+j)+1)\,b_{i+j}\bigr)_{0\le i,j\le n-1}
=\prod_{k=1}^{n-1}\Bigl(\frac{4k^6}{(2k-1)(2k+1)}\Bigr)^{n-k}
=\prod_{k=0}^{n-1}\frac{16^{k}\,(k!)^{8}}{(2k)!\,(2k+1)!}\,.
$$
\end{Proposition}

This is \eqref{eq:heilermann57} with $\mu_0=1$ and
$\lambda^{\mathcal T}_k=\tfrac{4k^6}{(2k-1)(2k+1)}$; the second form is the
product $\prod_{k=0}^{n-1}h^{\mathcal T}_k$ of the norms~\eqref{eq:xsin-hT}. The
first values are
$$
H_1,H_2,H_3,H_4,\dots=1,\,\tfrac43,\,\tfrac{4096}{135},\,\tfrac{50331648}{875},\dots
$$

\medskip
\emph{The reciprocal sine $b_k$ itself.}
One may equally take the coefficients $b_k$ as the moment sequence --- this is
the even functional $\mathcal S[y^k]=b_k$ of Section~\ref{sec:xsinx}. Unlike the
sequences above, $\sum_k b_kt^k$ is given not by a Jacobi but by a Stieltjes
continued fraction: by Proposition~\ref{prop:xsin-even},
$$
\sum_{k\ge0}b_k\,t^k
=\cfrac{1}{1-\cfrac{u_1t}{1-\cfrac{u_2t}{1-\cfrac{u_3t}{1-\ddots}}}},
\qquad
u_n=\frac{n^4}{(2n-1)(2n+1)},
$$
with $\mu_0=b_0=1$ and, explicitly, $u_1=\tfrac13,\ u_2=\tfrac{16}{15},\
u_3=\tfrac{81}{35},\dots$ For an $S$-fraction Heilermann's formula takes the
telescoped form~\eqref{eq:Snorm}, $H_n=\prod_{i=0}^{n-1}\prod_{j=1}^{2i}u_j$.

\begin{Proposition}\label{prop:hankel-sineeven}
For the sequence $b_k=(2k)!\,[x^{2k}](x/\sin x)$ the classical Hankel
determinant is, for all $n\ge1$,
$$
H_n(b)=\det\bigl(b_{i+j}\bigr)_{0\le i,j\le n-1}
=\prod_{i=0}^{n-1}\prod_{j=1}^{2i}\frac{j^4}{(2j-1)(2j+1)}
=\prod_{l=0}^{n-1}\frac{16^{l}\,\bigl((2l)!\bigr)^{6}}{(4l)!\,(4l+1)!}\,.
$$
\end{Proposition}

The second form is the product $\prod_{l=0}^{n-1}h^{\mathcal S}_l$ of the
norms~\eqref{eq:xsin-hS}. The first values are
$$
H_1,H_2,H_3,H_4,\dots
=1,\,\tfrac{16}{45},\,\tfrac{65536}{55125},\,\tfrac{4294967296}{18883865},\dots
$$
The shifted sequence $\mu'_k=b_{k+1}$ has, by the odd companion
of~\eqref{eq:Snorm},
$$
\det\bigl(b_{i+j+1}\bigr)_{0\le i,j\le n-1}
=\prod_{l=0}^{n-1}\prod_{j=1}^{2l+1}u_j
=\prod_{l=0}^{n-1}h^{\mathcal S}_l\,u_{2l+1},
\qquad u_{2l+1}=\frac{(2l+1)^4}{(4l+1)(4l+3)},
$$
with first values $\tfrac13,\,\tfrac{48}{175},\,\tfrac{65536}{11319},\dots$

\medskip
\emph{The Bessel-cosine moments.}
For $s>-1$ recall from Section~\ref{sec:besselst} the Bessel cosine
$\mathrm{cosb}_s(x)=\Gamma(s+1)(2/x)^sJ_s(x)$ and its scaled Taylor
coefficients $\mu_k$,
$$
\mathrm{cosb}_s(x)=\sum_{k\ge0}\mu_k\,\frac{x^{2k}}{(2k)!},
\qquad
\mu_k=(-1)^k\,\frac{(\tfrac12)_k}{(s+1)_k},
$$
the even moment sequence $\mu_k=a_{2k}$ of the family~\eqref{eq:besselfam}. Its
ordinary generating function is the Gauss series
$\sum_{k\ge0}\mu_kz^k={}_2F_1(\tfrac12,1;s+1;-z)$, which by
Lemma~\ref{lem:bessel-data} has the Stieltjes continued fraction
$$
\sum_{k\ge0}\mu_k\,z^k
=\cfrac{1}{1-\cfrac{u_1z}{1-\cfrac{u_2z}{1-\cfrac{u_3z}{1-\ddots}}}},
\qquad
u_j=-\,\frac{j\,(j+2s-1)}{4\,(s+j-1)(s+j)},
$$
with, explicitly,
$u_1=-\tfrac{1}{2(s+1)},\
u_2=-\tfrac{2s+1}{2(s+1)(s+2)},\
u_3=-\tfrac{3(s+1)}{2(s+2)(s+3)},\dots$

\begin{Proposition}\label{prop:hankel-bessel}
For the Bessel-cosine moments $\mu_k=(-1)^k(\tfrac12)_k/(s+1)_k$ the classical
Hankel determinant is, for all $n\ge1$ and as an identity in $\QQ(s)$,
$$
H_n(\mu)=\det\bigl(\mu_{i+j}\bigr)_{0\le i,j\le n-1}
=\prod_{i=0}^{n-1}h^{S}_i
=\prod_{i=0}^{n-1}\frac{(2i)!\;(s+\tfrac12)_i}{4^{i}\,(s+i)_i\,(s+1)_{2i}}\,.
$$
\end{Proposition}

This is the $S$-fraction form of Heilermann's formula~\eqref{eq:Snorm} with
coefficients $u_j$; the product is that of the norms $h^S_i$
of~\eqref{eq:bessel-norms}. The odd part
$a_{2k+1}=(-1)^k(\tfrac12)_k/(t+1)_k$ of the same family gives the identical
statement with $t$ in place of $s$. The first values are, e.g.,
$$
H_1=1,\qquad H_2=\frac{2s+1}{4(s+1)^2(s+2)}\,.
$$
This --- and, below, its two shifts --- is the classical Hankel determinant of
Jacobi (Beta) moments \cite{Krattenthaler1998}: the entries $\mu_k=(-1)^k(\tfrac12)_k/(s+1)_k$
are the moments of a Beta weight, and multiplication by $y$ keeps it Beta.
The shifted sequence $\mu_{k+1}=a_{2k+2}$ gives, again by the odd companion
of~\eqref{eq:Snorm},
$$
\det\bigl(\mu_{i+j+1}\bigr)_{0\le i,j\le n-1}
=\prod_{i=0}^{n-1}h^{S}_i\,u_{2i+1},
\qquad
u_{2i+1}=-\,\frac{(2i+1)(s+i)}{2\,(s+2i)(s+2i+1)},
$$
e.g.\ $H^{(1)}_1=-\tfrac1{2(s+1)}$ and
$H^{(1)}_2=\dfrac{3\,(2s+1)}{16\,(s+1)^2(s+2)^2(s+3)}$.
The double shift is clean as well: the doubly shifted sequence is again of
Beta type,
$$
\mu_{k+2}=a_{2k+4}
=\frac{3}{4(s+1)(s+2)}\,(-1)^k\,\frac{(\tfrac52)_k}{(s+3)_k},
$$
so that
$$
\det\bigl(\mu_{i+j+2}\bigr)_{0\le i,j\le n-1}
=\Bigl(\frac{3}{4(s+1)(s+2)}\Bigr)^{n}
\prod_{k=0}^{n-1}\frac{k!\,(\tfrac52)_k\,(s+\tfrac12)_k}{(s+k+2)_k\,(s+3)_{2k}},
$$
e.g.\ $H^{(2)}_1=\dfrac{3}{4(s+1)(s+2)}$ and
$H^{(2)}_2=\dfrac{45\,(2s+1)}{64\,(s+1)^2(s+2)^2(s+3)^2(s+4)}$.

\medskip
\emph{The Springer number family $1/(\cos x-t\sin x)^{r}$.}
Section~\ref{sec:springer} evaluates the \emph{dilated} determinant of this
family; its coefficients form, at $t=r=1$, the Springer numbers
$$
\frac{1}{\cos x-\sin x}=\sum_{n\ge0}S_n\,\frac{x^n}{n!},
\qquad
S_n=1,\,1,\,3,\,11,\,57,\,361,\,2763,\dots,
$$
a moment sequence \cite{Sokal2019}. We record the \emph{classical} Hankel
determinant of the full moment sequence
$a_n=n!\,[x^n](\cos x-t\sin x)^{-r}$ of the whole family, whose ordinary
generating function has a particularly clean Jacobi continued fraction: the
underlying orthogonal polynomials are, up to rescaling, the
Meixner--Pollaczek polynomials.

\begin{Proposition}\label{prop:hankel-springer}
Let $r>0$. The ordinary generating function of
$a_n=n!\,[x^n](\cos x-t\sin x)^{-r}$ has the Jacobi continued fraction
\begin{equation}\label{eq:springer-jfrac}
\sum_{n\ge0}a_n\,z^n
=\cfrac{1}{1-c_0z-\cfrac{\lambda_1z^2}{1-c_1z-\cfrac{\lambda_2z^2}{1-\ddots}}},
\qquad
c_n=t\,(2n+r),\qquad
\lambda_n=(1+t^2)\,n\,(n+r-1)
\end{equation}
\textup(so $\mu_0=1$ and, at $t=r=1$, $c_n=2n+1$, $\lambda_n=2n^2$\textup), and
consequently, for all $n\ge1$,
$$
H_n=\det\bigl(a_{i+j}\bigr)_{0\le i,j\le n-1}
=(1+t^2)^{\binom n2}\prod_{i=1}^{n-1}i!\,(r)_i\,.
$$
\end{Proposition}

\begin{proof}
Both $c_n,\lambda_n$ and the moments $a_n$ are rational in $(t,r)$ --- indeed
$a_n\in\QQ[t]$ of degree $\le n$, by Fa\`a di Bruno as in the proof of
Theorem~\ref{thm:springer} --- so it suffices to prove
\eqref{eq:springer-jfrac} for real $t$ and real $r>0$, a set with nonempty
interior. Write $\phi=\arctan t$, so that
$\cos x-t\sin x=\sqrt{1+t^2}\,\cos(x+\phi)$ and
$$
f(x):=(\cos x-t\sin x)^{-r}=(1+t^2)^{-r/2}\,\sec^{r}(x+\phi).
$$
By the Fourier evaluation of the Meixner--Pollaczek weight
\cite[\S9.7]{Koekoek2010KLS},
$$
\frac{1}{2\pi}\int_{-\infty}^{\infty}e^{2\psi\eta}\,
\Bigl|\Gamma\bigl(\tfrac r2+i\eta\bigr)\Bigr|^{2}\,d\eta
=\Gamma(r)\,(2\cos\psi)^{-r}
\qquad\bigl(|\psi|<\tfrac\pi2\bigr),
$$
so the probability measure
$$
d\mu(\xi)=\frac1N\,e^{\xi\arctan t}\,
\Bigl|\Gamma\bigl(\tfrac r2+\tfrac{i\xi}2\bigr)\Bigr|^{2}\,d\xi,
\qquad N=4\pi\,\Gamma(r)\,2^{-r}(1+t^2)^{r/2},
$$
has moment generating function $\int_{\RR}e^{x\xi}\,d\mu(\xi)=f(x)$ for
$|x+\phi|<\tfrac\pi2$ (substitute $\xi=2\eta$, $\psi=x+\phi$, and use
$\cos(x+\phi)^{-r}$). Matching Taylor coefficients at $x=0$ gives
$a_n=\int_{\RR}\xi^{n}\,d\mu(\xi)$: the $a_n$ are the moments of $\mu$.

Under the affine change $\xi=2\eta$, $\mu$ is the Meixner--Pollaczek measure
with parameters $\lambda=\tfrac r2$ and angle $\varphi=\phi+\tfrac\pi2\in(0,\pi)$,
whose monic orthogonal polynomials satisfy
$\hat p_{n+1}=(\eta-\beta_n)\hat p_n-\gamma_n\hat p_{n-1}$ with
\cite[\S9.7]{Koekoek2010KLS}
$$
\beta_n=-(n+\lambda)\cot\varphi,\qquad
\gamma_n=\frac{n(n+2\lambda-1)}{4\sin^{2}\varphi}.
$$
Here $\cot\varphi=-\tan\phi=-t$ and $\sin^{2}\varphi=\cos^{2}\phi=1/(1+t^2)$,
and the doubling $\xi=2\eta$ multiplies $\beta_n$ by $2$ and $\gamma_n$ by $4$;
hence the monic $\mu$-recurrence has coefficients
$$
c_n=2\beta_n=t\,(2n+r),\qquad
\lambda_n=4\gamma_n=(1+t^2)\,n\,(n+r-1),
$$
which is \eqref{eq:springer-jfrac} (the theorem of Stieltjes,
\eqref{eq:Jfrac}). Finally Heilermann's formula~\eqref{eq:heilermann57} with
$\mu_0=a_0=1$ and $h_i=\prod_{k=1}^{i}\lambda_k=(1+t^2)^{i}\,i!\,(r)_i$ gives
$H_n=\prod_{i=0}^{n-1}h_i=(1+t^2)^{\binom n2}\prod_{i=1}^{n-1}i!\,(r)_i$.
\end{proof}

The product $\prod_{i=1}^{n-1}i!\,(r)_i$ is itself the classical Hankel
determinant of $1/(1-x)^r$ (moments $(r)_n$), so the evaluation mirrors
Theorem~\ref{thm:springer} exactly, with the \emph{dilated} factor
$\bigl(t(t^2+1)\bigr)^{\binom n2}$ there replaced by the \emph{classical}
$(1+t^2)^{\binom n2}$ --- the extra $t$ dropping out because the classical
determinant depends only on the $\lambda_n$, not on the $c_n$. At $t=r=1$ it
gives the Springer Hankel determinant
$$
H_n=2^{\binom n2}\prod_{i=1}^{n-1}(i!)^2
=1,\,2,\,32,\,9216,\,84934656,\dots
$$
This last is known --- the Springer numbers are a Meixner--Pollaczek moment
sequence \cite{Barry2012,Sokal2019}, and the values $1,2,32,9216,\dots$ are
\texttt{A091804} in \cite{OEIS}. The two-parameter family
\eqref{eq:springer-jfrac}, by contrast, does not seem to have been recorded.

\section{Concluding remarks}\label{sec:concl}

We introduced the dilated Hankel determinant $\HH_n=\det(a_{2i+j})$
(Definition~\ref{def:double}) in order to study the Chapoton--Han root
conjecture (Section~\ref{sec:root}), and observed experimentally that for a
surprisingly large collection of classical sequences it factors into small
primes. Unlike the ordinary Hankel determinant, however, it admits no single
universal evaluation: the even and odd parts of $\mathbf a$ interact through a
connection determinant that collapses only under favourable circumstances. We
therefore developed the six methods $\M1$--$\M6$ of
Section~\ref{sec:methods} --- four running through the paper, and two more
specialised, the contiguous-relation method $\M5$ and the
rank-one/matrix-determinant-lemma method $\M6$ --- and with them evaluated,
each as an explicit product, the Euler number family and its
shifts, the Gaussian, secant, algebraic, Springer, elliptic and
derivative-ladder families, and the Beta family; the reciprocal-sine function
required, on top of $\M2$, the new
Catalan-determinant evaluation of Theorem~\ref{thm:xsin-catalan}.

What is still missing is a \emph{theory}. We have no criterion that decides, in
advance, whether a given generating function has a closed-form dilated Hankel
determinant, and the boundary is delicate: an innocuous modification of a
``nice'' function usually destroys the product form, and we do not know why one
perturbation succeeds where a neighbouring one fails. We record here the main
questions this leaves open.

\subsection{Further directions}

Four problems seem especially natural.

\begin{Problem}\label{prob:classify}
Find further sequences $\mathbf a$ --- or, more ambitiously, a criterion
characterising them --- whose dilated Hankel determinant $\HH_n$ has a closed
form, that is, factors into a product of linear factors in $n$ (equivalently,
into small primes for every $n$).
\end{Problem}

\begin{Problem}\label{prob:rstep}
The dilated Hankel determinant selects the rows $0,2,4,\dots$ of the infinite
Hankel matrix. For an integer $r\ge2$ one may equally form the \emph{$r$-step}
determinant $\det(a_{ri+j})$. For the Beta family this is again a closed
product for every step $r$: the Vandermonde reduction $\M1$ of
Section~\ref{sec:methods}, which factors $a_{x+j}=w(x)Q_j(x)$ with $Q_j$
triangular, applies verbatim to any set of row indices (see also the
determinant calculus of Krattenthaler
\cite{Krattenthaler1998,Krattenthaler2005}). The question is whether any family
\emph{beyond} the Beta family has a closed-form $r$-step determinant for
$r\ge3$.
\end{Problem}

\begin{Problem}\label{prob:lgv}
Give a combinatorial proof of our dilated Hankel determinant evaluations by
means of the Lindström--Gessel--Viennot lemma (Section~\ref{sec:lgv}). The
dilated Hankel determinants studied in this paper all have a simple closed
form --- a product of small integers. For the \emph{classical} Hankel
determinant, the Lindström--Gessel--Viennot lemma gives a nice explanation of
this fact: it realises the determinant as a positive count of families of
non-intersecting lattice paths, which visibly factors. Is there an analogous
combinatorial model that explains the product-of-small-integers property of the
\emph{dilated} Hankel determinant?
\end{Problem}

\begin{Problem}\label{prob:hfrac}
Our methods rest on orthogonal polynomials and the Jacobi continued fraction
($J$-fraction) of the even and odd moment sequences, which exists only when each
is quasi-definite. The \emph{Hankel continued fraction} ($H$-fraction)
\cite{Han2016Adv,Han2015NT} is more general: it exists for \emph{every}
sequence and reads off the Hankel determinants even when some vanish and the
$J$-fraction breaks down. When the even or odd part of $\mathbf a$ has no
$J$-fraction, can the dilated determinant still be evaluated by an $H$-fraction
analogue of the biorthogonal reduction $\M2$?
\end{Problem}

\subsection{Data availability}\label{ssec:data}

The closed forms were discovered and all the identities of this paper
were checked with the computer algebra system \textsf{SageMath}.
The \textsf{SageMath} programs and the output data used to discover and check
the results of this paper are available, for independent verification, in the
GitHub repository
$$\hbox{\tt https://github.com/GuoniuHan/dilated}.$$

\section*{Acknowledgement}

The author thanks Fr\'ed\'eric Chapoton for helpful discussions.

\section*{On the use of artificial intelligence}
An artificial-intelligence assistant based on a large language model was used to
help find and write some of the proofs. All arguments were verified independently
with \textsf{SageMath} and reviewed by the author, who takes full responsibility
for the correctness and the content of the paper.

\end{document}